\documentclass[11pt]{amsart}

\usepackage{geometry}
\usepackage[dvips]{graphicx}
\usepackage{amssymb,latexsym,amsmath}
\usepackage{amsfonts}
\usepackage{amsthm}
\usepackage{slashed}
\usepackage[all]{xy}
\usepackage{color}
\usepackage{slashed}
\usepackage{tikz}%
\usetikzlibrary{matrix,arrows}
\usepackage{float}
\usepackage{mathrsfs}
\usepackage{empheq}
\usepackage{multicol}

\usepackage{chngcntr}
\counterwithin{figure}{section}

\numberwithin{equation}{section}

\newcounter{introequation}
\newenvironment{introequation}{\refstepcounter{introequation}\equation}{\tag{\theintroequation}\endequation}

\usepackage{newfloat}
\DeclareFloatingEnvironment[
    fileext=los,
    listname=List of Schemes,
    name=Figure,
    placement=tbhp,
    within=section,
]{scheme}

\newtheorem{theorem}{Theorem}[section]
\newtheorem{THM}{Theorem}
\newtheorem{proposition}[theorem]{Proposition}
\newtheorem{lemma}[theorem]{Lemma}
\newtheorem{coro}[theorem]{Corollary}

\theoremstyle{definition}
\newtheorem{definition}[theorem]{Definition}
\newtheorem{remark}[theorem]{Remark}
\newtheorem{example}[theorem]{Example}

\newcounter{x}
\newtheorem{inneraxiom}{Axiom}
\newenvironment{axiom}[1]
  {\inneraxiom\stepcounter{x}}
  {\endinneraxiom}

\newcommand{\BB}[1]{\mathbb{#1}}
\newcommand{\CL}[1]{\mathcal{#1}}

\newcommand{\tr}{\text{tr}}
\newcommand{\Ad}{\text{Ad}}

\newcommand{\norm}[1]{\lVert #1 \rVert}
\newcommand{\bnorm}[1]{\bigg{\lVert} #1 \bigg{\rVert}}

\newcommand{\cris}[3]{\ensuremath{{\Gamma}_{#1#2}^{#3}}}

\newcommand{\inner}[2]{\langle #1,#2\rangle}

\newcommand{\sign}{\textnormal{sign}}

\newcommand{\ran}{\textnormal{ran}}
\newcommand{\ind}{\textnormal{ind}}
\newcommand{\Pf}{\text{Pf}}

\newcommand{\End}{\text{End}}
\newcommand{\Hom}{\text{Hom}}
\newcommand{\Aut}{\text{Aut}}

\newcommand{\Dom}{\text{Dom}}

\newcommand{\spec}{\textnormal{spec}}

\newcommand{\rk}{\text{rk}}

\newcommand{\vol}{\text{vol}}

\newcommand{\Sp}{\textnormal{Spec}}
\newcommand{\supp}{\textnormal{supp}}
\newcommand{\dom}{\textnormal{Dom}}

\usepackage{color}
\usepackage[pdftex,colorlinks=false,a4paper,bookmarks=true,bookmarksnumbered=true,
linkcolor=blue,citecolor=blue,citebordercolor={0 0 1}]{hyperref}
\usepackage{hyperref}

\begin{document}

\title{Induced Dirac-Schr\"odinger operators on $S^1$-semi-free quotients}
\author{Juan Camilo Orduz}
\address{Institut F\"ur Mathematik , Humboldt Universit\"at zu
Berlin, Germany.}
\email{juanitorduz@gmail.com}
\date{\today}
\maketitle

\begin{abstract}
John Lott has computed an integer-valued signature for the orbit space
of a compact orientable $(4k+1)$ manifold with a semi-free $S^1$-action,
which is a homotopy invariant of that space, but he did not construct
a Dirac type operator which has this signature as its index. In this
Thesis, we construct such operator on the orbit space, a Thom-Mather
stratified space with one singular stratum of positive dimension, and
we show that it is essentially unique and that its index coincides
with Lott's signature, at least when the stratified space satisfies
the so called Witt condition. We call this operator the
{\em induced Dirac-Schr\"odinger operator}. The strategy of the construction is
to ``push down" an appropriate $S^1$-invariant first order
transversally elliptic operator to the quotient space.\\
The Witt condition, a topological condition which in this case depends
on the codimension of the fixed point set, has various analytic
consequences.  In particular, when not satisfied, the Hodge-de Rham
operator on the quotient space does not need to be essentially
self-adjoint and therefore a choice of boundary conditions is
required. This choice freedom is not natural in view of the fact that
Lott's signature is well defined independently of the Witt condition.\\
The Dirac-Schr\"odinger operator constructed in this Thesis differs
from the Hodge-de Rham operator by a zero order term which ensures it
to be essentially self-adjoint. Moreover, this zero order term
anti-commutes with the chirality involution allowing the whole
operator to split so that the index can be computed even if the Witt
condition is not satisfied.
\end{abstract}

\pagenumbering{gobble} 

\newpage

\setcounter{tocdepth}{2}
\tableofcontents

\clearpage

\pagenumbering{roman}

\section*{Introduction}

\subsection*{Description of the problem}

Let $M$ be a $4k+1$ dimensional closed, connected and orientable smooth Riemannian manifold admitting an effective semi-free $S^1$-action by orientation preserving isometries. ``Semi-free" means that there are only two isotropy groups: the trivial group and the whole  $S^1$. We can decompose the manifold as $M=M_0\cup M^{S^1}$ where $M_0$ is the principal orbit space - an open and dense subset of $M$ on which the action is free - and $M^{S^1}$, the fixed point set, is a disjoint union of closed odd dimensional submanifolds of $M$. This decomposition is a  Whitney stratification (\cite[Section 2.7]{DK00}). Since the $S^1$-action on $M_0$ is free, the quotient space $M_0/S^1$ has a unique smooth structure such that the orbit map $\pi_{S^1}:M_0\longrightarrow M_0/S^1$ becomes a smooth fiber bundle. Moreover, we can induce a quotient metric by requiring $\pi_{S^1}$ to be a Riemannian submersion. In general, this metric will be incomplete. By studying the associated linear action, one can see that the quotient space close to a connected component $F$ of the singular stratum $M^{S^1}$ is diffeomorphic to the mapping cylinder $C(\mathcal{F})$ of a Riemannian fiber bundle $\pi_\mathcal{F}:\mathcal{F}\longrightarrow F$ with typical fiber $Y=\mathbb{C}P^N$ for some $N\in\mathbb{N}$ which depends on the codimension of $F$ in $M$ (see Figure \ref{Fig:Local}).\\
 
\begin{scheme}[h]
\begin{center}
\begin{tikzpicture}
\draw (-2,5)--(2,5);
\draw (-3,3)--(1,3);
\draw (-2,5)--(-3,3);
\draw (2,5)--(1,3);
\draw (-2.5,1)--(1.5,1);
\draw (-3,3)--(-2.5,1);
\draw (2,5)--(1.5,1);
\draw (1,3)--(1.5,1);
\draw [dashed](-2,5)--(-2.25,3);
\draw (-2.5,1)--(-2.25,3);
\draw (-1,3)--(-0.5,1);
\draw [dashed](0,5)--(-0.25,3);
\draw (-0.5,1)--(-0.25,3);
\draw (0,5)--(-1,3);
\node (a) at (-1,4) {$\mathbb{C}P^N$};
\node (b) at (-1.5,0.7) {$F$};
\node (b) at (-3,4) {$\mathcal{F}$};
\end{tikzpicture}
\caption{Mapping cylinder of the $\mathbb{C}P^N$-fiber bundle $\pi_{\mathcal{F}}:\mathcal{F}\longrightarrow F.$}\label{Fig:Local}
\end{center}
\end{scheme}

In this context, Lott defined in \cite{L00} a topological invariant of the action, the {\em equivariant $S^1$-signature} $\sigma_{S^1}(M)$, as follows: Let $V\in C^{\infty}(M,TM)$ denote the generating vector field of the $S^1$-action and consider the complex of basic differential forms with compact support $\Omega_{\text{bas},c}(M_0)\coloneqq \{\omega\in\Omega_c(M_0)\:|\:L_V\omega=\iota_V\omega=0\}$.
Here $L_V$ and $\iota_V$ denote the Lie derivative and the contraction with $V$, respectively. Let $H^{*}_{\text{bas},c}(M_0)$ be the associated cohomology groups with respect to the exterior derivative. Then, the {\em equivariant $S^1$-signature} is defined as the signature of the quadratic form
\begin{equation*}
\xymatrixrowsep{0.01cm}\xymatrixcolsep{2cm}\xymatrix{
H^{2k}_{\text{bas},c}(M_0)\otimes H^{2k}_{\text{bas},c}(M_0) \ar[r] & \mathbb{R}\\
[\omega_1]\otimes[\omega_2] \ar@{|->}[r]  &\displaystyle{\int_M \alpha\wedge \omega_1\wedge\omega_2,}
}
\end{equation*}
where $\alpha\in\Omega^1(M_0)$ satisfies $L_V\alpha=0$ and $\alpha(V)=1$. Lott showed that this signature does not depend on the metric on $M$ and that it is an invariant under orientation-preserving $S^1$-homotopy equivalences (\cite[Proposition 6]{L00}). Moreover, he proved the following remarkable formula (\cite[Theorem 4]{L00})
\begin{introequation}\label{Eqn:Lott}
\sigma_{S^1}(M)=\int_{M_0/S^1}L\left(T(M_0/S^1), g^{T(M_0/S^1)}\right)+\eta(M^{S^1}),
\end{introequation}
where $L\left(T(M_0/S^1), g^{T(M_0/S^1)}\right)$ is the $L$-polynomial of the curvature form of the tangent bundle $T(M_0/S^1)$ with respect to the quotient metric $g^{T(M_0/S^1)}$and $\eta(M^{S^1})$ is the eta invariant of the odd signature operator defined on the fixed point set. It is important to emphasize that part of the result is the convergence of the integral over $M_0/S^1$ (which is an open manifold if the action is not free). 
The main idea of the proof is to approximate each connected component of the fixed point set by a manifold with boundary in order to apply the Atiyah-Patodi-Singer signature theorem \cite[Theorem 4.14]{APSI} and then take the adiabatic limit, using the techniques developed in \cite{D91}, of the geometric and analytic quantities as they approach the singular stratum (see Figure \ref{Fig:Decomposition}). The proof also requires some equivariant methods, discussed for example in \cite{BGV}, to explain why the eta form of Bismut and Cheeger does not appear in the adiabatic formula (\cite{G00}) and also to compute the transgression term, which appears in the Atiyah-Patodi-Singer signature formula when the metric is not a product close to the boundary (\cite[Section 4.3]{G95}). \\

\begin{scheme}[h]
\begin{center}
\begin{tikzpicture}
\draw (-2,5)--(2,5);
\draw (-3,3)--(1,3);
\draw (-2,5)--(-3,3);
\draw (2,5)--(1,3);
\draw (-2.5,1)--(1.5,1);
\draw (-3,3)--(-2.5,1);
\draw (2,5)--(1.5,1);
\draw (1,3)--(1.5,1);
\draw [dashed](-2,5)--(-2.25,3);
\draw (-2.5,1)--(-2.25,3);
\draw (-1,3)--(-0.5,1);
\draw [dashed](0,5)--(-0.25,3);
\draw (-0.5,1)--(-0.25,3);
\draw (0,5)--(-1,3);
\draw (-2.4,2) arc (87:155:0.3 and 0.5);
\draw (1.6,2) arc (87:155:0.3 and 0.5);
\draw (-0.4,2) arc (87:155:0.3 and 0.5);
\draw (-2.7,1.7)--(1.33,1.7);
\draw (-2.4,2)--(1.62,2);
\node (a) at (-1,4) {$\mathbb{C}P^N$};
\node (b) at (0,0.7) {$F$};
\draw [<-](-2.84,1.68)--(-2.65,1);
\node (c) at (-3,1.3) {$r$};
\node (d) at (-2.5,0.8) {$0$};
\node (e) at (-3.2,3) {$1$};
\node (b) at (-3,4) {$\mathcal{F}$};
\end{tikzpicture}
\caption{Decomposition of $C(\mathcal{F})$ near the fixed point set $F$.}\label{Fig:Local2}
\end{center}
\end{scheme}

The question that arises naturally is whether there exists a Fredholm operator whose index computes $\sigma_{S^1}(M)$. This question was posed by Lott himself as a remark in his original work \cite[Section 4.2]{L00}. A natural candidate is the Hodge-de Rham operator $D_{M_0/S^1}
\coloneqq d_{M_0/S^1}+d^{\dagger}_{M_0/S^1}$ defined on the space of compactly supported differential forms $\Omega_c(M_0/S^1)$. Here $d^{\dagger}_{M_0/S^1}$ denotes the $L^2$ formal adjoint of the exterior derivative $d_{M_0/S^1}:\Omega^*_c(M_0/S^1)\longrightarrow \Omega^{*+1}_c(M_0)$.  As Fredholm operators are by definition closed,  we need to choose a closed extension of $D_{M_0/S^1}$. If the quotient metric is not complete, $D_{M_0/S^1}$ might have several closed extensions. In order to understand this phenomenon better we need to study the form of the operator close to the fixed point set. Following Br\"uning's work \cite{B09} one sees that $D_{M_0/S^1}$, close to $F\subset M^{S^1}$, is unitary equivalent to an operator of the form 
\begin{introequation}\label{Eqn:OpSing}
\Psi^{-1}{D}_{M_0/S^1}\Psi
=\gamma\left(\frac{\partial}{\partial r}+
\left(\begin{array}{cc}
I & 0\\
0 & -I
\end{array}\right)\otimes A(r)
\right).
\end{introequation}
Here $A(r)$ is an operator acting on differential forms $\Omega(\mathcal{F})$ of the total space of the fiber bundle $\mathcal{F}$ and $r>0$ is the radial coordinate. The operator $A(r)$ can be written as
\begin{introequation}\label{Eqn:A}
A(r)\coloneqq A_H(r)+\frac{1}{r}A_V,
\end{introequation}
where $A_H(r)$ is a first order horizontal operator, well defined for $r\geq 0$. The coefficient $A_V$ is a first order vertical operator, known as the {\em cone coefficient},  which can be itself written as $A_V=A_{0V}+\nu$, where $A_{0V}$ can be thought as the tangential signature operator of the fiber and $\nu\coloneqq \text{vd}-N$ where $\text{vd}$ is the vertical degree counting operator. Using the techniques developed in \cite{BS91}, Br\"uning showed in \cite[Section 4]{B09} that the operator \eqref{Eqn:OpSing} has a discrete self-adjoint extension. In addition, if the cone coefficient satisfies the spectral condition
\begin{introequation}\label{Eqn:AVgeg12}
|A_V|\geq\frac{1}{2}, 
\end{introequation} 
then the operator is in fact essentially self-adjoint. In the {\em Witt case}, i.e when there are no vertical harmonic forms in degree $N$, we can always achieve condition \eqref{Eqn:AVgeg12} by rescaling the vertical metric, which is an operation that preserves the index. To see this one needs to understand the spectrum of $A_V$. It was shown in \cite[Theorem 3.1]{B09} that the essential eigenvalues, those invariant under the rescaling, are the ones obtained when restricting to the space of vertical harmonic forms. These eigenvalues are explicitly given by $2j-N$, for $j=0,1,\cdots, N$. Observe that if $N$ is odd, zero does not appear as an essential eigenvalue and $|2j-N|\geq 1$. On the other hand, if $N=2\ell$ then zero appears as an eigenvalue when $j=\ell$ and the corresponding eigenspace is non-zero (the non-Witt case); therefore the rescaling procedure would fail to ensure condition \eqref{Eqn:AVgeg12}.\\

For the Witt case we apply the technques from the work of Br\"uning \cite[Section 5]{B09} to prove $\ind(D_{M_0/S^1}^+)=\sigma_{S^1}(M)$, where $D_{M_0/S^1}^+$ is the chiral Dirac operator with respect to the Clifford involution $\star_{M_0/S^1}$ (\cite[Proposition 3.58]{BGV}). The main tool for the index computation is the {\em Dirac-Schr\"odinger systems} formalism developed in \cite{BBC08}.  The strategy is to use the gluing index theorem \cite[Theorem 4.17]{BBC08} to compute the index as the sum of two contributions by approximating the singularity through a manifold with boundary, in the same spirit of Lott's proof of \eqref{Eqn:Lott}. This decomposition can be schematically visualized in Figure \ref{Fig:Decomposition}.

\begin{scheme}[H]
\begin{center}
\begin{tikzpicture}
\draw[rounded corners=32pt](7,-1)--(4,-1)--(2,-2)--(0,0) -- (2,2)--(4,1)--(7,1);
\draw (1.5,0.2) arc (175:315:1cm and 0.5cm);
\draw (3,-0.28) arc (-30:180:0.7cm and 0.3cm);
\draw (7.5,0) arc (0:360:0.5cm and 1cm);
\node (a) at (20:2.5) {$Z_t$};
\node (a) at (7,-1.5) {$\partial Z_t=\mathcal{F}_{t}$};
\draw (11,-0.6)--(11,0.6);
\draw (7,1)--(11,0.6);
\draw (7,-1)--(11,-0.6);
\node (a) at (9,0) {$U_t$};
\node (a) at (11.5,0) {$M^{S^1}$};
\end{tikzpicture}
\caption{Decomposition of $M/S^1$ as $M/S^1=Z_t\cup U_t$ where $Z_t$ is a compact manifold with boundary and $U_t=C(\mathcal{F}_t)$.}\label{Fig:Decomposition}
\end{center}
\end{scheme}

Imposing complementary APS boundary conditions on $Z_t$ and $U_t$ respectively, one obtains the following decomposition formula for the index of $D^+\coloneqq D^+_{M/S^1}$,
\begin{introequation}\label{Eqn:IndexDecompForm}
\textnormal{ind}({D}^+)=\textnormal{ind}\left({D}^{+}_{Z_t,Q_{<}({A}(t))(H)}\right)+\textnormal{ind}\left({D}^{+}_{U_t,Q_{\geq }({A}(t))(H)}\right).
\end{introequation}
One of the main ingredients of Br\"uning's  index computation is the fact that, for $t>0$ sufficiently small, the index contribution of $U_t$ vanishes. As a consequence, since the left hand side of \eqref{Eqn:IndexDecompForm} is independent of $t$, we can compute the index by taking the limit of $\textnormal{ind}\left({D}^{+}_{Z_t,Q_{<}({A}(t))(H)}\right)$ as $t\longrightarrow 0^+$ using the techniques of \cite{D91}. One final subtle point in the derivation of the index formula is the need, in order to compute a certain Kato index (\cite[Theorem 5.4]{B09}), of a remarkable result of Cheeger and Dai, presented in \cite{CD09}, relating the $\tau$ invariant of the fibration $\mathcal{F}$ and the $L^2$-signature of the associated generalized Thom space for the Witt case. 

\subsection*{Overview of the Thesis work}

Up to this point the picture looks incomplete as Lott's geometric proof of \eqref{Eqn:Lott} works without any distinction on the parity of $N$. In contrast, for the analytical counterpart one needs to distinguish between  the Witt and the non-Witt case since in the latter we are forced to impose boundary conditions. This motivates the following question: Does there exist an essentially self-adjoint operator on $M/S^1$, independent of the codimension of the fixed point set in $M$, whose index is precisely the $S^1$-signature?
Heuristically, should this operator exist, we hope it would be a perturbation of the Hodge-de Rham operator. That is, it would differ by a zero order potential whose contribution close to the singular stratum would be of the form $1/r$, in view of  the form of operator \eqref{Eqn:A}, and such that it would push the spectrum of the corresponding cone coefficient outside the open interval $(-1/2,1/2)$.
This would automatically make it essentially self-adjoint as the analogous condition \eqref{Eqn:AVgeg12} of the perturbed cone coefficient would be satisfied.\\

Hope for the existence of such operator relies on the fundamental work \cite{BH78} of Br\"uning and Heintze,  where the authors develop a machinery to ``push-down" self-adjoint operators to quotients of compact Lie group actions. The key observation of their formalism is that, whenever a self-adjoint operator commuting with the group action is restricted to the space of invariant sections, it remains self-adjoint in the restricted domain. Once this result was established, Br\"uning and Heintze constructed a unitary map $\Phi$ between the  space of square integrable invariant sections on the principal orbit and the space of square integrable sections of certain vector bundle defined on the quotient space. Altogether, the ``push-down" procedure for a self-adjoint operator consists of two steps: First, restrict its domain to the space of invariant sections. Second, compose this resulting domain with the unitary map $\Phi$.\\

This construction seems appropriate for our case of interest because all geometric differential operators on $M$, defined on smooth forms, are essentially self-adjoint since $M$ is closed. The next question is to determine which operator to choose in order to apply 
Br\"uning and Heintze's construction.  Two natural candidates are the Hodge-de Rham operator and the odd signature operator. Implementing the procedure described above for these two operators, one obtains only partially satisfactory results. Concretely, the induced operators are indeed self-adjoint by construction, but the resulting potentials do not anti-commute with $\star_{M_0/S^1}$. This is of course a problem since $\star_{M_0/S^1}$ is the natural involution which should split the desired push down operator in order to obtain the $S^1$-signature.\\
Despite these results, going back to the construction of \cite{BH78}, one can see that it is enough to push down a transversally elliptic operator in order to obtain an elliptic operator on the quotient. Using this observation, which enlarges the pool of candidates for the operator, and by analyzing the concrete form of the unitary transformation $\Phi$ defined by Br\"uning and Heintze, we are able to find an essentially self-adjoint $S^1$-invariant transversally elliptic operator whose induced push-down operator satisfies all the requirements described above. We call the resulting operator the {\em induced Dirac-Schr\"odinger operator}.\\

We now present the main steps of its construction, where the following geometric quantities are needed,
\begin{itemize}
\item $h:M_0/S^1\longrightarrow\mathbb{R}$ is the function that computes the volume of the orbit, i.e. $h(y)\coloneqq \vol(\pi^{-1}_{S^1}(y))$, where $\pi_{S^1}:M_0\longrightarrow M_0/S^1$ is the orbit map.
\item $\kappa\in\Omega^{1}(M_0)$ is the associated mean curvature $1$-form, given by $\kappa\coloneqq -d\log(\norm{V})$ where $V$ is the generating vector field of the $S^1$-action.
\item $\chi\in\Omega^1(M_0)$ is the characteristic $1$-form defined by the relation $\chi(V)\coloneqq\norm{V}$.
\item $\varphi_0\in\Omega^2(M_0)$  is the $2$-form defined by the equation  $\varphi_0\coloneqq d\chi+\kappa\wedge\chi$. 
\item $c:T^*M\longrightarrow \End(\wedge T^*M)$ denotes the Clifford multiplication, which is defined for an element $\alpha\in T^*M$ by $c(\alpha)\coloneqq \alpha\wedge-\iota_{\alpha^\sharp}$.
\item $\varepsilon\coloneqq (-1)^j$ on $j$-forms is the Gauss-Bonnet grading.
\end{itemize}
The forms $\kappa$ and $\varphi_0$ are basic, which is equivalent to saying that there exist unique forms $\bar{\kappa}$ and $\bar{\varphi}_0$ in $M_0/S^1$ such that $\kappa=\pi_{S^1}^*(\bar{\kappa})$ and $\varphi_0=\pi_{S^1}^*(\bar{\varphi}_0)$.\\

Consider now the operator $B\coloneqq -c(\chi)d+d^\dagger c(\chi):\Omega_c(M_0)\longrightarrow\Omega_c(M_0)$. It is not hard to see that $B$ has the following properties:
\begin{enumerate}
\item It is a first order symmetric differential operator.
\item It is transversally elliptic. 
\item It commutes with the $S^1$-action. 
\item It commutes with the Gauss-Bonnet grading $\varepsilon$. 
\end{enumerate}
This allows us to define the operator $\mathscr{D}':\Omega_c(M_0/S^1)\longrightarrow \Omega_c(M_0/S^1)$ through the following commutative diagram:
\begin{align*}
\xymatrixcolsep{2cm}\xymatrixrowsep{2cm}\xymatrix{
\Omega_c^\text{ev}(M_0)^{S^1} \ar[r]^-{B^\text{ev}} & \Omega_c^\text{ev}(M_0)^{S^1} \\
\Omega_c(M_0/S^1) \ar[u]^-{\psi_\text{ev}} \ar[r]^-{\mathscr{D}'}& \Omega_c(M_0/S^1) \ar[u]_-{\psi_\text{ev}}
 }
\end{align*}
Here $\psi_\text{ev}$ is a  modification of the unitary transformation introduced in \cite[Section 5]{BS88},
\begin{align*}
\psi_\text{ev}(\bar\omega)\coloneqq \left(\frac{1+\varepsilon}{2}\right)h^{-1/2}\pi_{S^1}^*(\bar{\omega}) +\left[\left(\frac{1-\varepsilon}{2}\right)h^{-1/2}\pi_{S^1}^*(\bar{\omega})\right]\wedge\chi.
\end{align*}
The operator $\mathscr{D}'$ is given explicitly by 
\begin{align*}
\mathscr{D}'=D_{M_0/S^1} +\frac{1}{2}c(\bar{\kappa})\varepsilon-\frac{1}{2}\widehat{c}(\bar{\varphi}_0)(1-\varepsilon),
\end{align*}
where $c(\bar{\kappa})$ is the  Clifford  action and 
$\widehat{c}(\bar{\varphi}_0)\coloneqq (\bar{\varphi}_0\wedge)-\star_{M_0/S^1}(\bar{\varphi}_0\wedge)\star_{M_0/S^1}$.\\

The following properties of $\mathscr{D}'$ are proven combining the results of  \cite[Section 4]{B09}, \cite{BH78} and \cite[Lemma 2.1]{GL02} .
\begin{THM}\label{THM1}
The operator $\mathscr{D}':\Omega_c(M_0/S^1)\longrightarrow \Omega_c(M_0/S^1)$ satisfies the following properties:
\begin{enumerate}
\item It anti-commutes with $\star_{M_0/S^1}$.
\item It is essentially self-adjoint. 
\item It has the same principal symbol as  the Hodge-de Rham operator $D_{M_0/S^1}$ .
\item It is discrete. 
\end{enumerate}
\end{THM}

Implementing the transformation $\Phi$ of \eqref{Eqn:OpSing} for the operator $\mathscr{D}'$ one sees that the term $\widehat{c}(\bar{\varphi_0})$ is actually proportional to $r$ and therefore bounded. Using the Kato-Rellich theorem we can remove this term and study instead the operator 
\begin{align*}
\mathscr{D}\coloneqq D_{M_0/S^1}+\frac{1}{2}c(\bar{\kappa})\varepsilon,
\end{align*}
which remains essentially self-adjoint. We call this operator the {\em induced Dirac-Schr\"odinger operator}. Since the mean curvature form can be written close to the fixed point set as $\bar{\kappa}=-dr/r$, one verifies that the whole transformation  of $\mathscr{D}$ is in fact
\begin{introequation}\label{Eqn:OpSingPot}
\Psi^{-1}\mathscr{D}\Psi
=\gamma\left(\frac{\partial}{\partial r}+
\left(\begin{array}{cc}
I & 0\\
0 & -I
\end{array}\right)\otimes \left(A(r)-\frac{\varepsilon}{2r}\right)
\right).
\end{introequation}
One can deduce from \cite[Theorem 3.1]{B09} that, by rescaling the vertical metric if necessary,
\begin{introequation}\label{CondSpecOpPot}
\spec\left(A_V-\frac{1}{2}\varepsilon\right)\cap\left(-\frac{1}{2},\frac{1}{2}\right)=\emptyset,
\end{introequation}
which verifies that $\mathscr{D}$ is indeed essentially self-adjoint. Furthermore, it is easy to verify that the parametrix's construction of \cite[Section 4]{B09} can be adapted to  $\mathscr{D}$, which allows us to prove that this operator is discrete. \\

Now we describe the index computation. As for the signature operator, we can split the index using the decomposition of Figure \ref{Fig:Decomposition} and \cite[Theorem 4.17]{BBC08} as 
\begin{align*}
\textnormal{ind}(\mathscr{D}^+)=\textnormal{ind}\left(\mathscr{D}^{+}_{Z_t,Q_{<}({A}(t))(H)}\right)+\textnormal{ind}\left(\mathscr{D}^{+}_{U_t,Q_{\geq }({A}(t))(H)}\right).
\end{align*}
Using some deformation arguments, motivated by  the proof of  \cite[Theorem 5.2]{B09}, we show that in the Witt case, the index contribution of $U_t$ is zero for $t>0$ small enough. Using this vanishing result we compute the index in a similar manner as for the signature operator by taking the limit as $t\longrightarrow 0^+$ of the index contribution of the manifold with boundary  $Z_t$. 
\begin{THM}\label{THM2}
In the Witt case for the graded Dirac-Schr\"odinger operator $\mathscr{D}^+$ we have the following index identity 
\begin{align*}
\ind(\mathscr{D}^+)=\sigma_{S^1}(M)=\int_{M_0/S^1}L\left(T(M_0/S^1), g^{T(M_0/S^1)}\right).
\end{align*}
\end{THM}
Here we have used the fact that in the Witt case the eta invariant of the odd signature operator of  fixed point set vanishes.

\subsection*{Organization of the document}
A first objective of the Thesis is to provide a clear and structured presentation of the subject. This does not mean that the presentation is self-contained, but that we present the most important arguments in a transparent way. The first three chapters are thus prerequisite material.\\

In Chapter \ref{Sect:BH} we describe in detail the general construction of ``pushing down" a self-adjoint operator of Br\"uning and Heintze treated in \cite[Sections 1 \& 2]{BH78}. In Chapter \ref{Sec:AdiabLimit} we recall some background material in order to understand the adiabatic limit formula for the odd signature operator \cite[Theorem 0.3]{D91}. In addition, we discuss the definition of $L^2$-cohomology and the construction of the signature operator in the context of Hilbert complexes. In Chapter \ref{Sect:Lott} we present the construction of the $S^1$-equivariant signature $\sigma_{S^1}(M)$ following \cite[Sections 2.1 \& 2.3]{L00}. In particular, we give a detailed and extended proof of \eqref{Eqn:Lott} using the tools introduced in Chapter \ref{Sec:AdiabLimit}.\\

The main results are developed in the following chapters. Chapter \ref{Sec:Induced} presents the push-down procedure for the Hodge-de Rham operator and the odd signature operator. This includes a careful implementation of the construction given in  Chapter \ref{Sect:BH}. We then construct $\mathscr{D}'$, the operator described above, and prove (1)-(3) of Theorem \ref{THM1}. In Chapter \ref{Sect:Examples} we present concrete examples of the theory. We first deal with some low dimensional examples. Next, we treat a semi-free $S^1$-action on the $5$-sphere such that $S^5/S^1$ is a Witt space with one singular stratum $M^{S^1}=S^1$. We verify \eqref{Eqn:Lott} in this example by exploiting the relation between $\sigma_{S^1}(M)$ and the signature of the pairing in intersection homology for Witt spaces. In Chapter \ref{Sect:LocalDesc} we dive into the local description of the operators $D$ and $\mathscr{D}$ to obtain the transformation \eqref{Eqn:OpSingPot} and verify condition \eqref{CondSpecOpPot}. Chapter \ref{Sect:Param} presents, in a synthesized manner,  Br\"uning's construction of the parametrix for the signature operator (\cite[Section 4]{B09}).  Furthermore, we show how this can be adapted to obtain a parametrix for $\mathscr{D}$ and hence to prove Theorem \ref{THM1}(4). Finally, in Chapter \ref{Sect:Index} we revisit the index computation presented in  \cite[Section 5]{B09}, we set up an analogous Dirac system for the operator $\mathscr{D}$ and prove the vanishing result for the index contribution on $U_t$ in the Witt case, proving Theorem \ref{THM2}.\\

We include two appendices: One collecting the main results of regular singular operators \cite[Sections 2 \& 3]{BS88} and another one containing the construction of a sequence of cut-off functions introduced in \cite[Section 6]{BS87} and its generalization in \cite[Lemma 5.1]{BS91}. This sequence is essential for the results of Chapter \ref{Sect:Param}. 

\subsection*{Acknowledgments}

First of all I would like to acknowledge my advisor Prof. Jochen  Br\"uning for all these years of teaching, guidance and unconditional support. Five years ago I had the privilege to attend a one year course offered by him on the Atiyah-Singer Index Theorem where I discovered this amazing area of mathematics. Through pleasant and illuminating discussions, Prof. Br\"uning taught me how to think about mathematics as a balanced mixture between intuition, creativity and rigor.  He was always patient and supportive with my questions. I will always be grateful for everything I have learned from him. \\

I wish to express my gratitude to Francesco Bei for his company and enlightening conversations. I will always be grateful with Asilya Suleymanova and Batu G\"uneysu for inspiring discussions and collaboration. I want to thank Kati Blaudzun for her support. I am profoundly grateful with Prof. Sylvie Paycha for her constant guidance, advice and of course her inspiring seminar at Universit\"at Potsdam.  I am deeply indebted to Sara Azzali, Georges Habib, Peter Patzt and  Prof. Ken Richardson for all the ideas that contributed to this work. \\

I gratefully acknowledge the financial support of the Berlin Mathematical School and the help and assistance from its amazing staff during my graduate studies. I am also grateful with the project  SFB 647: Space - Time - Matter for the financial support during the last period of this work.\\

I want to express my gratitude to all the great people that made this dream come true. In particular, Julio Backhoff, Todor Bilarev, Ana Botero, Gregor Bruns, Daniel Escobar, Sebastian Mart\'inez and Alejandra Rinc\'on, thank you for all the good moments.  Eva, thank you for your love! For my family I have only words of gratitude. None of this would have been possible without their unconditional support.  \\\\

\hfill {\em Esta tesis est\'a dedicada con todo mi amor a mi mam\'a y mi pap\'a.}

\pagenumbering{arabic}

\section{Compact Lie group actions and self-adjoint operators}\label{Sect:BH}

The aim on this first chapter is to give a detailed study of the construction, developed in the fundamental work \cite{BH78} of Br\"uning and Heintze, on induced self-adjoint operators on quotients of compact Lie group actions.

\subsection{$G$-Vector bundles}\label{Section:GVectorBundles}
We begin by describing some important results in the context of compact Lie group actions on smooth manifolds and vector bundles. We do not intend to give a complete self contained description but rather establish some notation and state relevant results that will be needed later on. For a thorough study we encourage to consult \cite{BGV}, \cite{B72}, \cite{DK00}, \cite{P01}, \cite{W71} among many other references. 

\subsubsection{$G$-manifolds}\label{Section:G-manifolds} 
Let $G$ be a compact Lie group acting on a smooth manifold $M$. We abbreviate this by simply saying that $M$ is a $G$-manifold. We will always assume that the maps are all of class $C^\infty$ unless otherwise stated. Let us begin with some basic definitions.  For an element $x\in M$ we define its {\em orbit} as $Gx\coloneqq \{gx\:|\: g\in G\}\subseteq M$.
Since $G$ is compact each orbit is an embedded submanifold of $M$. We denote the space of orbits by $M/G$ and we equip it with the quotient topology with respect to the orbit map $\pi_G:M\longrightarrow M/G$.  In general $M/G$ is not a smooth manifold, nevertheless since $G$ is compact then it is at least Hausdorff (\cite[Theorem I.3.1]{B72}). For $x\in M$, we define the its {\em isotropy} or {\em stabilizer group} at by $G_x\coloneqq \{g\in G\:|\: gx=x \}\subseteq G$. Since this set is closed then it follows that is in fact a compact Lie subgroup of $G$. In addition, the quotient map defines a principal $G_x$-bundle $G\longrightarrow G/G_x. $
\begin{definition}\label{Def:Semi-Free}
Let $e$ denote the identity element of $G$, then we say the action is 
\begin{itemize}
\item {\em Free} if $G_x=\{e\}$ for all $x\in M$. 
\item {\em Semi-free} if $G_x=\{e\}$ or $G_x=G$ for all $x\in M$. 
\item {\em Effective} if for all $g\in G$ with  $g\neq e$ there exists $x\in M$ such that $gx\neq x$.
\end{itemize}  
\end{definition}

Our first objective is to describe a local model for the $G$-action on $M$. This is the content of the so called {\em Slice Theorem} (see Theorem \ref{Thm:SliceThm} below). To do so we need to describe the induced $G$-action on each tangent space, i.e. its linearization. 

\begin{lemma}[{\cite[Proposition I.4.1]{B72}}]\label{Lemma:Ell}
For each $x\in M$, the multiplication map
\begin{align*}
\ell_x:\xymatrixrowsep{0.01cm}\xymatrixcolsep{2cm}\xymatrix{
G \ar[r] & M\\
g \ar@{|->}[r] & gx,
}
\end{align*}
induces a diffeomorphism between the quotient $G/G_x$ and the orbit $Gx$.
\end{lemma}

If we differentiate $\ell_x$ at $g=e$ we obtain a map
$\alpha_x\coloneqq d\ell_x|_{g=e}:\mathfrak{g}\longrightarrow T_x M,$
where $\mathfrak{g}\cong T_e G$ denotes the Lie algebra of $G$. We call $\alpha_x$ the {\em infinitesimal action at $x$}. Note that the kernel of $\alpha_x$ is equal to the Lie algebra of the subgroup $G_x$.\\

We are now ready to describe the local model of the action. Fix a point $x\in M$ and an element $g\in G_x$, then the derivative at $x$ of the induced function on $M$
\begin{align*}
g:\xymatrixrowsep{0.01cm}\xymatrixcolsep{2cm}\xymatrix{
M \ar[r] & M\\
y \ar@{|->}[r] & gy,
}
\end{align*}
induces a linear map $dg|_x:T_xM\longrightarrow T_x M$, called the {\em tangent action}. This defines a linear representation (\cite[Theorem 3.45]{W71})
\begin{align}\label{Eqn:rho}
\rho_x: \xymatrixrowsep{0.01cm}\xymatrixcolsep{2cm}\xymatrix{
G_x \ar[r] & \Aut(T_x M)\\
g \ar@{|->}[r] & dg|_x.
}
\end{align}
Starting from the principal $G_x$-bundle $G\longrightarrow G/G_x$ we can therefore construct the associated vector bundle $G\times_{G_x}T_x M\longrightarrow G/G_x$. Recall that this vector bundle is defined as
$G\times_{G_x}T_x M\coloneqq (G\times T_x M)/G_x$ where we define the action of an element $g\in G_x$ on $(h,v)\in G\times T_xM$ by $g(h,v)\coloneqq  (hg^{-1},\rho_x(g)v)$. Let $[(h,v)]\in G\times_{G_x} T_x M$ denote the equivalence class of $(h,v)\in G\times T_x M$, then the projection onto the $G/G_x$ is just $[(h,v)]\longmapsto hG_x$. Now we define an action of the whole group $G$ on the vector bundle $G\times_{G_x}T_x M$ by $h_0[(h,v)]\coloneqq [(h_0h,v)]$.  The action of $G$ on $G\times_{G_x}T_x M$ is well defined since it commutes with the $G_x$-action.  
\begin{theorem}[Slice Theorem, {\cite[Section 2.4]{DK00}}]\label{Thm:SliceThm}
Let $G$ be a compact Lie group acting on a manifold $M$. For each $x\in M$ there exists a $G$-invariant open neighborhood $\mathcal{O}_x$ of $x$ in $M$ such that the $G$-action in $\mathcal{O}_x$ is equivalent to the action of $G$ on $G\times_{G_x}B$. Here $B$ is an open $G_x$-invariant neighborhood of $0$ in $T_x /\alpha_x(\mathfrak{g})$, on which $G_x$ acts linearly, via the tangent action. 
\end{theorem}
Note in particular the following commutative diagram:
\begin{align*}
\xymatrixrowsep{2cm}\xymatrixcolsep{2cm}\xymatrix{
G\times_{G_x}B \ar[d] \ar[r] & \mathcal{O}_x \ar[d]\\
G/G_x \ar[r]^-{\ell_x} & Gx.
}
\end{align*}

\begin{coro}\label{Coro:FreeActMnfld}
If the action on $M$ is free, then the quotient space $M/G$ is a manifold, and $\pi_G:M\longrightarrow M/G$ is a principal $G$-bundle.
\end{coro}

Let $H\subset G$ be a subgroup, we define its associated fixed point set by
\begin{align*}
M^{H}\coloneqq \{x\in M\:|\:hx=x\:,\:\forall h\in H\}\subseteq M.
\end{align*}
If $H\subset G$ is a closed subgroup then each connected component of $M^H$ is a closed submanifold of $M$ (\cite{K58}).\\

For an element $x\in M$ we define its {\em orbit type} to be the conjugacy class of its istoropy group $G_x$ in $G$. Since for all $g\in G$ we have $G_{gx}=gG_x g^{-1}$ we see that the orbit type is constant along each orbit. For a closed subgroup $H\subset G$ we define 
\begin{align*}
M_{(H)}\coloneqq \{x\in M\:|\: G_x \:\text{is conjugate to}\: H\}\subset M. 
\end{align*}
It can be shown, using the Slice Theorem, that each $M_{(H)}$ is a submanifold of $M$. The following relation is the main idea for the proof. 
\begin{lemma}\label{Lemma:LocOrbitModel}
For the local model described by the Slice Theorem above we have 
$$(G\times_{H}B)_{(H)}=G\times_H B^H\cong (G/H)\times B^H.$$ 
\end{lemma}

\begin{proof}
Let $[(g,v)]\in (G\times_{H}B)_{(H)}$ and $\tilde{g}\in G_{[(g,v)]}$ then by definition there exists $h\in H$ such that $(\tilde{g}g,v)=(gh^{-1},\rho(h)v)$. In particular we see that $\tilde{g}=gh^{-1}g^{-1}$ and $h\in H_v$, where $H_v$ denotes the isotropy group of $v\in B$ with respect to the $H$-action. This shows the relation $G_{[(g,v)]}=g H_v g^{-1}$. Note that $H_v$ is conjugate to $H$ if, and only if, $H_v=H$. This can be seen as follows: Assume $H=gH_vg^{-1}$ and let $\Ad_g$ denote the conjugation by $g$, which is an isomorphism with inverse $\Ad_{g^{-1}}$. Let $j:H_v\longrightarrow H$ denote the inclusion, then the diagram
\begin{align*}
\xymatrixrowsep{2cm}\xymatrixcolsep{2cm}\xymatrix{
H\ar[r]^-{\Ad_g} & gHg^{-1}=H_v\\
H_v\ar@{^{(}->}[u]^-{j} \ar@{=}[ru]
}
\end{align*}
shows that $j$ must be an isomorphism. Hence, for $[(g,v)]\in (G\times_{H}B)_{(H)}$ we must have $H_v=H$, which shows $(G\times_{H}B)_{(H)}=G\times_H B^H$. Since the action of $H$ on $B^H$ is trivial we see that the twisted product becomes $G\times_H B^H\cong (G/H)\times B^H$. 
\end{proof}

\begin{remark}\label{Rmk:QuotientLocModelOrbTyp}
From the proof of last lemma we also obtain $(G\times_{H}B)_{(H)}/G=B^H$.
\end{remark}

The following is another important result which describes a particularly important orbit type.  

\begin{proposition}[Principal Orbit, {\cite[Section VI.3]{B72},\cite[Section 2.8]{DK00}}]\label{Prop:PrincipalOrTyp}
Suppose that $M$ is connected. Then there exists a closed subgroup $K\subset G$ such that $M_0\coloneqq M_{(K)}$ is open and dense in $M$. The subset $M_0$ is called the principal orbit and the conjugacy class class $(K)$ is called the principal orbit type of the action. Moreover, the quotient space $M_0/G$ is also a connected manifold. 
\end{proposition}

Now let us equip the $G$-manifold $M$ with a Riemannian metric and suppose the compact Lie group $G$ acts on $M$ by orientation preserving isometries. This can always achieved my means of the averaging procedure described below. Let $M_0$ be the principal orbit from Proposition \ref{Prop:PrincipalOrTyp}, which we know is an open and dense subset of $M$. We equip the manifold $M_0/G$ with a Riemannian metric such that $\pi_G|_{M_0}:M_0\longrightarrow M_0/G$ becomes a Riemannian submersion, i.e. we require $d\pi_G\big{|}_{(\ker d\pi_G)^\perp}:(\ker d\pi_G)^\perp\longrightarrow T(M_0/G)$ to be an isometry. We call this the associated {\em quotient metric}.

\subsubsection{$G$-vector bundles}
In addition, let $\pi_E: E\longrightarrow M$ be a complex vector bundle with Hermitian metric $\inner{\cdot}{\cdot}_E$.  This metric induces an inner product on the space of continuous sections with compact support $C_c(M,E)$ by
\begin{equation}\label{Eqn:HermE}
(s, s')_{L^2(E)}\coloneqq \int_M \inner{s(x)}{s'(x)}_E \vol_M(x),
\end{equation}
where $s,s'\in C_c(M,E)$, $x\in M$ and $\vol_M$ denotes the Riemannian volume element on $M$. Define the space $L^2(E)$  as the Hilbert space completion of $C_c(M,E)$ with respect to the inner product \eqref{Eqn:HermE}.
\begin{remark}\label{Rmk:M-M0}
By \cite[Proposition IV.3.7]{B72} it follows  that $M-M_0$ has measure zero with respect to the Riemannian measure, hence $L^2(E)=L^2(E|_{M_0})$. 
\end{remark}
Assume further that $E$ is a  {\em $G$-equivariant vector bundle}, i.e. $G$ acts on $E$ preserving the Hermitian metric and for each $g\in G$ the following diagram commutes 
$$\xymatrixrowsep{2cm}\xymatrixcolsep{2cm}\xymatrix{
E \ar[d]_{\pi_E} \ar[r]^-{g} & E\ar[d]^{\pi_E}\\
M \ar[r]^-{g} & M.
}$$
In this context there is an induced action of $G$ on the space of continuous sections $C(M,E)$ defined by the relation
\begin{equation}\label{Eqn:GActionSections}
(U_gs)(x)\coloneqq g(s(g^{-1}x)),
\end{equation}
where $g\in G$ $s\in C(M,E)$ and  $x\in M$. This action turns out to be an unitary representation of $G$ in $L^2(E)$. To see this we compute the norm
\begin{align*}
\norm{U_g s}^2_{L^{2}(E)}=\int_M \norm{g(s(g^{-1}x))}^2_E \vol_M(x)=\int_M \norm{s(g^{-1}x)}^2_E \vol_M(x)=\norm{s}^2_{L^{2}(E)},
\end{align*}
where the second equality follows since the $G$-action preserves the Hermitian metric and the third equality follows from the fact that $G$ acts on $M$ by orientation preserving isometries. We say that a section $s\in C(M,E)$ is {\em $G$-invariant} if $U_g s=s$ for all $g\in G$ and we denote by $C(M,E)^G$ and $L^2(E)^G$ the $G$-invariant subspaces of $C(M,E)$ and $L^2(E)$ respectively.\\
Let us discuss now the construction of the {\em orthogonal projection}  
$Q:L^2(E)\longrightarrow L^2(E)^G$.
We begin by recalling some results from the theory of integration on Lie groups (\cite[Section 4.2]{DK00}). Since $G$ is a compact Lie group then we can choose a left invariant Riemannian metric on it with associated volume element $dG$. In order to understand the construction of $Q$ it is essential to study the process of averaging over $G$ a continuous function $f\in C(G,\mathcal{V})$ with values in a complete, locally convex, topological vector space $\mathcal{V}$. Such a function is automatically uniformly continuous because $G$ is compact. A locally convex space is a vector space  $\mathcal{V}$ together with a family of seminorms $\{\nu_a\}_{a\in\mathcal{A}}$. The topology on $\mathcal{V}$ is the one generated by sets of the form 
$\mathcal{O}_{\mathcal{B},\epsilon}(p) \coloneqq \{q\in \mathcal{V}\:|\: \nu_a(p-q)<\epsilon\: ,\:\forall a\in\mathcal{B}\}$,
for $p\in\mathcal{V}$, $\mathcal{B}\subseteq \mathcal{A}$ and $\epsilon>0$ (\cite[Chapter 1]{RUDINFA}). For each $\epsilon >0$ and each seminorm $\nu_a$ we can use the uniform continuity of $f$ to construct a finite partition of unity $\{(\mathcal{O}_j,h_j)\}_{j=1}^N$ on $G$ where $\mathcal{O}_j\subset G$ is open, $h_j:G\longrightarrow \mathbb{C}$ is a non negative continuous function with $\supp(h_j)\subset\mathcal{O}_j$ for each $j$, the pointwise sum of all $h_j$ is equal to one and such that  $\nu_a(f(g_0)-f(g_1))<\epsilon$ whenever $g_0,g_1\in\supp(h_j)\subset G$. If we choose $g_j\in\supp(h_j)$ then we can consider the  ``Riemann sums" 
\begin{align}\label{RiemmSum}
\sum_{j=1}^N \left(\int_G h_j(g)dG(g)\right)f(g_j)\in \mathcal{V}. 
\end{align}
These sums depend of course on the partition of unity, however they form a Cauchy net and therefore converge to an element in $\mathcal{V}$, which is by definition the integral
\begin{align*}
\int_G f(g)dG(g). 
\end{align*}
Moreover, if $\phi:\mathcal{V}\longrightarrow \mathbb{C}$ is a continuous linear functional then $\phi\circ f\in C(G)$ and 
\begin{align*}
\phi\left(\int_G f(g)dG(g)\right)=\int_G \phi(f(g))dG(g). 
\end{align*}
Now let $U$ be a representation of $G$ in $\mathcal{V}$. It can be shown (using the Banach-Steinhaus theorem) that for each $h\in C(G)$, the map 
$$U(h):\xymatrixrowsep{0.01cm}\xymatrixcolsep{2cm}\xymatrix{
\mathcal{V}\ar[r] & \mathcal{V}\\
p \ar@{|->}[r] & \displaystyle{\int_G h(g)U(g)(p)dG(g)}
}$$
is continuous. In particular, we can consider the special case  $h(g)=1$ and define the {\em average of the representation} $U$ by
\begin{align}\label{Eqn:AV}
\text{av}(U):\xymatrixrowsep{0.01cm}\xymatrixcolsep{1.5cm}\xymatrix{
\mathcal{V}\ar[r] & \mathcal{V}\\
p \ar@{|->}[r] & \displaystyle{\frac{1}{\vol(G)}\int_G U(g)(p)dG(g)}. 
}
\end{align}
We summarize what we have described in the following theorem.
\begin{theorem}[Averaging Principle, {\cite[Proposition 4.2.1]{DK00}}]\label{Thm:AVPrincp}
Let $G$ be a compact group, and $U$ be a representation of $G$ in a complete, locally convex, topological vector space $\mathcal{V}$. Then $\textnormal{av}(U)$ is a linear projection of $\mathcal{V}$ onto the space 
$$\mathcal{V}^{U(G)}\coloneqq \{p\in \mathcal{V}\:|\: U(g)p=p\:,\:\forall g\in G\}.$$
If $U$ is a unitary representation of $G$ on a Hilbert space $\mathcal{V}$, then $\textnormal{av}(U)$ is equal to the orthogonal projection from $\mathcal{V}$ onto $\mathcal{V}^{U(G)}$. 
\end{theorem}

Applying Theorem \ref{Thm:AVPrincp} to our case of interest, namely $\mathcal{V}=L^2(M,E)$ and the unitary representation $U$ defined by \eqref{Eqn:GActionSections}, we see then that the orthogonal projection $Q\coloneqq\text{av}(U)$ is explicitly given by
\begin{align}\label{Def:Q}
Qs=\frac{1}{\vol(G)}\int _G(U_gs) d G(g),\quad\text{for $s\in L^2(M,E)$.}
\end{align}

\begin{example}[Exterior Algebra]\label{Ex:ExtAlg}
As before let $M$ be an oriented Riemannian manifold on which $G$ acts by orientation preserving isometries. The action on $M$ induces an action on the exterior algebra bundle  $E=\wedge T^*M=\bigoplus_r \wedge^r T^* M$ as follows: For $x\in M$, $\alpha\in\wedge^r T_x^* M$ and $g\in G$, we want to define the element $g\alpha\in\wedge^r T_{gx}^* M$. To do this is enough to describe its action on tangent vectors. For $v_1,\cdots v_r\in T_{gx} M$ define
\begin{align*}
(g\alpha )(v_1,\cdots v_r)\coloneqq  \alpha((dg^{-1}|_{gx})v_1,\cdots, (dg^{-1}|_{gx})v_r).
\end{align*}
This is well defined since $(dg^{-1}|_{gx}): T_{gx} M\longrightarrow T_x M$.  With this action $E$ becomes a $G$-vector bundle over $M$. We now describe the action \eqref{Eqn:GActionSections}. Let $X_1,\cdots, X_r\in C^\infty(M,TM)$ be vector fields on $M$ and  $\omega\in\Omega^r(M)\coloneqq C^\infty(M,\wedge T^rM)$ be a differential $r$-form, then by definition
\begin{align*}
(U_g\omega)(x)(X_1 (x)\cdots, X_r (x))&= (g\omega(g^{-1}x))(X_1 (x)\cdots, X_r (x))\\
&=\omega(g^{-1}x)((dg^{-1}|_{x})X_1 (x),\cdots,(dg^{-1}|_{x})X_r (x))\\
&=((g^{-1})^*\omega)(x)(X_1 (x)\cdots, X_r (x)).
\end{align*}
Thus, the $G$-action on differential forms is simply given by $U_g\omega=(g^{-1})^{*}\omega$. In particular, an element $\omega\in\Omega(M)^G$ is characterized by the relation $(g^{-1})^*\omega=\omega$ for all $g\in G$. 
\end{example}

The final aim of this section is to carefully study the construction, introduced in \cite[Section 1.]{BH78}, of a Hilbert space isomorphism between $L^2(M,E)^G$ and $L^2(M_0/G,F)$ where $F$ is a Hermitian vector bundle on $M_0/G$ constructed from $E$.  To begin consider the subset
\begin{equation}\label{Eqn:DefE'}
E'\coloneqq \bigcup_{x\in M_0} E^{G_x}_x
\end{equation}
where $E^{G_x}_x$ denotes the elements of the fiber $E_x\coloneqq \pi_E^{-1}(x)$ which are invariant under the $G_x$-action.

\begin{lemma}[{\cite[Lemma 1.2]{BH78}}]
If $M$ is connected then $E'$ is a $G$-equivariant subbundle of $E\big{|}_{M_0}$.
\end{lemma}

\begin{proof}
To show that $E'$ is $G$-equivariant we use the relation $G_{gx}=gG_xg^{-1}$ for $x\in M$ and $g\in G$ to conclude that $g(E_x^{G_x})=E_{gx}^{G_{gx}}$. To prove local triviality of $E'$ we use the local description of Lemma \ref{Lemma:LocOrbitModel}, that is, we assume that $M_0=G/H\times B$ where $B$ is an open ball in some Euclidean space. The action of $g\in G$ in this local model is $g(g_0H,v)=(gg_0H,v)$ for $(g_0 H,v)\in G/H\times B$. In particular if $x\in \mathcal{O}\coloneqq H\times B$ then $G_x=H$. Define now $E''\coloneqq E'\big{|}_{\mathcal{O}}=(E\big{|}_\mathcal{O})^H$, where $(E\big{|}_\mathcal{O})^H$ denotes the image of $(E\big{|}_\mathcal{O})$ under the averaging map (Theorem \ref{Thm:AVPrincp})
\begin{align*}
\frac{1}{\vol{(H)}}\int_H U_h(\cdot) dH(h).
\end{align*}
Being the image of a vector bundle homomorphism, as $U_h\in C^{\infty}(M,\Hom(E'\big{|}_\mathcal{O},E'\big{|}_\mathcal{O}))$, we can conclude that $E''$ is a vector bundle. Thus we can assume that $E''$ is trivial (since we are interested in a local condition). Now we want to prove that $E'$ is trivial by extending the trivialization of $E''$. Let $(s_j)_{j=1}^k$ be a maximal set of linear independent sections of $E''$. Define associated sections $(\bar{s}_j)_{j=1}^k$ on $E'$ by 
$
\bar{s}_j((gH,v))\coloneqq gs_j((H,v)).
$
Clearly the sections $(\bar{s}_j)_{j=1}^k$ are also maximal linearly independent, which implies that $E'$ is also locally trivial. Finally, since $M$ is connected,  $M_0/G$ is also connected by Proposition \ref{Prop:PrincipalOrTyp} and this implies the rank of $E'$ is constant on $M_0$.  
\end{proof}

From the proof of last lemma we see that $E'$ has just one orbit type $(H)$ with respect to the $G$ action. Hence,from Remark \ref{Rmk:QuotientLocModelOrbTyp}, it follows that $F\coloneqq E'/G$ is a manifold. Let $\pi'_G:E'\longrightarrow F$ denote the orbit map and $\pi_{E'}:E'\longrightarrow M_0$ denote the projection. Locally we can assume $E'=G/H\times \tilde{B}$ and $M_0=G/H\times B$ where $\tilde{B}$ and $B$ are two open balls in some Euclidean space with $B\subset \tilde{B}$. Using Remark \ref{Rmk:QuotientLocModelOrbTyp} we observe that $\pi_{E'}: \tilde{B}\longrightarrow B$ induces a map $\pi_F: \tilde{B}^{H}\longrightarrow B^H$ which makes the following diagram commute:
\begin{align}\label{Diag:LocalSliceBundle}
\xymatrixrowsep{2cm}\xymatrixcolsep{2cm}\xymatrix{
G/H\times \tilde{B} \ar[d]_-{\pi_{E'}}  \ar[r]^-{\pi'_G}  & \tilde{B} \ar[d]^-{\pi_{F}} \\
G/H\times B \ar[r]^-{\pi_G} & B.
}
\end{align}
This shows that $F$ is a vector bundle over $M_0/G$. Globally the last diagram looks like 
\begin{align}\label{Diag:OrbitMaps}
\xymatrixrowsep{2cm}\xymatrixcolsep{2cm}\xymatrix{
E' \ar[d]_-{\pi_{E'}}  \ar[r]^-{\pi'_G}  & F \ar[d]^-{\pi_{F}} \\
M_0 \ar[r]^-{\pi_G} & M_0/G.
}
\end{align}
\begin{remark}\label{Remark:rkF}
Observe from Lemma \ref{Lemma:LocOrbitModel} and Remark \ref{Rmk:QuotientLocModelOrbTyp} that the dimension  of the spaces involved  in the construction are related by the formulas
\begin{align*}
\dim (M_0)=&\dim(M_0/G)+\dim(G/H),\\
\dim (E')=&\dim (F)+\dim(G/H). 
\end{align*}
In particular we see that the rank, as vector bundles, of $E'$ and $F$ coincide. Indeed, from the relations above,
\begin{align*}
\rk(E')=\dim (E')-\dim(M_0)=\dim (F)-\dim(M_0/G)=\rk(F). 
\end{align*}
\end{remark}

\begin{lemma}\label{Lemma:InducedMetricF}
The bundle $F$ inherits a Hermitian metric $\inner{\cdot}{\cdot}_F$ from $E'$ defined by the relation 
$\inner{\pi'_G v_1}{\pi'_G v_2}_F(y)\coloneqq \inner{v_1}{v_2}_E(x),$ where $x\in M_0$, $v_1,v_2\in E'_x$ and $\pi_G(x)=y$.
\end{lemma}

\begin{proof}
We need to show that the relation above is well defined. Assume $\pi'_G v_1=\pi'_G \tilde{v}_1$ and $\pi'_G v_2=\pi'_G \tilde{v}_2$ for $\tilde{v}_1,\tilde{v}_2\in E'_{\tilde{x}}$, where $x=g\tilde{x}$. Then there exist $g_1,g_2\in G$ such that $v_1=g_1\tilde{v}_1$ and $v_2=g_2\tilde{v}_2$. Since $E'$ is  $G$-equivariant then
\begin{align*}
x=\pi_{E'}(v_1)=&\pi_{E'}(g\tilde{v}_1)=g_1\pi_{E'}(\tilde{v}_1)=g_1\tilde{x},\\
x=\pi_{E'}(v_2)=&\pi_{E'}(g\tilde{v}_2)=g_2\pi_{E'}(\tilde{v}_2)=g_2\tilde{x}.
\end{align*}
These relations show $g^{-1}g_1,g^{-1}g_2\in G_{\tilde{x}}$. Hence,
\begin{align*}
\inner{v_1}{v_2}_E(x)=&\inner{ g_1 \tilde{v}_1}{g_2 \tilde{v}_2}_E(x)\\
=&\inner{ g^{-1}g_1 \tilde{v}_1}{ g^{-1}g_2 \tilde{v}_2}_E(g^{-1}x)\\
=&\inner{ \tilde{v}_1}{ \tilde{v}_2}_E(\tilde{x}),
\end{align*}
where the second equality holds since the metric on $E$ is preserved by the $G$-action and the third equality follows from the definition of $E'$. 
\end{proof}

For $y\in M_0/G$ let $h(y)\coloneqq \vol(\pi_G^{-1}(y))$ be the {\em volume of the orbit} containing a point in $\pi_G^{-1}(y)$. We can consider the weighted inner product on $C_c(M_0/G,F)$ defined by the formula
\begin{equation}\label{Def:L2FhGen}
(s,s')_{L^2(F,h)}\coloneqq \int_{M_0/G}\inner{s(y)}{s'(y)}_F h(y) \vol_{M_0/G}(y),
\end{equation}
where  $\vol_{M_0/G}$ denotes the Riemannian volume element of $M_0/G$ with respect to the quotient metric.  Analogously we define $L^2(F,h)$ to be the completion of $C_c(M_0/G,F)$ with respect to this inner product.\\

We now describe one of the most important results of \cite{BH78}. The spirit of the proof relies on the construction above  and  Remark \ref{Rmk:M-M0}.

\begin{theorem}[{\cite[Theorem $1.3$]{BH78}}]\label{Thm:Fund}
There is an isometric isomorphism of Hilbert spaces
$$\Phi:L^2(E)^{G}\longrightarrow L^2(F,h).$$
With $\pi'_G:E'\longrightarrow F$ denoting the orbit map $\Phi$ is given by
$$\Phi s_1 \circ \pi_G(x)=\pi'_G\circ s_1(x),$$
where $s_1\in C_c(M_0,E)^G$ and $x\in M_0$. Its inverse map is given by 
$$\Phi^{-1}s_2(x)=s_2\circ \pi_G(x)\cap E_x,$$
where $s_2\in C_c(M_0/G,F)$ and $x\in M_0$. 
\end{theorem}

\begin{proof}
We divide the proof in several steps:
\begin{itemize}
\item \underline {Step 1: $\Phi$ is well defined}.\\
 Let $s_1\in C_c(M_0,E)^G$ and $x\in M_0$, then if $g\in G_x$ we have 
$$gs_1(x)=g(s_1(g^{-1}x))=(U_g s_1)(x)=s_1(x),$$
thus $C_c(M_0,E)^G\subset C_c(M_0,E')$. This shows that the map $\Phi$ is well defined. 
\item  \underline{Step 2: $\Phi(C_c^\infty(M_0,E)^G)=C_c^\infty(M_0/G,F)$.}\\
From the definition of $\Phi$ it follows directly that $\Phi(C_c^\infty(M_0,E)^G)\subset C_c^\infty(M_0/G,F)$. Let $s_2\in C^\infty_c(M_0/G,F)$  and $x\in M_0$, then using the slice theorem and \eqref{Diag:LocalSliceBundle} we observe that  we can write $x=(g_0H,b)\in G/H\times B$ and $s_2(\pi_G(x))=\pi'_G(g_0H,v)$ for some $(g_0 H,v)\in G/H\times\tilde{B}$. We now compute
\begin{center}
$(U_g(\Phi^{-1}s_2))(x)=g((\Phi^{-1}s_2)(g^{-1}x))=g(g^{-1}g_0H,v)=(g_0H,v)=(\Phi^{-1}s_2)(x).$
\end{center}
This completes the proof of the relation $\Phi(C_c^\infty(M_0,E)^G)=C_c^\infty(M_0/G,F)$.
\item \underline{Step 3: $\Phi$ is an isometry.}\\
Since $C_c(M_0,E)^G$ and $C_c(M_0/G,F)$ are dense in $L^{2}(E\big{|}_{M_0})=L^2(E)$ and $L^2(F)$ respectively, it is enough to show that $\Phi$ is an isometry on $C_c(M_0,E)^G$. For a section $s_1\in C_c(M_0,E)^G$ we compute $\norm{\Phi s_1}^2_{L^2(F,h)}$ using Fubini's theorem for Riemannian submersions (\cite[Proposition A.III.5]{BGM}),
\begin{center}
\begin{align*}
\norm{\Phi s_1}^2_{L^2(F,h)}=&\int_{M_0/G}\norm{\Phi s_1(y)}^2_F(y) h(y) \vol_{M_0/G}(y)\\
=&\int_{M_0/G}\norm{\pi'_G (s_1(x))}^2_F(y) h(y) \vol_{M_0/G}(y)\\
=&\int_{M_0/G}\left(\int_{\pi^{-1}_G(y)}\norm{ s_1(x)}^2_F(y) \vol_{\pi^{-1}_G(y)}(x)\right)\vol_{M_0/G}(y)\\
=&\int_{M_0}\norm{s_1(x)}^2_E(x)\vol_{M_0}(x)\\
=&\norm{s_1}_{L^2(E{|}_{M_0})}^2.
\end{align*}
\end{center}
\end{itemize}
\end{proof}
From the proof of the theorem above we have the following immediate consequences.
\begin{coro}[{\cite[Corollary 1.4]{BH78}}]\label{Coro:PropPhi}
The map $\Phi$ satisfies:
\begin{enumerate}
\item $\Phi(C_c^\infty(M_0,E)^G)=C_c^\infty(M_0/G,F)$.
\item $\norm{\Phi s_1(\pi_G(x))}_{F}=\norm{s_1(x)}_E$ for $s_1\in C_c^\infty(M_0,E)^G$ and $x\in M_0$.
\item $\Phi^{-1}(\phi s_2)=(\phi\circ\pi_G)\Phi^{-1}s_2$ for $s_2\in C_c(M_0/G,F)$ and $\phi\in C(M_0/G)$.
\item $\supp(\Phi s_1)=\pi_G(\supp(s_1))$ for $s_1\in C_c(M_0,E)^G$ and $\supp(\Phi^{-1} s_2)=\pi^{-1}_G(\supp(s_2))$ for $s_2\in C_c(M_0/G,F)$. 
\end{enumerate}
\end{coro}

Now we comment on a further property of the map $\Phi$ with respect to the jet maps, which will be relevant in the next section.
For $k\in\mathbb{Z}_+$ and $s\in C^\infty_c(M_0,F)$ consider 
\begin{align*}
j_k(s):
\xymatrixcolsep{2cm}\xymatrixrowsep{0.01cm}\xymatrix{
M_0 \ar[r] & J^k(F)\\
x \ar@{|->}[r] & j_k(s)(x),
}
\end{align*} 
where $j_k(s)(x)$ denotes the $k$-jet of $s$ at $x$  and $J^k(F)$ denotes that $k$-jet bundle associated with $s$ (\cite[Chapter IV.2]{PALAIS}) . 

\begin{coro}[{\cite[Corollary 1.4(5)]{BH78}}]\label{Coro:jets}
For $s\in C^\infty_c(M_0/G,F)$ and $x\in M_0$ we have $j_k(\Phi^{-1}s)(x)=0$ whenever $j_k(s)(\pi_G(x))=0$.
\end{coro}

\begin{proof}
Let $\pi_G(x)=y\in M_0/S^1$ and suppose $j_k(s)(y)=0$. This implies we can express
\begin{align*}
s=\sum_{i=1}^{r}\phi_i s_i
\end{align*}
with $\phi_i\in C^\infty_c(M_0/G)$ and $s_i\in C^\infty_c(M_0/G,F)$ with $j_{k}(\phi_i)(y)=0$ for $1\leq r\leq k$. From Corollary \ref{Coro:PropPhi}(3) it then follows
\begin{align*}
j_k(\Phi^{-1}s)(x)=\sum_{i=1}^r j_k((\phi_i\circ\pi_G)\Psi^{-1}s_i)(x)=\sum_{i=1}^{r}j_k(\phi_i)(y)j_k(\Phi^{-1}s_i)(x)=0.
\end{align*}
\end{proof}

\subsubsection{$G$-invariant and basic differential forms}
We end this section recalling some facts and fixing some notation regarding $G$-invariant differential forms.  Let $\mathfrak{g}$ denote the Lie algebra of $G$ and consider the setting as above. To each element $X\in\mathfrak{g}$ we can associate a first order differential operator acting on $s\in C^\infty(M,E)$, called the {\em Lie derivative} (\cite[Section 1.1]{BGV}), defined by the formula
\begin{equation}\label{Eqn:LieDerE}
L^E_X s\coloneqq \frac{d}{d t}\bigg{|}_{t=0}U_{\exp(tX)}s,
\end{equation}
where $\exp:\mathfrak{g}\longrightarrow G$ denotes the exponential map. It is easy to see that $s\in C^\infty(M,E)^G$ if and only if $L^E_X s=0$ for all $X\in\mathfrak{g}$ if $G$ is connected. 

\begin{remark}\label{Rmk:KillingVF}
In the special case when $E=M\times \mathbb{C}$ the operator \eqref{Eqn:LieDerE} defines a vector field, which we still denote by $X$ and is called the {\em generating} or {\em Killing vector field}. This vector field acts on a function $f\in C^{\infty}(M)$ as
\begin{align*}
(Xf)(x)=\frac{d}{d t}\bigg{|}_{t=0}f(\exp(-tX)x).
\end{align*}
\end{remark}

Consider now the setting of Example \ref{Ex:ExtAlg} in which $E=\wedge T^* M$. In this case the differential operator defined by \eqref{Eqn:LieDerE} is the usual Lie derivative $L_X$ acting on differential forms. For future reference we recall {\em Cartan's formula}
\begin{equation}\label{Eqn:Cartan}
L_X=d\iota_X+\iota_X d,
\end{equation}
where $\iota_X: \Omega^r(M)\longrightarrow\Omega^{r-1}(M)$ denotes the inner multiplication, or contraction,  by $X$ and $d:\Omega^r(M)\longrightarrow\Omega^{r+1}(M)$ is the exterior derivative. We now define the space of {\em basic differential forms} as
\begin{align}\label{Def:BasicForms}
\Omega_\text{bas}(M)\coloneqq \{\omega\in\Omega(M)\:|\:L_X\omega=0\:\:\text{and}\:\:\iota_X\omega=0 \:\:\forall X\in\mathfrak{g}\}
\end{align}
Using Cartan's formula it is easy to verify that $(\Omega_\text{bas}(M),d)$ is a sub-complex of $(\Omega(M),d)$, the de Rham complex of $M$. Observe also the inclusion $\Omega_\text{bas}(M)\subset \Omega(M)^G$ where $\Omega(M)^G\coloneqq C^\infty(M,\wedge T^*M)^G$ is the space of {\em $G$-invariant differential forms}. This is basically due to the condition $L_X\omega=0$. On the other hand, a differential forms $\omega$ satisfying the condition $\iota_X\omega=0$ for all $X\in\mathfrak{g}$ is called {\em horizontal}. The space of horizontal differential forms on $M$ is denoted by $\Omega_\text{hor}(M)$. Thus, $\Omega_\text{bas}(M)=\Omega_\text{hor}(M)\cap \Omega(M)^G$.\\

The following result relates the space of basic forms and the space of differential forms on the orbit space when the action is free. 
\begin{proposition}[{\cite[Proposition 1.9]{BGV},\cite[Lemma 6.44]{M00}}]\label{Prop:BasicPullback}
Let $G$ act on $M$ freely, then for each basic form $\omega\in \Omega_\textnormal{bas}(M)$ there exists a unique differential form $\bar{\omega}\in\Omega(M/G)$ on the orbit space such that $\omega=\pi_G^*\bar{\omega}$.  
\end{proposition}

\subsection{Induced operators on the principal orbit type}\label{Section:InduedOperatorsGen}

In this subsection we are going to describe how to use Theorem \ref{Thm:Fund} to construct a self-adjoint operator on $L^2(F,h)$ from a given $G$-invariant self-adjoint operator on $L^2(E)$. Let $\pi_E:E\longrightarrow M$ be a $G$-equivariant vector bundle as above and let $R:\dom(R)\subseteq L^{2}(E)\longrightarrow L^2(E)$ be a self-adjoint operator with domain of definition $\dom(R)$. 
\begin{definition}\label{Def:GInvOp}
We say that $R$ is {\em $G$-invariant} if it commutes with $G$, that is
\begin{enumerate}
\item $U_g(\dom(R))\subset \dom(R)$ for all $g\in G$.
\item $U_g R(s)=RU_g(s)$ for all $g\in G$ and $s\in\dom(R)$.
\end{enumerate}

\end{definition}

Let us denote its spectral resolution by 
\begin{align}\label{Eqn:SpecResR}
R=\int_{-\infty}^\infty t dR_t,
\end{align}
where $t\longmapsto R_t$ is  left continuous (\cite[Section VI.5]{KATO},\cite[Section 5.3]{S12}). When we identify $R_t$ with its image in $L^2(E)$ we see that $R_t$ is $G$-invariant for $t\in\mathbb{R}$. 
\begin{lemma}[{\cite[Lemma $2.2$]{BH78}}]\label{Lemma:OpS}
The operator $S\coloneqq R|_{\dom(R)\cap L^2(E)^G}$ is a well defined self-adjoint operator on $L^2(E)^G$ with $\dom(S)=\dom(R)^G$.
\end{lemma}

\begin{proof}
First we want to show the restriction of $Q$ to $\dom(R)$ maps into $\dom(R)$. Let $s\in\dom(R)$ then, using the notation of the construction of the average representation from the last section, we see each finite Riemann sum \eqref{RiemmSum} satisfies
\begin{align*}
\sum_{j=1}^N\left(\frac{1}{\vol(G)}\int_G h_j(g) dG(g) \right) U_{g_j} s\in\dom(R),
\end{align*}
since $U_g(\dom(R))\subset\dom(R)$ for all $g\in G$. Moreover, applying $R$ to such a sum we get
\begin{align*}
R\left(\sum_{j=1}^N\left(\frac{1}{\vol(G)}\int_G h_j(g) dG(g) \right) U_{g_j} s\right)=&\left(\sum_{j=1}^N\left(\frac{1}{\vol(G)}\int_G h_j(g) dG(g) \right) R(U_{g_j} s)\right)\\
=&\left(\sum_{j=1}^N\left(\frac{1}{\vol(G)}\int_G h_j(g) dG(g) \right) U_{g_j} (Rs)\right),
\end{align*}
so we see that these sums converge in $(\dom(R),\norm{\cdot}_R)$, where $\norm{\cdot}_R$ denotes the graph norm of $R$, as this space is a Hilbert space. As a result, $Qs$ belongs also to $\dom(R)$.\\
Regarding $R:\dom(R)\longrightarrow L^2(E)$ as a bounded linear operator we have easily see $U_gRQ=RU_gQ=RQ$, on $\dom(R)$,  for all $g\in G$. Thus, $R (\dom(R)\cap L^2(E)^G)\subset L^2(E)^G.$ Using the fact that both $Q$ and $R$ are symmetric we compute for $s,s'\in\dom(R)$,
\begin{align*}
(s,QRQs')_{L^2(E)}=(QRQs,s')_{L^2(E)}=(RQs,s')_{L^2(E)}=(s,QRs')_{L^2(E)},
\end{align*}
which implies $QRQ=QR$ since $R$ is densely defined.  Hence, $RQ=QRQ=QR$ on $\dom(R)$. This shows the operator $S$ is well defined. Finally note that as bounded operators $(RQ)^*=Q^*R^*=QR=RQ$ on $\dom(S)$, thus $S$ is self-adjoint. 
\end{proof}
\begin{remark}\label{Rmk:SpecResS}
The associated spectral resolution of $S$ is obtained from \eqref{Eqn:SpecResR} as 
\begin{align*}
S= \int_{-\infty}^\infty t \: dS_t\quad \text{with}\quad S_t=R_t^G.
\end{align*}
\end{remark}

From Theorem \ref{Thm:Fund} and Lemma \ref{Lemma:OpS} we deduce the following remarkable result. 
\begin{proposition}[{\cite[pg. 178-179]{BH78}}]\label{Prop:OpT}
Define an operator $$T:\dom(T)\subseteq L^2(F,h)\longrightarrow L^2(F,h)$$ by the relation
\begin{equation*}
T \coloneqq \Phi\circ S\circ \Phi^{-1}\big{|}_{\Phi(\dom(S))},
\end{equation*}
with $\dom(T)\coloneqq \Phi(\dom(S))$. Then $T$ is a self-adjoint operator. 
\end{proposition}

\begin{lemma}[{\cite[Lemma 2.3]{BH78}}]
We have the following relation between the spectrum of $R$ and the spectrum of $T$:
\begin{enumerate}
\item $\spec(T)\subset \spec(R)$.
\item If $\lambda\in\spec(T)$ is an isolated eigenvalue of $R$ then $R$ is also an eigenvalue of $T$.
\end{enumerate} 
\end{lemma}

\begin{proof}
For $z\notin\Sp(R)$ the projection $Q$ commutes with the resolvent $(R-z)^{-1}$ and therefore the $(R-z)^{-1}$ can be regarded as an operator on $L^2(E)^G$, thus $z\notin\spec(S)$. This shows $\spec(S)\subset \spec(R)$, and since $S$ and $T$ are unitarily equivalent then the first assertion follows. For the second assertion recall first that each isolated point of the spectrum of a self-adjoint operator is en eigenvalue (\cite[Corollary 5.11]{S12}). Now assume that $\lambda$ is an isolated eigenvalue of $R$ with eigenspace $W$ and $Q(W)=0$ then it follows from Remark \ref{Rmk:SpecResS} that $S_t$ is constant near $t=\lambda$ and therefore $\lambda\notin\spec(S)$ (\cite[Proposition 5.10(i)]{S12}). 
\end{proof}

\subsubsection{Example: Differential operators}\label{Section:DiffOp}

To finish this section we are going to describe a special case of the construction above for which the operator $R$ is generated by a differential operator. One defines a {\em $k$-th order differential operator} $D$ acting between sections of two vector bundles $E_1$ and $E_2$ over a manifold $M$ as a linear map $D:C^\infty(M,E_1)\longrightarrow C^\infty(M,E_2)$ satisfying the following property:
For each $x\in M$ and $s\in C^\infty(M,E_1)$ the condition $j_k(s)(x)=0$ implies $Ds(x)=0$ (\cite[Chapter IV.3]{PALAIS}). Given such an operator we can consider its restriction to smooth sections with compact support $D:C^\infty_c(M,E_1)\longrightarrow C_c^\infty(M,E_2)$.
Let us assume further that $M$ is an oriented, Riemannian and that $E_1, E_2$ are Hermitian vector bundles. These additional structures allow us to define the {\em formal adjoint} operator  $D^\dagger :C^\infty_c(M,E_2)\longrightarrow C_c^\infty(M,E_1)$ of $D$ by the relation
\begin{align*}
(Ds_1,s_2)_{L^2(E_2)}=(s_1,D^\dagger s_2)_{L^2(E_1)}, \quad\forall s_1\in C^\infty_c(M,E_1), s_2\in C^\infty_c(M,E_2).
\end{align*}
It can be shown, using integration by parts in local charts, that $D^\dagger$ exists and it is also a differential operator. The existence of the formal adjoint implies that $D$ is a {\em closable operator,} i.e. given a sequence $(s_n)_n\subset C^\infty_c(M,E_1)$ such that $s_n\longrightarrow s$ and $Ds_n\longrightarrow 0$ (here we mean $L^2$-convergence) then $s=0$. Hence, we can define its closure $\bar{D}\eqqcolon D_\text{min}$, also called the {\em minimal extension}, which is a closed operator with  domain of definition
\begin{align*}
\dom(D_\text{min})\coloneqq 
\left\{
\begin{array}{c}
s\in L^2(M,E_1)\: : \: \exists (s_n)_n \subset C_c(M,E_1), \tilde{s}\in L^2(M,E_2)\:\:\text{such that}\\
\quad \quad \quad \:\quad s_n\longrightarrow s\:\:\text{and}\:\: Ds_n\longrightarrow \tilde{s}\in L^2(M,F).
\end{array}
\right\}
\end{align*}
For such $s\in\dom(D_\text{min})$ we define $D_\text{min}s\coloneqq \tilde{s}$. \\
There is another important closed extension of a differential operator, the so called {\em maximal extension}, which is defined by its distributional action. More precisely, for $s\in L^2(M,E_1)$ we say that $Ds\in L^2(M,E_2)$ if there exists $\tilde{s} \in L^2(M,E_2)$ such that for all $s'\in C_c^\infty(M,E_1)$ the relation 
$$(s',s)_{L^2(E_1)}=(D^\dagger s',\tilde{s})_{L^2(E_2)},$$  
holds true. We define
\begin{align*}
\dom(D_\text{max})\coloneqq
\{s\in L^2(M,E_1)\:|\: Ds\in L^2(M,E_2)\},
\end{align*}
and for $s\in\dom(D_\text{max})$ we set $D_\text{max}s\coloneqq Ds$. Any choice of closed extension between the minimal extension and maximal extension is called an {\em ideal boundary condition} (\cite{Ch79}). 

The following result shows that the difference between the minimal and the maximal extension for a first order differential operator happens at ``infinity''.
\begin{lemma}[{\cite[Lemma 2.1]{GL02}}]\label{LemmaGL02}
Let $D:C^\infty_c(M,E_1)\longrightarrow C_c^\infty(M,E_2)$ be a first order differential operator. If $s\in\dom(D_\textnormal{max})$ and $\phi\in C_c^\infty (M)$ then $\phi s\in\dom(D_\textnormal{min})$. 
\end{lemma}

Let us recall the notion of {\em Hilbert space adjoint}  $D^*$ of the operator $D$. We say that $s\in\dom(D^*)\subset L^2(M,E_2)$ if the map
\begin{align*}
\xymatrixcolsep{2cm}\xymatrixrowsep{0.01cm}\xymatrix{
C_c(M,E_1 ) \ar[r] & \mathbb{C}\\
\tilde{s} \ar@{|->}[r] & (D\tilde{s},s)_{L^2(E_2)}
}
\end{align*} 
is continuous. In this case, since $D$ is densely defined, there exists a unique section $s'\in L^2(M,E_2)$ such that $(D\tilde{s},s)_{L^2(E_1)}=(\tilde{s},s')_{L^2(E_1)}$. For such $s\in\dom(D^*)$ we set $D^* s\coloneqq s'$. 
\begin{definition}
A differential operator $D:C^\infty_c(M,E_1)\longrightarrow C^\infty_c(M,E_2)$ is called 
\begin{itemize}
\item {\em Symmetric} if $D=D^\dagger.$
\item {\em Self-adjoint} if $D=D^*$.
\item {\em Essentially self-adjoint} if $\bar{D}=D^*$.
\end{itemize}
\end{definition}

\begin{remark}
Note that the definition of the maximal extension of $D $ given above is equivalent to $D_\textnormal{max}=(D^\dagger)^*$, i.e. $D_\text{max}$ is the Hilbert space adjoint of $D^\dagger$. 
\end{remark}

\begin{remark}\label{Remark:GL02}
Lemma \ref{LemmaGL02} shows in particular that any symmetric first order differential operator $D:C^\infty(M,E_1)\longrightarrow C^\infty(M,E_2)$ over a closed manifold $M$, i.e. compact without boundary, satisfies 
\begin{align*}
D_\textnormal{min}=D_\textnormal{max}=(D^\dagger)^*=D^*=(D_\textnormal{min})^*,
\end{align*}
i.e. is $D$ essentially self-adjoint. 
\end{remark}

Let $D:C_c^\infty(M,E_1)\longrightarrow C_c^\infty(M,E_2)$ be a differential operator of order $k$, recall that the {\em principal symbol} $\sigma_P(D)(x,\xi):E_{1,x}\longrightarrow E_{2,x}$ of $D$ at a point $(x,\xi)\in T^*M$ is the linear map defined by
\begin{align}\label{Def:PrincipalSymbol}
\sigma_P(D)(x,\xi)(e)\coloneqq D\left(\frac{(-i)^k}{k!}\phi^k s\right)(x), 
\end{align}
where $\phi\in C^\infty (M)$ is any function with $d\phi(x)=\xi$ and $\phi(x)=0$, $s\in C^\infty(M,E_1)$ such that $s(x)=e$. We say that $D$ is {\em elliptic} if $\sigma_P(D)(x,\xi)$ is invertible for all $\xi\neq 0$.\\
In the context in which a compact Lie group $G$ acts on $M$, as in last section, we define
\begin{align*}
T_G^*M\coloneqq \{(x,\xi)\in T^*M\:|\:\xi(v)=0,\:\forall v\in T_x M\:\text{tangent to the orbit $Gx$}\}.
\end{align*}
We say that $D$ is {\em transversally elliptic} if $\sigma_P(D)(x,\xi)$ is inverible for all non-zero $\xi\in T^*_GM$.\\

The following theorem shows one, among many others, important consequence of the ellipticity condition. 
\begin{theorem}[{\cite[Theorem III.5.8]{LM89}}]\label{Thm:MainEllipticOp}
Let $D:C^\infty(M,E)\longrightarrow C^\infty(M,E)$ be a symmetric $k$-th order elliptic differential operator over a closed Riemanian manifold $M$. Then $D$ is essentially self-adjoint and discrete. Moreover, we have an orthogonal decomposition 
\begin{align*}
L^2(E)=\bigoplus_{\lambda\in\spec D}E_\lambda,
\end{align*}
where $E_\lambda\subset C^\infty(M,E)$ is the finite dimensional eigenspace of $\lambda\in\spec(D)\subset \mathbb{R}$.
\end{theorem}

After this brief discussion on differential operators on vector bundles and in view of Theorem \ref{Thm:MainEllipticOp} it is natural to be interested in the particular  case of the construction described in Section \ref{Section:InduedOperatorsGen} when $R$ {\em is generated by a differential operator} $D$, that is, when $C^\infty_c(M,E)\subset\dom(R)$ and $R|_{C^\infty_c(M,E)}=D$. The following result states that in this case the induced operator $T$  on $M_0/G$ of Proposition \ref{Prop:OpT} is also generated by a differential operator of the same order. 
\begin{proposition}[{\cite[Theorem 2.4]{BH78}}]\label{Prop:SDiff}
If $R$ is generated by a differential operator $D$ of order $k$, then $T$ is also generated by a certain differential operator $D'$ of order $k$. Their principal symbols are related by the formula
\begin{equation*}
\sigma_P(D')(y,\xi)(\pi_G'(e))=\pi_G'(\sigma_P(D)(x,\pi_G^*\xi)(e)),
\end{equation*}
where $y\in M_0/G$, $\xi\in T^*_y (M_0/G)$, $x\in\pi^{-1}_G(y)$ and $e\in E'_x$. 
\end{proposition}

\begin{proof}
From Corollary \ref{Coro:PropPhi} we know that 
$$C^\infty_c(M_0/G,F)\subset\dom(T) \quad\text{and}\quad \Phi(C_c^\infty(M_0,E)^G)=C_c^\infty(M_0/G,F),$$
thus for $s\in C^\infty_c(M_0/G,F)$  we have $\supp(Ts)=\supp(\Phi(D(\Phi^{-1}s)))\subset\supp(s)$,
since $R$ is generated by the differential operator $D$. Now let $\{\phi_i\}_i$ be a partition of unity on $M_0/G$. Define for $s\in C^\infty(M_0/G,F)$ 
\begin{align*}
D'(s)\coloneqq \sum_{i}T(\phi_i s).
\end{align*}
Then $D'$ is a linear extension of $T|_{C^\infty_c(M_0/G,F)}$ to $C^\infty(M_0/G,F)$ and from the discussion above we have $\supp(D's)\subset\supp(s)$ for all $s\in C^\infty(M_0/G,F)$. Next we verify that $D'$ is indeed a differential operator of order $k$.  Fix $y\in M_0/G$, $x\in\pi^{-1}_G(y)$ and let  $s\in C^{\infty}(M_0/G,F)$ be such that $j_ks(y)=0$, we want to show $D's(y)=0$. Choose a cut-off function $\psi\in C_c^\infty(M_0/G)$ such that $\psi=1$ in a neighborhood around $y$, then from Corollary \ref{Coro:PropPhi} and Corollary \ref{Coro:jets} we get
\begin{align*}
D's(y)=D'(\psi s)(y)=\Phi(D(\Phi^{-1}\psi s))(\pi_G(x))=\pi'_G(D(\Phi^{-1}(\psi s)(x))=0. 
\end{align*}
Hence, $T$ is generated by a differential operator of order $k$.\\
To derive the relation between the principal symbols we just compute \eqref{Def:PrincipalSymbol} using Corollary \ref{Coro:PropPhi}. Concretely, choose $y\in M_0/G$, $x\in \pi_G^{-1}(y)$, $e\in E'_x$, $\xi\in T^*_y(M_0/G)$, $\phi\in C^{\infty}_c(M_0/G)$ so that $\phi(y)=0$ and $d\phi(y)=\xi$, and $s\in C_c(M_0/G,F)$ such that $s(y)=\pi'_G(e)$. Now we calculate
\begin{align*}
\sigma_P(D')(y,\xi)(\pi_G'(e))=&D'\left(\frac{(-1)^k}{k!}\phi^k s\right)(y)\\
=&\Phi\left(D\left(\Phi^{-1}\left(\frac{(-1)^k}{k!}\phi^k s\right)\right)\right)(\pi_G(x))\\
=&\pi_G'\left(D\left(\frac{(-1)^k}{k!}(\phi\circ\pi_G)^k \Phi^{-1}s\right)(x)\right)\\
=&\pi_G'(\sigma(D)(x,\pi_G^*\xi)(e)). 
\end{align*}
\end{proof}
\begin{coro}
The operator $D'$ is elliptic if $D$ is transversally elliptic. 
\end{coro}

\section{Adabiatic limit of the eta invariant}\label{Sec:AdiabLimit}

This chapter is intended to collect of various kown results from the literature which will serve as elements of a toolbox for the subsequent chapters.  In spite of the wide range of topics presened here, there is a concrete goal: to understand Dai's formula for the adiabatic limit of the eta invariant associated to the signature operator studied in detail in \cite{D91}. This formula contains various kind of terms from different nature and it is essential for our objective to understand their definitions and main properties. We begin by collecting some geometric and analytic properties of the signature operator and its relation to the signature of manifolds. Then, we explain how to construct this operator in  the framework of Hilbert complexes, introduced by Br\"uning and Lesch in \cite{BL92}, and  its relation to $L^2$-cohomology. Next we describe how to decompose the exterior derivative on a fibered manifold. This decomposition will be a recurrent tool throughout the whole document. In addition, we discuss an important invariant associated to the Leray spectral sequence of such fibration.  This will allow us to understand the right hand side of Dai's result.  Finally we will briefly recall some important equivariant techniques, which will play a fundamental role in the proof of Lott's formula for the $S^1$-signature. These equivariant methods are treated in detail in \cite{BGV}.

\subsection{The signature operator}\label{Section:SignatureOp} 
Let $M$ be an oriented Riemannian manifold of dimension $m$. The metric on $M$, denoted by $\inner{\cdot}{\cdot}$, defines two {\em musical isomorphisms}
\begin{align}\label{Eqn:Musical}
\flat:
\xymatrixcolsep{2cm}\xymatrixrowsep{0.01cm}\xymatrix{
TM \ar[r] & T^*M\\
Y \ar@{|->}[r] & Y^\flat\coloneqq \inner{Y}{\cdot} 
}
\quad\quad\quad
\sharp:
\xymatrixcolsep{2cm}\xymatrixrowsep{0.01cm}\xymatrix{
T^*M \ar[r] & TM\\
\omega \ar@{|->}[r] & \flat^{-1}(\omega). 
}
\end{align}
These maps induce a metric on the cotangent bundle $\wedge T^*M$, which we still denote by $\inner{\cdot}{\cdot}$, so that $\wedge T^*M=\bigoplus_{r}\wedge^r T^* M$ is a orthogonal decomposition. The metric and the orientation of $M$ allow us to define the  {\em Hodge star operator} $*:\wedge^r T^*M\longrightarrow \wedge ^{m-r} T^*M$ by the relation
\begin{equation}
\alpha\wedge *\beta\coloneqq \inner{\alpha}{\beta}\vol_M \quad\text{for}\quad \alpha,\beta\in\wedge^r T^*M,
\end{equation}
where $\vol_M$ is the Riemannian volume element defined by the orientation and the metric. 
This operator satisfies  $\inner{*\alpha}{*\beta}=\inner{\alpha}{\beta}$ for $\alpha,\beta\in\wedge^r T^*M$ and $*^2=(-1)^{r(m-r)}$ on $r$-forms (\cite[Proposition 4.7]{M00}). \\
For $\omega,\omega\in\Omega_c(M)$  differential forms with compact support, we can use the Hodge star operator to express the  $L^2$-inner product of  \eqref{Eqn:HermE}  as
$$(\omega,\omega')_{L^2(\wedge T^* M)}=\int_M\inner{\omega}{\omega'}\vol_M=\int _M \omega\wedge *\omega '.$$
Let $d:\Omega_c^r(M)\longrightarrow\Omega_c^{r+1}(M)$ be the exterior derivative on $M$. As this operator is a first order differential, we can consider its formal adjoint $d^\dagger:\Omega_c^r(M)\longrightarrow\Omega_c^{r-1}(M)$ defined by the condition
\begin{align*}
(d\omega,\omega')_{L^2(\wedge T^* M)}=(\omega,d^\dagger \omega')_{L^2(\wedge T^* M)}\quad\forall\omega,\omega'\in\Omega_c(M).
\end{align*} 
We can express $d^\dagger$ explicitly in terms of the Hodge star operator as $d^\dagger=(-1)^{m(r+1)+1}*d*$ on $r$-forms.\\

The Hermitian vector bundle $\wedge_\mathbb{C}T^*M\coloneqq \wedge T^*M\otimes\mathbb{C}$ carries an additional geometric structure induced from the metric, the Clifford action (\cite[Section 3.6]{BGV},\cite{LM89}). That is, $\wedge_\mathbb{C}T^*M$ is equipped with the left Clifford multiplication,
\begin{align*}
c:\xymatrixcolsep{2cm}\xymatrixrowsep{0.01cm}\xymatrix{
T^*M \ar[r] & \End(\wedge_\mathbb{C} T^*M)\\
\alpha\ar@{|->}[r] & c(\alpha)\coloneqq \alpha\wedge-\iota_{\alpha^\sharp}.
}
\end{align*}
It is easy to verify the relation $c(\alpha)c(\beta)+c(\beta)c(\alpha)=-2\inner{\alpha}{\beta}$, for all $\alpha,\beta\in T^*M$. The left Clifford multiplication satisfies $(c(\alpha)\omega,\omega')_{L^2(\wedge T^* M)}=-(\omega,c(\alpha)\omega')_{L^2(\wedge T^* M)},$ i.e. $c(\alpha)^\dagger=-c(\alpha)$. This relation follows from the fact that $(\alpha\wedge)^\dagger=\iota_{\alpha^\sharp}$ with respect to the $L^2$-inner product. One can also define the {\em right Clifford multiplication} by $\widehat{c}(\alpha)\coloneqq\alpha\wedge+\iota_{\alpha^\sharp}$.\\

To each Clifford bundle we can associate its corresponding {\em chirality operator} or {\em complex volume element} (\cite[Lemma 3.17]{BGV}), which for $\wedge_\mathbb{C} T^*M$ equipped with the left Clifford action described above it is explicitly given by 
\begin{align}\label{Def:Chirl}
\star\xymatrixcolsep{2cm}\xymatrixrowsep{0.01cm}\xymatrix{
\wedge^r_\mathbb{C}T^*M \ar[r] & \wedge^{m-r}_\mathbb{C}T^*M\\
\omega \ar@{|->}[r] & \star\omega\coloneqq i^{[(m+1)/2]+2mr+r(r-1)}*\omega,
}
\end{align}
where  $[\cdot]$ denotes the integer part and $i\coloneqq \sqrt{-1}$. The following proposition summarizes some important properties of the chirality operator. 
\begin{proposition}[{\cite[Proposition 3.58]{BGV}}]\label{Prop:Chirl}
The chirality operator $\star$ satisfies:
\begin{enumerate}
\item $\star^2=1$.
\item $\star^\dagger =\star$.
\item For $\alpha\in T^*M$, $\star (\alpha \wedge)\star=(-1)^m\iota_{\alpha^\sharp}$.
\item $d^\dagger =(-1)^{m+1}\star d\star$.
\end{enumerate}
\end{proposition}

\begin{coro}
The chirality operator and the Clifford map satisfy
\begin{align*}
\star c(\alpha)&=(-1)^{m+1}c(\alpha)\star,\\
\star \widehat{c}(\alpha)&=(-1)^{m}\widehat{c}(\alpha)\star.
\end{align*}
\end{coro}

From the last proposition we see that $\star$ is a self-adjoint involution in $\wedge_\mathbb{C} T^*M$. Another endomorphism of this type is the so called {\em Gau\ss-Bonnet involution} defined by $\varepsilon\coloneqq (-1)^r$ on $r$-forms. This involution satisfies the relations
\begin{align}
\varepsilon\star=&(-1)^{m}\star\varepsilon, \label{PropertiesEpsilonStar}\\
\varepsilon c(\alpha)=&-c(\alpha)\varepsilon, \label{PropertiesEpsilonCliff}\\
\varepsilon \widehat{c}(\alpha)=&-\widehat{c}(\alpha)\varepsilon.
\end{align}

Given the Clifford bundle structure described above we can associate to it a first order differential operator $D:\Omega_c(M)\longrightarrow\Omega_c(M)$ called the {\em Dirac operator}. For our case of study it is explicitly given by $D=d+d^\dagger$. This operator is also known in the literature as the {\em Hodge-de Rahm operator}. Clearly $D$ is a symmetric operator on $\Omega_c(M)$. The operator $D$ satisfies the following relations
\begin{align}
D\star=&(-1)^{m+1}\star D,\label{Eqn:DChirl}\\
\varepsilon D=&-D\varepsilon.
\end{align}
 
\begin{proposition}\label{Prop:PSD}
The principal symbol of the  Hodge-de Rham operator $D$ is given by $\sigma_P(D)(x,\xi)=-ic(\xi)$; therefore, it is an elliptic operator. 
\end{proposition}
\begin{proof}
Using Equation \eqref{Def:PrincipalSymbol} we can compute the principal symbols
\begin{align*}
\sigma_P(d)(x,\xi)=&-i\xi\wedge,\\
\sigma_P(d^\dagger)(x,\xi)=&i\iota_{\xi^\sharp},
\end{align*}
and conclude that $\sigma_P(D)(x,\xi)=-ic(\xi)$. Note that $-c(\xi)c(\xi)= \norm{\xi}^2\neq 0$ if $\xi\neq 0$, hence $D$ is elliptic. 
\end{proof}
The Hodge-de Rham operator satisfies the following important property
\begin{proposition}[{\cite[Proposition 3.38]{BGV}}]\label{Prop:CommDirOpFunct}
If $f\in C^\infty(M)$ and $\omega\in\Omega(M)$, then 
\begin{align*}
D(f\omega)= c(df)\omega+fD\omega.
\end{align*}
\end{proposition}

Another relevant differential operator associated to this Clifford bundle is the {\em Laplacian} $\Delta\coloneqq D^2=dd^\dagger+d^\dagger d$, defined also on $\Omega_c(M)$.  

\begin{proposition}
The operator $\Delta:\Omega_c(M)\longrightarrow\Omega_c(M)$ is a second order, symmetric differential operator which satisfies the relations
\begin{enumerate}
\item $\varepsilon\Delta=\Delta\varepsilon$. 
\item $\star\Delta=\Delta\star$. 
\item $(\Delta \omega,\omega)_{L^2(\wedge T^* M)}\geq 0$ for all $\omega\in \Omega_c(M)$. 
\item $\sigma_P(\Delta)(x,\xi)=\norm{\xi}^2$.
\end{enumerate}
In particular, $\Delta$ is an elliptic operator. 
\end{proposition}

\begin{remark}\label{Rmk:ScaleHodgeStar}
For $a>0$ consider the rescaled metric $g(a)\coloneqq a^2 g$. We want to compare the Laplacians $\Delta_{g(a)}$ and $\Delta_g$  associated to the metrics $g$ and $g(a)$ respectively. To do this we first compare the Hodge star operators $*_a$ and $*$.  From definition we compute for $r$-forms $\omega_0$ and $\omega_1$, 
\begin{align*}
\omega_0\wedge *_{a}\omega_1=\inner{\omega_0}{\omega_1}_a  \vol_{g(a)}=(a^{-2r}\inner{\omega_0}{\omega_1})a^{m}\vol_{g}=\omega_0\wedge (a^{m-2r}*\omega_1). 
\end{align*}
Thus, we see that $*_a=a^{m-2r}*$ on $r$-forms and in particular
\begin{align*}
*_a d*_a=a^{m-2(m-r+1)} * d a^{m-2r} *=a^{-2}*d*. 
\end{align*}
Hence, we see that the Laplacians are related by $\Delta_{g(a)}=  a^{-2}\Delta_{g}$.
\end{remark}

Let us assume $m$ is even. By the properties of $\star$ described in Proposition \ref{Prop:Chirl}  we see we can decompose  $\Omega(M)=\Omega^+(M)\oplus\Omega^-(M)$ where $\Omega^\pm(M)$ are the $\pm1$-eigenspaces of $\star$. These eigenspaces are known as spaces of {\em self-dual} ($+1$-eigenvalue) and {\em anti-self-dual} ($-1$-eigenvalue) forms respectively. From \eqref{Eqn:DChirl} it follows that  $D$ can be also decomposed with respect to this splitting as 
\begin{equation}
D=\left(
\begin{array}{cc}
0 & D^-\\
D^+ & 0
\end{array}
\right),
\end{equation}
where $D:\Omega_c^\pm(M)\longrightarrow\Omega_c^{\mp}(M)$. We define the {\em signature operator} as the component
\begin{equation}\label{Eqn:SigOp}
D^+:\Omega_c^+(M)\longrightarrow\Omega_c^-(M).
\end{equation}

\subsubsection{The signature theorem for closed manifolds}
Suppose now $M$ closed and $m=4k$. In this case, Poincar\'e duality ensures that the quadratic form
\begin{align*}
\xymatrixcolsep{2cm}\xymatrixrowsep{0.01cm}\xymatrix{
H^{2k}(M;\mathbb{C})\times H^{2k}(M;\mathbb{C}) \ar[r] & H^{4k}(M;\mathbb{C}) \ar[r] & \mathbb{C}\\
(\omega,\omega') \ar@{|->}[r] & \omega\wedge\omega' \ar@{|->}[r]  & \displaystyle{\int_M \omega\wedge\omega'},
}
\end{align*}
is non-degenerate. Here $H^*(M;\mathbb{C})$ denotes the de Rham cohomology ring of $M$ with $\mathbb{C}$ coefficients.
\begin{definition} 
The {\em signature of $M$}, denoted by $\sigma(M)$, is defined to be the signature of this quadratic form. 
\end{definition}

\begin{proposition}[{\cite[Chapter 2]{HZ74}}]\label{Prop:PropSignatureClosed}
The signature satisfies the following properties:
\begin{enumerate}
\item Let  $-M$ denote the manifold $M$ with reversed orientation,  then $\sigma(-M)=-\sigma(M)$.
\item If $M=\partial \widetilde{M}$ is the boundary of a compact manifold $\widetilde{M}$, then $\sigma(M)=0$.
\item $\sigma(M_1\times M_2)=\sigma(M_1)\sigma(M_2)$.
\end{enumerate}
\end{proposition}
Equipp $M$ with a Riemannian metric $g^{TM}$. As $M$ is assumed to be a closed manifold then $D$ is a Fredholm operator since it is elliptic (\cite[Theorem III.5.2]{LM89}). Using the Hodge Theorem (\cite[Proposition 3.48]{BGV}) one verifies that the index of the signature operator is equal to the signature of $M$, that is, $\ind(D^{+})=\sigma(M)$.  
From this relation and applying the celebrated Atiyah-Singer index theorem, one can show that the signature of $M$ can be expressed as (\cite[Theorem 4.8]{BGV}, \cite[Theorem III.13.9]{LM89})
\begin{equation}\label{Eqn:SignTheoClosed}
\sigma(M)=\int_M L(TM,g^{TM}),
\end{equation}
where $L(TM,g^{TM})$ is {\em Hirzebruch's $L$-polynomial} of the Levi-Civita connection of $g^{TM}$ (\cite[Section 1.5]{BGV}).  It is important to bear in mind that $L(TM,g^{TM})$ is an even polynomial in the components of the curvature $\Omega\in\Omega(M,\mathfrak{gl}(4k))$. The first terms are 
\begin{align*}
L(TM,g^{TM})=1-\frac{1}{24\pi^2}\tr(\Omega\wedge\Omega)+\cdots \in\Omega(M).
\end{align*}
The relation \eqref{Eqn:SignTheoClosed}, known as the {\em signature theorem}, was first proven by Hirzebruch using Thom's computations of the oriented cobordism rings (\cite[Chapter 19]{MS74}).
\begin{remark}
Actually, since $M$ is closed, the right hand side of Equation \eqref{Eqn:SignTheoClosed} can be computed using any connection since the difference will be an exact form, which integrates to zero. This of course something to be expected since the signature of $M$ does not depend on the metric. 
\end{remark} 

\subsubsection{The signature of theorem for compact manifolds with boundary}\label{Section:SigThmMB}

When the dimension of $M$ is odd one can also define a signature operator. In this case, set $\mathbb{R}_+\coloneqq [0,\infty)$ and  consider the even dimensional manifold  $M\times \mathbb{R}_+$ equipped with the product metric $dr^2+g^{TM}$ where $r\geq 0$ denotes the $\mathbb{R}_+$ coordinate. It can be shown  that the associated signature operator $D^{+}_{M\times\mathbb{R}_+}$ on $M\times\mathbb{R}_+$ can be written as (\cite[Section 4.1]{D91}\cite[Section 4.3]{G95} and Section \ref{Section:DecConeOp} below)
\begin{equation}
D^{+}_{M\times\mathbb{R}_+}=c(dr)\left(\frac{\partial}{\partial r}+A\right),
\end{equation}
where $A$ a first order, self-adjoint elliptic operator when restricted to $M$. The operator $A$, which can be seen as the tangential part of the signature operator $D^{+}_{M\times\mathbb{R}_+}$, is called the {\em tangential} or {\em odd signature operator}. This operator can be written explicitly using only geometric information on $M$ as
\begin{equation}\label{Eqn:DefOddSignOp}
A=\star_M(d_M+d^\dagger_M),
\end{equation}
where $\star_M$ and $d_M+d_M^\dagger $ are the corresponding chirality and Hodge-de Rham  operator on $M$. For the Gau\ss-Bonnet involution $\varepsilon_M$ on $M$ we can deduce from \eqref{PropertiesEpsilonStar}, 
\begin{align*}
\varepsilon_M A=\varepsilon_M\star_M(d_M+d^\dagger_M)=-\star_M\varepsilon_M(d_M+d^\dagger_M)=\star_M(d_M+d^\dagger_M)\varepsilon_M=A\varepsilon_M,
\end{align*}
thus we can decompose $A=A^\text{ev}\oplus A^\text{odd}$ where $A^\text{ev/odd}:\Omega^\text{ev/odd}(M)\longrightarrow\Omega^\text{ev/odd}(M)$. 

\begin{remark}
Observe that $A^2=\Delta_M$, thus $A$ is a generalized Dirac operator in the sense of \cite[Definition 2.2]{BGV}.
\end{remark}

Although the index of an elliptic operator over a closed odd dimensional manifold always vanishes (\cite[Theorem III.13.12]{LM89}), the odd signature operator described above plays an essential role in the generalization of the signature theorem for manifolds with boundary derived by Atiyah, Patodi and Singer in the fundamental article \cite{APSI}. We will briefly recall the main ingredients of its formulation. First we come to the definition of the signature for a manifold with boundary (\cite[Section I.2]{HZ74}). 
\begin{definition}
Let $N$ be a compact, oriented $4k$-dimensional manifold with boundary $\partial N=M$. The {\em signature} $\sigma(N)$ of $N$ is defined as the signature of the quadratic form induced by the cup product on the image of the map $H^{2k}(N,M)\longrightarrow H^{2k}(N)$, which is the associated map in cohomology of the inclusion of the pair $(N,\emptyset)\longrightarrow (N,M)$. 
\end{definition}
The following remarkable result can be shown using purely topological methods (see \cite[Theorem 2.8.1]{H95}, \cite[Chapter 2]{HZ74}).
\begin{proposition}[Novikov additivity of the signature]\label{Prop:Novikov}
Let $N_1,N_2$ be two oriented $4k$-dimensional manifolds with boundary $\partial N_1=\partial N_2=M$. By reverting the orientation of $N_2$ we can glue it to $N_1$ along $M$  and obtain a closed oriented $4k$-dimensional manifold $N\coloneqq N_1\cup_{M}(-N_2)$. The signature of $N$ satisfies $\sigma(N)=\sigma(N_1)+\sigma(-N_2)$.
\end{proposition}

In view of formula \eqref{Eqn:SignTheoClosed} for closed manifolds, a natural question which arises is whether one can express the signature of a manifold with boundary as the integral of Hirzebruch's $L$-polynomial with respect to a given Riemannian metric on $N$. By computing some examples (e.g. Lens spaces) one can see that \eqref{Eqn:SignTheoClosed} does not hold true. The correction term, known as the signature defect, was discovered by Atiyah, Patodi and Singer long after (around 10 years) the proof of the index theorem for closed manifolds. The main difficulty was to set up an appropriate boundary condition. To describe it let us consider as before the cylinder $M\times\mathbb{R}_+$ with the product metric $dr^2\oplus g^{TM}$. The model operator considered by Atiyah, Patodi and Singer on the cylinder is
\begin{align}\label{Eqn:ModelOpCylinder}
\frac{\partial }{\partial r}+A,
\end{align} 
where $A$ is the odd signature operator on $M$. Since $A$ is elliptic and self-adjoint we obtain from Theorem \ref{Thm:MainEllipticOp}  a discrete orthogonal spectral resolution 
$$L^2(\wedge T^* M)=\bigoplus_{\lambda\in\spec A}E_\lambda,$$
where $E_\lambda\subset \Omega(M)$ is the finite dimensional eigenspace associated with the eigenvalue $\lambda$. Let $P\coloneqq Q_{\geq}(A)$ denote the spectral projection onto the space generated by the eigensections corresponding to non-negative eigenvalues.  The boundary condition imposed for the model operator \eqref{Eqn:ModelOpCylinder} is $Ps|_M\coloneqq Ps(\cdot, 0)=0$. This boundary condition is now called the {\em APS boundary condition}. In \cite[Section 2]{APSI} they constructed the solution operator associated to this boundary value problem. They also showed that the spectral function
\begin{equation}\label{Def:eta}
\eta_A(z)\coloneqq\sum_{\substack{\lambda\in\spec A \\ \lambda\neq 0}}\sign (\lambda)|\lambda|^{-z},
\end{equation}
is holomorphic near $z=0$ and its value at zero  $\eta_A(0)$, called the {\em eta invariant} of $A$, is proportional to the signature defect. Observe from the decomposition $A=A^{\text{ev}}\oplus A^{\text{odd}}$ that $\eta_{A}(0)=2\eta_{A^{\text{ev}}}(0)$. 

\begin{remark}[Adjoint boundary condition]\label{Rmk:AdjointBoundCond}
The formal adjoint of the operator \eqref{Eqn:ModelOpCylinder} is simply
\begin{align*}
-\frac{\partial }{\partial r}+A.
\end{align*}
We want to investigate now what would be the corresponding adjoint boundary condition associated to $P$. Let $s_1,s_2\in C_c(M\times\mathbb{R}_+)$, then integrating by parts
\begin{align*}
\left(\left(\frac{\partial }{\partial r}+A\right)s_1,s_2\right)_{L^2(M\times\mathbb{R}_+)}
-\left(s_1,\left(-\frac{\partial }{\partial r}+A\right)s_2\right)_{L^2(M\times\mathbb{R}_+)}
=(s_1|_M,s_2|_M)_{L^2(M)}.
\end{align*}
Thus, if $Ps_1|_M=0$ then, in order to get $(s_1|_M,s_2|_M)_{L^2(M)}=0$ we must require $s_2|_M$ to be in the image of $P$, i.e. $(I-P)s_2|_M=0$. That is, the adjoint boundary condition is defined by the spectral projection $I-P=Q_<(A)$.
\end{remark}

Now let us describe how these results can be used to derive the signature formula for a closed, oriented Riemannian manifold $N$ with boundary $M$. Assume the metric on $N$ is a product metric near the boundary $M$. Using the solution operator of \eqref{Eqn:ModelOpCylinder} with the APS boundary condition Atiyah, Patodi and Singer showed that the signature operator $D^+$ on $N$ has a well defined index and derived the index formula 
\begin{equation}\label{Eqn:IndexDBoundary}
\ind(D^+)=\int_N L(TN,g^{TN})-\dim(\ker A^{\text{ev}})-\eta_{A^{\text{ev}}}(0).
\end{equation}
On the other hand, they also showed using a version of the Hodge theorem for manifolds with cylindrical ends, the equality
\begin{equation}\label{Eqn:IndexVSSignatureAPS}
\ind(D^{+})=\sigma(N)-\dim(\ker A^{\text{ev}}).
\end{equation}
Combining \eqref{Eqn:IndexDBoundary} and \eqref{Eqn:IndexVSSignatureAPS} one concludes
\begin{align*}
\sigma(N)=\int_N L(TN,g^{TN})-\eta_{A^{\text{ev}}}(0).
\end{align*}
One can relax the metric condition around the boundary by introducing a transgression term, which is a secondary characteristic class, that interpolates the $L$-polynomial between the given metric and the product metric close to the boundary. These transgression forms will be discussed in Section \ref{Sect:PfoofSignatureFormula}.

\begin{theorem}[{\cite{APSI},\cite[Theorem 4.3.10]{G95}}]\label{Thm:SignThmMBound}
Let $(N,g^{TN})$ be a compact, oriented Riemannian manifold of dimension $4k$ with boundary $\partial N=M$. Then we can compute the signature of $N$ as
$$\sigma(N)=\int_N L\left(TN,g^{TN}\right)-\int_M TL(g^{TN})-\eta_{A^\textnormal{ev}}(0).$$
Here $L(TN,g^{TN})$ is Hirzebruch's $L$-polynomial constructed from the Riemannian metric $g^{TN}$, $TL(g^{TN})$ is a transgression form and  $\eta_{A^{\textnormal{ev}}}(0)$ is the eta invariant of the even part of the tangential signature operator $A^{\textnormal{ev}}$. 
\end{theorem}

\begin{remark}[Induced orientation on the boundary]\label{Rmk:OrBound}
Given an orientation on $N$ we choose an orientation on the boundary $\partial N=M$ by requiring a local orthonormal basis $\{e_1, \cdots, e_{4k-1}\}$ of $TM$ to be oriented if and only if $\{-\partial_r, e_1, \cdots, e_{4k-1}\}$ is oriented in $TN$. Here $r\geq 0$ denotes the inward normal coordinate close to the boundary. 
\end{remark}
\begin{remark}
We will show in Section \ref{Section:DiracSystems} how to derive Novikov's additivity formula for the signature (Proposition \ref{Prop:Novikov}) using  Theorem \ref{Thm:SignThmMBound} and a gluing index theorem.
\end{remark}

\subsubsection{Vanishing results for the eta invariant}\label{Section:VanishEta}
It is important to emphasize that the fact that the eta invariant is finite at $z=0$ for operators arising in Riemannian geometry (studied in \cite{APSI}), for example the odd signature operator, is a consequence of the APS index theorem \cite[Theorem 3.10]{APSI}. This regularity result can be generalized to a larger family of operators.  Indeed, let $A$ be a self-adjoint elliptic pseudo-differential operator over a closed manifold $M$, then $A$ is discrete and therefore we can define its corresponding  eta function $\eta_A(z)$ by means of \eqref{Def:eta}. Atiyah, Patodi and Singer proved in \cite{APSIII}, using the regularity result for geometric operators and some K-theoretical methods, that $\eta_A(z)$ is holomorphic at $z=0$ for odd dimensional manifolds 
(\cite[Theorem 4.5]{APSIII}). \\

Let us now recall some vanishing results (\cite[Remark 3]{APSI}) of the eta invariant for geometric operators on a closed manifold $M$ of dimension $\dim M=2n-1$ for $n$ odd, i.e. $n=2\ell-1$. Both observations come from the representation theory of Clifford algebras. Let $A$ denote a geometric operator (think about the odd signature operator). We have to distinguish two cases separately:
\begin{enumerate}
\item If $\ell=2k$ then  $n=2(2k)-1\equiv 3(\text{mod 4})$ and $\dim M=8k-3\equiv 5(\text{mod 8})$.\\
In this case the associated spinor bundle is endowed with quaternionic structure (\cite[Remark I.5.13]{LM89}) which anti-commutes with $A$. This shows that the spectrum of $A$ is symmetric about the origin which implies $\eta_A(0)=0$.
\item If $\ell=2k-1$ then  $n=2(2k-1)-1\equiv 1(\text{mod 4})$ and $\dim M=8k-7=1(\text{mod 8})$.\\
In this case the operator $A$ is of the form $A=iB$ where $B$ is a real skew-adjoint operator. We claim this implies that also in this case the spectrum of $A$ is symmetric about the origin. \\
This follows from the fact that since $B$ is similar to its transpose operator then the must have the same eignevalues, which must be purely imaginary by the spectral theorem. In particular, if $i\lambda\in\spec(B)$ then we must have also $-i\lambda\in\spec(B)$. 
\end{enumerate}

\subsubsection{The Euler characteristic}
Another well known topological invariant, which can also be computed as an index, is the Euler characteristic. If $M$ is a closed and oriented manifold, then its {\em Euler characteristic} is defined by
\begin{align*}
\chi(M)\coloneqq \sum_{j=0}^{\dim M} \dim H^j(M;\mathbb{R}). 
\end{align*} 
The famous {\em Chern-Gau\ss-Bonnet theorem} states that we can compute this invariant as 
\begin{align*}
\chi(M)=\int_{M}e(TM, g^{TM}), 
\end{align*}
where $e(TM, g^{TM})$ is the {\em Euler class} of the tangent bundle, which can be constructed as a polynomial in the curvature of any Riemannian metric. For a purely topological proof see \cite[Chapter 11, Appendix C]{MS74}. On the other hand, this theorem can also be obtained from the Atiyah-Singer index theorem. Indeed, it is not hard to verify, using the Hodge decomposition, that the index of the graded Hodge-de Rham operator $D^{\text{ev}}:\Omega^{\text{ev}}(M)\longrightarrow \Omega^{\text{odd}}$, with respect to the involution $\varepsilon$, computes the Euler characteristic of $M$ (see \cite[Section 4.1]{BGV}). \\

In the case of a compact and oriented manifold $N$ with non-empty boundary $\partial N =M$, we can define a {\em relative Euler characteristic} by means of the relative cohomology groups 
\begin{align*}
\chi(N,M)\coloneqq \sum_{j=0}^{\dim N}(-1)^j \dim H^j(N,M;\mathbb{R}). 
\end{align*} 
Using the long exact sequence in cohomology of the pair $(N,M)$,
\begin{center}
\begin{tikzpicture}[descr/.style={fill=white,inner sep=1.5pt}]
        \matrix (m) [
            matrix of math nodes,
            row sep=1em,
            column sep=2.5em,
            text height=1.5ex, text depth=0.25ex
        ]
        { 0 & H^0(N,M) & H^0(N) & H^0(M) \\
            & H^1(N,M) & H^1(N) & H^1(M) \\
             & H^2(N, M) & H^2(N) & H^2(M) \\
            &\cdots \mbox{}         &                 & \mbox{}         \\
        };

        \path[overlay,->, font=\scriptsize,>=latex]
        (m-1-1) edge (m-1-2) 
        (m-1-2) edge (m-1-3)
        (m-1-3) edge (m-1-4)
        (m-1-4) edge[out=355,in=175] node {} (m-2-2)
        (m-2-2) edge (m-2-3)
        (m-2-3) edge (m-2-4)
        (m-2-4) edge[out=355,in=175] node {} (m-3-2)
        (m-3-2) edge (m-3-3)
        (m-3-3) edge (m-3-4)
        (m-3-4) edge[out=355,in=175] node {} (m-4-2);
\end{tikzpicture}
\end{center}
one easily verifies the relation
\begin{align}\label{Eqn:EulerCharLES}
\chi(N)=\chi(N,M)+\chi(M). 
\end{align}
We now analyze two different cases:
\begin{itemize}
\item If $\dim N$ is even then $\chi(M)=0$ by Poincar\'e duality, and therefore \eqref{Eqn:EulerCharLES} implies $\chi(N)=\chi(N,M)$.
\item If $\dim N$ is odd then we can consider the double $2N\coloneqq N\cup (-N)$ of $N$, which is the closed odd dimensional manifold  obtained by gluing two copies of $N$ along the boundary $M$. Again, by Poincar\'e duality $\chi(2N)=0$. On the other hand, we can use the Mayer-Vietoris sequence to verify $\chi(2N)=2\chi(N)-\chi(M)$. Combining this relation with \eqref{Eqn:EulerCharLES} we get $\chi(N)=-\chi(N,M)=\chi(M)/2$.
\end{itemize}

In contrast to the signature theorem, it is not hard to derive the generalization of the Chern-Gau\ss-Bonnet formula to manifolds with boundary. In fact, it can be deduced from the formula for closed manifolds.  

\begin{theorem}[{\cite[Section 4.2]{G95}}]\label{Thm:GBManBound}
Let $(N, g^{TN})$ be a compact, oriented Riemannian manifold with boundary $\partial N=M$. 
\begin{enumerate}
\item If the dimension of $N$ is even, then 
\begin{align*}
\chi(N)=\int_N e(TN, g^{TN})-\int_M Te(g^{TN}), 
\end{align*}
where $e(TN, g^{TN})$ is the Euler form on $N$ constructed from the metric $g^{TN}$ and  $Te(g^{TN})$ is a transgression  form. 
\item If the dimension of $N$ is odd, then 
\begin{align*}
\chi(N)=\frac{1}{2}\int_{M} e(TM, g^{TM}),
\end{align*}
where $e(TM, g^{TM})$ is the Euler form on $M$ constructed from induced metric  $g^{TM}$ on the boundary. 
\end{enumerate}
\end{theorem}
In the even dimensional case, the Gau\ss-Bonnet formula can be obtained as the index of the operator $D^{\text{ev}}$ induced from the maximal extension of the de Rham complex (the absolute boundary condition), which we will describe in some detail in the next section (\cite[Theorem 4.1]{BL92},\cite[Theorem 4.2.7]{G95}). Alternatively, it was proven in \cite[Section 4]{APSI} that Theorem \ref{Thm:GBManBound} can be derived as the difference between two index problems in which the eta invariant cancels out. The main idea is to add up the contributions of the  index theorem \cite[Theorem 3.10]{APSI} applied to even and odd components of the signature operator. Indeed, note that $D^{+}$ interchanges even forms and odd forms.

\subsection{$L^2$-cohomology}\label{Sec:L2}

In this section we are going to study the de Rham complex as a Hilbert complex in the sense of Br\"uning and Lesch (\cite{BL92}) by considering closed extensions of the exterior derivative. We will also recall the definition of $L^2$-cohomology and the $L^2$-signature and discuss when this latter invariant can be computed as an index. As an intermediate step we we will describe the strong Hodge theorem. For a detailed and comprehensive treatment of these topics refer to \cite{BL92}, \cite{CD09}, \cite{D11} or \cite{GL02}.\\

Let us continue in the general setting where $M$ is an oriented Riemannian manifold of dimension $m$. Consider the associated de Rham complex of $M$
\begin{align*}
\xymatrixcolsep{2pc}\xymatrix{
0 \ar[r] &\Omega_c^0(M) \ar[r]^-{d_0} &\Omega_c^1(M) \ar[r]^-{d_1} &\cdots \ar[r]^-{d_{m-1}} &\Omega_c^m(M) \ar[r] & 0.
}
\end{align*}
This elliptic complex admits {\em ideal boundary conditions}, i.e. consistent closed extensions of each $d_r:\Omega_c^r(M)\longrightarrow\Omega_c^{r+1}(M)$, such that it becomes a Hilbert complex  (\cite[Lemma 3.1]{BL92}). There are two particularly important choices: the minimal extension and the maximal extension of the exterior derivative (see Section \ref{Section:DiffOp}), i.e 
\begin{align*}
\xymatrixcolsep{2.1pc}\xymatrix{
0 \ar[r] &\dom(d_{0,\text{min}})\ar[r]^-{d_{0,\text{min}}} &\dom(d_{1,\text{min}}) \ar[r]^-{d_{1,\text{min}}} &\cdots \ar[r]^-{d_{m-1,\text{min}}} & \:\:\dom(d_{m,\text{min}}) \ar[r] & 0.
}
\end{align*}
and
\begin{align}\label{Diag:MaximalExtdeRham}
\xymatrixcolsep{2.1pc}\xymatrix{
0 \ar[r] &\dom(d_{0,\text{max}})\ar[r]^-{d_{0,\text{max}}} &\dom(d_{1,\text{max}}) \ar[r]^-{d_{1,\text{max}}} &\cdots \ar[r]^-{d_{m-1,\text{max}}} & \:\:\dom(d_{m,\text{max}}) \ar[r] & 0.
}
\end{align}
These are also known as the {\em relative} and {\em absolute} boundary conditions respectively.  In the case where $d_{r,\text{min}}=d_{r,\text{max}}$ for all $0\leq r\leq m$ we say that {\em the case of uniqueness} or {\em $L^2$-Stokes theorem} holds for $M$. 
\begin{remark}
From Lemma \ref{LemmaGL02} it follows that if $M$ is compact then the $L^2$-Stokes theorem holds for $M$. Moreover, Gaffney proved in \cite{G51} that this is also true for $M$ complete. 
\end{remark}

The cohomology groups of the complex \eqref{Diag:MaximalExtdeRham}, denoted by $H^*_{(2)}(M)$, are called the {\em $L^2$-cohomology groups} of the Riemannian manifold $M$. Since the de Rham complex is elliptic, the cohomology of the complex \eqref{Diag:MaximalExtdeRham} is isomorphic to the cohomology of the complex of smooth $L^2$-forms (\cite[Theorem 3.5]{BL92}),
\begin{align*}
\xymatrixcolsep{2pc}\xymatrix{
0 \ar[r] &\Omega_{(2)}^0(M) \ar[r]^-{d_0} &\Omega_{(2)}^1(M) \ar[r]^-{d_1} &\cdots \ar[r]^-{d_{m-1}} &\Omega_{(2)}^m(M) \ar[r] & 0, 
}
\end{align*}
where $\Omega^r_{(2)}(M)\coloneqq \{\omega\in\Omega^r(M)\cap L^2(\wedge T^*M)\:|\:d_r\omega\in L^2(\wedge T^*M)\}$. That is, 
\begin{align*}
H^r_{(2)}(M)\cong \frac{\{\omega\in\Omega^r(M)\cap L^2(\wedge T^*M)\:|\:d_r\omega=0\}}{\{d_{r-1}\omega\:|\: \omega\in\Omega^{r-1}(M)\cap L^2(\wedge T^*M),\: d_{r-1}\omega\in\Omega^{r}(M)\cap L^2(\wedge T^*M)\}}.
\end{align*}
In view of this result one defines the space of $L^2$-harmonic forms as
\begin{align*}
\mathcal{H}^*_{(2)}(M)\coloneqq\{\omega\in\Omega^*(M)\cap L^2(\wedge T^*M)\:|\: d\omega=0, d^\dagger\omega=0\}.
\end{align*}
When the natural inclusion $\mathcal{H}_{(2)}(M)\longrightarrow H^*_{(2)}(M)$ is an isomorphism we say that the {\em strong Hodge theorem} holds. Let us analyze under which conditions this is the case. To begin with, consider the weak Hodge orthogonal decomposition of \cite[Lemma 2.1]{BL92} induced by the maximal extension, 
\begin{align}\label{Eqn:WeakHodgeMax}
L^2(\wedge^r T^*M)=\widehat{\mathcal{H}}^r(M)\oplus \overline{\text{ran}(d_{r-1,\text{max}})}
\oplus \overline{\text{ran}(d_{r,\text{max}}^*)}.
\end{align}
Here $\widehat{\mathcal{H}}^r(M)\coloneqq\ker(d_{r,\text{max}})\cap \ker(d_{r-1,\text{max}}^*)$. From this decomposition and the definition of the $L^2$-cohomology groups we see 
\begin{align*}
H^{r}_{(2)}(M)=\widehat{\mathcal{H}}^r(M)\oplus\left(\frac{\overline{\text{ran}(d_{r-1,\text{max}})}}{\text{ran}(d_{r-1,\text{max}})}\right). 
\end{align*}
In particular, $H^{r}_{(2)}(M)\cong \widehat{\mathcal{H}}^r(M)$ if, and only if, $d_{r-1,\text{max}}$ has closed image; which is the case if, for example, $H^{r}_{(2)}(M)$ is finite dimensional. This follows from the restriction
\begin{align*}
\dim\left(\frac{\overline{\text{ran}(d_{r-1,\text{max}})}}{\text{ran}(d_{r-1,\text{max}})}\right)=\text{$0$ or $\infty$},
\end{align*}
by the closed graph theorem (\cite[Lemma 19.1.1]{HOIII}).\\

Now we want to understand the relation between $\mathcal{H}^*_{(2)}(M)$ and $\widehat{\mathcal{H}}^*(M)$. On the one hand we have the inclusion
\begin{align*}
\ker(d_{\text{max}})\cap \ker(d^\dagger_{\text{max}})\subseteq \ker(d+d^\dagger)_\text{max},
\end{align*}
and on the other $\ker(d+d^\dagger)_\text{max}\subset \Omega(M)$ by elliptic regularity. This shows 
\begin{align}\label{Eqn:L2HarmonicForms}
\mathcal{H}^*_{(2)}(M)=\ker(d_\text{max})\cap\ker(d^\dagger_\text{max})=\ker(d_\text{max})\cap\ker((d_\text{min})^*). 
\end{align}
Summarizing we have the following result.
\begin{proposition}[{\cite{Ch79}}]\label{Prop:StrongHodgeTheorem}
If the $L^2$-cohomology of $M$ has finite dimension and the $L^2$-Stokes’ theorem holds then the strong Hodge theorem holds, i.e  the $L^2$-cohomology of $M$ is isomorphic to the space of $L^2$-harmonic forms.
\end{proposition}

\begin{remark}
When $M$ is closed all the conditions of Proposition \ref{Prop:StrongHodgeTheorem} are satisfied and we obtain the usual Hodge theorem. 
\end{remark}

\begin{remark}\label{L2StokesQI}
The $L^2$-Stokes theorem condition as well as the $L^2$-cohomology groups are not topological invariants. They depend on the quasi-isometry class of the Riemannian metric. Two metrics $g$ and $g'$ are {\em quasi-isometric} if there exists a constant $C>0$ such that $C^{-1}g\leq g'\leq C g.$
\end{remark}
The following result will be used in later chapters. 
\begin{proposition}[{\cite[Lemma 2.3]{BL92}}]\label{Prop:Lemma2.3BL}
If the Hodge-de Rham operator $D=d+d^\dagger$ is essentially self-adjoint, then the $L^2$-Stokes theorem holds. 
\end{proposition}

Let us now describe how the Hodge star operator behaves under the maximal and minimal extensions. First of all note that this operator induces a unitary map
$$*:L^2(\wedge^r T^*M)\longrightarrow L^2(\wedge^{m-r} T^*M).$$
It is  not hard to see that 
\begin{align*}
*\dom(d_{r,\text{min/max}})\subseteq \dom(d^\dagger_{m-r,\text{min/max}})
\end{align*}
and
\begin{align*}
d^\dagger_{r,\text{min/max}}= (-1)^{m(r+1)+1}* d_{m-r,\text{min/max}} *.
\end{align*}
\begin{example}
Let $\omega\in\dom(d_{r,\text{min}})$, then there exists a sequence $(\omega_n)_n\subset \Omega^r_c(M)$ such that $\omega_n\longrightarrow \omega$ and $d_r\omega\longrightarrow d_{r,\text{min}}\omega$. Since $*$ is unitary we see that $*\omega_n\longrightarrow *\omega$. On the other hand 
\begin{align*}
d^\dagger_{m-r}*\omega_n=\pm * d_r\omega_n\longrightarrow \pm *  d_{r,\text{min}}\omega, 
\end{align*}
again since $*$ is unitary. Hence verify that $*\dom(d_{r,\text{min}})\subseteq \dom(d^\dagger_{m-r,\text{min}})$.
\end{example}
From the relations above we get
\begin{align*}
*:\ker(d_\text{max/min})\cap\ker((d_\text{max/min})^*)\longrightarrow\ker(d_\text{min/max})\cap\ker((d_\text{min/max})^*).
\end{align*}
This observation and \eqref{Eqn:L2HarmonicForms} imply the following result. 
\begin{proposition}[{\cite[Lemma 3.7]{BL92}}]\label{Prop:PoincareDuality}
Let $M$ be an oriented Riemannian manifold. If the $L^2$-Stokes’ theorem holds then the Hodge star operator induces an isomorphism  
\begin{align*}
*:\mathcal{H}^r_{(2)}(M)\longrightarrow \mathcal{H}^{m-r}_{(2)}(M). 
\end{align*}
\end{proposition}

For $m=4k$ we define the {\em $L^2$-signature} of $M$, denoted by $\sigma_{(2)}(M)$, as the signature of the quadratic form 
\begin{align*}
\xymatrixcolsep{2cm}\xymatrixrowsep{0.01cm}\xymatrix{
H^{2k}_{(2)}(M)\times H^{2k}_{(2)}(M)  \ar[r] & \mathbb{C}\\
(\omega,\omega') \ar@{|->}[r] & \displaystyle{\int_M \omega\wedge\omega'}.
}
\end{align*}

In view of the signature theorem for closed manifolds, one can wonder whether the $L^2$-signature can be computed as an index. This is not always the case. Nevertheless, under certain conditions one can obtain such an index theorem. Since the $L^2$-cohomology groups are defined using the complex \eqref{Diag:MaximalExtdeRham} a natural candidate to consider would be the operator $D_\textnormal{max}\coloneqq d_\text{max}+(d_\text{max})^*$ with domain of definition $\text{Dom}(D_\textnormal{max})\coloneqq \text{Dom}(d_\text{max})\cap \text{Dom}((d_\text{max})^*)$. Since this operator is induced from a Hilbert complex it is automatically is self-adjoint  (\cite[pg 92]{BL92}). Suppose that the $L^2$-Stokes theorem holds, then from the discussion above we see that 
$*\text{Dom}(D_\textnormal{max})\subset \text{Dom}(D_\textnormal{max})$ and therefore $\star\text{Dom}(D_\textnormal{max})\subset \text{Dom}(D_\textnormal{max})$, where $\star$ is the chirality operator of $M$. This allow us to define 
\begin{align}\label{Def:Dmax+}
D^{+}_{\textnormal{max}}\coloneqq \frac{1}{2}(1-\star)D_\text{max}\frac{1}{2}(1+\star). 
\end{align} 

\begin{theorem}[{\cite[Lemma 3.2]{BL92}}]\label{Prop:L2singnatureIndex}
Let $M$ be an oriented Riemannian manifold of dimension $4k$. If the $L^2$-cohomology of $M$ has finite dimension and the $L^2$-Stokes theorem holds then $\sigma_{(2)}(M)=\ind(D^{+}_{\textnormal{max}})$.
\end{theorem}

\begin{proof}
Since $d_\text{max}=d_\text{min}$ then the operator
\begin{align*}
D\coloneqq D_\textnormal{max}=d_\text{max}+(d_\text{max})^*=d_\text{min}+d^\dagger_\text{min},
\end{align*}
is defined on 
\begin{align*}
\text{Dom}(D)\coloneqq \text{Dom}(d_\text{max})\cap \text{Dom}((d_\text{max})^*)=\text{Dom}(d_\text{min})\cap \text{Dom}(d^\dagger_\text{min}).
\end{align*}
Assume $\omega\in \ker D$, then there exist sequences $(\omega_n)_n,(\alpha_l)_l\subset\Omega_c(M)$ such that 
\begin{align*}
\omega_n &\longrightarrow \omega,  &\alpha_l  &\longrightarrow \omega,\\
 d\omega_n &\longrightarrow d_\text{min}\omega,  &d^\dagger\alpha_l &\longrightarrow d^\dagger_\text{min}\omega,
\end{align*}
and $d_\text{min}\omega+d^\dagger_\text{min}\omega=0$. Thus, by continuity of the $L^2$-inner product 
\begin{align*}
(d_\text{min}\omega,d^\dagger_\text{min}\omega)_{L^2(\wedge T^*M)}=&\lim_{n\rightarrow\infty}\lim_{l\rightarrow\infty}(d\omega_n,d^\dagger\alpha_l)_{L^2(\wedge T^*M)}\\
=&\lim_{n\rightarrow\infty}\lim_{l\rightarrow\infty}(d^2\omega_n,\alpha_l)_{L^2(\wedge T^*M)}\\
=&0.
\end{align*}
Hence, $\ker D=\ker(d_\text{min})\cap\ker(d^\dagger_\text{min})$. Now, from Proposition \ref{Prop:StrongHodgeTheorem} we conclude that $\ker D=\mathcal{H}^*_{(2)}(M)=H^r_{(2)}(M)$. One can prove that the image of $D$ is closed as in the proof of  \cite[Theorem 2.4]{BL92}.  The index relation follows directly from   Proposition \ref{Prop:PoincareDuality}.
\end{proof}

\subsection{Exterior derivative of a Riemannian submersion}\label{Sect:RiemFibr}
This section collects some geometric constructions in the setting of Riemannian fiber bundles. We describe how the exterior derivative on a total space  of a fibration can be decomposed into a ``vertical'' and a ``horizontal'' part. This decomposition will often be used during the next chapters. Then we recall the structure equations of satisfied by the connection and the curvature forms of the Levi-Civita connection of a submersion metric.  The aim of this section is to fix some notation and state some results without proofs. A detailed treatment can be found in \cite{BGV} and  \cite{BL95}. 

\subsubsection{Decomposition of the exterior derivative} 
Let $\pi:N\longrightarrow B$ be a fiber bundle with fiber $Y$ and set ${v}\coloneqq \dim Y$ and ${h}\coloneqq \dim B$. Define the {\em vertical tangent bundle} as the collection of subspaces
$$T_V N_x\coloneqq \{v\in T_x N\: : \: d\pi |_x(v)=0\in T_{\pi(x)}B\}\subset T_x N,\quad\text{for $x\in B$}.$$ 
Let us assume we are given a {\em connection}, i.e an {\em horizontal subbundle} $T_H N\subset TN$ such that $TN=T_H N\oplus T_V N$ and which we can identify $T_HN\cong \pi^*TB$.  This splitting of $T M$ induces a splitting of the exterior bundle 
$$\wedge T^* N=\wedge T_H^*N\otimes\wedge T^*_V N =\bigoplus_{p,q}\wedge^p T^*_H N\otimes \wedge^q T^*_V N\eqqcolon \bigoplus_{p,q}\wedge^{p,q}T^* N.$$ 
We denote the space of differential $(p,q)$-forms by $\Omega^{p,q}(N)\coloneqq C^\infty(N,\wedge^{p,q}T^*N)$. With respect to this decomposition we define the operators which count the horizontal and vertical degree by
\begin{align*}
\text{hd}|{\wedge^{p,q} T^*N}\coloneqq p\quad \text{and} \quad \text{vd}|{\wedge^{p,q} T^*N}\coloneqq q.
\end{align*}
Similarly, let $\varepsilon_H\coloneqq (-1)^{\text{hd}}$ and $\varepsilon_V\coloneqq (-1)^\text{vd}$ be the horizontal and vertical Gau\ss-Bonnet involution respectively, i.e 
 \begin{align*}
\varepsilon_H|{\wedge^{p,q} T^*N}\coloneqq (-1)^p\quad \text{and} \quad \varepsilon_V|{\wedge^{p,q} T^*N}\coloneqq (-1)^q.
\end{align*}

Let $d_N:\Omega^r(N)\longrightarrow\Omega^{r+1}(N)$ be the exterior derivative operator of $N$.  It is proven in \cite[Section 10.1]{BGV}, \cite[Proposition 3.4]{BL92} and \cite[Lemma 2.1]{B09} that we can decompose $d_N$ as the sum 
\begin{equation}\label{Eqn:d}
d_N=d_{N,H}^{(1)}+d_{N,H}^{(2)}+d_{N,V},
\end{equation}
where 
\begin{align*}
d_{N,H}^{(1)}&:\Omega^{p,q}(N)\longrightarrow \Omega^{p+1,q}(N),\\
d_{N,H}^{(2)}&:\Omega^{p,q}(N)\longrightarrow \Omega^{p+2,q-1}(N),\\
d_{N,V}&:\Omega^{p,q}(N)\longrightarrow \Omega^{p,q+1}(N).
\end{align*}
Here we list some properties of these operators:
\begin{itemize}
\item $d_{N,V}$ is given explicitly by $d_{N,V}(\pi^*\omega_1\otimes\omega_2)\coloneqq \varepsilon_H(\pi^*\omega_1)\otimes d_{Y}\omega_2$, where $d_Y$ is the exterior derivative on the fiber. Thus, $d_{N,V}$ is a first order vertical differential operator. 
\item $d_{N,H}^{(1)}$ is a first order horizontal operator. 
\item $d_{N,H}^{(2)}$ is a bounded operator involving the curvature of the fibration.
\end{itemize}
Using the equation $d_N^2=0$ we obtain the following relations (\cite[Equations (3.12)]{BL95}),
\begin{align*}
(d_{N,H}^{(2)})^2=&0.\\
d_{N,V}d_{N,H}^{(1)}+d_{N,H}^{(1)}d_{N,V}=&0,\\
(d_{N,H}^{(1)})^2+d_{N,V}d_{N,H}^{(2)}+d_{N,H}^{(2)}d_{N,V}=&0,\\
d_{N,H}^{(1)}d_{N,H}^{(2)}+d_{N,H}^{(2)}d_{N,H}^{(1)}=&0,\\
(d_{N,V})^2=&0.
\end{align*}

\begin{remark}\label{Rmk:FlatConnHarmForms}
Let us consider the {\em vertical cohomology bundle} $\mathscr{H}^*(Y)\longrightarrow B$ defined by the vertical differential $d_{N,V}$, i.e. it is the bundle of de Rham cohomology groups of the fibers. This collection of spaces is indeed locally trivial by the results of \cite[Section 9.2]{BGV}. Moreover, from the first two relations above  it follows that $d_{N,H}^{(1)}$ induces a flat connection on this bundle (\cite[pg. 12]{ALP13}). 
\end{remark}

\subsubsection{The Levi-Civita connection}
Let us equip the vertical tangent bundle $T_V N$ with a Riemannian metric $g^{T_V N}$ and the horizontal bundle $T_H N$  with a metric $g^{T_H N}$ by pulling up with $\pi$ a metric $g^{TB}$ from the base. With these two metrics we can construct a {\em submersion metric}  on $TN$ by setting $g^{TN}\coloneqq g^{T_H N}\oplus g^{T_V N}$ and denote by $\nabla$ its associated Levi-Civita connection . Consider now a local oriented orthonormal basis of $TN$ of the form $\{e_i\}_{i=1}^{{v}}\cup\{f_\alpha\}_{\alpha=1}^{{h}}$ where $e_i\in T_V N$ and $f_\alpha\in T_H N\cong TB$.  Let $\{e^i\}_{i=1}^{{v}}\cup\{f^\alpha\}_{\alpha=1}^{{h}}$ denote its associated dual basis.  With respect to this local frame we construct the associated connection $1$-form $\omega$ of the Levi-Civita connection $\nabla$, whose components are defined by the relations
\begin{equation}\label{Eqns:DefConnection1form}
\nabla e_j\eqqcolon \omega^i_{j} \otimes e_i+\omega^\alpha_j \otimes f_\alpha\quad \text{and}\quad \nabla f_\alpha\eqqcolon \omega^i_{\alpha} \otimes e_i+\omega^\alpha_\beta\otimes f_\beta,
\end{equation}
where the sum over repeated indices is understood. Since the Levi-Civita connection preserves the metric then $\omega$ is anti-symmetric, i.e. $\omega^I_J=-\omega^J_I$, where $I,J$ denote vertical (latin) or horizontal (greek) indices.  On the other hand since the Levi-Civita connection is torsion free then it satisfies the  {\em structure equations}
\begin{align}\label{Eqns:StructureEquations}
de^i+\omega^i_j\wedge e^j+\omega^i_\alpha\wedge f^\alpha=0,\\
df^\alpha+\omega^\alpha_j\wedge e^j+\omega^\alpha_\beta\wedge f^\beta=0.\notag
\end{align}
We can express each component of $\omega$ in terms of the basis elements, for example
\begin{align*}
\omega^\alpha_j=\omega^\alpha_{jk}e^k+\omega^{\alpha}_{j\beta} f^\beta\in\Omega^1(N),
\end{align*}
where $\omega^\alpha_{jk}$ and $\omega^{\alpha}_{j\beta}$ are smooth (local) functions on $N$. These components are computed as follows: First we see from \eqref{Eqns:DefConnection1form},
\begin{align*}
\nabla_{e_k}e_j=\omega^i_{j} (e_k)\otimes e_i+\omega^\alpha_j(e_k)\otimes f_\alpha.
\end{align*}
Then to pick up the desired component we just contract with the metric $\omega^\alpha_{jk}=\inner{\nabla_{e_k}e_j}{f_\alpha}$.
\begin{align*}
\omega^\alpha_{jk}=\inner{\nabla_{e_k}e_j}{f_\alpha}. 
\end{align*}
\begin{lemma}[{\cite[Equations (3.21)]{BL95}}]\label{Lemma:SymmOmega}
The components of $\omega$ satisfy 
\begin{enumerate}
\item $\omega^{\alpha}_{jk}=\omega^{\alpha}_{kj}$.
\item $\omega^{\alpha}_{j\beta}=\omega^{\alpha}_{k\beta}$.
\end{enumerate}
\end{lemma}
\begin{proof}
For the first identity we just using the relation above
\begin{align*}
\omega^{\alpha}_{jk}-\omega^{\alpha}_{kj}=&\inner{\nabla_{e_k}e_j}{f_\alpha}-\inner{\nabla_{e_j}e_k}{f_\alpha}\\
=&\inner{\nabla_{e_k}e_j-\nabla_{e_j}e_k}{f_\alpha}\\
=&\inner{[e_k,e_j]}{f_\alpha}\\
=&0, 
\end{align*}
since $[e_k,e_j]$ is a vertical vector field.  Similarly we verify $\omega^{\alpha}_{j\beta}-\omega^{\alpha}_{\beta j}=\inner{[e_k,f_\beta]}{f_\alpha}=0$, since $[f_\beta,e_j]$  is also vertical.
\end{proof}

The associated associated curvature $2$-form $\Omega$ of $\nabla$ has components (\cite[Theorem 5.21]{M00})
\begin{align}\label{Eqns:ComponentsCurvatureForm}
\Omega^i_j=& d\omega^i_j+\omega^i_k\wedge\omega^k_j +\omega^i_\alpha\wedge\omega^\alpha_j,\notag \\
\Omega^i_\alpha=& d\omega^i_\alpha+\omega^i_k\wedge\omega^k_\alpha +\omega^i_\beta\wedge\omega^\beta_\alpha, \\
\Omega^\alpha_\beta=& d\omega^\alpha_\beta+\omega^\alpha_k\wedge\omega^k_\beta +\omega^\alpha_\gamma\wedge\omega^\gamma_\beta. \notag
\end{align} 
These components also satisfy the relations $\Omega^I_J=-\Omega^J_I$.

\subsection{Adiabatic limit of the eta invariant of the signature operator}\label{Section:Dai}

In this section we are going to describe Dai's formula for the adiabatic limit of the eta invariant derived in \cite{D91}. The original motivation for such a formula goes back to Witten in his study of certain types of anomalies (\cite{W85}). The adiabatic formula was studied in the particular case when the base space had dimension one. Rigorous treatments of Witten's ideas were developed further in \cite{BFI86}, \cite{BFII86}. A remarkable generalization for general compact base manifolds was developed by Bismut and Cheeger in \cite{BC89}. One of the main ingredients of their adiabatic limit formula is the appearance  of the $\widetilde{\eta}$-form,a differential form on the base space. Here are two important features of $\widetilde{\eta}$:
\begin{enumerate}
\item It arises as the transgression form on Bismut's family index theorem (\cite[Section 10.5]{BGV}).
\item It can be seen as a higher dimensional analogue of the eta invariant in view of the family index theorem for manifolds with boundary of Bismut and Cheeger (\cite{BCI90}, \cite{BCII90}). 
\end{enumerate}
In the fundamental paper \cite{BC89}  Bismut and Cheeger studied in detail the case when the Dirac operators along the fibers are invertible. Note that this is not necessary the case for the vertical signature operator since its kernel, by Hodge theory, is given by the cohomology of the fiber.  In Dai's Ph.D. thesis he studied the case when the dimension of the kernel of the vertical Dirac operators is locally constant (\cite{D91}). He found for the adiabatic limit formula of the eta invariant of the signature operator a new term arising from the Leray spectral sequence of the fibration, the so-called $\tau$-invariant. The aim of this section is to understand the meaning of the terms involved in Dai's result. Nevertheless, we are not going to dive into the details of the derivation of the formula. Instead, we will focus in some particular situations in which we have vanishing results for the $\widetilde{\eta}$-form and the $\tau$-invariant. These  results will be essential for next chapter.

\subsubsection{Description of the adiabatic limit formula for the signature operator}
Let us consider a fibration of closed manifolds $N\longrightarrow B$ with typical fiber $Y$ and such that the dimension of the total space $N$ is $4k-1$ . We will be mainly interested in the case where fiber $Y$ has even dimension $\dim Y=2N$, which implies that the dimension of the base space is $\dim B=4k-2N-1$ is odd. We further assume the fibration is {\em oriented}, meaning that $TB$ and $T_VN$ are both oriented. As before consider the submersion metric $g^{TN}=g^{T_H N}\oplus g^{T_V N}$ on $TN$ such that $g^{T_H N}$ comes from a metric $g^{TB}$ on $TB$. For a parameter $r>0$ define a new metric by  
\begin{align*}\label{Eqn:AdiabaticMetric}
g^{TN}(r)\coloneqq (r^{-2}g^{T_H M})\oplus g^{T_V M}.
\end{align*}
Let $A_r$ denote the odd signature operator \eqref{Eqn:DefOddSignOp} of $N$ with respect to the metric $g^{TN}(r)$. In various applications one is interested in the behavior,  as the parameter $r$ goes to zero, of $\eta(A_r)\coloneqq\eta_{A_r}(0)$. This limit is called the {\em adiabatic limit of the eta invariant}. In \cite{D91} Dai found an expression of this limit in terms of several geometric and topological quantities which we now describe (see \cite[Section 4.1]{D91})
\begin{itemize}
\item Let $L(TB,g^{TB})$ be the $L$-polynomial of the Levi-Civita connection of $g^{TB}$.
\item Let $A_B$ denote the odd signature operator of $B$ with respect to the metric $g^{TB}$.
\item Let $D_Y$ denote the family of Hodge-de Rham operators along the fibers and let $\ker D_Y$ denote its corresponding index bundle. Note that $\ker D_Y=\mathscr{H}^*(Y)$ admits a flat connection (see Remark \ref{Rmk:FlatConnHarmForms}).
\item Let $\widetilde{\eta}$ be the eta-form introduced in \cite{BC89} by Bismut and Cheeger. This is an odd differential form on $B$ since the dimension of the fiber $Y$ is even. It is constructed using the Bismut super connection (\cite[Section 10.5]{BGV}, \cite{BC89}, \cite[Section 1.1]{D91}).
\item Let $\tau$ be Dai's invariant coming from the Leray's spectral sequence of the fibration (which will be explained in detail below).
\end{itemize}
\begin{theorem}[{\cite[Theorem 0.3]{D91}}]\label{Thm:Dai}
Suppose that the fibration $N\longrightarrow B$ with typical fiber $Y$ is oriented, then
\begin{align*}
\lim_{r\rightarrow 0}\eta(A_r)=2\int_B L(TB,g^{TB})\wedge \widetilde{\eta}+\eta(A_B\otimes \ker D_Y)+2\tau,
\end{align*}
where $\eta(A_B\otimes \ker D_Y)$ is the eta invariant of the odd signature operator $A_B$ of $B$ twisted by the bundle of vertical harmonic forms $\ker D_Y$.
\end{theorem}

\begin{remark}
The theorem above also holds when the dimension of the fiber is odd. 
\end{remark}

\subsubsection{Definition of the $\widetilde{\eta}$-form}	
We will comment very briefly the definition of the form $\widetilde{\eta}$ for the case of even dimensional fibers. Let $\mathcal{B}_t$ be the {\em Bismut super-connection} associated to the fibration $N\longrightarrow B$ and the odd signature operator on $N$ introduced in \cite{B86(2)}. The $\widetilde{\eta}$-form is defined, up to constants, by 
\begin{equation}\label{Def:EtaForm}
\widetilde{\eta}\coloneqq \int_{0}^\infty \tr_s\left(\frac{d \mathcal{B}_t}{dt}e^{-{\mathcal{B}_t}^2}\right)dt\in \Omega^{\text{odd}}(B),
\end{equation}
where $\tr_s$ denotes the super-trace in the Clifford algebra (\cite[Section 1.5]{BGV}). It is not straightforward to verify that this integral is is well defined. In general, this is not true for an arbitrary super-connection.		

\subsubsection{Definition of the $\tau$-invariant}
Now we are going to illustrate the construction of the  invariant $\tau$ appearing in the adiabatic limit formula of Theorem \ref{Thm:Dai}. To begin, we recall the properties of the Leray spectral sequence of a fibration. 
\begin{theorem}[Leray, {\cite[Theorem 14.18]{BT82}}]\label{Thm:Leray}
Given a fiber bundle ${N}\longrightarrow B$ with fiber $Y$ over a manifold $B$ and a good cover $\mathfrak{U}$ on $B$ there is a spectral sequence $\{(E^{p,q}_r,d_r)\}$, with
\begin{align*}
E^s_r\coloneqq \bigoplus_{p+q=s}E^{p,q}_r\quad \text{and}\quad E_r\coloneqq \bigoplus_{s\geq 1}E^s_r,
\end{align*}
which converges to the cohomology of the total space $H^*({N})$ and has $E_2$-term
$$E^{p,q}_2=H^{p}(\mathfrak{U},\mathscr{H}^q),$$
where $\mathscr{H}^q$ is the locally constant presheaf $\mathscr{H}^q(U)\coloneqq H^q(\pi^{-1}U)$ for $U\in \mathfrak{U}$. 
\end{theorem}

In addition, this spectral sequence has a {\em multiplicative structure} (\cite[Chapter 1.2]{H04}): It is equipped with a bilinear product $E^{p,q}_r\times E^{s,t}_r\longrightarrow E^{p+s,q+t}_r$ such that for each $r\geq 1$ the differential $$d_r:E^{p,q}_r\longrightarrow E^{p+r,q-r+1}_r$$
acts as a derivation, i.e. $d_r(ab)=(d_ra)b+(-1)^{p+q}a(d_rb)$ for $a\in E^{p,q}_r$. Moreover,  since $E^{p,q}_{r+1}=\ker( d_r)/\ran  (d_r)$ at $E^{p,q}_r$ the product structure on $E^{p,q}_{r+1}$ is induced by the product structure on $E^{p,q}_{r}$.
This multiplicative structure satisfies the relation 
\begin{equation}\label{Eqn:CommSS}
ab=(-1)^{(p+q)(s+t)}ba,\quad\text{for $a\in E^{p,q}_r$ and $b\in E^{s,t}_r$}.
\end{equation}

Now let us see how these results allow us to define the invariant $\tau$. Assume that the fiber bundle $N\longrightarrow B$ is oriented, then the orientation of $T_V N$ induces a  trivalization of the top-degree flat line bundle $\mathscr{H}^{v}(Y)$ (Remark \ref{Rmk:FlatConnHarmForms}). Here we again use the previous notation $v\coloneqq \dim Y$ and $h\coloneqq \dim B$. Together with the orientation of $B$ we obtain an identification
\begin{align*}
E^{h,v}_2=H^{h}(B,\mathscr{H}^{v}(Y))\cong H^{h}(B)\cong \mathbb{R}.
\end{align*}
In particular, the multiplicative structure gives a Poincar\'e pairing
$$E^{p,q}_r\times E^{h-p,v-q}_r\longrightarrow E^{h,v}_r\cong\mathbb{R}.$$
Denote by $\xi_2\in E^{h+v}_2$ the generator induced by the orientations of $B$ and $T_V N$. By Theorem \ref{Thm:Leray} we know that $E_\infty^{m}=H^{h+v}(N)\cong \mathbb{R}$ so $\dim(E_\infty^{h+v})=1$, which implies $\dim(E^{h+v}_r)=1$ for all $r\geq 2$. Let us denote by $\xi_r\in E^{m}_r$ the generator induced by $\xi_2$.\\

We now consider, for each $r$, a  the bilinear pairing induced from the multiplicative structure described above.  Define
\begin{align*}
\tau_r:
\xymatrixcolsep{2cm}\xymatrixrowsep{0.01cm}\xymatrix{
E_r \times E_r \ar[r] & \mathbb{R}\\
(a,b) \ar@{|->}[r] & \inner{a(d_rb)}{\xi_r}.
}
\end{align*}
More concretely let $a\in E^{p,q}_r$ and $b\in E^{s,t}_r$ such that 
\begin{align*}
s=&h-p-r,\\
t=&v-q+r-1,
\end{align*}
then $d_r b\in E^{h-p,v-q}$ and so $a(d_r b)\in E^{h,v}_r$. Thus $a(d_r b)$ must be a multiple of $\xi_r$, we define $\inner{a}{d_rb}\in\mathbb{R}$ to be such a multiple. This bilinear pairing satisfies 
\begin{align*}
\tau_r(a,b)=(-1)^{(p+q+1)(s+t+1)}\tau_r(b,a),
\end{align*}
for $a\in E^{p,q}_r$ and $b\in E^{s,t}_r$. To see this observe from the derivation property of $d_r$ and \eqref{Eqn:CommSS},
\begin{align*}
a(d_r b)=&d_r(ab)+(-1)^{p+q+1}(d_ra)b\\
=&d_r(ab)+(-1)^{p+q+1}(-1)^{(p+q+1)(s+t)}b(d_ra)\\
=&d_r(ab)+(-1)^{(p+q+1)(s+t+1)}b(d_ra).
\end{align*}
Finally note that the term $d_r(ab)$ is zero in cohomology. 
In particular, if $h+v=4k-1$, the map $\tau_r: E_r^{2k-1}\times E_r^{2m-1}\longrightarrow\mathbb{R}$ is symmetric so it has a well-defined signature $\sigma(\tau_r)$. Define the $\tau$ {\em invariant} associated to the fibration $N\longrightarrow B$ as the sum of these signatures,
\begin{equation}\label{Def:TauInv}
\tau\coloneqq \sum_{r\geq 2}\sigma(\tau_r). 
\end{equation}

\begin{example}[Projectivization bundle]\label{Example:ProjBundle}
Let $E{\longrightarrow}B$ be a complex vector bundle of rank $N+1$ and let $P(E)\longrightarrow B$ be its associated projectivization bundle. It can be shown that there exists a cohomology class $c\in H^{2}(P(E))$ such that $1,c,\cdots, c^{N}$ are global classes on $P(E)$ whose restrictions to the fiber $P(E_x)$ freely generate the cohomology of the fiber over $x\in B$ (\cite[pg. 270]{BT82}). This implies the $E_2$-term of the spectral sequence of Theorem \ref{Thm:Leray} associated  to the fibration $P(E)\longrightarrow B$ is $E^{p,q}_2=H^{p}(B)\otimes H^{q}(\mathbb{C}P^N)$. Since each term of $E_2$ is already global, it can be seen from the proof of Theorem \ref{Thm:Leray} that the differentials $d_2=d_3=\cdots=0$ and therefore $E_2=E_\infty$. In particular, the $\tau$ invariant of $\pi:P(E)\longrightarrow B$ is zero.
\end{example}

\subsection{Equivariant methods}\label{Section:CompactStrGroup}
The aim of this last section of the chapter is to gather some results concerning fibrations with compact structure group. These results will be used in the next chapter. For a detailed treatment on the subject we refer, for example, to \cite{BGV}, \cite[Chapte IV]{BL95} and \cite{DK00}.

\subsubsection{Equivariant Chern-Weil homomorphism}

Let $G$ be a Lie group with associated Lie algebra $\mathfrak{g}$ and let $Y$ be an oriented $G$-manifold. Denote by $\mathbb{C}[\mathfrak{g}]$ the algebra of complex valued polynomial functions on $\mathfrak{g}$. An element $g\in G$ acts on $\alpha\in\mathbb{C}[\mathfrak{g}]\otimes\Omega(Y)$ as 
\begin{align*}
(g\alpha)(X)\coloneqq g(\alpha(g^{-1}X))\:,\quad \text{for $X\in\mathfrak{g}$.}
\end{align*}
Here $G$ acts on $\mathfrak{g}$ via the adjoint representation. We denote by $(\mathbb{C}[\mathfrak{g}]\otimes\Omega(Y))^G$ the space of {\em equivariant differential forms}, i.e. forms which are invariant under this $G$-action. We can define a $\mathbb{Z}$-grading on the space $\mathbb{C}[\mathfrak{g}]\otimes\Omega(Y)$ by setting 
\begin{align*}
\deg(f\otimes\beta)\coloneqq 2\deg(f)+\deg(\beta), 
\end{align*}
where $f\in\mathbb{C}[\mathfrak{g}]$ and $\beta\in\Omega(Y)$. The {\em equivariant exterior differential} $d_\mathfrak{g}$ on $\mathbb{C}[\mathfrak{g}]\otimes\Omega(M)$ is defined by the relation
\begin{align*}
d_\mathfrak{g}(f\otimes\beta)(X)\coloneqq d(f(X)\beta)-\iota_X(f(X)\beta)). 
\end{align*}
It is easy to see that $d_\mathfrak{g}$ increases the total degree by one and when restricted to $(\mathbb{C}[\mathfrak{g}]\otimes\Omega(Y))^G$ it satisfies $d_\mathfrak{g}^2=0$. \\

In addition let $P\longrightarrow B$ be a $G$-principal bundle with connection form $\boldsymbol{\omega}\in\Omega^1(P,\mathfrak{g})^G$ and curvature form $\boldsymbol{\Omega}=d\boldsymbol{\omega}+\boldsymbol{\omega}\wedge \boldsymbol{\omega}\in\Omega^2(P,\mathfrak{g})^G$. We can construct the associated bundle $N\coloneqq P\times_G Y$ over $B$ with typical fiber $Y$ which has an induced connection induced from $\boldsymbol{\omega}$ (\cite[Proposition 1.6]{BGV}). Recall that this bundle is defined as a quotient of a $G$ free action on $P\times Y$ (Section \ref{Section:G-manifolds}). Let $f\otimes\beta\in (\mathbb{C}[\mathfrak{g}]\otimes\Omega(Y))^G$, then by Proposition \ref{Prop:BasicPullback} we can regard $f(\boldsymbol{\Omega})\otimes\beta\in\Omega_\text{bas}(P\times Y)\cong \Omega(N)$. With this observation in mind we can make sense of the following fundamental result.
\begin{theorem}[Chern-Weil, {\cite[Theorem 7.33]{BGV}}]
Let $G$ be a Lie group, $P\longrightarrow B$ be a principal bundle with structure group $G$ and connection form $\boldsymbol{\omega}$ and let $Y$ be a $G$-manifold. Then the Chern-Weil homomorphism
$$\phi_{\boldsymbol{\omega}}:
\xymatrixcolsep{2cm}\xymatrixrowsep{0.01cm}\xymatrix{(\mathbb{C}[\mathfrak{g}]\otimes\Omega(Y))^G,d_\mathfrak{g}) \ar[r] & (\Omega(N),d_N)},$$
defined by $\phi_{\boldsymbol{\omega}}(f\otimes\beta)\coloneqq f(\boldsymbol{\Omega})\otimes\beta$
is a homomorphism of differential graded Lie algebras. 
\end{theorem}

\begin{coro}[{\cite[Proposition 7.35]{BGV})}]
Due the functoriality of the Chern-Weil homomorphism, if $Y$ is compact and oriented then we have the following commutative diagram
\begin{equation}\label{DiagCW}
\xymatrixcolsep{4pc}\xymatrixrowsep{4pc}\xymatrix{
(\mathbb{C}[\mathfrak{g}]\otimes\Omega(Y))^G \ar[d] \ar[r]^-{\phi_{\boldsymbol{\omega}}} & \Omega(N) \ar[d]\\
\mathbb{C}[\mathfrak{g}]^G \ar[r]^-{\phi_{\boldsymbol{\omega}}}  & \Omega(B),
}
\end{equation}
where the left vertical arrow is integration over $Y$ and the right vertical arrow is integration along the fibers of $N\longrightarrow B$ (\cite[Chapter 1.6]{BT82}).
\end{coro}

\subsubsection{Vanishing of the $\widetilde{\eta}$-form}

Let us continue in the setting described above where $N=P\times_GY\longrightarrow B$ is a bundle associated to a principal $G$-bundle $P\longrightarrow B$ with connection. Let us further assume that the Lie group $G$ is compact. Equip the $N$ with a submersion metric $g^{T_H N}\oplus g^{T_V N}$ as described in Section \ref{Sect:RiemFibr}.  Since $G$ is compact we can assume without loss of generality that it preserves the vertical metric. It was proven in \cite{HERMANN} that in this case the fibers of $N$ are totally geodesic, i.e. the mean curvature $1$-form of the Riemannian fibration vanishes (\cite[Section 10.1]{BGV}). In particular, the Bismut super-connection $\mathcal{B}_t$ simplifies and its derivative with respect to $t$, as well as its square, can be written conveniently  so that the argument of the super-trace \eqref{Def:EtaForm} has coefficients only in the odd part of the Clifford algebra (\cite[Sections 9.4, 10.7]{BGV}, \cite[Section 1.c]{G00}). As a result, the form $\widetilde{\eta}$ vanishes as the supertrace of the integrand is zero.

\begin{proposition}[{\cite[Remark 1.15]{G00}}]\label{Prop:VanishinEtaForm}
Let $N\longrightarrow B$ be a fibration with compact structure group $G$ such that $\dim N=4k-1$. If we equip $N$ with a  submersion metric $g^{T_H N}\oplus g^{T_V N}$ such that $g^{T_V N}$ is preserved by $G$, then the $\widetilde{\eta}$-form of Theorem \ref{Thm:Dai} is zero.
\end{proposition}

\section{Lott's $S^1$-equivariant signature formula}\label{Sect:Lott}
In this chapter we study the definition and some fundamental properties of the equivariant $S^1$-signature. In particular we provide a detailed proof of Lott's formula \eqref{Eqn:Lott} for semi-free $S^1$-actions, presented in \cite{L00}, using the tools described in Chapter \ref{Sec:AdiabLimit}.

\subsection{Basics and definitions}\label{Section:BasicsDef}

Let $M$ be an $4k+1$ dimensional Riemannian, closed, oriented manifold on which $S^1$ acts by orientation-preserving isometries. Let us denote by $V$ the generating vector field of the action discussed in Remark \ref{Rmk:KillingVF}.
\begin{lemma}\label{Lemma:KillingEq}
Let $\nabla$ be the Levi-Civita connection associated to the Riemannian metric on $M$. The generating vector field $V$ satisfies the relation 
\begin{align*}
\inner{\nabla_Y V}{Z}+\inner{Y}{\nabla_Z V}=0,
\end{align*}
for all $Y,Z\in C^{\infty}(M,TM)$. 
\end{lemma}
\begin{proof}
As the $S^1$-action is isometric then the Lie derivative of the metric $L_V\inner{\cdot}{\cdot}$ vanishes. We compute this Lie derivative explicitly for $Y,Z\in C^{\infty}(TM)$,
\begin{align*}
(L_V\inner{\cdot}{\cdot})(Y,Z)=&V\inner{Y}{Z}-\inner{[V,Y]}{Z}-\inner{Y}{[V,Z]}\\
=&\inner{\nabla_V Y-[V,Y]}{Z}+\inner{Y}{\nabla_V Z-[V,Z]}\\
=&\inner{\nabla_Y V}{Z}+\inner{Y}{\nabla_Z V}. 
\end{align*}
Here we have used that $\nabla$ is metric and torsion-free. 
\end{proof}
Let  $\imath:M^{S^1}\longrightarrow M$  be the inclusion of the fixed point set. We can define two sub-complexes of the de Rham complex of $M$ (see \eqref{Def:BasicForms}),
\begin{align*}
\Omega_\text{bas}(M)\coloneqq &\{\omega\in\Omega(M)\:|\:L_V\omega=0\:\:\text{and}\:\:\iota_V\omega=0\},\\
\Omega_\text{bas}(M,M^{S^1})\coloneqq &\{\omega\in\Omega(M)_\text{bas}\:|\: \imath^*\omega=0\}.
\end{align*}
Denote by $H^*_\text{bas}(M)$ and $H^*_\text{bas}(M,M^{S^1})$ their corresponding cohomology groups.  It can be shown that there exist isomorphisms (\cite[Proposition 1]{L00})
\begin{align}\label{IsomsBasic}
H^{*}_\text{bas}(M,M^{S^1})\cong H^{*}_{\text{bas},c}(M-M^{S^1})\cong H^{*}(M/S^1,M^{S^1};\mathbb{R}),
\end{align}
where the subscript $c$ denotes cohomology with compact support. These  cohomology groups are all $S^1$-homotopy invariant (\cite[Proposition 2]{L00}).

\subsubsection{Definition of the $S^1$-equivariant signature}
Using the musical isomorphim \eqref{Eqn:Musical} we define the $1$-form $\alpha$ on $M-M^{S^1}$ by 
\begin{equation*}
\alpha\coloneqq \frac{V^\flat}{\norm{V}^2},
\end{equation*}
so that $\alpha(V)=1$. Note that as $\norm{V}\neq 0$ on $M-M^{S^1}$ then $\alpha$ is well defined. 
\begin{lemma}[{\cite[Proposition 3]{L00}}]\label{Lema:dalphabas}
The $2$-form $d\alpha$ is basic. 
\end{lemma}
\begin{proof}
First we show that $L_V\alpha=0$. Indeed, since the action preserves the metric then 
\begin{align*}
L_V\alpha=\frac{1}{\norm{V}^2}L_V V^\flat.
\end{align*}
Let $\nabla$ denote the Levi-Civita connection on $M$, then we compute for a vector field $Z\in C^\infty(M,TM)$, 
\begin{align*}
(L_V V^\flat)(Z)=&V(\inner{V}{Z})-\inner{V}{[V,Z]}\\
=&(\inner{\nabla_V V}{Z}+\inner{V}{\nabla_V Z})-(\inner{V}{\nabla_V Z}+\inner{V}{\nabla_Z V})\\
=&\inner{\nabla_V V}{Z}+\inner{V}{\nabla_Z V}\\
=&0.
\end{align*}
For the second equality we have used that the Levi-Civita connection preserves the metric and that is torsion free. The last equality follows by Lemma \ref{Lemma:KillingEq}.
Hence, $L_Vd\alpha=d(L_V\alpha)=0$, which shows that $d\alpha$ is $S^1$-invariant. On the other hand using Cartan's equation \eqref{Eqn:Cartan} we calculate $\iota_{V}d\alpha={L}_V\alpha-d\iota_V \alpha=-d(\alpha(V))=0$.
\end{proof}

Form the proof of the lemma above we obtain the following fact. 
\begin{coro}\label{Coro:alphaInv}
The form $\alpha$ satisfies $L_V\alpha=0$, that is, $\alpha$ is $S^1$-invariant. 
\end{coro}

\begin{proposition}[{\cite[Proposition 4]{L00}}]\label{Prop:WellDefTau}
If $\omega\in\Omega_{\textnormal{bas},c}^{4k-1}(M-M^{S^1})$ then 
$$\int_{M}\alpha\wedge d\omega=0. $$
\end{proposition} 
\begin{proof}
On the one hand, we use Stoke's theorem
$$\int_{M}\alpha\wedge d\omega=-\int_{\partial M}\alpha\wedge\omega+\int_{M}d\alpha\wedge\omega.$$
The first integral is zero because $\partial M=\emptyset$. On the other hand since $d\alpha$ and $\omega$ are basic forms then so is $d\alpha\wedge\omega$. Observe however that every top degree basic form must be zero. 
\end{proof}

\begin{definition}[{\cite[Definition 2]{L00}}]
We define the  {\em $S^1$-fundamental class of $M$ } to be the linear map 
$$
\xymatrixcolsep{2cm}\xymatrixrowsep{0.01cm}\xymatrix{
H^{4k}_{\text{bas},c} (M-M^{S^1}) \ar[r] &\mathbb{R}\\
\omega \ar@{|->}[r]& \displaystyle{\int_{M}\alpha\wedge\omega.}
}$$
\end{definition}
Note from Proposition \ref{Prop:WellDefTau} that the $S^1$-fundamental class of $M$ is a well-defined map in cohomology. 
\begin{proposition}[{\cite[Proposition 5]{L00}}]
The $S^1$-fundamental class of $M$ is independent of the Riemannian metric.
\end{proposition}

\begin{proof}
Let $\alpha_1$ and $\alpha_2$ be two $1$-forms constructed from two Riemannian metrics on $M$. Since $\alpha_1 -\alpha_2$ is a basic $1$-form it follows, as before, that
$$\int_M (\alpha_1-\alpha_1)\wedge\omega=0,\quad \text{for all $\omega\in\Omega^{4k}_{\text{bas},c} (M-M^{S^1}).$}$$
\end{proof}
\begin{definition}[{\cite[Definition 4]{L00}}]
We define the {\em  equivariant $S^1$-signature}  $\sigma_{S^1}(M)$ of $M$ with respect to the $S^1$-action as the signature of the symmetric quadratic form
\begin{align*}
\xymatrixcolsep{2cm}\xymatrixrowsep{0.01cm}\xymatrix{
H^{2k}_{\text{bas},c}(M-M^{S^1})\times H^{2k}_{\text{bas},c}(M-M^{S^1}) \ar[r] & \mathbb{R}\\
(\omega,\omega') \ar@{|->}[r] & \displaystyle{\int_M\alpha\wedge \omega\wedge\omega'}.
}
\end{align*}
\end{definition}
Observe that this intersection form is well defined: Let $\omega\in\Omega^{2k-1}_{\text{bas},c}(M-M^{S^1})$ and $\omega'\in\Omega^{2k}_{\text{bas},c}(M-M^{S^1})$  be forms such that $d\omega'=0$, then by Proposition \ref{Prop:WellDefTau}
\begin{align*}
\int_M \alpha\wedge d\omega\wedge \omega'=\int_M\alpha\wedge d(\omega\wedge \omega')=0.
\end{align*}

\begin{proposition}[{\cite[Proposition 6]{L00}}]
If $f:M\longrightarrow N$ is a orientation-preserving $S^1$-homotopy equivalence then $\sigma_{S^1}(M)=\sigma_{S^1}(N)$. 
\end{proposition}
\begin{proof}
Let $\alpha_N$ be the $1$-form defined by a metric on $N$. For $\omega\in\Omega_{\text{bas},c}^{4k}(N-N^{S^1})$ we have 
$$\int_{N}\alpha_N\wedge\omega=\int_{M}f^*\alpha_N\wedge f^*\omega.$$
If $\alpha_M$ is the $1$-form defined by the metric of $M$ we claim that $f^{*}\alpha_{N}-\alpha_{M}$ is a basic $1$-form  on $M$. This is easy to see since $L_{V_M}(f^{*}\alpha_{N}-\alpha_{M})=f^{*}(L_{V_N}\alpha_N)=0$ and on the other hand $\iota_{X_M}(f^{*}\alpha_{N}-\alpha_{M})=f^{*}(\iota_{X_N}\alpha_N)-\iota_{X_M}\alpha_M=1-1=0$. Therefore
$$\int_M (f^{*}\alpha_{N}-\alpha_{M})\wedge f^{*}\omega=0.$$
This shows that
$$\int_N \alpha_N\wedge\omega=\int_M \alpha_M\wedge f^{*}\omega,$$
i.e. the $S^1$-fundamental class of $M$ pushes forward to the $S^1$-fundamental class of $N$ and the result follows.
\end{proof}

\subsubsection{$S^1$-signature for semi-free actions}\label{Section:S1 Signature}

Let us now assume that the action is semi-free (Definition \ref{Def:Semi-Free}) and let $M_0:=M-M^{S^1}$ be the principal orbit where the action is free (Proposition \ref{Prop:PrincipalOrTyp}). We equip the manifold $M_0/S^1$ with the quotient metric $g^{T(M_0/S^1)}$ as in Section \ref{Section:G-manifolds}. As we will see below, in this case the dimension of the fixed point set $M^{S^1}$ must be odd, so we can consider its associated  odd signature operator \eqref{Eqn:DefOddSignOp} and the corresponding eta invariant which we denote  by $\eta(M^{S^1})$. One of the most important results in \cite{L00} is the following {\em index-type formula} for the  equivariant $S^1$-signature. 
\begin{theorem}[{\cite[Theorem 4]{L00}}]\label{Thm:S1SignatureThm}
Suppose $S^1$ acts effectively and semifreely on $M$, then 
$$\sigma_{S^1}(M)=\int_{M_0/S^1}L\left(T(M_0/S^1),g^{T(M_0/S^1)}\right)+\eta(M^{S^1}).$$
\end{theorem} 

\begin{remark}
It is important to emphasize  that part of the conclusion of Theorem \ref{Thm:S1SignatureThm} it the convergence of the integral of the $L$-polynomial over the open manifold $M_0/S^1$. 
\end{remark}

The following result establishes an analogous property as for the signature for closed manifolds (Proposition \ref{Prop:PropSignatureClosed}).
\begin{coro}[{\cite[Proposition 7]{L00}}]
Let $W$ be a semifree $S^1$-cobordism between $M_1$ and $M_2$. Then, $\sigma_{S^1}(W)=\sigma_{S^1}(M_1)-\sigma_{S^1}(M_2)$. 
\end{coro}

Now we study two examples where we see that the $S^1$-signature coincides with the notions of signature discussed in Section \ref{Section:SignatureOp}. \label{Ex:SpiningM}
\begin{example}[Spinning a closed manifold]
Let $X$ be closed oriented Riemannian manifold of dimension $4k$. Let us consider the product space $M\coloneqq X\times S^1$ equipped with the product metric and on which $S^1$ acts by multiplication on the second factor. It is straightforward to see that the action is free and $\sigma_{S^1}(M)=\sigma(X)$. 
\end{example}

\begin{example}[Spinning a manifold with boundary ({\cite[pg. 628]{L00}})]\label{Ex:SpiningMwB}
Let $X$ be an oriented compact manifold with boundary of dimension $4k$. We construct from it a closed manifold $M$ of dimension $4k+1$ as follows: Let $D^2$ be the closed disk of dimension $2$ and consider the product space $D^2\times X$. Define $M\coloneqq\partial(D^2\times X)$. An equivalent construction is $M=(D^2\times\partial X)\cup_{S^1\times\partial X}(S^1\times X)$, where we glue using the appropriate orientations (see Figure \ref{Fig:SpinningMnfld}). The natural $S^1$-action on $M$, induced by rotations on $D^2$,  is semi-free and has as fixed point set  $M^{S^1}=-\partial X$. Hence, it follows that $H^*_{\text{bas}}(M,M^{S^1})=H^{*}(X,\partial X;\mathbb{R})$ and $\sigma_{S^1}(M)=\sigma(X)$. Moreover, Theorem \ref{Thm:S1SignatureThm} reduces to the APS signature theorem (Theorem \ref{Thm:SignThmMBound}).

\begin{figure}[h]
\begin{center}
\begin{tikzpicture}
\draw (0,-2)--(0,2);
\draw (4,-2)--(4,2);
\draw[dashed,->] (5,-2.5)--(5,3.5);
\draw (6,-2)--(6,2);
\draw (5,2) ellipse (28.5pt and 10pt);
\draw (5,-2) ellipse (28.5pt and 10pt);
\draw (5,2) node {$\bullet$};
\draw (5,-2) node {$\bullet$};
\draw (0,2) node {$\bullet$};
\draw (0,-2) node {$\bullet$};
\draw [thick,->] (5.5,3) arc ( 0:315:15pt and 5pt);
\draw (10,2) node {$\bullet$};
\draw (10,-2) node {$\bullet$};
\draw (0,-3) node {$X$};
\draw (5,-3) node {$M$};
\draw (10,-3) node {$M^{S^1}=-\partial X$};
\end{tikzpicture}
\caption{Schematic description of the process of spinning a manifold with boundary $X$.}\label{Fig:SpinningMnfld}
\end{center}
\end{figure}
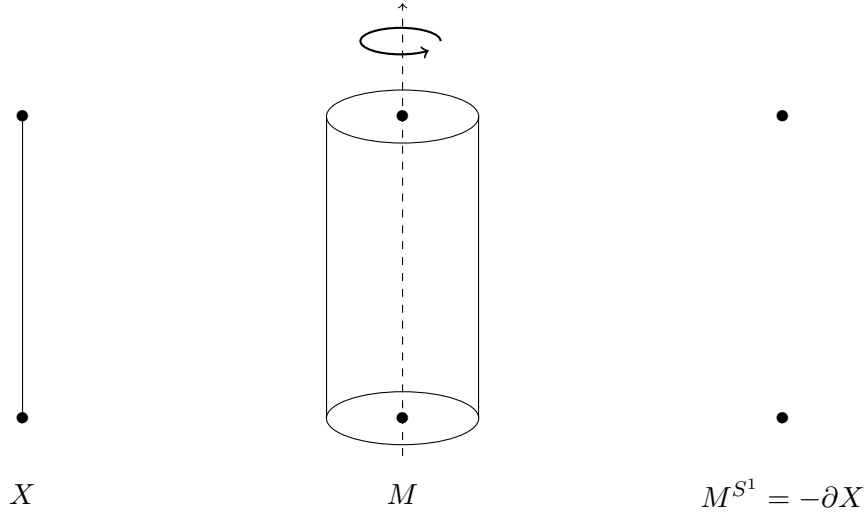

\end{example}

\subsection{Proof of the signature formula} \label{Sect:PfoofSignatureFormula}

The aim of this section is to study in detail the proof of Theorem \ref{Thm:S1SignatureThm} given in \cite[Section 2.3]{L00}. In order to clarify the exposition we are going to divide the proof in several steps. 

\subsubsection*{Step 1: Local description}

The first step of the proof is to understand the geometric model of a neighborhood of a connected component $F\subseteq M^{S^1}$ of the fixed point set. The set $F$ is a closed submanifold of $M$ and for each $x\in F$, by the Slice Theorem (Theorem \ref{Thm:SliceThm}), there exists an $S^1$-invariant open neighborhood $\mathcal{O}_x$ of $x$ in $M$ such that the $S^1$-action in $\mathcal{O}_x$ is equivalent to the action on the associated bundle  $S^1\times_{G_x} T_x M$ where $G_x=S^1$.

\begin{figure}[h]
\begin{center}
\begin{tikzpicture}
\draw [style=very thick](-1,1)--(2,0);
\draw [style=very thick](-1,-2)--(2,-3);
\draw [style=very thick](-1,1)--(-1,-2);
\draw [style=very thick](2,0)--(2,-3);
\draw [style=very thick](-3,-1)--(4,-1);
\draw[rounded corners=30pt](-2,-4)--(-4,-3)--(-2.5,-1)--(3.5,-1)--(5,0.5)--(6,0.7);
\node at (0.5,-1) {\pgfuseplotmark{*}};
\node at (0.5,-1.3) {$x$};
\node at (-3.5,-1) {$HF_x$};
\node at (-0.2,0.2) {$VF_x$};
\node at (5,0.7) {$F$};
\end{tikzpicture}
\caption{Local slice for a fixed point set $x\in F$.}
\end{center}
\end{figure}
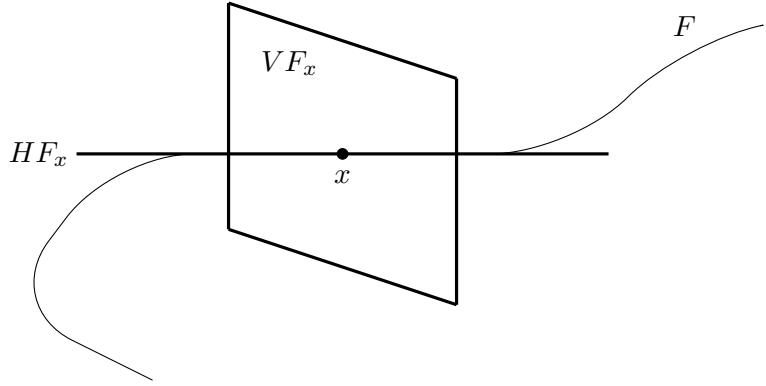

Next we want to understand the induced action $\rho_x:S^1\longrightarrow \text{Aut}(T_xM)$ of \eqref{Eqn:rho}. Note that since the action is orientation and metric preserving then we can regard it as a representation $\rho_x:S^1\longrightarrow SO(T_xM)$. This implies that for each $g\in S^1$ we have the following commutative diagram
\begin{equation}\label{Diag:ExpG}
\xymatrixcolsep{2cm}\xymatrixrowsep{2cm}\xymatrix{
T_x M\ar[r]^-{\rho_x(g)} \ar[d]_-{\exp_x} & T_x M \ar[d]^-{\exp_x}\\
M \ar[r]^-{g}& M,
}
\end{equation}
where $\exp_x:T_xM\longrightarrow M$ is the exponential map at $x$. Consider now the following two subspaces of $T_xM$,
\begin{align*}
HF_x\coloneqq \{v\in T_x M\:|\: \rho_x(g)v=v\:,\:\forall g\in S^1\}\quad\text{and}\quad NF_x\coloneqq HF_x^{\perp},
\end{align*}
called the {\em horizontal} and {\em normal} spaces respectively.
Let $v\in HF_x$, then by the commutativity of the diagram \eqref{Diag:ExpG} we have $g\exp_x(v)=\exp(\rho_x(g)v)=\exp_x(v)$ for all $g\in S^1$, thus $\exp_x(v)\in F$. This shows that the restriction of the exponential map to $HF_x$ is a local diffeomorphism into  $F$ since $\exp_x:T_xM\longrightarrow M$ is a local diffeomorphism itself. 
Similarly it is easy to see that $NF_x$ is a $S^1$-invariant subspace of $T_x M$ and that  the action of  $\rho_x$ on $NF_x-\{0\}$ is free. Hence,  we can decompose $\rho_x(g)$, for each $g\in S^1$, as 
\begin{align*}
\rho_x(g):
\xymatrixcolsep{3cm}\xymatrixrowsep{1.5cm}\xymatrix{
{\begin{array}{c}
HF_x\\
\oplus\\
NF_x
\end{array}}
\ar[r] ^{\left(
\begin{array}{cc}
1 & 0\\
0 & \rho_x(g)
\end{array}\right)
}&
{\begin{array}{c}
HF_x\\
\oplus\\
NF_x.
\end{array}}
}
\end{align*}
 
For two distinct points $x, y\in F$ the representations $\rho_x$ and $\rho_y$ are equivalent. This can be seen for example by considering the parallel transport along a curve joining $x$ and $y$ (\cite[pg. 263]{LM89}). The family of subspaces $\{NF_x\:|\:x\in F\}$ form a vector bundle over $F$ called the {\em normal bundle} of $F$ in $M$. \\

From representation theory we know that there is an orthogonal decomposition 
\begin{equation}\label{Eqn:DecompNF}
NF_x=NF_x(\pi)\oplus \left(\bigoplus_{0<\theta<\pi} NF_x(\theta)\right),
\end{equation}
where $\rho_x(g)| NF_x(\pi)=-1$ and $NF_x(\theta)$ decomposes into $2$-dimensional subspaces on which $\rho_x(g)$ acts as a rotation by an angle $\theta$ for some  fixed $g\in S^1$. Since the action is orientation preserving then the dimension of $NF_x(\pi)$ must be even and therefore the rank of $NF$ must also be even. We write this rank as $\text{rk}(NF)=2(N+1)$ for some $N\in\mathbb{N}_0$. In particular the dimension of $F$ should be $\dim F=(4k+1)-2(N+1)=4k-2N-1$, i.e. each connected component of the fixed point set is a odd dimensional closed submanifold of $M$. We claim that since the action is free on $NF-F$ (where $F$ is regarded as the image of the zero section) then the decomposition of \eqref{Eqn:DecompNF} only contains one value of $\theta$. For example let us assume that the action is given for $N=0$ by 
$$
e^{i\theta}\longrightarrow
\left(\begin{array}{cc}
e^{i\theta} & 0\\
0 & e^{i\mu\theta}
\end{array}
\right),
$$
for some $0<\mu< 1$. Then for $\theta=\mu^{-1}$ the vector
$$
\left(\begin{array}{c}
0\\
1
\end{array}
\right)$$
is fixed and therefore contradicts the fact that the action is free. Arguing by induction the claim follows. We  rephrase this in the following proposition.
\begin{proposition}[{\cite[Lemma 2.2]{U70}}]\label{Prop:S1actionNF}
The normal bundle $NF\longrightarrow F$ has naturally a complex structure such that the induced action on $NF$ is a scalar multiplication. 
\end{proposition}

Let $\pi_\mathcal{D}:\mathcal{D}\longrightarrow F$ be the associated disc bundle of $NF$, i.e. is the fiber bundle with fiber over $x\in F$ given by $\mathcal{D}_x:=\{v\in NF_x\: : \:\norm{v}\leq 1\}$. Then, by means of the exponential map we see that $\mathcal{D}$ is $S^1$-diffeomorphic to a neighborhood of $F$ in $M$. If we restrict the action to the associated sphere bundle $\pi_\mathcal{S}:\mathcal{S}\longrightarrow F$, then Proposition \ref{Prop:S1actionNF} shows that the quotient space $ \mathcal{S}/S^{1}$ becomes the total space of a Riemannian fiber bundle $\pi_{\mathcal{F}}:\mathcal{F}\longrightarrow F$ whose fibers $Y$ are copies of $\mathbb{C}P^N$.
The quotient $\mathcal{D}/S^{1}$ is homeomorphic to the mapping cylinder $C(\mathcal{F})$ of the projection $\pi_{\mathcal{F}}:\mathcal{F}\longrightarrow F$, see Figure \ref{Fig:LocalDescrF}.


 \begin{figure}[h]
\begin{center}
\begin{tikzpicture}
\draw (-2,5)--(2,5);
\draw (-3,3)--(1,3);
\draw (-2,5)--(-3,3);
\draw (2,5)--(1,3);
\draw (-2.5,1)--(1.5,1);
\draw (-3,3)--(-2.5,1);
\draw (2,5)--(1.5,1);
\draw (1,3)--(1.5,1);
\draw [dashed](-2,5)--(-2.25,3);
\draw (-2.5,1)--(-2.25,3);
\draw (-1,3)--(-0.5,1);
\draw [dashed](0,5)--(-0.25,3);
\draw (-0.5,1)--(-0.25,3);
\draw (0,5)--(-1,3);
\draw (-2.4,2) arc (87:155:0.3 and 0.5);
\draw (1.6,2) arc (87:155:0.3 and 0.5);
\draw (-0.4,2) arc (87:155:0.3 and 0.5);
\draw (-2.7,1.7)--(1.33,1.7);
\draw (-2.4,2)--(1.62,2);
\node (a) at (-1,4) {$\mathbb{C}P^N$};
\node (b) at (0,0.7) {$F$};
\draw (-2.84,1.68)--(-2.65,1);
\node (c) at (-3,1.3) {$r$};
\node (d) at (-2.5,0.8) {$0$};
\node (e) at (-3.2,3) {$1$};
\end{tikzpicture}
\caption{Mapping cylinder of the $\mathbb{C}P^N$-fibration $\pi_{\mathcal{F}}:\mathcal{F}\longrightarrow F.$}\label{Fig:LocalDescrF}
\end{center}
\end{figure}
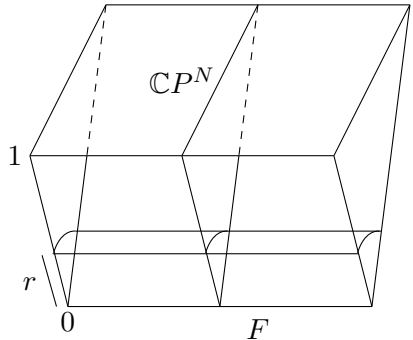

\subsubsection*{Step 2: Strategy of the proof}  
Having understood the local description close to the a singular stratum we explain now the spirit of the proof. As a first approximation we will assume that a neighborhood of $F$ in $M$ is actually $S^1$-isometric to $\mathcal{D}$. We will comment more about this assumption at the end of the section.  For $r>0$, let $N_r(F)$ be the $r$-neighborhood of $F$ in $M/S^1$. We claim that for $r$ small enough
$\sigma_{S^1}(M)$ is the signature $\sigma(M/S^1- N_r(F))$ of the manifold with boundary $M/S^1 -N_r(F)$. To see this let $\omega\in \Omega^{4k}_{\text{bas,c}}(M_0)$, then by Proposition \ref{Prop:BasicPullback} we know that there exists $\bar{\omega}\in\Omega^{4k}_c(M_0/S^1)$ with $\pi^*(\bar{\omega})=\omega$ and $\supp(\bar{\omega})\subset M_0/S^1-N_r(F)$ for $r$ small enough. Since the action on $M_0$ is free we can compute the $S^1$-fundamental class of $M$ using the Fubini theorem for Riemannian submersions (\cite[Proposition A.III.5]{BGM}) 
\begin{align*}
\int_M\alpha \wedge\omega=\int_{M_0/S^1-N_r(F)}\left(\int_{S^1 x}\alpha(x)\right)\bar{\omega},
\end{align*}
where $S^1 x$ denotes the orbit of $x\in M$.  Observe that the integral over the orbit is a constant function on $M$. To see this we just apply the exterior derivative and use \cite[Proposition 6.14.1]{BT82}, which states that this operation commutes with integration along the fiber, i.e
\begin{align*}
d \left(\int_{S^1 x}\alpha(x)\right)=\int_{S^1 x}d\alpha (x)=0,
\end{align*}
where the last equality holds by Lemma \ref{Lema:dalphabas} since as the form $d\alpha$ is basic it is in particular horizontal. Hence, 
we see that, for $r$ sufficiently small, 
\begin{align}\label{Eqn:S1signSignMwB}
\sigma_{S^1}(M)=\sigma(M/S^1-N_{r}(F)).
\end{align} 
Then, applying the Atiyah-Patodi-Singer signature theorem (Theorem \ref{Thm:SignThmMBound}) we get
\begin{align}\label{Eq:MainIdea}
\sigma( M/S^1 - N_r(F))=&\int_{M/S^1-N_r(F)}L\left(T(M/S^1 - N_r(F)),g^{T(M/S^1 - N_r(F))}\right)\\
& +\int_{\partial N_r(F)}TL(\partial N_r(F))+\eta(\partial N_r(F))\notag.
\end{align}
The sign difference of the last two terms of equation \eqref{Eq:MainIdea} with respect to Theorem \ref{Thm:SignThmMBound} is because we give $\partial N_r (F)$ the orientation induced from $N_r(F)$. The main idea of Lott's proof is to study the behavior of the terms in \eqref{Eq:MainIdea} as $r\longrightarrow 0$.

\subsubsection*{Step 3: Computation of the curvature}
In this step we are going deal with the first term of the right hand side of \eqref{Eq:MainIdea}.  Let $\{e^i\}_{i=1}^{v}\cup\{f^\alpha\}_{\alpha=1}^{h}$, where $v\coloneqq 2N$ and  $h\coloneqq 4k-2N-1$, be a local orthonormal basis for $T^*{\mathcal{F}}$ with respect to the submersion metric
\begin{align*}
g^{T\mathcal{F}}= g^{T_H \mathcal{F}}\oplus g^{T_V \mathcal{F}},
\end{align*}
 as in Section \ref{Sect:RiemFibr}. Let $\omega$ be the connection $1$-form of the Levi-Civita connection associated to this basis. As we have seen before, the structure equations satisfied by $\omega$ are 
\begin{align*}
de^i+\omega^i_j\wedge e^j+\omega^i_\alpha\wedge f^\alpha=0,\\
df^\alpha+\omega^\alpha_j\wedge e^j+\omega^\alpha_\beta\wedge f^\beta=0.
\end{align*}
Let $\Omega$ denote the curvature $2$-form on $\mathcal{F}$ of the Levi-Civita connection defined by \eqref{Eqns:ComponentsCurvatureForm}. Our aim is to compute the analog quantities for the metric of $C(\mathcal{F})$
\begin{align}\label{Eqn:MetricC(F)}
g^{TC(\mathcal{F})}=dr^2\oplus g^{T_H \mathcal{F}}\oplus r^2 g^{T_V \mathcal{F}},
\end{align}
in terms of $\omega$, $\Omega$ and $r$. Observe that we can construct a local orthonormal basis for $T^*C(\mathcal{F})$ from the basis above as $\{dr\}\cup\{\widehat{e}^i\}_{i=1}^{v}\cup\{\widehat{f}^\alpha\}_{\alpha=1}^{h}$ where
\begin{align*}
\widehat{e}^i\coloneqq &re^i,\\
\widehat{f}^\alpha\coloneqq &f^\alpha.
\end{align*}
Denote by $\widehat{\omega}$ the connection $1$-form of the metric \eqref{Eqn:MetricC(F)} with respect to this basis. The structure  equations defining  the components of $\widehat{\omega}$ are
\begin{align*}
\widehat{\omega}^r_i\wedge \widehat{e}^i+\widehat{\omega}^r_\alpha\wedge \widehat{f}^\alpha=&0,\\
d\widehat{e}^i+\widehat{\omega}^i_j\wedge \widehat{e}^j+\widehat{\omega}^j_\alpha\wedge \widehat{f}^\alpha+\widehat{\omega}^i_r\wedge dr=&0,\\
d\widehat{f}^\alpha+\widehat{\omega}^\alpha_j\wedge \widehat{e}^j+\widehat{\omega}^\alpha_\beta\wedge \widehat{f}^\beta++\widehat{\omega}^\alpha_r\wedge dr=&0.
\end{align*}
Expanding the second equation, using the structure the equations for $\omega$, we get 
\begin{align*}
0&=d\widehat{e}^i+\widehat{\omega}^i_j\wedge \widehat{e}^j+\widehat{\omega}^i_\alpha\wedge \widehat{f}^\alpha+\widehat{\omega}^i_r\wedge dr\\
&=dr\wedge {e}^i+rde^i+r\widehat{\omega}^i_j\wedge {e}^j+\widehat{\omega}^i_\alpha\wedge {f}^\alpha+\widehat{\omega}^i_r\wedge dr\\
&=dr\wedge(e^i-\widehat{\omega}^i_r)+(rde^i+r\widehat{\omega}^i_j\wedge {e}^j+\widehat{\omega}^i_\alpha\wedge {f}^\alpha)\\
&=dr\wedge(e^i-\widehat{\omega}^i_r)+r(-\omega^i_j+\widehat{\omega}^i_j)\wedge e^j+(-r\omega^i_\alpha+\widehat{\omega}^i_\alpha)\wedge {f}^\alpha,
\end{align*}
so we conclude that $\widehat{\omega}^i_r=e^i$, $\widehat{\omega}^i_j=\omega^i_j$ and $\widehat{\omega}^j_\alpha=r\omega^i_\alpha$. From the first equation we get similarly
\begin{align*}
0=&\widehat{\omega}^r_i\wedge \widehat{e}^i+\widehat{\omega}^r_\alpha\wedge \widehat{f}^\alpha\\
=& -e^i\wedge(re^i)+\widehat{\omega}^r_\alpha\wedge {f}^\alpha\\
=&\widehat{\omega}^{r}_{\alpha j}e^j\wedge f^\alpha+\widehat{\omega}^r_{\alpha\beta}f^\beta\wedge f^\alpha+\omega^r_{\alpha r}dr\wedge f^\alpha,
\end{align*}
thus $\widehat{\omega}^r_\alpha=0$. Finally, from the third structure equation,
\begin{align*}
0&=d\widehat{f}^\alpha+\widehat{\omega}^\alpha_j\wedge \widehat{e}^j+\widehat{\omega}^\alpha_\beta\wedge \widehat{f}^\beta\\
&=d{f}^\alpha+r^2{\omega}^\alpha_j\wedge {e}^j+\widehat{\omega}^\alpha_\beta\wedge {f}^\beta.
\end{align*}
We now expand the $1$-forms $\omega^\alpha_j$ and  $\widehat{\omega}^\alpha_\beta$ in terms of the given basis
\begin{align*}
\omega^\alpha_j=&\omega^\alpha_{jk}e^k+\omega^\alpha_{j\beta}f^\beta,\\
\widehat{\omega}^\alpha_\beta=&\widehat{\omega}^\alpha_{\beta j}\widehat{e}^j+\widehat{\omega}^\alpha_{\beta\gamma}\widehat{f}^\gamma
=r\widehat{\omega}^\alpha_{\beta j}{e}^j+\widehat{\omega}^\alpha_{\beta\gamma} f^\gamma,
\end{align*}
and insert them in the expression above to get
\begin{align}\label{Eqn:dfalpha}
0=df^\alpha+r^2{\omega}^\alpha_{jk}e^k\wedge{e}^j+r^2{\omega}^\alpha_{j\beta} f^\beta\wedge{e}^j+r\widehat{\omega}^\alpha_{\beta j}e^j\wedge {f}^\beta+\widehat{\omega}^\alpha_{\beta\gamma}f^\gamma\wedge {f}^\beta.
\end{align}
Observe that Lemma \ref{Lemma:SymmOmega}(1) implies ${\omega}^\alpha_{jk}e^k\wedge{e}^j=0$, thus by replacing  the structure equation $df^\alpha=-\omega^\alpha_j\wedge e^j-\omega^\alpha_\beta\wedge f^\beta$ in \eqref{Eqn:dfalpha} we obtain
\begin{align*}
0=&(- \omega^\alpha_{j\beta} f^\beta\wedge e^j-\omega^\alpha_{\beta j}e^j\wedge f^\beta-\omega^{\alpha}_{\beta\gamma}f^\gamma\wedge f^\beta)\\
&+(-r^2{\omega}^\alpha_{j\beta} +r\widehat{\omega}^\alpha_{\beta j})e^j\wedge {f}^\beta+\widehat{\omega}^\alpha_{\beta\gamma}f^\gamma\wedge {f}^\beta\\
=&(-r^2{\omega}^\alpha_{j\beta} +r\widehat{\omega}^\alpha_{\beta j})e^j\wedge {f}^\beta+(-\omega^\alpha_{\beta\gamma}+\widehat{\omega}^\alpha_{\beta\gamma})f^\gamma\wedge {f}^\beta,
\end{align*}
where we have used Lemma \ref{Lemma:SymmOmega}(2). Altogether we conclude that  
\begin{align*}
\widehat{\omega}^\alpha_{\beta j}=&r{\omega}^\alpha_{j\beta},\\
\omega^\alpha_{\beta\gamma}=&\widehat{\omega}^\alpha_{\beta\gamma}, 
\end{align*}
i.e. $\widehat{\omega}^{\alpha}_\beta = \omega^{\alpha}_{\beta\gamma}f^\gamma+r^2\omega^{\alpha}_{\beta j}e^j$.
Summarizing, the components of the connection $1$-form $\widehat{\omega}$ are
\begin{align}\label{Eqns:Conn1From}
\widehat{\omega}^{i}_j &= \omega^i_j,\nonumber\\
\widehat{\omega}^{i}_r  &= e^i,\nonumber\\
\widehat{\omega}^{i}_\alpha &= r\omega^i_\alpha ,\\
\widehat{\omega}^{\alpha}_\beta &= \omega^{\alpha}_{\beta\gamma}f^\gamma+r^2\omega^{\alpha}_{\beta j}e^j,\nonumber \\
\widehat{\omega}^{\alpha}_r &=0 .\nonumber  
\end{align}
Having computed the connection $1$-form we now want to compute the curvature $2$-form $\widehat{\Omega}(r)$ using the equations 
\begin{equation*}
\widehat{\Omega}^I_J(r)=d\widehat{\omega}^I_J+\widehat{\omega}^I_K\wedge\widehat{\omega}^K_J,
\end{equation*}
 where $I,J,K\in\{i,\alpha,r\}$. We are mainly interested in computing $\widehat{\Omega}^I_J(r)$  as $r\longrightarrow 0$.

\begin{itemize}
\item \underline{$\widehat{\Omega}^i_j$ component} 
\begin{align*}
\widehat{\Omega}^i_j (r)=&d\widehat{\omega}^i_j+\widehat{\omega}^i_k\wedge\widehat{\omega}^k_j
+\widehat{\omega}^i_\alpha\wedge\widehat{\omega}^\alpha_j+
\widehat{\omega}^i_r\wedge\widehat{\omega}^r_j\\
=&d{\omega}^i_j+{\omega}^i_k\wedge{\omega}^k_j
+(r\omega^i_\alpha)\wedge (r{\omega}^\alpha_j)-e^i\wedge e^j.
\end{align*}
Now we take the limit as $r\longrightarrow 0$
$$\widehat{\Omega}^i_j:=\displaystyle\lim_{r\rightarrow 0}\widehat{\Omega}^i_j(r)=d{\omega}^i_j+{\omega}^i_k\wedge{\omega}^k_j-e^i\wedge e^j
=(R^{TY})^i_j-e^i\wedge e^j,$$
where $R^{TY}$ denotes the curvature in the vertical direction.\\

\item \underline{$\widehat{\Omega}^i_r$ component} 
\begin{align*}
\widehat{\Omega}^i_r (r)&=d\widehat{\omega}^i_r+\widehat{\omega}^i_j\wedge\widehat{\omega}^j_r\\
&=d e^i+{\omega}^i_j\wedge e^j\\
&=-{\omega}^i_j\wedge e^j-{\omega}^i_\alpha\wedge f^\alpha+
{\omega}^i_j\wedge e^j\\
&=-{\omega}^i_\alpha\wedge f^\alpha.
\end{align*}
Therefore $\widehat{\Omega}^i_r:=\displaystyle\lim_{r\rightarrow 0}\widehat{\Omega}^i_r(r)=-\widehat{\omega}^i_\alpha\wedge f^\alpha$.

\item \underline{$\widehat{\Omega}^i_\alpha$ component} 
\begin{align*}
\widehat{\Omega}^i_\alpha(r)&=d\widehat{\omega}^i_\alpha+\widehat{\omega}^i_j\wedge\widehat{\omega}^j_\alpha
+\widehat{\omega}^i_\beta\wedge\widehat{\omega}^\beta_\alpha\\
&=d(r\omega^i_\alpha)+\omega^i_j\wedge(r\omega^j_\alpha)+ (r\omega^i_\beta)\wedge(r^2\omega^\beta_{\alpha j}e^j+\omega^\beta_{\alpha\gamma} f^\gamma),
\end{align*}
so $\widehat{\Omega}^i_\alpha:=\displaystyle\lim_{r\rightarrow 0}\widehat{\Omega}^i_\alpha(r)=dr\wedge\omega^i_\alpha$.

\item \underline{$\widehat{\Omega}^\alpha_r$ component} 
\begin{align*}
\widehat{\Omega}^\alpha_r(r)=\widehat{\omega}^\alpha_i\wedge\widehat{\omega}^i_r=r\omega^\alpha_i\wedge e^i\quad\text{so}\quad \widehat{\Omega}^\alpha_r:=\lim_{r\rightarrow 0}\widehat{\Omega}^\alpha_r(r)=0.
\end{align*}

\item \underline{$\widehat{\Omega}^\alpha_\beta$ component}
\begin{align*}
\widehat{\Omega}^\alpha_\beta (r)=&d\widehat{\omega}^\alpha_\beta+
\widehat{\omega}^\alpha_i\wedge\widehat{\omega}^i_\beta+
\widehat{\omega}^\alpha_\gamma\wedge\widehat{\omega}^\gamma_\beta\\
=& d(r^2\omega^\alpha_{\beta i}e^i+\omega^\alpha_{\beta\gamma}f^\gamma)
+\widehat{\omega}^\alpha_i\wedge(r\omega^i_\beta)\\
&+(r^2\omega^\alpha_{\gamma i}e^i+\omega^\alpha_{\gamma \delta}f^\delta)\wedge(r^2\omega^\gamma_{\beta j}e^j+\omega^\gamma_{\beta \rho} f^\rho).
\end{align*}
Hence $$\widehat{\Omega}^\alpha_\beta:=\displaystyle\lim_{r\rightarrow 0}\widehat{\Omega}^\alpha_\beta(r)=d(\omega^\alpha_{\beta\gamma}\tau^\gamma)+\omega^\alpha_{\gamma \delta} f^\delta\wedge\omega^\gamma_{\beta \rho}f^\rho=(R^{TF})^\alpha_\beta,$$
where $R^{TF}$ denotes the curvature in the horizontal direction.
\end{itemize}
Thus, the components of the curvature $2$-form as $r\longrightarrow 0$ are given by 
\begin{align}\label{Eqns:Curv}
\widehat{\Omega}^{i}_j &= (R^{TY})^i_j-e^i\wedge e^j,\nonumber\\
\widehat{\Omega}^{i}_r  &= -{\omega}^i_\alpha\wedge f^\alpha, \nonumber\\
\widehat{\Omega}^{i}_\alpha &= dr\wedge \omega^i_\alpha, \\
\widehat{\Omega}^{\alpha}_\beta &= (R^{TF})^\alpha_\beta,\nonumber\\
\widehat{\Omega}^{\alpha}_r &=0. \nonumber  
\end{align}

We can conclude that the limit
$$\lim_{r\rightarrow 0}\int_{M/S^1-N_r(F)}L\left(T(M/S^1 - N_r(F)),g^{T(M/S^1 - N_r(F))}\right),$$
is well defined.

\subsubsection*{Step 4:  Transgression term}
In this step we are going to study the transgression term $TL(N_r(F))$ in \eqref{Eq:MainIdea}.This term measures how the metric close to $\partial N_r(F)$ fails to be a product. To begin with, we recall the concrete definition of the transgression form (\cite[Section 11.1]{N03}). 
Let $P_l:\mathfrak{gl}(N,\mathbb{C})\longrightarrow\mathbb{C}$ be an invariant homogeneous polynomial of degree $l$. Given a connection $1$-form $\omega$, with associated curvature form $\Omega$, we can define the characteristic class associated to $P_l$ as the cohomology class $[P_l(\Omega)]$. If we choose another connection $1$-form $\omega'$, with curvature $\Omega'$, then by Chern-Weil theory we know that $[P_l(\Omega)]=[P_l(\Omega')]$ in cohomology. Nevertheless, as differential forms $P_l(\Omega)$ and $P_r(\Omega')$ can differ by at most an exact form, which we call a {\em transgression form} between $\omega$ and $\omega'$. Despite the fact that such a form is not unique, there is a concrete method  to construct one as follows: Set $\theta\coloneqq\omega'-\omega$ and consider for $0\leq t\leq 1$ the path $\omega(t):=\omega+t\theta$ between $\omega$ and $\omega'$, i.e for each fixed $t$ the form $\omega(t)$ is a connection $1$-form, $\omega(0)=\omega$ and $\omega(1)=\omega'$. The associated curvature form is computed again from the structure equation $\Omega(t)=d\omega(t)+\omega(t)\wedge\omega(t).$
\ If ${p}_r$ is the associated polarization polynomial of $P_l$ (\cite[pg. 375]{N03}) then one verifies that
$$P_l(\Omega')-P_l(\Omega)=d\left[l\int_0^1{p}_{l}(\omega'-\omega,\Omega(t), \cdots, \Omega(t))dt\right],$$
thus a transgression form is then
\begin{equation}\label{Eqn:Transgression}
T{P}_l(\omega',\omega)=l\int_0^1{p}_{l}(\omega'-\omega,\Omega(t), \cdots, \Omega(t))dt. 
 \end{equation}
Now we want to use this construction for our particular case of interest where we want to compute the limit of this form as $r\longrightarrow 0$. To do so we consider as above a path of connections which interpolates between the Riemannian connection of $C(\mathcal{F})$, pulled back to $\partial N_r(F)$, and the Riemannian connection of a product metric, at least as $r\longrightarrow 0$. For $t\in[0,1]$ set
\begin{align*}
\widehat{\omega}^i_j(t)\coloneqq &\omega^i_j,\\
\widehat{\omega}^i_r(t)\coloneqq &te^i,\\
 \widehat{\omega}^\alpha_\beta(t)\coloneqq & \omega^\alpha_{\beta\gamma} f^\gamma.
\end{align*}
This path satisfies:
\begin{itemize}
\item At $t=0$, the connection $\widehat{\omega}(0)$ is the connection of a product metric.
\item At $t=1$ we obtain
\begin{align}\label{Eqn:tConnection}
\widehat{\omega}^i_j(1)\coloneqq &\omega^i_j,\notag\\
\widehat{\omega}^i_r(1)\coloneqq &e^i,\\
 \widehat{\omega}^\alpha_\beta(1)\coloneqq & \omega^\alpha_{\beta\gamma} f^\gamma,\notag
\end{align}
which are precisely the non-zero components of the pullback connection $1$-form on $\partial N_r(F)$ as $r\longrightarrow 0$ obtained in \eqref{Eqns:Conn1From}.
\end{itemize}

\begin{remark}\label{Rmk:SFF}
Note that the only component of the second fundamental form of $\partial N_r(F)$ is $\widehat{\omega}^i_r(1)-\widehat{\omega}^i_r(0)=e^i$.
\end{remark}

Now we compute the curvature $2$-form $\Omega_t$ using the relation $\Omega_t=d\omega_t+\omega_t\wedge\omega_t$. In view of  \eqref{Eqns:Conn1From} and \eqref{Eqn:tConnection} we can easily see that the non-zero components of $\Omega_t$ are
\begin{align*}
\widehat{\Omega}^i_j(t)=&(R^{TY})^{i}_j-t^2 e^i\wedge e^j,\notag \\
\widehat{\Omega}^i_r(t)=&- t\omega^i_\alpha\wedge f^\alpha, \\
\widehat{\Omega}^\alpha_\beta(t)=&(R^{TF})^{\alpha}_\beta.\notag
\end{align*}

Note however, that if we consider the form $\widehat{\omega} (t)$ as a form $\check{\omega}$ on $[0,1]\times\mathcal{F}$, then the associated curvature components in this product space are
\begin{align*}
\check{\Omega}^i_j(t)=&(R^{TY})^{i}_j-t^2 e^i\wedge e^j,\notag \\
\check{\Omega}^i_r(t)=&dt\wedge e^i- t\omega^i_\alpha\wedge f^\alpha, \\
\check{\Omega}^\alpha_\beta(t)=&(R^{TF})^{\alpha}_\beta.\notag
\end{align*} 
The difference term between these to curvature forms is $dt\wedge e^{i}$, which by Remark \ref{Rmk:SFF} is precisely the term $dt\wedge(\widehat{\omega}(1)-\widehat{\omega}(0))$ in the transgression formula. This observation allows us to conclude that  
$$\lim_{r\rightarrow 0}\int_{\partial N_r(F)}TL(\partial N_r(F))=\int_{[0,1]\times{\mathcal{F}}}L(\check{\Omega}(t)).$$
Finally note that $\check{\Omega}^i_j(t)$ and $\check{\Omega}^i_r(t)$ do not contain horizontal terms and $\check{\Omega}^\alpha_\beta(t)$ consists only of horizontal terms. Hence, if we define $\underline{\Omega}\in\Omega^2([0,1]\times\CL{F})\otimes\End(TY\oplus\BB{R})$ by  
\begin{align*}
\underline{\Omega}^i_j(t):=&(R^Z)^{i}_j-t^2 e^i\wedge e^j\\
\underline{\Omega}^i_r(t):=&dt\wedge e^i- t\omega^i_\alpha\wedge f^\alpha
\end{align*}
we  we can factorize the $L$-polynomial and obtain
\begin{equation*}
\int_{[0,1]\times{\mathcal{F}}}L(\widehat{\Omega}(t))=\int_{F}\widehat{L}(F)\wedge L(TF),
\end{equation*}
where
$$\widehat{L}(F):=\int_Y\int_{0}^1L(\underline{\Omega})\in \Omega^{\text{odd}}(F).$$
Here the integral over $Y$ means integration along the fibers of $\mathcal{F}$. We will handle the form $\widehat{L}(F)$ later in the proof.

\subsubsection*{Step 5: Adiabatic limit computation}
In this step we want to study the behavior of the term $\eta(\partial N_{r}(F))$ as $r\longrightarrow 0$. Here is where the adiabatic limit methods of the eta invariant described in Section \ref{Section:Dai} appear on the scene. Namely, by Theorem \ref{Thm:Dai} we obtain 
\begin{equation}\label{Eqn:AdiabaticFormulaProof}
\lim_{r\rightarrow 0} \eta(\partial N_r(F))=\int_F L(TF, g^{TF})\wedge \widetilde{\eta}+\eta(A_F\otimes \ker D_Y)+\tau.
\end{equation}
Let us see various features of this limit formula:
\begin{itemize}
\item The global factor $2$ cancels in both sides because $\eta(\partial N_r(F))$ is the eta invariant for the even part of the tangential signature operator.
\item From Remark \ref{Rmk:FlatConnHarmForms} and Hodge theory we see that $\ker(D_Y)$ is the difference of the bundles $H^N_{+}(Y)$ and $H^N_{-}(Y)$ of self-dual and anti-self-dual cohomology groups. In our particular case $Y=\BB{C}P^N$ so $H^N_{\pm}(\BB{C}P^N)$ vanishes if $N$ is odd. If $N$ is even, then $H^N_{+}(Y;\BB{R})=\BB{R}$ and $H^N_{-}(Z)=0$. This real line bundle admits a flat connection by Remark \ref{Rmk:FlatConnHarmForms}. Moreover since it admits a non-vanishing global section as described in Example \ref{Example:ProjBundle} we see that this bundle is actually trivial. 
Hence we see that $ \eta(A_F\otimes \ker D_Y)=\eta(F)$. 
\item From Example \ref{Example:ProjBundle} we see that the $\tau$-invariant in \eqref{Eqn:AdiabaticFormulaProof} vanishes. 
\end{itemize}

\subsubsection*{Step 6: Equivariant methods}
This is the final step of the proof. So far we have shown that 
$$\sigma_{S^1}(M)=\int_{M_0/S^1}L\left(T(M_0/S^1),g^{T(M_0/S^1)}\right)+\eta(F)+\int_F \widehat{L}(TF)\wedge L(TF)+\int_F  L(TF)\wedge\widetilde{\eta}.$$
Lott claims in his work that  $\widehat{L}(F)$ and $\widetilde{\eta}$ are zero as a result of some equivariant techniques. We want to develop his arguments in some more detail. The key observation is that, from Theorem \ref{Thm:SliceThm} and the local description given in {\em Step 2}, the fibration $\mathcal{F}\longrightarrow F$ is associated to a principal $G$-bundle over $F$ with compact structure group, thus we can apply the results described in Section \ref{Section:CompactStrGroup}. 
\begin{enumerate}
\item Since $G$ is compact, we can assume that the vertical metric $g^{T_V\mathcal{F}}$ is $G$-invariant. Using the commutativity of the diagram \eqref{DiagCW} we see that 
$\widehat{L}(F)$ belongs indeed to the image of the Chern-Weil homomorphism and therefore it should be an even form. More concretely, from the commutative diagram 
\begin{equation*}
\xymatrixcolsep{4pc}\xymatrixrowsep{4pc}\xymatrix{
(\mathbb{C}[\mathfrak{g}]\otimes\Omega([0,1]\times Y))^G \ar[d]_-{\displaystyle{\int_Y}} \ar[r]^-{\phi_{\boldsymbol{\omega}}} & \Omega([0,1]\times\mathcal{F}) \ar[d]^-{\displaystyle{\int_Y \int_0^1}}\\
\mathbb{C}[\mathfrak{g}]^G \ar[r]^-{\phi_{\boldsymbol{\omega}}}  & \Omega(F),
}
\end{equation*}
it follows that $\widehat{L}(F)\in \ran(\phi_{\boldsymbol{\omega}})\in\Omega^{\text{ev}}(F)$. However, by construction this form is odd so we must have that $\widehat{L}(F)=0$. 
\item From Proposition \ref{Prop:VanishinEtaForm} it follows that in this case $\widetilde{\eta}=0$. 
\end{enumerate}

This concludes the proof of Theorem \ref{Thm:S1SignatureThm} under the assumption a neighborhood of $F$ in $M$ is $S^1$-isometric to $DNF$. In the general case, as we will see in the examples below, in the limit as one approaches the fixed point set, the calculations above remain still valid. 

\subsection{The Witt condition}

In this section we comment on an important topological interpretation of the $S^1$-signature in the context of intersection homology theory introduced by Goresky and McPhearson in \cite{GMcP80}.\\

As we saw in the proof of Theorem \ref{Thm:S1SignatureThm} above, each connected component $F\subseteq M^{S^1}$ of the fixed point set is an odd-dimensional closed manifold, whose dimension can be written as
$$\dim F=4k-2N-1, \quad \text{for some $N\in\mathbb{N}_0$}.$$  
We now distinguish the two possible cases for $N$. 
\begin{definition}\label{Def:Witt}
We say that $M/S^1$ satisfies the {\em Witt condition}, if $N$ is odd, that is, the codimension of the fixed point set $M^{S^1}$ in $M$ is divisible by four.
\end{definition}

\begin{lemma}\label{Lemma:VanishingEtaWitt}
If $M/S^1$ satisfies the Witt condition then $\eta(M^{S^1})=0$. 
\end{lemma}

\begin{proof}
This follows from the discussion in Section \ref{Section:VanishEta} since $4k-2N-1=2(2k-N)-1$ and $2k-N$ is odd if and only if $N$ is odd. 
\end{proof}

Stratified spaces for which the middle dimensional cohomology of the links vanish are called {\em Witt spaces} (see \cite{S83}). This is of course consistent with Definition \ref{Def:Witt} since $H^N(\mathbb{C}P^N)$ vanishes if and only if $N$ is odd. For this kind of stratified spaces one can always construct the Goresky-McPhearson $L$-class following the procedure described in detail in \cite[Section 5.3]{B07}.  The following result describes an explicit form of such a $L$-homology class for our case of interest.  

\begin{proposition}[$L$-homology Class, {\cite[Proposition 8]{L00}}]\label{Prop:LHomologyClass}
In the Witt case the differential form $L(T(M/S^1))$ represents the $L$-homology class of $M/S^1$.
\end{proposition}

\begin{proof}
In \cite[Section 4]{BC91} Bismut and Cheeger constructed homology classes for certain singular manifolds  where the singularities are modeled as the mapping cone of a certain fibration as described above for $M/S^1$. By \cite[Theorem 5.7]{BC91} the $L$-homology class of $M/S^1$ in the Witt case is represented by the pair $(L(T(M/S^1),L(T M^{S^1})\wedge\widetilde{\eta})$, where $\widetilde{\eta}$ is the eta form of the $\mathbb{C}P^N$-fibration over $M^{S^1}$. From {\em Step 6} in the section above we know that this eta form vanishes.  Hence, the differential form  $L(T(M/S^1))$ represents the $L$-homology class of $M/S^1$ in the mentioned Bismut-Cheeger model.
\end{proof}

In addition, for Witt spaces there is a well defined non-degenerate  pairing in intersection homology (\cite[Section 4.4]{B07}) that gives rise to a signature invariant. As a consequence of Lemma \ref{Lemma:VanishingEtaWitt} and Proposition \ref{Prop:LHomologyClass} we obtain the following remarkable corollary of Theorem \ref{Thm:S1SignatureThm}. 

\begin{coro}[{\cite[Corollary 1]{L00}}]\label{Coro:IH}
In the Witt case, 
\begin{align*}
\sigma_{S^1}(M)=\int_{M_0/S^1} L(T(M_0/S^1)),
\end{align*}
equals the intersection homology signature of $M/S^1$. 
\end{coro}

Although we did not really discuss what intersection homology is in detail, we will illustrate Corollary \ref{Coro:IH} in an example later on. 

\subsection{The Equivariant $S^1$-Euler characteristic}

Inspired in the work of Lott, we define an analogous invariant.
\begin{definition}
We define the {\em equivariant $S^1$-Euler characteristic} as
\begin{align*}
\chi_{S^1}(M)\coloneqq \sum_{j}^{n}(-1)^j\dim H_{\text{bas},c}^j(M-M^{S^1}). 
\end{align*}
\end{definition}

\begin{remark}
It follows directly from the isomorphisms \eqref{IsomsBasic} that $\chi_{S^1}(M)=\chi(M,M^{S^1})$, the relative Euler characteristic. 
\end{remark}

As a consequence of the proof of Theorem \ref{Thm:S1SignatureThm} and from the fact that for $r>0$ sufficiently small we have isomorphisms 
\begin{align*}
H^j(M/S^1-N_t(F), \partial N_r(F))\cong H^j(M/S^1,M^{S^1};\mathbb{R}),
\end{align*}
due the homotopy invariance of cohomology, we can use Theorem \ref{Thm:GBManBound} to prove the following formulas for $\chi_{S^1}(M)$. 
\begin{theorem}\label{Thm:S1EC}
Suppose $S^1$ acts effectively and semifreely on $M$, then 
\begin{enumerate}
\item If the dimension of $M$ is odd, 
\begin{align*}
\chi_{S^1}(M)=\int_{M_0/S^1} e(T(M_0/S^1), g^{T(M_0/S^1)}).
\end{align*}
\item If the dimension of $M$ is even, 
\begin{align*}
\chi_{S^1}(M)=\frac{1}{2}\int_{M^{S^1}}e(TM^{S^1}, g^{TM^{S^1}}).
\end{align*}
\end{enumerate}
\end{theorem}

\begin{example}[Spinning a closed manifold II]
For the free action on the product manifold $M\coloneqq X\times S^1$ of Example \ref{Ex:SpiningM}, Theorem \ref{Thm:S1EC} describes
the Chern-Gau\ss -Bonnet formula for closed manifolds. 
\end{example}

\begin{example}[Spinning a closed manifold with boundary II]
In the case we $M$ is obtained by spinning a compact manifold with boundary as in Example \ref{Ex:SpiningMwB}, Theorem \ref{Thm:S1EC} reduces to Theorem \ref{Thm:GBManBound}. 
\end{example}

\section{Induced Dirac-Schr\"odinger operator on $M_0/S^1$}\label{Section:InducedOpConstruction}\label{Sec:Induced}

The goal of this chapter is to implement the constructions studied in Chapter \ref{Sect:BH} for the special case of a semi-free $S^1$-action discussed in  Section \ref{Section:S1 Signature}. We will push down various  types of Dirac operators from $M$ to the quotient space $M_0/S^1$. 

\subsection{The mean curvature $1$-form} \label{Sect:MeanCurvFrom}

Let $M$ be an $(n+1)$-dimensional oriented, closed  Riemannian manifold on which $S^1$ acts by orientation preserving isometries (recall that in Chapter  \ref{Sect:Lott} we had $n=4k$, but here we will treat the general case). Denote by $\nabla$ its associated Levi-Civita connection. As in Section \ref{Section:BasicsDef} let $V$ be the generating vector field of the $S^1$-action. The flow of the  vector field $V$ generates a $1$-dimensional foliation on $M_0$ induced by the {\em distribution}
\begin{equation}
L_x:=\{v\in T_xM_0\:|\: v=\lambda V(x)\quad\text{for some}\quad \lambda\in\mathbb{R}\}\leq T_x M_0,\quad \text{for}\: x\in M_0.
\end{equation}
We will denote by $L^\perp$ the {\em transverse distribution}, i.e.  $TM_0=L\oplus L^\perp$. The distribution $L$ is always integrable and the corresponding integral curves are precisely the $S^1$-orbits of  the action. In contrast, the transverse distribution will not be in general integrable. 

\begin{definition}
Let $X\coloneqq V/\norm{V}$ be the {\em unit vector field} which defines the foliation $L$.  Using the musical isomorphism \eqref{Eqn:Musical} we define:
\begin{enumerate}
\item The associated {\em characteristic $1$-form} $\chi\coloneqq X^\flat$.
\item The {\em mean curvature vector field} $H\coloneqq\nabla_X X$.
\item The {\em mean curvature $1$-form} $1$-form $\kappa\coloneqq H^\flat$.
\end{enumerate}

\end{definition}
As in Section \ref{Section:BasicsDef} we are interested in the complex of basic differential forms on $M_0$,
\begin{equation}\label{Eqn:BasicFormsM0}
\Omega_\text{bas}(M_0)\coloneqq \{\omega\in\Omega(M_0)\:|\:L_V\omega=0\:\:\text{and}\:\: \iota_V=0\}.
\end{equation}
\begin{remark}\label{Rmk:LieDer}
Note that since 
\begin{align*}
L_X=L_{\frac{V}{\norm{V}}}=\frac{1}{\norm{V}}L_V+d\left(\frac{1}{\norm{V}}\right)\wedge\iota_V,
\end{align*}
we see that a form $\omega$ is basic if and only if $L_X\omega=0$ and $\iota_X\omega=0$. 
\end{remark}
Now we will collect some important properties of the mean curvature $1$-form $\kappa$. 

\begin{lemma}\label{Lemma:VFH}
The mean curvature vector field satisfies $H\in C^\infty (L^\perp)$. As a consequence $\kappa$ is horizontal, i.e. $\iota_X\kappa=0$.
\end{lemma}
\begin{proof}
This follows from the fact that the Levi-Connection is metric preserving, i.e. 
$$0=X\inner{X}{X}=2\inner{X}{\nabla_X X}=2\inner{X}{H}.$$
\end{proof}

\begin{lemma}[{\cite[Chapter 6]{T97}}]\label{Lemma:kappa}
The mean curvature form $\kappa$ satisfies $\kappa=L_X \chi$. 
\end{lemma}

\begin{proof}
Using the defining properties of the Levi-Civita connection and the relation $0=Y\inner{X}{X}=2\inner{X}{\nabla_Y X}$ for any vector field $Y$,  we compute 
\begin{align*}
(L_X\chi)(Y)=&X(\chi(Y))-\chi([X,Y])\\
=&X\inner{X}{Y}-\inner{X}{[X,Y]}\\
=&\inner{H}{Y}+\inner{X}{\nabla_X Y-[X,Y]}\\
=&\inner{H}{Y}+\inner{X}{\nabla_Y X}\\
=&\kappa(Y).
\end{align*}
\end{proof}

Recall that  $\alpha\coloneqq V^\flat/\norm{V}^2$  is the $1$-form considered in Section  \ref{Section:BasicsDef} which satisfies $\alpha(V)=1$ . It is easy to see that $\chi=\norm{V}\alpha$. Indeed,
 $\norm{V}\alpha(X)=\alpha(\norm{V}X)=\alpha(V)=1$.
The following statement is a consequence of this relation, Corollary \ref{Coro:alphaInv} and the lemma above.

\begin{coro}\label{Coro:ChiInv}
The characteristic $1$-form $\chi$ satisfies $L_V\chi=0$, that is $\chi$ is $S^1$-invariant. 
\end{coro}

We combine Cartan's formula and Lemma \ref{Lemma:kappa}  to get  $\kappa=\iota_X d\chi$, or equivalently 
\begin{equation}\label{Eqn:dchi}
d\chi + \kappa\wedge \chi \eqqcolon \varphi_0,
\end{equation}
where $\varphi_0$ satisfies $\iota_X\varphi_0=0$, i.e. $\varphi_0\in \Omega^2_{\text{hor}}(M_0)$. Equation \eqref{Eqn:dchi} is known as {\em Rummler's formula} and it holds for general tangentially oriented foliations (\cite[Lemma 10.4]{BGV}, \cite[Chapter 4]{T97}).\\

Observe that the characteristic form $\chi$ can be viewed as the volume form on each leaf of the foliation as the vector field $X$ satisfies $X\in C^\infty (L)$, it is $S^1$-invariant and $\norm{\chi}=1$. In particular, the volume of the orbit function $h:M_0/S^1\longrightarrow \mathbb{R}$ used in \eqref{Def:L2FhGen} can be written explicitly as  
\begin{equation}\label{Eqn:h}
h(y)=\int_{\pi^{-1}_{S^1}(y)} \chi.
\end{equation}
\begin{lemma}\label{Lemma:dh}
The exterior derivative of the volume of the orbit function is $dh=-h\kappa$. Thus, the mean curvature form $\kappa$ measures the volume change of the orbits. 
\end{lemma}
\begin{proof}
We use  \cite[Proposition 6.14.1]{BT82} and  \eqref{Eqn:dchi} to compute 
\begin{align*}
d h= d\int_{\pi^{-1}_{S^1}(y)} \chi= \int_{\pi^{-1}_{S^1}(y)}d\chi=-\int_{\pi^{-1}_{S^1}(y)}\kappa\wedge\chi +\int_{\pi^{-1}_{S^1}(y)}\varphi_0=-h\kappa, 
\end{align*}
where we have used that the integral of $\varphi_0$ is zero because this is a horizontal $2$-form. 
\end{proof}
The next proposition shows that all the geometric quantities discussed above are encoded in the norm of the generating vector field $V$. 
\begin{proposition}\label{Prop:NormV}
In terms of $\norm{V}$ we can express
\begin{enumerate}
\item $\chi=\norm{V}\alpha$.
\item $\kappa=-d\log(\norm{V})$.
\item $\varphi_0=\norm{V} d\alpha$.
\end{enumerate}
\end{proposition}
\begin{proof}
\begin{enumerate}
\item This was shown above. 
\item We use Remark \ref{Rmk:LieDer}  and Corollary \ref{Coro:ChiInv} to compute 
\begin{align*}
L_X(\chi)=\frac{1}{\norm{V}}L_V\chi+d\left(\frac{1}{\norm{V}}\right)\wedge\iota_V\chi=-\frac{d\norm{V}}{\norm{V}^2}\norm{V}=-d\log(\norm{V}).
\end{align*}
 The result then follows by Lemma \ref{Lemma:kappa}.
\item Using (1), (2) and \eqref{Eqn:dchi} we obtain
\begin{align*}
d\alpha=d\left(\frac{1}{\norm{V}}\chi\right)=-\frac{d\norm{V}}{\norm{V}^2}\wedge\chi+\frac{1}{\norm{V}}d\chi=\frac{1}{\norm{V}}\varphi_0.
\end{align*}
\end{enumerate}
\end{proof}

\begin{coro}\label{Coro:kappa}
The mean curvature $1$-form $\kappa$ is closed and basic.
\end{coro}
\begin{proof}
By Proposition \ref{Prop:NormV}(2) we actually see that $\kappa$ is exact and therefore closed. Since $H\in C^{\infty}(L^\perp)$ we see that $\kappa$ is horizontal. On the other hand by Cartan's formula and Lemma \ref{Lema:dalphabas} we see that $L_V\kappa =0$.
\end{proof}

\begin{proposition}\label{Prop:varphi0}
The following relations for $\varphi_0$ hold:
\begin{enumerate}
\item If $Y_1, Y_2\in C^{\infty}(L^\perp)$, then $\varphi_0(Y_1,Y_2)=-\chi([Y_1,Y_2])$.
\item $d\varphi_0+\kappa\wedge\varphi_0=0$.
\item $\varphi_0\in\Omega^2_\textnormal{bas}(M_0)$.
\end{enumerate}
\end{proposition}

\begin{proof}
Part (1) and (2) follow immediately from \eqref{Eqn:dchi} and Corollary \ref{Coro:kappa}. On the other hand (3) follows directly from Lemma \ref{Lema:dalphabas} and Proposition \ref{Prop:NormV}(3).
\end{proof}
\begin{remark}
Form Proposition \ref{Prop:varphi0}(1) we see that $\varphi_0$ can be regarded as the curvature form of the principal $S^1$-bundle $M_0\longrightarrow M_0/S^1$. In particular we see that $L^\perp$ is an integrable distribution if and only if $\varphi_0$ vanishes. 
\end{remark}

\subsection{The operator $T(D)$}

Let us now consider the Hermitian vector bundle $E\coloneqq  \wedge_\mathbb{C} T^* M$ of  Example \ref{Ex:ExtAlg}. As explained before, the $S^1$-action on $M$ endows $E$ with a $S^1$-vector bundle structure, i.e. for each $g\in S^1$ we have a commutative diagram 
$$\xymatrixrowsep{2cm}\xymatrixcolsep{2cm}\xymatrix{
\wedge_\mathbb{C} T^* M \ar[d]_{} \ar[r]^-{g} & \wedge_\mathbb{C} T^* M\ar[d]^{}\\
M \ar[r]^-{g} & M.
}$$
Moreover, we also showed in  Example \ref{Ex:ExtAlg} that the action on differential forms is given by the pullback 
$U_g\omega\coloneqq  (g^{-1})^*\omega$ for $g\in S^1$.
\begin{lemma}
The Hodge star operator on $M$ commutes with the $S^1$-action on differential forms. 
\end{lemma}
\begin{proof}
Since the $S^1$-action is metric and orientation preserving, then $g^*\vol_M=\vol_M$ for all $g\in S^1$.  From this observation and from the fact that pullback respects the wedge product we compute for $\omega_1,\omega_2\in\Omega(M)$, 
\begin{align*}
\omega_1\wedge (U_g *\omega_2)=&\omega_1\wedge  (g^{-1})^* *\omega_2\\
=& (g^{-1})^*((g^*\omega_1)\wedge  *\omega_2)\\
=& (g^{-1})^*(\inner{g^*\omega_1}{\omega_2}\vol_M)\\
=& \inner{\omega_1}{(g^{-1})^*\omega_2}\vol_M\\
= & \omega_1\wedge (* U_{g} \omega_2).
\end{align*}
\end{proof}
From this lemma and Proposition \ref{Prop:Chirl}(4) we obtain the following result. 
\begin{proposition}\label{Prop:S1commD}
The Hodge-de Rham operator $D=d+d^\dagger$ of $M$ defined on the core  $\Omega(M)$ is $S^1$-invariant in the sense of Definition \ref{Def:GInvOp}. 
\end{proposition}
This result shows that we are in position to apply the construction of Br\"uning and Heintze described in Chapter \ref{Sect:BH}. The strategy is then as  follows:
\begin{enumerate}
\item Explicitly construct the vector bundle $F\longrightarrow M_0/S^1$ introduced in Section \ref{Section:GVectorBundles} and describe the $L^2$-inner product \eqref{Def:L2FhGen}. 
\item Understand the isomorphism $\Phi$ of Theorem \ref{Thm:Fund}.
\item Describe the self-adjoint operator $D:\Omega(M)^{S^1}\longrightarrow\Omega(M)^{S^1}$ of Lemma \ref{Lemma:OpS} 
\item Obtain an explicit characterization of the self-adjoint operator $T$ of Proposition \ref{Prop:OpT} and describe its properties. For example, compute its principal symbol (Proposition \ref{Prop:SDiff}). 
\end{enumerate}

\subsubsection{Decomposition of $S^1$-invariant differential forms}
We begin with a decomposition result of the space of $S^1$-invariant forms in terms of the basic forms \eqref{Eqn:BasicFormsM0}. Recall that we have the inclusion  $\Omega_\text{bas}(M_0)\subset \Omega(M_0)^{S^1}$.

\begin{proposition}[{\cite[Proposition 6.12]{T97}}]\label{Prop:SESBasic}
There is a short exact sequence of complexes
\begin{equation*}
\xymatrixcolsep{2cm}\xymatrix{
0 \ar[r] & \Omega^*_{\textnormal{bas}}(M_0)  \ar@{^{(}->}[r] & \Omega^* (M_0)^{S^1} \ar[r]^-{\iota_V}& \Omega^{*-1}_{\textnormal{bas}}(M_0) \ar[r] & 0. 
}
\end{equation*}
\end{proposition}

\begin{proof}
Using Cartan's formula \eqref{Eqn:Cartan} it follows that, up to a sign, the diagram
\begin{align*}
\xymatrixrowsep{2cm}\xymatrixcolsep{2cm}\xymatrix{
0 \ar[r] & \Omega^r_{\textnormal{bas}}(M_0) \ar[d]_-{d}  \ar@{^{(}->}[r] & \Omega^r (M_0)^{S^1} \ar[d]_-{d} \ar[r]^-{\iota_V}& \Omega^{r-1}_{\textnormal{bas}}(M_0) \ar[d]_-{d} \ar[r] & 0\\
0 \ar[r] & \Omega^{r+1}_{\textnormal{bas}}(M_0)  \ar@{^{(}->}[r] & \Omega^{r+1} (M_0)^{S^1} \ar[r]^-{\iota_V}& \Omega^{r}_{\textnormal{bas}}(M_0) \ar[r] & 0
}
\end{align*} 
commutes since $L_V=0$ on $\Omega(M_0)^{S^1}$. Hence we have indeed a map between complexes. We now verify the exactness of the sequence:
\begin{itemize}
\item The map $\iota_V:\Omega^{r}(M_0)^{S^1}\longrightarrow \Omega^{r-1}_{\textnormal{bas}}(M_0)$ is well defined: if $\tilde{\omega}\in\Omega^r(M_0)^{S^1}$, then $L_V\iota_V\tilde{\omega}=\iota_VL_V\tilde{\omega}=0$.
\item By construction  the sequence is exact at $\Omega^r(M_0)^{S^1}$.
\item The map $\iota_V:\Omega^{r}(M_0)^{S^1}\longrightarrow \Omega^{r-1}_{\textnormal{bas}}(M_0)$ is surjective: Let us consider the form $\tilde{\omega}:=(-1)^{r-1}\omega\wedge\alpha\in\Omega^r(M_0)^{S^1}$ where $\omega\in\Omega^{r-1}_{\text{bas}}(M_0)$. Then
\begin{itemize}
\item $\iota_V\tilde{\omega}=\omega$.
\item $L_V\tilde{\omega}=(-1)^{r-1}(L_V\omega)\wedge\alpha+(-1)^{r-1}\omega\wedge L_V\alpha=0$.
\end{itemize}
\end{itemize}
\end{proof}

\begin{coro}\label{Coro:DecInvForm}
Any $S^1$-invariant form $\omega\in \Omega(M_0)^{S^1}$ can be uniquely decomposed as $\omega=\omega_0+\omega_1\wedge\chi$, where $\omega_0,\omega_1\in\Omega_{\textnormal{bas}}(M_0)$. With respect to this decomposition we will represent the form $\omega$ as the column vector
\begin{equation*}
\omega=
\left(
\begin{array}{c}
\omega_0\\
\omega_1
\end{array}
\right).
\end{equation*}
\end{coro}

\begin{proof}
By Proposition \ref{Prop:SESBasic} we can express uniquely $\omega=\omega_0+\tilde{\omega}_1\wedge\alpha$ for some basic forms $\omega_0,\tilde{\omega_1}\in\Omega_\text{bas}(M_0)$. Next, in view of Proposition \ref{Prop:NormV}(1), we set $\omega_1\coloneqq\tilde{\omega}_1/\norm{V}$ so that $\omega_1\in \Omega_\text{bas}(M_0)$ and $\omega=\omega_0+{\omega}_1\wedge\chi$.
\end{proof}

\begin{example}\label{Ex:EpsilonInvDec}
We can write the action of the Gau\ss-Bonnet involution $\varepsilon$ with respect to this decomposition as 
\begin{equation*}
\varepsilon\bigg{|}_{\Omega(M_0)^{S^1}}=
\left(\begin{array}{cc}
\varepsilon & 0 \\
0 & -\varepsilon
\end{array}
\right),
\end{equation*}
since $\varepsilon \omega=\varepsilon \omega_0 +\varepsilon(\omega_1\wedge\chi)=\varepsilon \omega_0 -(\varepsilon\omega_1)\wedge\chi$.
\end{example}

\subsubsection{Construction of the bundle $F$}\label{Section:ConstrF}

We now construct the Hermitian vector bundle $F\longrightarrow M_0/S^1$ following its explicit description given in Section \ref{Section:GVectorBundles} . We start by pointing out some important remarks:  
\begin{itemize}
\item The action on  $M_0$ is free and therefore the $S^1$-invariant bundle $E'$ of \eqref{Eqn:DefE'} is nothing else but $E'=\wedge_\mathbb{C} T^*M_0$.
\item From Remark \ref{Remark:rkF} we know that the rank of $F$ must agree with the rank of $E'$, which is $\rk(E')= 2^{n+1}$.
\item From Proposition \ref{Prop:BasicPullback} it follows that for each basic form $\beta\in\Omega^r_\text{bas}(M_0)$ there exists a unique $\bar{\beta}\in\Omega^r(M_0/S^1)$ such that $\pi^*_{S^1}\bar{\beta}=\beta$. Thus, using Corollary \ref{Coro:DecInvForm} we can identify $\Omega(M_0)^{S^1}\cong \Omega(M_0/S^1)\otimes\mathbb{C}^2$, via the orbit map $\pi_{S^1}$.
\end{itemize}
These observations indicate that 
\begin{align*}
F\coloneqq E'/S^1=\wedge_\mathbb{C} T^*(M_0/S^1)\oplus \wedge_\mathbb{C} T^*(M_0/S^1). 
\end{align*}
Indeed, given $x\in M_0$ and $\omega_x=\omega'_x+\omega''_x\wedge\chi_x\in \wedge_\mathbb{C} T^*_x M_0$ where $\iota_{V_x}\omega'_x=\iota_{V_x}\omega''_x=0$, the orbit map on $E'$ is explicitly given by
\begin{align}\label{Def:BundleF}
\pi'_{S^1}:
\xymatrixrowsep{0.01cm}\xymatrixcolsep{2cm}\xymatrix{
E'=\wedge_\mathbb{C} T^*M_0 \ar[r] & F=\wedge_\mathbb{C} T^*(M_0/S^1)\oplus \wedge_\mathbb{C} T^*(M_0/S^1)\\
\omega_x=\omega'_x+\omega''_x\wedge\chi_x \ar@{|->}[r] & 
{\left(\begin{array}{c}
\bar{\omega}_y' \\
{\bar{\omega}_y''}
\end{array}\right),}
}
\end{align}
where $\pi_{S^1}(x)=y$ and the form $\bar{\omega}_y'\in\wedge_\mathbb{C} T^*(M_0/S^1)$ (similarly for $\bar{\omega}_y''$) is defined by the relation $\omega_x(v_x)=\bar{\omega}_y((\pi_{S^1})_* v_x)$ for all $v_x\in T_x M_0$. \\

Hence, the the diagram \eqref{Diag:OrbitMaps}  becomes,
\begin{align*}
\xymatrixrowsep{2cm}\xymatrixcolsep{2cm}\xymatrix{
E' =\wedge_\mathbb{C} T^*M_0\ar[d]_-{\pi_E} \ar[r]^-{\pi'_{S^1}} & F=\wedge_\mathbb{C} T^*(M_0/S^1)\oplus \wedge_\mathbb{C} T^*(M_0/S^1)  \ar[d]^-{\pi_F}\\
M_0 \ar[r]^-{\pi_{S^1}}& M_0/S^1.
}
\end{align*}

Using the notation above we see from Lemma \ref{Lemma:InducedMetricF} that the metric on the bundle $F$, which is inherited from the metric of $E'$, is given by
\begin{align*}
\left\langle
\left(\begin{array}{c}
\bar{\omega}'_y \\
{\bar{\omega}''_y}
\end{array}\right),
\left(\begin{array}{c}
\bar{\beta}'_y \\
{\bar{\beta}''_y}
\end{array}\right)
\right\rangle_{F}(y)=\inner{\omega'_x}{\beta'_x}_{E}(x) +\inner{\omega''_x}{\beta''_x}_{E}(x),
\end{align*}
since $\norm{\chi}=1$. 
 The associated  weighted $L^2$-inner product \eqref{Def:L2FhGen} on $L^2(F,h)$ is
\begin{align*}
\left(
\left(
\begin{array}{c}
\bar{\omega}_0\\
\bar{\omega}_1
\end{array}
\right),
\left(
\begin{array}{c}
\bar{\beta}_0\\
\bar{\beta}_1
\end{array}
\right)
\right)_{L^2(F,h)}=
\int_{M_0/S^1}\left(\left\langle
\left(\begin{array}{c}
\bar{\omega}_0\\
{\bar{\omega}_1}
\end{array}\right),
\left(\begin{array}{c}
\bar{\beta}_0 \\
{\bar{\beta}_1}
\end{array}\right)
\right\rangle_{F}(y)\right)h(y)\vol_{M_0/S^1}(y),
\end{align*}
where $h(y)=\vol(\pi^{-1}_{S^1}(y))$ and $\vol_{M_0/S^1}$ is the volume form of the induced quotient metric. 

\begin{remark}\label{Rmk:L2F}
The $L^2$-inner product induced just from the Hermitian metric on $F$ and from the quotient metric is, for compactly supported smooth forms, 
\begin{align*}
\left(
\left(
\begin{array}{c}
\bar{\omega}_0\\
\bar{\omega}_1
\end{array}
\right),
\left(
\begin{array}{c}
\bar{\beta}_0\\
\bar{\beta}_1
\end{array}
\right)
\right)_{L^2(F)}=
\int_{M_0/S^1}\left\langle
\left(\begin{array}{c}
\bar{\omega}_0\\
{\bar{\omega}_1}
\end{array}\right),
\left(\begin{array}{c}
\bar{\beta}_0 \\
{\bar{\beta}_1}
\end{array}\right)
\right\rangle_{F}(y)\vol_{M_0/S^1}(y),
\end{align*}
\end{remark}

\subsubsection{Description of the isomorphism $\Phi$}

Using the description of the bundle $F$ above we want to describe its image under the isomorphism $\Phi$ of Theorem \ref{Thm:Fund}. Given an $S^1$-invariant form with compact support $\omega\in\Omega_c(M_0)^{S^1}$  there are two unique compactly supported basic differential forms $\omega_0,\omega_1\in\Omega_{\text{bas},c}(M_0)$ such that  $\omega=\omega_0+\omega_1\wedge\chi$. With respect to the vector notation introduced in Corollary \ref{Coro:DecInvForm} we write
\begin{align*}
\left(\begin{array}{c}
\omega_0\\
\omega_1\\
\end{array}\right)=
\left(\begin{array}{c}
\pi_{S^1}^*\bar{\omega}_0\\
\pi_{S^1}^*\bar{\omega}_1\\
\end{array}\right),
\end{align*}
where $\bar{\omega}_0,\bar{\omega}_1\in\Omega_c(M_0/S^1)$. This representation allow us to express the isomorphism $\Phi$, on compactly supported forms,  as

\begin{align*}
\Phi: \xymatrixrowsep{0.01cm}\xymatrixcolsep{1.7cm}\xymatrix{
\Omega_c(M_0)^{S^1} \ar[r] & \Omega_{\text{bas},c}(M_0)\oplus \Omega_{\text{bas},c}(M_0)  \ar[r] & \Omega_c(M_0/S^1)\oplus \Omega_c(M_0/S^1)\\
\omega  \ar@{|->}[r] &
{
\left(\begin{array}{c}
\omega_0\\
\omega_1\\
\end{array}\right)=
\left(\begin{array}{c}
\pi_{S^1}^*\bar{\omega}_0\\
\pi_{S^1}^*\bar{\omega}_1\\
\end{array}\right)
}
\ar@{|->}[r] &
{
\left(\begin{array}{c}
\bar{\omega}_0\\
\bar{\omega}_1\\
\end{array}\right).
}
}
\end{align*}
We can extend this map to $\Phi:L^2(M)^{S^1}\longrightarrow L^2(F,h)$ by density. 

\subsubsection{Description of the operator $S(D)$}
Now we want to understand the operator $S$ of Lemma \ref{Lemma:OpS} associated to the Hodge-de Rham operator $D$, i.e. 
\begin{align*}
S(D)\coloneqq (d+d^\dagger)\bigg{|}_{\Omega_c(M_0)^{S^1}}=\left(d+(-1)^{(n+1)+1}\star d\star\right)\bigg{|}_{\Omega_c(M_0)^{S^1}}. 
\end{align*}
The idea is to view $S$ thorough the decomposition of Corollary \ref{Coro:DecInvForm}, that is 
\begin{align*}
S(D):=\left(d+(-1)^{n}\star d\star\right)\bigg{|}_{\Omega_c(M_0)^{S^1}}:
\xymatrixrowsep{2cm}\xymatrixcolsep{2cm}\xymatrix{
{
\begin{array}{c}
\Omega_{\text{bas},c}(M_0)^{S^1}\\
\bigoplus\\
\Omega_{\text{bas},c}(M_0)^{S^1}
\end{array}
}\ar[r] &
{
\begin{array}{c}
\Omega_{\text{bas},c}(M_0)^{S^1}\\
\bigoplus\\
\Omega_{\text{bas},c}(M_0)^{S^1}.
\end{array}
}
}
\end{align*}
First we  begin with the decomposition of the Hodge star operator following the techniques of \cite[Chapter 7]{T97}. 
\begin{definition}\label{Def:Bar*}
The {\em basic Hodge star operator} is defined as the linear map
\begin{align*}
\bar{*}:\xymatrixrowsep{0.01cm}\xymatrixcolsep{2cm}\xymatrix{
\Omega^r_{\textnormal{bas}}(M_0)  \ar[r] & \Omega^{n-r}_{\textnormal{bas}} (M_0),
}
\end{align*}
satisfying the conditions
\begin{align}
 \bar{*}\beta=&(-1)^{n-r}*(\beta\wedge\chi), \label{Eqn:TransHosgeStar1} \\
 *\beta=&\bar{*}\beta\wedge\chi, \label{Eqn:TransHosgeStar2}
\end{align}
where $*$ is the Hodge star operator defined by the metric and orientation of $M$.
\end{definition} 

\begin{remark}\label{Rmk:VolumeQuot}
Observe that the volume form $\vol_{M_0}$ of $M_0$ can be written as  
\begin{align*}
\vol_{M_0}=* 1=\bar{*}1\wedge\chi.
\end{align*}
\end{remark}

\begin{lemma}
The operator $\bar{*}$ satisfies $\bar{*}^2=(-1)^{r(n-r)}$ on $r$-forms.
\end{lemma}
\begin{proof}
We know that the Hodge star operator on $M$ satisfies $*^2\beta=(-1)^{r(n+1-r)}$ for $\beta\in\Omega^r_\text{bas}(M_0)$. On the other hand by \eqref{Eqn:TransHosgeStar2} we have $*^2\beta=*(\bar{*}\beta\wedge\chi )$. We now compute using \eqref{Eqn:TransHosgeStar1}, 
\begin{align*}
\bar{*}^2\beta=(-1)^{n-(n-r)}{*}(\bar{*}\beta\wedge\chi)=(-1)^{r+r(n+1-r)}=(-1)^{r(n-r)}. 
\end{align*}
\end{proof}

In view of this lemma we can define a chirality operator on basic differential forms as in \cite[Section 5]{HR13}. The following result follows from Proposition \ref{Prop:Chirl}.

\begin{proposition}\label{Prop:BarChirl}
The basic chirality operator
\begin{align*}
\bar{\star}:\xymatrixrowsep{0.01cm}\xymatrixcolsep{2cm}\xymatrix{
\Omega^r_{\textnormal{bas}}(M_0)  \ar[r] & \Omega^{n-r}_{\textnormal{bas}} (M_0)
}
\end{align*}
defined by $\bar{\star}\coloneqq i^{[(n+1)/2]+2nr+r(r-1)}\bar{*}$, satisfies the relations
\begin{enumerate}
\item $\bar{\star}^2=1$.
\item $\varepsilon\bar{\star}=(-1)^n\bar{\star}\varepsilon$. 
\item $\bar{\star}\circ(Y^\flat)\circ\bar{\star}=(-1)^n\iota_Y \quad\text{for}\quad Y^\flat\in\Omega_\textnormal{bas}(M_0)$.
\end{enumerate}
\end{proposition}

\begin{lemma}\label{Lemma:star}
With respect to the decomposition of Corollary \ref{Coro:DecInvForm} we can express the operator $\star$ as 
\begin{equation*}
\star\bigg{|}_{\Omega(M_0)^{S^1}}=i^{q(n)}(-1)^n\left(\begin{array}{cc}
0 & -\varepsilon\bar{\star}\\
\varepsilon\bar{\star} & 0
\end{array}
\right),
\end{equation*}
where 
\begin{displaymath}
   q(n)\coloneqq (n-1)\textnormal{mod}(2)=\left\{
     \begin{array}{ccc}
       1 & ,& \:n \:\text{even},\\
       0 & ,& \:n \:\text{odd}.
     \end{array}
   \right.
\end{displaymath}
\end{lemma}
\begin{proof}
Recall that $[\cdot]$ denotes the integer part function. First observe the relation
 \begin{align*}
\left[\frac{n+1}{2}\right]+q(n)=\left[\frac{n}{2}\right]+1.
\end{align*}
For $\beta$ a basic $r$-form we calculate using \eqref{Def:Chirl} with $m=n+1$,
\begin{align*}
\star\beta =&i^{[n/2]+1+2(n+1)r+r(r-1)}*\beta\\
=&i^{q(n)+2r+[(n+1)/2]+2nr+r(r-1)}\bar{*}\beta\wedge\chi\\
=&(i^{q(n)}\bar{\star}\varepsilon \beta)\wedge\chi\\
=&(i^{q(n)}(-1)^n\varepsilon\bar{\star} \beta)\wedge\chi.
\end{align*}
On the other hand using Proposition \ref{Prop:Chirl} we compute,
\begin{align*}
\star (\beta\wedge\chi)=&(\star\circ(\chi\wedge)\circ \varepsilon )\beta= (-1)^{n+1}(\iota_{\chi^\sharp}\circ \star\circ\varepsilon)\beta=(\iota_{\chi^\sharp}\circ\varepsilon\circ \star)\beta.
\end{align*}
Finally, using the first computation above we conclude that 
\begin{align*}
(\iota_{\chi^\sharp}\circ\varepsilon\circ \star)\beta=&\iota_{\chi^\sharp}\varepsilon((i^{q(n)}(-1)^n\varepsilon\bar{\star} \beta)\wedge\chi)\\
=&-i^{q(n)}(-1)^n\iota_{\chi^\sharp}(\chi\wedge (\varepsilon\bar{\star} \beta))\\
=&-i^{q(n)}(-1)^n\varepsilon\bar{\star} \beta.
\end{align*}
\end{proof}

We are now ready to describe the operator $S(D)$ of Lemma \ref{Lemma:OpS}.
\begin{theorem}\label{Thm:OpInv}
With respect to the decomposition of Corollary \ref{Coro:DecInvForm} the exterior derivative decomposes as
\begin{align*}
d\bigg{|}_{\Omega(M_0)^{S^1}}=\left(\begin{array}{cc}
d & \varepsilon\varphi_0\wedge\\
0 & d-\kappa\wedge
\end{array}
\right)
\end{align*}
and its formal adjoint as
\begin{align*}
d^\dagger\bigg{|}_{\Omega(M_0)^{S^1}}=
\left(\begin{array}{cc}
(-1)^{n+1}\bar{\star}d\bar{\star} +\iota_H & 0\\
-\varepsilon\bar{\star}(\varphi_0\wedge)\bar{\star} & (-1)^{n+1}\bar{\star}d\bar{\star}
\end{array}
\right).
\end{align*}
Hence, the restriction of the Hodge-de Rham operator $D$ to the space of $S^1$-invariant forms with respect to this decomposition is
\begin{equation*}
S(D)\coloneqq
D\bigg{|}_{\Omega(M_0)^{S^1}}=
\left(\begin{array}{cc}
d+(-1)^{n+1}\bar{\star}d\bar{\star} +\iota_H & \varepsilon(\varphi_0\wedge)\\
-\varepsilon\bar{\star}(\varphi_0\wedge)\bar{\star} & d+(-1)^{n+1}\bar{\star}d\bar{\star}-\kappa\wedge
\end{array}
\right).
\end{equation*}
\end{theorem}
\begin{proof}
Let $\omega=\omega_0+\omega_1\wedge\chi$ be an invariant $S^1$-form, we use \eqref{Eqn:dchi} to compute (compare with  \cite[Proposition 10.1]{BGV}),
\begin{align*}
d(\omega_0+\omega_1\wedge\chi)=&d\omega_0+d\omega_1\wedge\chi+(\varepsilon\omega_1)\wedge d\chi\\
=&d\omega_0+d\omega_1\wedge\chi-(\varepsilon\omega_1)\wedge \kappa\wedge\chi+(\varepsilon\omega_1)\wedge\varphi_0\\
=&(d\omega_0+\varepsilon\varphi_0\wedge\omega_1)+(d\omega_1-\kappa\wedge\omega_1)\wedge\chi, 
\end{align*}
from where we obtain the desired decomposition for the exterior derivative. To compute the analogous expression for  the adjoint we use the relation $d^\dagger=(-1)^{n}\star d\star$.  We first calculate using the decomposition of $d$ and Lemma \ref{Lemma:star},
\begin{align*}
d\star \bigg{|}_{\Omega(M_0)^{S^1}}=&
i^{q(n)}(-1)^n
\left(\begin{array}{cc}
d & \varepsilon\varphi_0\wedge\\
0 & d-\kappa\wedge
\end{array}
\right)
\left(\begin{array}{cc}
0 & -\varepsilon\bar{\star}\\
\varepsilon \bar{\star}& 0 
\end{array}
\right)\\
=& 
i^{q(n)}(-1)^n\left(\begin{array}{cc}
\varphi_0\wedge\bar{\star}& -d\varepsilon \bar{\star}\\
(d-\kappa\wedge)\varepsilon\bar{\star} & 0
\end{array}
\right)\\
=&
i^{q(n)}(-1)^n\left(\begin{array}{cc}
\varphi_0\wedge\bar{\star}& \varepsilon d\bar{\star}\\
-\varepsilon(d-\kappa\wedge)\bar{\star} & 0
\end{array}
\right)
\end{align*}
Finally, using the relation $(i^{q(n)})^2=(-1)^{n+1}$ and  Proposition \ref{Prop:BarChirl} we expand the product 
\begin{align*}
d^\dagger\bigg{|}_{\Omega(M_0)^{S^1}}
=& (-1)^{n}\star d\star \bigg{|}_{\Omega(M_0)^{S^1}}\\
=&-\left(\begin{array}{cc}
0 & -\varepsilon\bar{\star}\\
\varepsilon \bar{\star}& 0 
\end{array}
\right)
\left(\begin{array}{cc}
\varphi_0\wedge\bar{\star}& \varepsilon d\bar{\star}\\
-\varepsilon(d-\kappa\wedge)\bar{\star} & 0
\end{array}
\right)\\
=& 
\left(\begin{array}{cc}
0 & -(-1)^{n+1}\bar{\star}\varepsilon\\
(-1)^{n+1}\bar{\star}\varepsilon& 0 
\end{array}
\right)
\left(\begin{array}{cc}
\varphi_0\wedge\bar{\star}& \varepsilon d\bar{\star}\\
-\varepsilon(d-\kappa\wedge)\bar{\star} & 0
\end{array}
\right)\\
=&
\left(\begin{array}{cc}
(-1)^{n+1}\bar{\star}(d-\kappa\wedge)\bar{\star} & 0\\
(-1)^{n+1}\bar{\star}(\varepsilon\varphi_0\wedge)\bar{\star} & (-1)^{n+1}\bar{\star}d\bar{\star}
\end{array}
\right)\\
=&
\left(\begin{array}{cc}
(-1)^{n+1}\bar{\star}d\bar{\star} +\iota_H & 0\\
-\varepsilon\bar{\star}(\varphi_0\wedge)\bar{\star} & (-1)^{n+1}\bar{\star}d\bar{\star}
\end{array}
\right).
\end{align*}
\end{proof}

\begin{proposition}
Let $c(\kappa)=\kappa\wedge-\iota_{\kappa^\sharp}=\kappa\wedge-\iota_{H}$ be the left Clifford multiplication by the mean curvature $1$-form. Then the following relations hold true
\begin{enumerate}
\item $c(\kappa)\varepsilon(\varepsilon\varphi_0\wedge)+(\varepsilon\varphi_0\wedge)c(\kappa)\varepsilon=-(\iota_H\varphi_0)\wedge$.
\item $d (\varepsilon\varphi_0\wedge)+(\varepsilon\varphi_0\wedge) d=-\kappa\wedge (\varepsilon\varphi_0\wedge)$.
\end{enumerate}
\end{proposition}
\begin{proof}
For the first relation we just compute,
\begin{align*}
c(\kappa)\varepsilon(\varepsilon\varphi_0\wedge)=&(\kappa\wedge-\iota_H)(\varphi_0\wedge)\\
=&\varphi_0\wedge\kappa\wedge-(\iota_H\varphi_0)\wedge-\varphi_0\wedge\iota_H\\
=&(\varphi_0\wedge)c(\kappa)-(\iota_H\varphi_0)\wedge\\
=&-(\varepsilon\varphi_0\wedge)c(\kappa)\varepsilon-(\iota_H\varphi_0)\wedge.
\end{align*}
For the second one we use Proposition \ref{Prop:varphi0}(2),
\begin{align*}
d(\varepsilon\varphi_0\wedge)=-\varepsilon d(\varphi_0\wedge)=\varepsilon \kappa\wedge\varphi_0\wedge-\varepsilon\varphi_0\wedge d.
\end{align*}
\end{proof}

\subsubsection{Construction of the operator $T(D)$}

Now that we have described the isomorphism $\Phi$ and the operator $S(D)$ we can compute the self-adjoint operator $T\coloneqq\Phi\circ S\circ \Phi^{-1}$ of Proposition \ref{Prop:OpT}. As $D$ is a first order differential operator,  Proposition \ref{Prop:SDiff} states that $T$ is also generated by a differential operator of the same order.  Let us begin by analyzing the zero order terms. Since the forms $\kappa$ and $\varphi_0$ are both basic there exist unique $\bar{\kappa}\in\Omega^1(M_0/S^1)$ and $\bar{\varphi}_0\in\Omega^2(M_0/S^1)$ such that $\kappa=\pi^*_{S^1}(\bar{\kappa})$ and $\varphi_0=\pi^*_{S^1}(\bar{\varphi}_0)$. Moreover, as pullback commutes with the wedge product, the following diagrams commute
\begin{align}\label{Diag:KappaAction}
\xymatrixrowsep{2cm}\xymatrixcolsep{2cm}\xymatrix{
\Omega_{\textnormal{bas}}(M_0)  \ar[r]^-{\kappa\wedge} & \Omega_{\textnormal{bas}} (M_0)\\
\Omega(M_0/S^1) \ar[u]^-{\pi_{S^1}^*} \ar[r]^-{\bar{\kappa}\wedge}& \Omega(M_0/S^1), \ar[u]_-{\pi_{S^1}^*}
}
\end{align}
\begin{align}\label{Diag:Varphi0Action}
\xymatrixrowsep{2cm}\xymatrixcolsep{2cm}\xymatrix{
\Omega_{\textnormal{bas}}(M_0)  \ar[r]^-{\varphi_0\wedge} & \Omega_{\textnormal{bas}} (M_0)\\
\Omega(M_0/S^1) \ar[u]^-{\pi_{S^1}^*} \ar[r]^-{\bar{\varphi}_0\wedge}& \Omega(M_0/S^1). \ar[u]_-{\pi_{S^1}^*}
}
\end{align}

Now we study the Hodge star operator. In view of Remark \ref{Rmk:VolumeQuot}, we choose the sign of the  volume form $\vol_{M_0/S^1}$ on $M_0/S^1$ so that $\pi_{S^1}^*(\vol_{M_0/S^1})\coloneqq\bar{*}1$. This means that we can express $\vol_{M_0}=\pi_{S^1}^*(\vol_{M_0/S^1})\wedge\chi$. With this choice we can identify $\bar{*}$, via the orbit map $\pi_{S^1}$, with the Hodge star operator $*_{M_0/S^1}$ of $M_0/S^1$ with respect to the quotient metric.

\begin{lemma}\label{Lemma:CommStar}
Then the following diagram commutes:
\begin{align*}
\xymatrixrowsep{2cm}\xymatrixcolsep{2cm}\xymatrix{
\Omega_{\textnormal{bas}}(M_0)  \ar[r]^-{\bar{*}} & \Omega_{\textnormal{bas}} (M_0)\\
\Omega(M_0/S^1) \ar[u]^-{\pi_{S^1}^*} \ar[r]^-{*_{M_0/S^1}}& \Omega(M_0/S^1). \ar[u]_-{\pi_{S^1}^*}
}
\end{align*}
\end{lemma}
\begin{proof}
Using \eqref{Eqn:TransHosgeStar1} and \eqref{Eqn:TransHosgeStar2}, we compute for $\bar{\omega}\in\Omega(M_0/S^1)$,  

\begin{align*}
\pi^*_{S^1}\bar{\omega}\wedge\bar{*}\pi^*_{S^1}\bar{\omega}\wedge\chi=&\pi^*_{S^1}\bar{\omega}\wedge*\pi^*_{S^1}\bar{\omega}\\
=&\inner{\pi^*_{S^1}\bar{\omega}}{\pi^*_{S^1}\bar{\omega}}\vol_{M_0}\\
=&\pi_{S^1}^*(\inner{\bar{\omega}}{\bar{\omega}})\bar{*}1\wedge\chi,
\end{align*}

and therefore, by contracting with $\iota_X$, we obtain $\pi^*_{S^1}\bar{\omega}\wedge\bar{*}\pi^*_{S^1}\bar{\omega}=\pi_{S^1}^*(\inner{\bar{\omega}}{\bar{\omega}})\bar{*}1$. Now, in view of our convention 
\begin{align*}
\pi^*_{S^1}\bar{\omega}\wedge\bar{*}\pi^*_{S^1}\bar{\omega}=&\pi_{S^1}^*(\inner{\bar{\omega}}{\bar{\omega}})\bar{*}1\\
=&\pi_{S^1}^*(\inner{\bar{\omega}}{\bar{\omega}}\vol_{M_0/S^1})\\
=&\pi_{S^1}^*(\bar{\omega}\wedge*_{M_0/S^1}\bar{\omega})\\
=&\pi_{S^1}^*\bar{\omega}\wedge \pi_{S^1}^**_{M_0/S^1}\bar{\omega}.
\end{align*}
Combining these two computations we obtain the desired result. 
\end{proof}

\begin{coro}\label{Coro:ChirBasicDiag}
We have an analogous commutative diagram for the chirality operators 
\begin{align*}
\xymatrixrowsep{2cm}\xymatrixcolsep{2cm}\xymatrix{
\Omega_{\textnormal{bas}}(M_0)  \ar[r]^-{\bar{\star}} & \Omega_{\textnormal{bas}} (M_0)\\
\Omega(M_0/S^1) \ar[u]^-{\pi_{S^1}^*} \ar[r]^-{\star_{M_0/S^1}}& \Omega(M_0/S^1).\ar[u]_-{\pi_{S^1}^*}
}
\end{align*}
\end{coro}

Next we are going to study  the term $-\bar{\star}({\varphi}_0\wedge)\bar{\star}$ as an operator on the quotient space. In view of Corollary \ref{Coro:ChirBasicDiag} and to lighten the notation we are going to identify $\bar{\star}\coloneqq\star_{M_0/S^1}$.
\begin{proposition}\label{Prop:CliffMultVarphi}
With respect to the quotient metric on $M_0/S^1$ we have
$$(\bar{\varphi}_0\wedge)^\dagger=-\bar{\star}(\bar{\varphi_0}\wedge)\bar{\star}.$$
In particular, the symmetric operator $\widehat{c}(\bar{\varphi}_0)\coloneqq(\bar{\varphi}_0\wedge)+(\bar{\varphi}_0\wedge)^\dagger$ satisfies 
\begin{enumerate}
\item $\widehat{c}(\bar{\varphi}_0)\bar{\star}+\bar{\star}\widehat{c}(\bar{\varphi}_0)=0$.
\item $\widehat{c}(\bar{\varphi}_0)\varepsilon-\varepsilon\widehat{c}(\bar{\varphi}_0)=0$. 
\end{enumerate}
\end{proposition}
\begin{proof}
We want to compute the adjoint of $(\bar{\varphi}_0\wedge)$ with respect to the quotient metric. To do this we expand $\bar{\varphi}_0$ with respect to an orthonormal basis $\{\bar{e}^i\}^{n}_{i=1}$ of $\wedge_\mathbb{C}T^*(M_0/S^1)$ as
\begin{align*}
\bar{\varphi_0}\wedge=\sum\inner{\varphi_0}{\bar{e}^i\wedge \bar{e}^j} \bar{e}^i\wedge \bar{e}^j\wedge,
\end{align*}
and use  Proposition \ref{Prop:Chirl}(3) to verify
\begin{align*}
(\bar{\varphi_0}\wedge)^\dagger =&\sum_{i<j}\inner{\bar{\varphi}_0}{\bar{e}^i\wedge \bar{e}^j}\iota_{\bar{e}_j}\circ\iota_{\bar{e}_i}=\sum_{i<j}\inner{\bar{\varphi}_0}{\bar{e}^i\wedge \bar{e}^j}\bar{\star}\circ \bar{e}^i\wedge \bar{e}^j\circ \bar{\star}=-\bar{\star}(\bar{\varphi}_0\wedge)\bar{\star}.
\end{align*}
\end{proof}

We are ready to handle the first order terms of $S(D)$ in Theorem \ref{Thm:OpInv}.  For the exterior derivative, as it also commutes with pullbacks, we have an analogous commutative diagram, 
\begin{align}\label{Diag:dBasic}
\xymatrixrowsep{2cm}\xymatrixcolsep{2cm}\xymatrix{
\Omega_{\textnormal{bas}}(M_0)  \ar[r]^-{d} & \Omega_{\textnormal{bas}} (M_0)\\
\Omega(M_0/S^1) \ar[u]^-{\pi_{S^1}^*} \ar[r]^-{d_{M_0/S^1}}& \Omega(M_0/S^1), \ar[u]_-{\pi_{S^1}^*}
}
\end{align}
where $d_{M_0/S^1}$ is the exterior derivative of $M_0/S^1$. Hence, it remains to study the operator 
\begin{align*}
(-1)^{n+1}\bar{\star}d\bar{\star}:
\xymatrixrowsep{2cm}\xymatrixcolsep{2cm}\xymatrix{
\Omega^r_{\textnormal{bas}}(M_0)  \ar[r] & \Omega^{r-1}_{\textnormal{bas}} (M_0).
}
\end{align*}
\begin{remark}\label{Rmk:AdjointNotPresBas}
Let $d^\dagger_{M_0/S^1}=(-1)^{n+1}\star_{M_0/S^1}d_{M_0/S^1}\star_{M_0/S^1}$ be the $L^2$-formal adjoint of $d_{M_0/S^1}$ with respect to the quotient metric (Proposition \ref{Prop:Chirl}(4)). One might think that there is an analogous commutative diagram as \eqref{Diag:dBasic} where $d^{\dagger}$ and $d^\dagger_{M_0/S^1}$ are placed instead. Note however that $d^\dagger$ does not preserve the space of basic forms, as it can be explicitly seen from Theorem \ref{Thm:OpInv}, and in general
$$d^\dagger\circ \pi_{S^1}^* \neq\pi_{S^1}^*\circ d^\dagger_{M_0/S^1}.$$
\end{remark}

Observe that \eqref{Diag:dBasic} and Corollary \ref {Coro:ChirBasicDiag} can be combined to obtain the following commutative diagram 

\begin{align*}
\xymatrixrowsep{2cm}\xymatrixcolsep{2cm}\xymatrix{
\Omega_{\textnormal{bas}}(M_0)  \ar[r]^-{(-1)^{n+1}\bar{\star}} & \Omega_{\textnormal{bas}} (M_0)   \ar[r]^-{d} & \Omega_{\textnormal{bas}} (M_0) \ar[r]^-{\bar{\star}} & \Omega_{\textnormal{bas}} (M_0)\\
\Omega(M_0/S^1) \ar[u]^-{\pi_{S^1}^*} \ar[r]^-{(-1)^{n+1}\star_{M_0/S^1}}& \Omega(M_0/S^1) \ar[u]^-{\pi_{S^1}^*}
 \ar[r]^-{d_{M_0/S^1}} & \Omega(M_0/S^1) \ar[u]_-{\pi_{S^1}^*}   \ar[r]^-{\star_{M_0/S^1}}& \Omega(M_0/S^1), \ar[u]_-{\pi_{S^1}^*}
}
\end{align*}
which shows that 
\begin{align*}
\pi^*_{S^1}\circ d^\dagger_{M_0/S^1}=(-1)^{n+1}\bar{\star}d\bar{\star}\circ \pi^*_{S^1}.
\end{align*}
That is, the following diagram commutes
\begin{align}\label{Diag:ddaggerBasic}
\xymatrixrowsep{2cm}\xymatrixcolsep{2cm}\xymatrix{
\Omega_{\textnormal{bas}}(M_0)  \ar[r]^-{(-1)^{n+1}\bar{\star}d\bar{\star}} & \Omega_{\textnormal{bas}} (M_0)\\
\Omega(M_0/S^1) \ar[u]^-{\pi_{S^1}^*} \ar[r]^-{d^{\dagger}_{M_0/S^1}}& \Omega(M_0/S^1). \ar[u]_-{\pi_{S^1}^*}
}
\end{align}

Altogether, from the discussion of Section \ref{Section:ConstrF}, Theorem \ref{Thm:OpInv}, Corollary \ref{Coro:ChirBasicDiag} and \eqref{Diag:KappaAction}, \eqref{Diag:Varphi0Action}, \eqref{Diag:dBasic}, \eqref{Diag:ddaggerBasic} we can describe explicitly  the operator $T$ of Proposition \ref{Diag:dBasic}.

\begin{theorem}\label{Theorem:OpT}
The operator $T$ of Proposition \ref{Prop:OpT} for a semi-free $S^1$-action can be written as
\begin{equation*}
T=\left(
\begin{array}{cc}
D_{M_0/S^1}+\iota_{\bar{\kappa}^\sharp} & \varepsilon(\bar{\varphi_0}\wedge) \\
 \varepsilon (\bar{\varphi}_0\wedge)^\dagger & D_{M_0/S^1}-\bar{\kappa}\wedge
\end{array}
\right),
\end{equation*}
where $D_{M_0/S^1}\coloneqq d_{M_0/S^1}+d^\dagger_{M_0/S^1}$is the Hodge-de Rham operator on $M_0/S^1$. The operator $T$, when defined on the core $\Omega_c(M/S^1)$, is essentially self adjoint. 
\end{theorem}

\begin{remark}[Principal symbol]
From Theorem \ref{Theorem:OpT} we clearly see that $T$ is a first order differential operator, which was to be expected by Proposition \ref{Prop:SDiff}. Recall from Proposition \ref{Prop:PSD} that the principal symbol of the Hodge-de Rahm operator is $\sigma_P(D)(x,\xi)=-ic(\xi)$ for $(x,\xi)\in T^*M$. In particular if $(x,\omega'_x+\omega''_x\wedge\chi_x)\in\wedge_\mathbb{C}T^*M$ and $(y,\bar{\xi})\in T_{\mathbb{C}}^*(M_0/S^1)$ are such that $\pi_{S^1}(x)=y$, then 
\begin{align*}
c(\pi^*_{S^1}\bar{\xi})(\omega'_x+\omega''_x\wedge\chi)=&c(\pi^*_{S^1}\bar{\xi})\omega'_x+c(\pi^*_{S^1}\bar{\xi})(\omega''_x\wedge\chi_x)\\
=&c(\pi^*_{S^1}\bar{\xi})\omega'_x+(c(\pi^*_{S^1}\bar{\xi})\omega''_x)\wedge\chi_x +\chi_x((\pi^*_{S^1}\bar{\xi})^\sharp)(\varepsilon\omega''_x)\\
=&c(\pi^*_{S^1}\bar{\xi})\omega'_x+(c(\pi^*_{S^1}\bar{\xi})\omega''_x)\wedge\chi_x ,
\end{align*} 
since $\chi_x((\pi^*_{S^1}\bar{\xi})^\sharp)=\inner{\chi}{\pi^*_{S^1}\bar{\xi}}=0$.
Therefore, with respect to the decomposition of Corollary \ref{Coro:DecInvForm} we have
\begin{align*}
\sigma_P(D)(x,\pi_{S^1}^*\bar{\xi})
\left(\begin{array}{c}
\omega'_x\\
\omega''_x
\end{array}
\right)
=-ic(\pi_{S^1}^*\bar{\xi})\left(\begin{array}{c}
\omega'_x\\
\omega''_x
\end{array}
\right)
=
\left(\begin{array}{c}
-ic(\pi_{S^1}^*\bar{\xi})\omega'_x\\
-ic(\pi_{S^1}^*\bar{\xi})\omega''_x
\end{array}
\right).
\end{align*} 
On the other hand, we see from the explicit expression of the operator $T$ described in Theorem \ref{Theorem:OpT} that its principal symbol is 
\begin{align*}
\sigma_{P}(T)(y,\bar{\xi})=\sigma_{p}(D_{M_0/S^1})(y,\bar{\xi})\otimes
\left(\begin{array}{cc}
1&0\\
0&1
\end{array}
\right)
=-ic(\bar{\xi})\otimes
\left(\begin{array}{cc}
1&0\\
0&1
\end{array}
\right),
\end{align*}
where $(y,\bar{\xi})\in T_{\mathbb{C}}^*(M_0/S^1)$. Hence, using the notation of Section \ref{Section:ConstrF}, we see that
\begin{align*}
\pi_{S^1}'\left(\sigma_P(D)(x,\pi_{S^1}^*\bar{\xi})
\left(\begin{array}{c}
\omega'_x\\
\omega''_x
\end{array}
\right)\right)
=\left(\begin{array}{c}
-ic(\bar{\xi})\omega'_y\\
-ic(\bar{\xi})\omega''_y\\
\end{array}\right)
=
\sigma_P(T)(y,\bar{\xi})
\pi'_{S^1}
\left(\begin{array}{c}
\omega'_x\\
\omega''_x
\end{array}
\right),
\end{align*}
which verifies the symbol equation of Proposition \ref{Prop:SDiff}. In particular, we see that $T$ is elliptic. 
\end{remark}

\subsection{Dirac-Schr\"odinger operators}

We have obtained in Theorem \ref{Theorem:OpT} an operator $T$ which is self-adjoint in $L^2(F,h)$ but not in $L^{2}(F)$, the $L^2$-inner product without the weight $h$ (Remark \ref{Rmk:L2F}). For example if we took the adjoint in $L^2(F)$  of $D_{M_0/S^1}+\iota_{\bar{H}}$ we would obtain $(D_{M_0/S^1}+\iota_{\bar{H}})^\dagger=D_{M_0/S^1}+\bar{\kappa}\wedge$. To obtain a self-adjoint operator in $L^2(F)$ we perform the following unitary transformation:
\begin{equation}\label{Eqn:U}
\omega=\left(
\begin{array}{c}
\omega_0\\
\omega_1
\end{array}
\right)\longmapsto 
U(\omega)\coloneqq
h^{-1/2}
\left(
\begin{array}{c}
\omega_0\\
\omega_1
\end{array}
\right),
\end{equation}
for $\omega_0,\omega_1\in\Omega_c(M_0/S^1)$. Note that $\norm{U(\omega)}_{L^2(F,h)}=\norm{\omega}_{L^2(F)}$. Using this transformation we want to compute an explicit formula for the operator $\widehat{T}\coloneqq U^{-1}TU$, i.e. the operator defined by the following commutative diagram:
\begin{align*}
\xymatrixrowsep{2cm}\xymatrixcolsep{2cm}\xymatrix{
 L^2(F)  \ar[d]_-{U}\ar[r]^-{\widehat{T}} & L^2(F) \ar[d]^-{U}\\
\text{Dom}(T)\subset L^2(F,h)  \ar[r]^-{T} & L^2(F,h),
}
\end{align*}
with $\text{Dom}(\widehat{T} )\coloneqq U^{-1}(\text{Dom}(T))$. Clearly the operator $\widehat{T}$ is a self-adjoint operator on $L^2(F)$ since it is unitarily equivalent to $T$. To begin the transformation of this operator it is convenient to compute the exterior derivative of $h^{-1/2}$.
\begin{lemma}\label{Lemmadh}
The volume of the orbit function $h:M_0/S^1\longrightarrow \mathbb{R}$ satisfies 
\begin{align*}
d(h^{\pm1/2})=\mp\frac{1}{2}h^{\pm1/2}\bar{\kappa}.
\end{align*}
\end{lemma}
\begin{proof}
We compute using Lemma \ref{Lemma:dh},
\begin{equation*}
d(h^{\pm1/2})=\pm\frac{1}{2}h^{\pm1/2}\frac{dh}{h}=\mp\frac{1}{2}h^{\pm1/2}\bar{\kappa}.
\end{equation*}
\end{proof}

\begin{theorem}\label{Thm:THat}
The operator $\widehat{T}=U^{-1}TU$ is given by 
\begin{equation*}
\widehat{T}=
\left(
\begin{array}{cc}
D_{M_0/S^1}+\frac{1}{2}\widehat{c}(\bar{\kappa})&\varepsilon \bar{\varphi}_0\wedge\\
\varepsilon(\bar{\varphi}_0\wedge)^\dagger &  D_{M_0/S^1}-\frac{1}{2}\widehat{c}(\bar{\kappa})
\end{array}
\right),
\end{equation*}
where $\widehat{c}(\bar{\kappa})\coloneqq\bar{\kappa}\wedge+\iota_{\bar{\kappa}^\sharp}$ is the right Clifford multiplication. 
\end{theorem}
\begin{proof}
Using the last lemma we get
\begin{align*}
h^{1/2}d_{M_0/S^1}h^{-1/2}=d_{M_0/S^1}+\frac{1}{2}\bar{\kappa}\wedge.
\end{align*}
Similarly, by taking the adjoint, we get 
\begin{align*}
h^{1/2}d^\dagger_{M_0/S^1}h^{-1/2}=(h^{-1/2}d_{M_0/S^1}h^{1/2})^\dagger=\left(d_{M_0/S^1}-\frac{1}{2}\bar{\kappa}\wedge\right)^\dagger=d^\dagger_{M_0/S^1}-\frac{1}{2}\iota_{\bar{\kappa}^\sharp}.
\end{align*}
Therefore, the operator $D_{M_0/S^1}$ transforms as
\begin{align*}
h^{1/2}D_{M_0/S^1}h^{-1/2}=D_{M_0/S^1}+\frac{1}{2}c(\bar{\kappa}).
\end{align*}
The result follows now immediately from Theorem \ref{Theorem:OpT}.
\end{proof}

\begin{remark}
Observe from Lemma \ref{Lemma:star} that
\begin{align*}
\left(U^{-1}\circ\Phi\circ\star\bigg{|}_{\Omega(M_0)^{S^1}}\circ \Phi^{-1}\circ U\right)^\dagger
=&i^{q(n)}(-1)^n(-1)^{q(n)}\left(\begin{array}{cc}
0 & (\varepsilon\bar{\star})^\dagger\\
(-\varepsilon\bar{\star})^\dagger & 0
\end{array}
\right)\\
=&i^{q(n)}(-1)^n(-1)^{q(n)}\left(\begin{array}{cc}
0 & \bar{\star}\varepsilon\\
- \bar{\star}\varepsilon& 0
\end{array}
\right)\\
=&i^{q(n)}(-1)^n(-1)^{q(n)+n+1}\left(\begin{array}{cc}
0 & -\varepsilon\bar{\star}\\
 \varepsilon \bar{\star}& 0
\end{array}
\right)\\
=&i^{q(n)}(-1)^n\left(\begin{array}{cc}
0 & -\varepsilon\bar{\star}\\
 \varepsilon \bar{\star}& 0
\end{array}
\right)\\
=&\star\bigg{|}_{\Omega(M_0)^{S^1}},
\end{align*}
where we have used the relation $(q(n)+n+1)\text{mod}(2)=0$. This computation just verifies what we expected.
\end{remark}

\subsubsection{An involution on $F$}\label{Sec:AnInvolution}
As we are interested in Fredholm indices, we would like to find a self-adjoint involution which anti-commutes with $\widehat{T}$ in order to split this operator. Since the dimension of $M_0/S^1$ is $n$ then $\bar{\star}D_{M_0/S^1}+(-1)^n D_{M_0/S^1}\bar{\star}=0$, thus a  first natural candidate is 
\begin{align}\label{Def:BigStar}
{\bigstar}\coloneqq
\left(
\begin{array}{cc}
0 & \bar{\star}\\
\bar{\star} &0
\end{array}
\right).
\end{align}
From Proposition \ref{Prop:Chirl} and Proposition \ref{Prop:CliffMultVarphi} we compute 
\begin{align*}
\bigstar\widehat{T} =&
\left(
\begin{array}{cc}
0 & \bar{\star}\\
\bar{\star} &0
\end{array}
\right)
\left(
\begin{array}{cc}
D_{M_0/S^1}+\frac{1}{2}\widehat{c}(\bar{\kappa})&\varepsilon \bar{\varphi}_0\wedge\\
\varepsilon(\bar{\varphi}_0\wedge)^\dagger &  D_{M_0/S^1}-\frac{1}{2}\widehat{c}(\bar{\kappa})
\end{array}
\right)\\
=&
\left(
\begin{array}{cc}
\bar{\star} \varepsilon (\bar{\varphi}_0\wedge)^\dagger & \bar{\star}( D_{M_0/S^1}-\frac{1}{2}\widehat{c}(\bar{\kappa}))\\
\bar{\star}( D_{M_0/S^1}+\frac{1}{2}\widehat{c}(\bar{\kappa})) & \bar{\star} \varepsilon \bar{\varphi}_0\wedge 
\end{array}
\right)\\
=&
(-1)^{n+1}\left(
\begin{array}{cc}
 \varepsilon (\bar{\varphi}_0\wedge)\bar{\star} & ( D_{M_0/S^1}+\frac{1}{2}\widehat{c}(\bar{\kappa}))\bar{\star}\\
( D_{M_0/S^1}-\frac{1}{2}\widehat{c}(\bar{\kappa})) \bar{\star}& \varepsilon (\bar{\varphi}_0\wedge)^\dagger \bar{\star}
\end{array}
\right)\\
=&
(-1)^{n+1}\widehat{T}\bigstar.
\end{align*}
This implies that if $n$ is even then we can decompose 
\begin{align*}
\widehat{T}=
\left(
\begin{array}{cc}
0 & \widehat{T}^-\\
\widehat{T}^+ & 0
\end{array}
\right),
\end{align*}
with respect to the involution $\bigstar$. 
\begin{example}[Spinning a closed manifold III]\label{Ex:ClosedManifold}
Consider again the free action treated in Example \ref{Ex:SpiningM} for the closed manifold $M\coloneqq X\times S^1$. Since the volume of the orbits are constant then $\kappa=0$ and $\varphi_0=0$. Therefore the operator $\widehat{T}$ of Theorem \ref{Thm:THat} reduces to 
\begin{align*}
\widehat{T}=
\left(
\begin{array}{cc}
d_X+d^\dagger_X & 0\\
0 & d_X+d^\dagger_X
\end{array}
\right).
\end{align*}
Taking its square we get  
\begin{align*}
\widehat{T}^2=
\left(
\begin{array}{cc}
\Delta_X & 0\\
0 & \Delta_X 
\end{array}
\right),
\end{align*}
where $\Delta_X$ denotes the Laplacian on $X$. Observe that $\ker\widehat{T}=\ker \widehat{T}^2=\mathcal{H}_\Delta \otimes\mathbb{C}^2$, where $\mathcal{H}_\Delta$ denotes the space of harmonic differential forms on $X$. The involution \eqref{Def:BigStar} in this example is 
\begin{align*}
\bigstar\coloneqq
\left(
\begin{array}{cc}
0 & \star_X\\
\star_X &0
\end{array}
\right),
\end{align*}
 where $\star_X$ is the chirality operator of $X$. Observe that if $\alpha\in \mathcal{H}^l_\Delta$ is an harmonic differential form of degree $l$ then $\star_X\alpha\in \mathcal{H}^{4k-l}_\Delta$ is an harmonic $(4k-l)$-differential form and 
\begin{align*}
\bigstar\left(\begin{array}{c}
\alpha\\
\pm\star_X\alpha
\end{array}\right)
=\pm
\left(\begin{array}{c}
\alpha\\
\pm \star_X\alpha
\end{array}\right).
\end{align*}
This shows that the map 
\begin{align*}
\left(\begin{array}{c}
\alpha\\
\star_X\alpha
\end{array}\right)
\longmapsto
\left(\begin{array}{c}
\alpha\\
-\star_X\alpha
\end{array}\right),
\end{align*}
induces an isomorphism between $\ker\widehat{T}^+$ and $\ker\widehat{T}^-$, so we see that $\ind (\widehat{T}^+)=0$. 
\end{example}

We can give a geometric interpretation of the involution $\bigstar$ in terms of an involution on $M$ in the general case of a semi-free $S^1$-action. Assume that $n$ is even so that the closed manifold $M$ has odd dimension. Define the operator $a\coloneqq i\varepsilon\star$ acting on $\wedge_\mathbb{C}T^*M$. Using \eqref{PropertiesEpsilonStar} we can easily verify
\begin{itemize}
\item $a^\dagger=-i\star\varepsilon=i\varepsilon\star=a$.
\item $a^2 =(i\varepsilon\star)^2=-\varepsilon\star \varepsilon\star=1$.
\end{itemize}
In addition, if $D$ is the Hodge- de Rham operator of $M$, then
\begin{align*}
aD=i\varepsilon\star D=i\varepsilon D\star=-iD\varepsilon\star=-Da.
\end{align*}
Thus, we can decompose $D$ with respect to $a$ as 
\begin{align*}
D=
\left(\begin{array}{cc}
0 & D^{(+)}\\
D^{(-)} & 0
\end{array}\right).
\end{align*}
Recall however, the index of an elliptic operator on an odd dimensional closed manifold vanishes (\cite[Theorem III.13.12]{LM89}). We claim that the involution $a$ restricted to $S^1$-invariant forms is precisely $\bigstar$. To see this recall from Example \ref {Ex:EpsilonInvDec}  and Lemma \ref{Lemma:star},  
\begin{equation*}
\star\bigg{|}_{\Omega(M_0)^{S^1}}= 
\left(\begin{array}{cc}
0 & -i\varepsilon\bar{\star}\\
i\varepsilon \bar{\star}& 0 
\end{array}
\right)\quad\text{and}\quad
\varepsilon\bigg{|}_{\Omega(M_0)^{S^1}}=
\left(\begin{array}{cc}
\varepsilon & 0\\
0 & -\varepsilon
\end{array}
\right).
\end{equation*}
Thus, we simply compute and verify
\begin{align*}
a\bigg{|}_{\Omega(M_0)^{S^1}}=
\left(\begin{array}{cc}
i\varepsilon & 0\\
0 & -i\varepsilon
\end{array}
\right)
\left(\begin{array}{cc}
0 & -i\varepsilon\bar{\star}\\
i\varepsilon \bar{\star}& 0 
\end{array}
\right)=
\left(\begin{array}{cc}
0 & \bar{\star}\\
\bar{\star} & 0
\end{array}
\right)=\bigstar.
\end{align*}
This explains the relation $\bigstar\widehat{T}+\widehat{T}\bigstar=0$ and clarifies why $\ind(\widehat{T}^+)=\ind(D^{(+)})=0$ in Example \ref{Ex:ClosedManifold} is zero.\\

This result motivates the need to consider two possible directions:
\begin{enumerate}
\item Find another involution anti-commuting with $\widehat{T}$ which produces a non-trivial index. 
\item Perform the push down procedure for other geometric operators on $M$ so that the resulting operators anti-commute with $\bar{\star}$. 
\end{enumerate}
As we are interested in the equivariant $S^1$-signature and in view of Example \ref{Ex:SpiningM} and Example \ref{Ex:SpiningMwB}, it seems more convenient to follow the second option. This will be the content of the following sections.

\subsubsection{The positive signature operator}

Let us assume now that $n=4k$ . In this situation the Hodge-de Rham operator $D$ on $M$ commutes with $\star$ and therefore it defines the graded component $D^+:\Omega^+(M)\longrightarrow\Omega^+(M)$ where $\Omega^+(M)$ denotes the $+1$-eigenforms of $\star$. Our aim is to perform the analogous push down construction for the operator $D^+$ to obtain the corresponding self-adjoint operator $\widehat{T}(D^+)$.  We begin by computing the restriction operator 
$$S(D^+)\coloneqq D^+\bigg{|}_{\Omega_c^+(M_0)^{S^1}}=(d+\star d)\bigg{|}_{\Omega_c^+(M_0)^{S^1}}.$$ 
Observe from Corollary \ref{Coro:DecInvForm} and  Lemma \ref{Lemma:star}  that we can identify
\begin{align}\label{Eqn:BasS1InvPlus}
\xymatrixrowsep{0.01cm}\xymatrixcolsep{2cm}\xymatrix{
\Omega_{\text{bas}}(M_0) \ar[r] & \Omega^{+}(M_0)^{S^1}\\
\omega\ar@{|->}[r] & 
{\left(\begin{array}{c}
\omega \\
{i\varepsilon \bar{\star}}\omega
\end{array}\right).}
}
\end{align}
Indeed, for $\omega\in\Omega_\text{bas}(M_0)$ we verify the parity condition
\begin{align*}
\left(\begin{array}{cc}
0 & -i\varepsilon \bar{\star} \\
i\varepsilon \bar{\star} & 0 
\end{array}\right)
\left(\begin{array}{c}
\omega \\
{i\varepsilon \bar{\star}}\omega
\end{array}\right)
=
\left(\begin{array}{c}
\varepsilon\bar{\star}\varepsilon\bar{\star}\omega\\
i\varepsilon\bar{\star}\omega
\end{array}\right)
=
\left(\begin{array}{c}
\omega \\
{i\varepsilon \bar{\star}}\omega
\end{array}\right).
\end{align*}

Now we calculate $S(D^+)$ by means of this identification. Using Theorem \ref{Thm:OpInv} we compute for $\omega\in\Omega_{\text{bas}}(M_0)$,
\begin{align*}
d\left(\begin{array}{c}
\omega\\
i\varepsilon\bar{\star}\omega
\end{array}
\right)=&\left(
\begin{array}{cc}
d & \varepsilon\varphi_0\wedge\\
0 & d-\kappa\wedge
\end{array}
\right)
\left(
\begin{array}{c}
\omega\\
i\varepsilon\bar{\star}\omega
\end{array}
\right)\\
=&
\left(
\begin{array}{c}
d \omega+\varepsilon\varphi_0\wedge i\varepsilon\bar{\star}\omega\\
d i\varepsilon\bar{\star}\omega-\kappa \wedge i\varepsilon\bar{\star}\omega
\end{array}
\right)\\
=&
\left(
\begin{array}{c}
d \omega+i\varphi_0\wedge \bar{\star}\omega\\
i\varepsilon\bar{\star}((-\bar{\star}d\bar{\star})\omega+\iota_H\omega)
\end{array}
\right),
\end{align*}
and 
\begin{align*}
\star d\left(\begin{array}{c}
\omega\\
i\varepsilon\bar{\star}\omega
\end{array}
\right)=&
\left(\begin{array}{cc}
0 & -i\varepsilon\bar{\star}\\
 i\varepsilon\bar{\star} & 0
\end{array}
\right)
\left(
\begin{array}{cc}
d & \varepsilon\varphi_0\wedge\\
0 & d-\kappa\wedge
\end{array}
\right)
\left(
\begin{array}{c}
\omega\\
i\varepsilon\bar{\star}\omega
\end{array}
\right)\\
=&
\left(\begin{array}{cc}
0 & -i\varepsilon\bar{\star}\\
 i\varepsilon\bar{\star} & 0
\end{array}
\right)
\left(
\begin{array}{c}
d \omega+i\varphi_0\wedge \bar{\star}\omega\\
i\varepsilon\bar{\star}((-\bar{\star}d\bar{\star})\omega+\iota_H\omega)
\end{array}
\right)\\
=&
\left(
\begin{array}{c}
(-\bar{\star}d\bar{\star})\omega+\iota_H\omega\\
i\varepsilon\bar{\star} (d \omega+i\varphi_0\wedge \bar{\star}\omega)\\
\end{array}
\right).
\end{align*}
Therefore,
\begin{align*}
S(D^+):=D^+\bigg{|}_{\Omega_c^+(M_0)^{S^1}}:
\left(
\begin{array}{c}
\omega\\
i\varepsilon\bar{\star}\omega
\end{array}
\right)
\longmapsto
\left(
\begin{array}{c}
d \omega+(-\bar{\star}d\bar{\star})\omega+\iota_H\omega +i\varphi_0\wedge \bar{\star}\omega\\
i\varepsilon\bar{\star}(d \omega+(-\bar{\star}d\bar{\star})\omega+\iota_H\omega +i\varphi_0\wedge \bar{\star}\omega)\\
\end{array}
\right). 
\end{align*}
Under the identification $\Omega_{\text{bas}}(M_0)$ with $\Omega^+(M_0)^{S^1}$ of \eqref{Eqn:BasS1InvPlus} we see that the induced operator on basic forms $S_\text{bas}(D^+):\Omega_{\text{bas},c}(M_0)\longrightarrow \Omega_{\text{bas},c}(M_0)$ is 
\begin{align*}
S_\text{bas}(D^+)\coloneqq d +(-\bar{\star}d\bar{\star})+\iota_H+i\varphi_0\wedge \bar{\star}. 
\end{align*}
The associated operator $T(D^+):\Omega_c(M_0/S^1)\longrightarrow\Omega_c(M_0/S^1)$ induced by $\pi^*_{S^1}$ is then
\begin{align*}
T(D^+)=D_{M_0/S^1}+\iota_{\bar{\kappa}^\sharp}+i\bar{\varphi}_0\wedge \bar{\star}.
\end{align*}
As before, in order to obtain an essentially self-adjoint operator $\widehat{T}(D^+)$ on the core $\Omega_c(M_0/S^1)$ with respect to the $L^2$-inner product on sections of $\wedge_\mathbb{C} T^* (M_0/S^1)$ without the weight $h$, we implement the unitary transformation \eqref{Eqn:U} and use Lemma \ref{Lemmadh},
\begin{align*}
\widehat{T}(D^+)=D_{M_0/S^1}+\frac{1}{2}\widehat{c}(\bar{\kappa})+i\bar{\varphi}_0\wedge\bar{\star}.
\end{align*}
\begin{remark}
From Proposition \ref{Prop:CliffMultVarphi} we see that 
\begin{align*}
(i\bar{\varphi}_0\wedge\bar{\star})^\dagger=-i\bar{\star}(\bar{\varphi}_0\wedge)^\dagger=i\bar{\varphi}_0\wedge\bar{\star},
\end{align*}
thus we see that $\widehat{T}(D^+)$ is indeed symmetric.
\end{remark}
Note however, that the zero order part does not anti-commute with $\bar{\star}$ since
\begin{align*}
\widehat{c}(\bar{\kappa})\bar{\star}=\bar{\star}\widehat{c}(\bar{\kappa}) \quad\text{and}\quad \bar{\star}(i\bar{\varphi}_0\wedge \bar{\star})=-i(\bar{\varphi}_0\wedge)^\dagger.
\end{align*}

\subsubsection{The odd signature operator}

Now we study, still in the case $n=4k$, the odd signature operator of $M$,  
$$A\coloneqq\star D=\star d+ d\star.$$ 
We have seen that $A$ commutes with the involution $\varepsilon$ so we can decompose it into the components $A^{\text{ev/ odd}}:\Omega^{\text{ev/ odd}}(M)\longrightarrow \Omega^{\text{ev/ odd}}(M)$. Our objective is to compute the corresponding operator $\widehat{T}(A^{\text{ev}})$, which is possible since $A^{\text{ev}}$ commutes with the $S^1$-action. \\

Using the decomposition of Corollary \ref{Coro:DecInvForm} and Example \ref{Ex:EpsilonInvDec} we  identify
\begin{align}\label{Eqn:BasS1InvEv}
\xymatrixrowsep{0.01cm}\xymatrixcolsep{2cm}\xymatrix{
{
\begin{array}{c}
\Omega^{\text{ev}}_{\text{bas}}(M_0)\\
\oplus\\
\Omega^{\text{odd}}_{\text{bas}}(M_0)
\end{array}
}
 \ar[r] & \Omega^{\text{ev}}(M_0)^{S^1}\\
{\left(\begin{array}{c}
\omega_0 \\
\omega_1
\end{array}\right)}
\ar@{|->}[r] & 
\omega_0+\omega_1\wedge\chi.}
\end{align}
From this identification and Theorem \ref{Thm:OpInv} we calculate
\begin{align*}
\star d\bigg{|}_{\Omega^\text{ev}(M_0)^{S^1}}=&
\left(
\begin{array}{cc}
0 & -i\varepsilon\bar{\star}\\
i\varepsilon\bar{\star} & 0
\end{array}
\right)
\left(
\begin{array}{cc}
d & \varepsilon\varphi_0\wedge\\
0 & d -\kappa\wedge
\end{array}
\right)=
\left(
\begin{array}{cc}
0 & -i\bar{\star}(d -\kappa\wedge)\\
-i\bar{\star}d & i\bar{\star}(\varphi_0\wedge)
\end{array}
\right),\\
 d\star\bigg{|}_{\Omega^\text{ev}(M_0)^{S^1}}=&
\left(
\begin{array}{cc}
d & \varepsilon\varphi_0\wedge\\
0 & d -\kappa\wedge
\end{array}
\right)
\left(
\begin{array}{cc}
0 & -i\varepsilon\bar{\star}\\
i\varepsilon\bar{\star} & 0
\end{array}
\right)=
\left(
\begin{array}{cc}
i\varphi_0\wedge\bar{\star} & id\bar{\star}\\
i(d-\kappa\wedge)\bar{\star} & 0
\end{array}
\right),
\end{align*}
and obtain
\begin{align*}
 S(A_\text{ev})\coloneqq
 d\star+\star d\bigg{|}_{\Omega^\text{ev}(M_0)^{S^1}}=&
 \left(
\begin{array}{cc}
i\varphi_0\wedge\bar{\star} & id\bar{\star}-i\bar{\star}(d -\kappa\wedge)\\
i(d-\kappa\wedge)\bar{\star}-i\bar{\star}d & i\bar{\star}\varphi_0\wedge
\end{array}
\right)\\
=&
i \left(
\begin{array}{cc}
\varphi_0\wedge\bar{\star} & (d-\bar{\star}d\bar{\star})\bar{\star}+\bar{\star}(\kappa\wedge)\\
(d-\bar{\star}d\bar{\star})\bar{\star}-(\kappa\wedge)\bar{\star} & \bar{\star}\varphi_0\wedge
\end{array}
\right).
\end{align*}
Hence the induced operator in $ L^2(\wedge_{\mathbb{C}}T^*(M_0/S^1),h)$ is 
\begin{align*}
T(A^\text{ev})=
i \left(
\begin{array}{cc}
\bar{\varphi}_0\wedge\bar{\star} & D_{M_0/S^1}\bar{\star}+\bar{\star}(\bar{\kappa}\wedge)\\
 D_{M_0/S^1}\bar{\star}-(\bar{\kappa}\wedge)\bar{\star} & \bar{\star}\bar{\varphi}_0\wedge
\end{array}
\right).
\end{align*}
After performing the unitary transformation \eqref{Eqn:U} we finally obtain 
\begin{align*}
\widehat{T}(A^\text{ev})=i \left(
\begin{array}{cc}
\bar{\varphi}_0\wedge\bar{\star} &  D_{M_0/S^1}\bar{\star}+\frac{1}{2}\widehat{c}(\bar{\kappa})\bar{\star} \\
D_{M_0/S^1}\bar{\star}-\frac{1}{2}\widehat{c}(\bar{\kappa})\bar{\star} & \bar{\star}\bar{\varphi}_0\wedge
\end{array}
\right)
: \begin{array}{c}
\Omega^\text{ev}(M_0/S^1)\\
\bigoplus\\
\Omega^\text{odd}(M_0/S^1)
\end{array}
\longrightarrow
\begin{array}{c}
\Omega^\text{ev}(M_0/S^1)\\
\bigoplus\\
\Omega^\text{odd}(M_0/S^1),
\end{array}
\end{align*}
which can be written in the simpler way
\begin{align*}
\widehat{T}(A^\text{ev})=i\left( D_{M_0/S^1}+\frac{1}{2}\varepsilon \widehat{c}(\bar{\kappa})\right)\bar{\star}+i\left(\frac{1+\varepsilon}{2}\right)\bar{\varphi}_0\wedge\bar{\star}-i\left(\frac{1-\varepsilon}{2}\right)(\bar{\varphi}_0\wedge)^\dagger\bar{\star}.
\end{align*}

\begin{remark}
It is easy to verify that this operator is symmetric. For example, using Proposition \ref{Prop:CliffMultVarphi} we calculate the adjoint of a zero order term
\begin{align*}
\left(i\left(\frac{1+\varepsilon}{2}\right)\bar{\varphi}_0\wedge\bar{\star}\right)^{\dagger}=-i\bar{\star}(\bar{\varphi}_0\wedge)^\dagger\left(\frac{1+\varepsilon}{2}\right)
=i(\bar{\varphi}_0\wedge)\bar{\star}\left(\frac{1+\varepsilon}{2}\right)
=i\left(\frac{1+\varepsilon}{2}\right)(\bar{\varphi}_0\wedge)\bar{\star}.
\end{align*}
Moreover, for the first order part we see that $(iD_{M_0/S^1}\bar{\star})^\dagger=-i\bar{\star}D_{M_0/S^1}=iD_{M_0/S^1}\bar{\star}.$
\end{remark}

Observe however that also in this case the zero order part does not anti-commute with $\bar{\star}$ since $\bar{\star}\varepsilon \widehat{c}(\bar{\kappa})=\varepsilon \widehat{c}(\bar{\kappa})\bar{\star}$.

\subsubsection{The Dirac-Schr\"odinger signature operator}\label{Sect:ConstDiracSchrOp}

Our aim is to implement a variant of the construction described above to construct an operator defined on $\Omega_c(M_0/S^1)$ such that: 
\begin{enumerate}
\item Its first order part is  $D_{M_0/S^1}$. 
\item It anti-commutes with the chirality operator $\bar{\star}$. 
\item It is essentially self-adjoint. 
\end{enumerate}
\begin{remark}\label{Rmk:Heuristic}
For instance, in view of the zero order part of the induced operator $\widehat{T}$ of Theorem \ref{Thm:THat},  we would like to have a term containing $c(\bar{\kappa})$ instead of $\widehat{c}(\bar{\kappa})$ since $\bar{\star}c(\bar{\kappa})+c(\bar{\kappa})\bar{\star}=0$. However, this term would need to me modified somehow since $c(\bar{\kappa})^\dagger=-c(\bar{\kappa})$. 
\end{remark}
In view of Proposition \ref{Prop:SESBasic} and \eqref{Eqn:U} we introduce, for $r=0,1,\cdots 4k$, the unitary transformation
\begin{align*}
\psi_r:
\xymatrixcolsep{2cm}\xymatrixrowsep{0.01cm}\xymatrix{
\Omega_c^{r-1}(M_0/S^1)\oplus \Omega_c^r(M_0/S^1) \ar[r]& \Omega_c^r(M_0)^{S^1}\\
(\omega_{r-1},\omega_{r})\ar@{|->}[r] & h^{-1/2}\left(\pi_{S^1}^*\omega_r+(\pi_{S^1}^*\omega_{r-1})\wedge \chi\right).
}
\end{align*}
From the proof of Theorem \ref{Thm:THat} we immediately obtain the following expressions.
\begin{lemma}\label{Lemma:Psir}
For the maps $\psi_r$ we have the relations
\begin{align*}
d\psi_r(\omega_{r-1},\omega_r)
=&\psi_{r+1}\left(d_{M_0/S^1}\omega_{r-1}-\frac{1}{2}\bar{\kappa}\wedge\omega_{r-1},d_{M_0/S^1}\omega_{r}+\frac{1}{2}\bar{\kappa}\wedge\omega_r+\varepsilon(\bar{\varphi_0}\wedge)\omega_{r-1}\right),\\
d^\dagger\psi_r(\omega_{r-1},\omega_r)
=&\psi_{r-1}\left(d^{\dagger}_{M_0/S^1}\omega_{r-1}-\frac{1}{2}\iota_{\bar{\kappa}^\sharp} \omega_{r-1}+\varepsilon(\bar{\varphi}_0\wedge)^\dagger \omega_r ,d^{\dagger}_{M_0/S^1}\omega_{r}+\frac{1}{2}\iota_{\bar{\kappa}^\sharp}\omega_r\right).
\end{align*}
\end{lemma}

Now consider the transformations introduced in \cite[Section 5]{BS88},
\begin{align*}
\psi_\text{ev}:&
\xymatrixcolsep{3pc}\xymatrixrowsep{0.5pc}\xymatrix{
\Omega_c(M_0/S^1)\ar[r]& \Omega_c^\text{ev}(M_0)^{S^1}\\
(\omega_0,\cdots,\omega_{4k} )\ar@{|->}[r] & (\psi_0(0,\omega_0),\psi_2(\omega_1,\omega_2),\cdots,\psi_{4k}(\omega_{4k-1},\omega_{4k})),
}\\\\\
\psi_\text{odd}:&
\xymatrixcolsep{3pc}\xymatrixrowsep{0.5pc}\xymatrix{
\Omega_c(M_0/S^1)\ar[r]& \Omega_c^\text{odd}(M_0)^{S^1}\\
(\omega_0,\cdots,\omega_{4k} )\ar@{|->}[r] & (\psi_1(\omega_0,\omega_1),\psi_3(\omega_2,\omega_3),\cdots,\psi_{4k-1}(\omega_{4k-2},\omega_{4k-1})).
}
\end{align*}
Motivated by Remark \ref{Rmk:Heuristic} and Lemma \ref{Lemma:Psir} we define the operator
\begin{equation}\label{Eqn:ConstrD'}
\mathscr{D}'\coloneqq \psi_{\text{odd}}^{-1}d\psi_{\text{ev}}+\psi_{\text{ev}}^{-1}d^\dagger\psi_{\text{odd}}:\Omega_c(M_0/S^1)\longrightarrow\Omega_c(M_0/S^1).
\end{equation}
Clearly $\mathscr{D}'$ is symmetric since
\begin{align*}
(\mathscr{D}')^\dagger
= \psi_{\text{ev}}^{\dagger}d^\dagger(\psi_{\text{odd}}^{-1})^\dagger+\psi_{\text{odd}}^{\dagger}d (\psi_{\text{ev}}^{-1})^\dagger
= \psi_{\text{ev}}^{-1}d^\dagger\psi_{\text{odd}}+\psi_{\text{odd}}^{-1}d\psi_{\text{ev}}
= \mathscr{D}'.
\end{align*}
\begin{example}
Let us describe the action of the operator $\mathscr{D}'$, using Lemma \ref{Lemma:Psir}, in two particular cases:
\begin{itemize}
\item Let $\omega_1\in\Omega^1_c(M_0/S^1)$, then 
\begin{align*}
\psi_{\text{odd}}^{-1}d\psi_{\text{ev}}\omega_1=&\psi_{\text{odd}}^{-1}d\psi_{2}(\omega_1,0)\\
=&\psi_{\text{odd}}^{-1}\psi_{3}\left(d_{M_0/S^1}\omega_1-\frac{1}{2}\bar{\kappa}\wedge\omega_1,-\bar{\varphi}_0\wedge\omega_1\right)\\
=&d_{M_0/S^1}\omega_1-\frac{1}{2}\bar{\kappa}\wedge\omega_1-\bar{\varphi}_0\wedge\omega_1.
\end{align*}
Similarly, we compute
\begin{align*}
\psi_{\text{ev}}^{-1}d^\dagger\psi_{\text{odd}}\omega_1=&\psi_{\text{odd}}^{-1}d^\dagger \psi_{1}(0,\omega_1)\\
=&\psi_{\text{ev}}^{-1}\psi_{0}\left(0 ,d^\dagger_{M_0/S^1}\omega_1+\frac{1}{2}\iota_{\bar{\kappa}^\sharp}\omega_1\right)\\=& d^\dagger_{M_0/S^1}\omega_1+\frac{1}{2}\iota_{\bar{\kappa}^\sharp}\omega_1.
\end{align*}
Hence 
\begin{align*}
\mathscr{D}'\omega_1=D_{M_0/S^1}\omega_1-\frac{1}{2}c(\bar{\kappa})\omega_1-\bar{\varphi}_0\wedge\omega_1.
\end{align*}
\item On the other hand, for $\omega_2\in\Omega^2_c(M_0/S^1)$, we have
\begin{align*}
\psi_{\text{odd}}^{-1}d\psi_{\text{ev}}\omega_2=&\psi_{\text{odd}}^{-1}d\psi_{2}(0,\omega_2)\\
=&\psi_{\text{odd}}^{-1}\psi_{3}\left(0,d_{M_0/S^1}\omega_2+\frac{1}{2}\bar{\kappa}\wedge\omega_2\right)\\
=&d_{M_0/S^1}\omega_1+\frac{1}{2}\bar{\kappa}\wedge\omega_2
\end{align*}
and 
\begin{align*}
\psi_{\text{ev}}^{-1}d^\dagger\psi_{\text{odd}}\omega_2=&\psi_{\text{odd}}^{-1}d^\dagger \psi_{3}(\omega_2,0)\\
=&\psi_{\text{ev}}^{-1}\psi_{2}\left(0 ,d^\dagger_{M_0/S^1}\omega_1+\frac{1}{2}\iota_{\bar{\kappa}^\sharp}\omega_1\right)\\=& d^\dagger_{M_0/S^1}\omega_1+\frac{1}{2}\iota_{\bar{\kappa}^\sharp}\omega_1.
\end{align*}
Thus we get
\begin{align*}
\mathscr{D}'\omega_2=D_{M_0/S^1}\omega_2+\frac{1}{2}c(\bar{\kappa})\omega_2.
\end{align*}
\end{itemize}
\end{example}

In view of this example it is easy to verify, using Lemma \ref{Lemma:Psir} and the explicit expression of $\psi_{\text{ev/odd}}$, that with respect to the decomposition  $\Omega_c(M_0/S^1)=\bigoplus_{r\geq 0} \Omega^r_c(M_0/S^1)$ we can express $\mathscr{D}'$ in matrix form as
\begin{align*}
\mathscr{D}'
=&\left(
\begin{array}{ccccc}
0 & d^{\dagger}_{M_0/S^1} & 0& \cdots & 0\\
d_{M_0/S^1} & 0 & 0 & \cdots & 0\\
\vdots & & & & \vdots\\
0 & \cdots& 0&  0 &   d^{\dagger}_{M_0/S^1} \\
0 & \cdots& 0& d_{M_0/S^1} & 0
\end{array}
\right)+\\
&\left(
\begin{array}{cccccccccc}
0 & \frac{1}{2}\iota_{\bar{\kappa}^\sharp} & 0& 0 &0 &0&0&\cdots&& 0\\
\frac{1}{2}\bar{\kappa}\wedge & 0 & -\frac{1}{2}\iota_{\bar{\kappa}^\sharp} & \varepsilon(\bar{\varphi}_0\wedge)^\dagger& 0 &0&0&\cdots&& 0\\
0 &-\frac{1}{2}\bar{\kappa}\wedge& 0& \frac{1}{2}\iota_{\bar{\kappa}^\sharp} &0 &0&0&\cdots&& 0\\
0 & \varepsilon(\bar{\varphi}_0\wedge)& \frac{1}{2}\bar{\kappa}\wedge & 0 &-\frac{1}{2}\iota_{\bar{\kappa}^\sharp} & \varepsilon(\bar{\varphi}_0\wedge)^\dagger &0&\cdots&& 0\\
0&0 & 0&-\frac{1}{2}\bar{\kappa}\wedge & 0&\frac{1}{2}\iota_{\bar{\kappa}^\sharp} &0&\cdots&& 0\\
0&0 & 0&  \varepsilon(\bar{\varphi}_0\wedge)& \frac{1}{2}\bar{\kappa}\wedge&0& &&&\vdots\\
\vdots & & & & &&\ddots&&& \\
0 & \cdots& &&&   &  &0&-\frac{1}{2}\iota_{\bar{\kappa}^\sharp}& \varepsilon(\bar{\varphi}_0\wedge)^\dagger\\
0 & \cdots& &&&  &  &-\frac{1}{2}\bar{\kappa}\wedge&0& \frac{1}{2}\iota_{\bar{\kappa}^\sharp} \\
0 & \cdots&  &&& &  & \varepsilon(\bar{\varphi}_0\wedge)&\frac{1}{2}\bar{\kappa}\wedge& 0
\end{array}
\right).
\end{align*}
From this description we obtain the following result.
\begin{proposition}
The operator $\mathscr{D}'\coloneqq \psi_{\textnormal{odd}}^{-1}d\psi_{\textnormal{ev}}+\psi_{\textnormal{ev}}^{-1}d^\dagger\psi_{\textnormal{odd}}:\Omega_c(M_0/S^1)\longrightarrow\Omega_c(M_0/S^1)$ is explicitly given by 
\begin{equation*}
\mathscr{D}'=D_{M_0/S^1}+\frac{1}{2}c(\bar{\kappa})\varepsilon-\widehat{c}(\bar{\varphi}_0)\left(\frac{1-\varepsilon}{2}\right).
\end{equation*}
Thus, it is a first order elliptic operator. Moreover, it is symmetric on $\Omega_c(M_0/S^1)$ and anti-commutes with the chirality operator $\bar{\star}$. 
\end{proposition}
\begin{proof}
Observe for the principal symbol  that $\sigma_P(\mathscr{D}')=\sigma_P(D_{M_0/S^1})$, so indeed $\mathscr{D}'$ is first order elliptic. For the symmetry assertion we just compute 
\begin{align*}
(c(\bar{\kappa})\varepsilon)^\dagger=\varepsilon c(\bar{\kappa})^{\dagger}=-\varepsilon c(\bar{\kappa})=c(\bar{\kappa})\varepsilon,
\end{align*}
and 
\begin{align*}
(\widehat{c}(\bar{\varphi}_0)(1-\varepsilon))^\dagger=(1-\varepsilon)\widehat{c}(\bar{\varphi}_0)=\widehat{c}(\bar{\varphi}_0)(1-\varepsilon). 
\end{align*}
For the last assertion we calculate using Proposition \ref{Prop:CliffMultVarphi}, 
\begin{align*}
\bar{\star}\mathscr{D}'=&\bar{\star}D_{M_0/S^1} +\bar{\star}\frac{1}{2}c(\bar{\kappa})\varepsilon-\bar{\star}\widehat{c}(\bar{\varphi}_0)\left(\frac{1-\varepsilon}{2}\right)\\
=&-D_{M_0/S^1}\bar{\star}-\frac{1}{2}c(\bar{\kappa})\bar{\star}\varepsilon+\widehat{c}(\bar{\varphi}_0)\bar{\star}\left(\frac{1-\varepsilon}{2}\right)\\
=&-D_{M_0/S^1}\bar{\star}-\frac{1}{2}c(\bar{\kappa})\varepsilon\bar{\star}+\widehat{c}(\bar{\varphi}_0)\left(\frac{1-\varepsilon}{2}\right)\bar{\star}\\
=&-\mathscr{D}'\bar{\star}. 
\end{align*}
\end{proof}

Due the nature of the potential of the operator $\mathscr{D}'$ and in view of Theorem \ref{Thm:THat} we would expect $\mathscr{D}'$ to be essentially self-adjoint on the core $\Omega_c(M_0/S^1)$. In order to apply the construction of Br\"uning and Heintze we need to find an operator on $M$, commuting with the $S^1$-action, so that when pushed down to $M_0/S^1$ coincides with $\mathscr{D}'$. Let us explore how to find such an operator. In view of \eqref{Eqn:ConstrD'} we define 
\begin{align*}
\text{d}\coloneqq \psi_\text{odd}^{-1}d\psi_\text{ev} \quad \text{and}\quad \text{d}^\dagger\coloneqq \psi_\text{ev}^{-1}d^\dagger\psi_\text{odd},
\end{align*}
so that $\mathscr{D}'=\text{d}+\text{d}^\dagger$.  Observe that these operators fit in the commutative diagrams
\begin{align*}
\xymatrixcolsep{4pc}\xymatrixrowsep{3pc}\xymatrix{
\Omega_c^\text{ev}(M_0)^{S^1} \ar[r]^-{d} & \Omega_c^\text{odd}(M_0)^{S^1} \ar[r]^-{\psi_\text{ev}\psi_\text{odd}^{-1}} & \Omega_c^\text{ev}(M_0)^{S^1}\\
\Omega_c(M_0/S^1) \ar[u]^-{\psi_\text{ev}} \ar[r]^{\text{d}} & \Omega_c(M_0/S^1) \ar[u]^-{\psi_\text{odd}} \ar@{=}[r] & \Omega_c(M_0/S^1), \ar[u]^-{\psi_\text{ev}} 
}
\end{align*}
and
\begin{align*}
\xymatrixcolsep{4pc}\xymatrixrowsep{3pc}\xymatrix{
\Omega_c^\text{ev}(M_0)^{S^1} \ar[r]^-{\psi_\text{odd}\psi_\text{ev}^{-1}} & \Omega_c^\text{odd}(M_0)^{S^1} \ar[r]^-{d^\dagger} & \Omega_c^\text{ev}(M_0)^{S^1}\\
\Omega_c(M_0/S^1) \ar[u]^-{\psi_\text{ev}} \ar@{=}[r]& \Omega_c(M_0/S^1) \ar[u]^-{\psi_\text{odd}} \ar[r]^{\text{d}^\dagger} & \Omega_c(M_0/S^1).\ar[u]^-{\psi_\text{ev}} 
}
\end{align*}
\begin{lemma}
Let $c(\chi)$ be the left Clifford multiplication by the characteristic $1$-form. Then the following relations hold,
\begin{align*}
\psi_\textnormal{ev}\psi_\textnormal{odd}^{-1}=&-c(\chi),\\
\psi_\textnormal{odd}\psi_\textnormal{ev}^{-1}=& c(\chi).
\end{align*}
\end{lemma}
\begin{proof}
We compute the left hand side of the first relation for an element in the image of $\psi_\text{odd}$,
\begin{align*}
\psi_\text{ev}\psi_\text{odd}^{-1}\left((h^{-1/2}\left(\pi_{S^1}^*\omega_{2r+1}+(\pi_{S^1}^*\omega_{2r})\wedge \chi\right)\right)=&\psi_\text{ev}\psi_\text{odd}^{-1}(\psi_{2r+1}(\omega_{2r},\omega_{2r+1}))\\
=&\psi_\text{ev}(\omega_{2r},\omega_{2r+1})\\
=&\psi_{2r}(0,\omega_{2r})+\psi_{2r+2}(\omega_{2r+1},0)\\
=&h^{-1/2}\left(\pi_{S^1}^*\omega_{2r}+(\pi_{S^1}^*\omega_{2r+1})\wedge \chi\right).
\end{align*}
On the other hand,
\begin{align*}
c(\chi)\left((h^{-1/2}\left(\pi_{S^1}^*\omega_{2r+1}+(\pi_{S^1}^*\omega_{2r})\wedge \chi\right)\right)
=&h^{-1/2}\left(\chi\wedge(\pi_{S^1}^*\omega_{2r+1})-\pi_{S^1}^*\omega_{2r}\right)\\
=&-h^{-1/2}\left((\pi_{S^1}^*\omega_{2r+1})\wedge\chi+\pi_{S^1}^*\omega_{2r}\right).
\end{align*}
This shows the first relation. The second one follows by taking adjoints.
\end{proof}
From this lemma we obtain the commutative diagram
\begin{align}\label{Diag:B}
\xymatrixcolsep{3cm}\xymatrixrowsep{2cm}\xymatrix{
\Omega_c^\text{ev}(M_0)^{S^1} \ar[r]^-{-c(\chi)d+d^\dagger c(\chi)} & \Omega_c^\text{ev}(M_0)^{S^1} \\
\Omega_c(M_0/S^1) \ar[u]^-{\psi_\text{ev}} \ar[r]^-{\mathscr{D}'}& \Omega_c(M_0/S^1), \ar[u]^-{\psi_\text{ev}}
 }
\end{align}
and conclude that the operator $\mathscr{D}'$ is unitary equivalent to the operator $ -c(\chi)d+c(\chi)d^\dagger$ when restricted to $\Omega^\text{ev}_c(M_0)^{S^1}$.
\begin{proposition}\label{Prop:OpB}
The operator $B\coloneqq -c(\chi)d+d^\dagger c(\chi)$ satisfies:
\begin{enumerate}
\item It is a transversally elliptic first order differential operator with principal symbol
\begin{align*}
\sigma_P(B)(x,\xi)=-i(\inner{\chi}{\xi}+c(\xi)c(\chi)).
\end{align*}
\item It can be extended to $M$ and $B:\Omega(M)\longrightarrow \Omega(M)$  it essentially self-adjoint when defined on this core. 
\item It commutes with the $S^1$-action on differential forms.
\item It commutes with the Gau\ss-Bonnet involution $\varepsilon$ and therefore it can be decomposed as $B=B^\textnormal{ev}\oplus B^\textnormal{odd}$ where $B^\textnormal{ev/odd}:\Omega^{\textnormal{ev/odd}}(M)\longrightarrow \Omega^{\textnormal{ev/odd}}(M)$. 
\end{enumerate}
\end{proposition}
\begin{proof}
To prove the first statement recall from the proof of Proposition \ref{Prop:PSD} the expressions for the principal symbols
\begin{align*}
\sigma_P(d)(x,\xi)=&-i\xi\wedge,\\
\sigma_P(d^\dagger)(x,\xi)=&i\iota_{\xi^\sharp}.
\end{align*}
Using the relation
\begin{align*}
c(\chi)\circ (\xi\wedge)=&\chi\wedge\xi\wedge-\iota_{X}\circ (\xi\wedge)=\chi\wedge\xi\wedge -\inner{\chi}{\xi}+\xi\wedge \iota_X =-\xi\wedge c(\chi)-\inner{\chi}{\xi},
\end{align*}
we calculate the principal symbol of the operator $B$,
\begin{align*}
\sigma_p(B)(x,\xi)=& -c(\chi)\sigma_P(d)(x,\xi)+\sigma_P(d^\dagger)(x,\xi)c(\chi)\\
=&i\left(c(\chi)\circ(\xi\wedge)  +(\iota_{\xi^\sharp})\circ c(\chi)\right)\\
=&i\left(-\xi\wedge c(\chi)-\inner{\chi}{\xi} +(\iota_{\xi^\sharp})\circ c(\chi)\right)\\
=&-i(\inner{\chi}{\xi}+c(\xi)c(\chi)).
\end{align*}
In particular we see that if $\inner{\chi}{\xi}=0$ then 
\begin{align*}
\sigma_P(B)(x,\xi)^2=(-ic(\xi)c(\chi))^2=-c(\xi)c(\chi)c(\xi)c(\chi)=c(\xi)c(\chi)^2c(\xi)=\norm{\xi}^2,
\end{align*}
from which follows that $B$ is a transversally elliptic first order differential operator. To prove the second assertion observe that the Clifford multiplication operator $c(\chi)$ has domain $\dom(c(\chi))=\Omega_c(M_0)$. For $\omega\in\Omega_c(M_0)$ compute the $L^2$-norm
\begin{align*}
\norm{c(\chi)\omega}^2_{L^2(\wedge T^* M)}=\int_M\inner{c(\chi)\omega}{c(\chi)\omega}(x)\vol_M(x)=\int_M\inner{\omega}{\omega}(x)\vol_M(x)=\norm{\omega}^2_{L^2(\wedge T^* M)},
\end{align*}
where we have used that $c(\chi)^\dagger=-c(\chi)$ and $c(\chi)^2=-1$. Hence we can extend the operator $c(\chi)$ to all $L^2(M,\wedge_\mathbb{C} T^*M)$ by density. Using this fact we see that $B$ is indeed densely defined with core $\Omega(M)$.
Since $M$ is compact we can use Remark \ref{Remark:GL02} to conclude that $B$ is an essentially self-adjoint operator. The last two assertions follow easily from Proposition \ref{Prop:S1commD} and \eqref{PropertiesEpsilonCliff}.
\end{proof}

Now we can implement the construction described in Section \ref{Section:InduedOperatorsGen} in this setting: the restriction of $B$ to the $S^1$-invariant forms remains essentially self-adjoint and since this operator is unitary equivalent to $\mathscr{D}'$ through $\psi_\text{ev}$ we conclude that $\mathscr{D}'$ is essentially self-adjoint with core $\Omega_c(M_0/S^1)$. We summarize these results in the next theorem.

\begin{theorem}\label{Thm:InducedDiracOp}
The Dirac-Schr\"odinger operator 
\begin{align*}
\mathscr{D}'=D_{M_0/S^1} +\frac{1}{2}c(\bar{\kappa})\varepsilon-\frac{1}{2}\widehat{c}(\bar{\varphi}_0)(1-\varepsilon),
\end{align*}
defined on on $\Omega_c(M_0/S^1)$, is a first order elliptic differential operator which is essentially self-adjoint.  As the dimension of  $M$ is odd then $\mathscr{D}'$ anti-commutes with the chirality operator $\bar{\star}$ on $M_0/S^1$ and therefore we can define the operator
\begin{align*}
\mathscr{D}'^{+}:\Omega^+_c(M_0/S^1)\longrightarrow \Omega^-_c(M_0/S^1),
\end{align*}
where $\Omega^\pm_c(M_0/S^1)$ is the  $\pm 1$-eigenspace of $\bar{\star}$. 
\end{theorem}

\begin{example}[Closed manifold IV]
In the situation of  Example \ref{Ex:ClosedManifold} we of course have $\ind(\mathscr{D}'^+)=\ind(D_X^+)=\sigma(X)$. 
\end{example}

\begin{example}[Free action]
Let us consider the case in which the action on $M$ is free. In this situation, by  Corollary \ref{Coro:FreeActMnfld}, the quotient space $M/S^1$ is a closed manifold. Moreover,  $\kappa$ and $\varphi_0$ are both bounded, so the zero order term
\begin{align*}
\frac{1}{2}c(\bar{\kappa})\varepsilon-\widehat{c}(\bar{\varphi}_0)\left(\frac{1-\varepsilon}{2}\right),
\end{align*}
is a bounded operator on $L^2(\wedge_\mathbb{C} T^*(M/S^1))$. Since the operators $\mathscr{D}$ and $D$ have the same principal symbol then $\ind(\mathscr{D}'^+)=\ind(D^+)=\sigma(M/S^1)$. 
\end{example}

\begin{remark}
Let us explore in some detail the relation $\mathscr{D}'\bar{\star}=-\bar{\star}\mathscr{D}'$ at the level of $S^1$-invariant forms on $M$.  To do so, consider the operator $\blacklozenge\coloneqq i\varepsilon\star c(\chi)$ on $\Omega(M_0)^{S^1}$, then
\begin{enumerate}
\item $\blacklozenge^\dagger=-i(-c(\chi))\star\varepsilon=i\star c(\chi)\varepsilon=-i\star \varepsilon c(\chi)=i\varepsilon\star  c(\chi)=\blacklozenge$.
\item $\blacklozenge^2=-\varepsilon\star c(\chi)\varepsilon\star c(\chi)=\varepsilon\star \varepsilon c(\chi)\star c(\chi)=-\star c(\chi)\star c(\chi)=1$.
\end{enumerate}
This shows that $\blacklozenge$ is a self-adjoint involution. Moreover, 
\begin{align*}
\blacklozenge B=&i\varepsilon\star c(\chi) (-c(\chi)d+d^\dagger c(\chi))\\
=&i\varepsilon\star (d+c(\chi)d^\dagger c(\chi))\\
=&i\varepsilon\star (-dc(\chi)+c(\chi)d^\dagger)c(\chi)\\
=&i\varepsilon(-d^\dagger\star c(\chi)+c(\chi)d\star)c(\chi)\\
=&(-d^\dagger c(\chi)+c(\chi)d)i\varepsilon\star c(\chi)\\
=&-B\blacklozenge.
\end{align*}
We now compute using Lemma \ref{Lemma:star},
\begin{align*}
\blacklozenge\bigg{|}_{\Omega(M_0)^{S^1}}=& 
i
\left(\begin{array}{cc}
\varepsilon & 0\\
0 & -\varepsilon 
\end{array}
\right)
\left(\begin{array}{cc}
0 & -i\varepsilon\bar{\star}\\
i\varepsilon\bar{\star}& 0 
\end{array}
\right)
\left(\begin{array}{cc}
0 & -\varepsilon\\
\varepsilon & 0
\end{array}
\right)\\
=&
i
\left(\begin{array}{cc}
\varepsilon & 0\\
0 & -\varepsilon 
\end{array}
\right)
\left(\begin{array}{cc}
-i\bar{\star} & 0\\
0 & -i\bar{\star} 
\end{array}
\right)\\
=&
\left(\begin{array}{cc}
\varepsilon\bar{\star} & 0\\
0 & -\varepsilon\bar{\star} 
\end{array}
\right). 
\end{align*}
In particular,
\begin{align*}
\blacklozenge\bigg{|}_{\Omega^\text{ev}(M_0)^{S^1}}=
\left(\begin{array}{cc}
\bar{\star} & 0\\
0 & \bar{\star}
\end{array}
\right),
\end{align*}
which shows that $\bar{\star}\mathscr{D}=-\mathscr{D}\bar{\star}$ since $\blacklozenge B=-B \blacklozenge$.
\end{remark}

\begin{remark}\label{Rmk:OpD}
We will later give a detailed description of the operator $\mathscr{D}'$ close to a connected component of the fixed point set. We will see that the term containing $\widehat{c}(\bar{\varphi}_0)$ is actually bounded and therefore, by the Kato-Rellich theorem, it will be enough to consider the operator
\begin{align*}
\mathscr{D}\coloneqq D_{M_0/S^1}+\frac{1}{2}c(\bar{\kappa})\varepsilon.
\end{align*}
We will also see that the factor $1/2$ is fundamental for the essential self-adjointness of $\mathscr{D}$. 
\end{remark}

\begin{remark}[Gau\ss-Bonet Involution]\label{Rmk:GBInv}
It is straightforward to verify that $\mathscr{D}$ anti-commutes with $\varepsilon$, while $\mathscr{D}'$ does not. Thus, we can split
\begin{align*}
\mathscr{D}=
\left(
\begin{array}{cc}
0 & \mathscr{D}^{\text{odd}}\\
\mathscr{D}^{\text{ev}} & 0 
\end{array}
\right),
\end{align*}
where $\mathscr{D}^{\text{ev/odd}}:\Omega^{\text{ev/odd}}_c(M_0/S^1)\longrightarrow \Omega^{\text{odd/ev}}_c(M_0/S^1)$.
\end{remark}

The following result shows that $\mathscr{D}^2$ is a generalized Laplacian in the sense of \cite[Definition 2.2]{BGV}. 

\begin{lemma}\label{Lemma:SquareD}
Let $D_{M_0/S^1}^2=\Delta_{M_0/S^1}$ be the Laplacian on $M_0/S^1$ with respect to the quotient metric, then 
\begin{align*}
\left(D_{M_0/S^1}+\frac{1}{2}c(\bar{\kappa})\varepsilon\right)^2=\Delta_{M_0/S^1}-\nabla_{\bar{\kappa}^\sharp}^{M_0/S^1}\varepsilon-c(\bar{\kappa})D_{M_0/S^1}\varepsilon+\frac{1}{2}d_{M_0/S^1}^\dagger(\bar{\kappa})\varepsilon+\frac{1}{4}\norm{\bar{\kappa}^\sharp}^2,
\end{align*}
where $\nabla^{M_0/S^1}$ denotes the induced Levi-Civita connection.
\end{lemma}

\begin{proof}
We first expand
\begin{align*}
\left(D_{M_0/S^1}+\frac{1}{2}c(\bar{\kappa})\varepsilon\right)^2=&\Delta_{M_0/S^1}+\frac{1}{2}D_{M_0/S^1}c(\bar{\kappa})\varepsilon+\frac{1}{2}c(\bar{\kappa})\varepsilon D_{M_0/S^1}+\frac{1}{4}\norm{\bar{\kappa}^\sharp}^2.
\end{align*}
Then, from \cite[Proposition 3.45]{BGV} and the fact that $d\kappa=0$ we obtain
\begin{align*}
c(\bar{\kappa})D_{M_0/S^1}+D_{M_0/S^1}c(\bar{\kappa})=-2\nabla_{\bar{\kappa}^\sharp}^{M_0/S^1}+d_{M_0/S^1}^\dagger(\bar{\kappa}). 
\end{align*}
\end{proof}

\begin{remark}
Observe that we can write 
\begin{align*}
\left(D_{M_0/S^1}+\frac{1}{2}c(\bar{\kappa})\varepsilon\right)^2=\Delta_{M_0/S^1}+\frac{\varepsilon}{2}[D_{M_0/S^1},c(\bar{\kappa})]+\frac{1}{4}\norm{\bar{\kappa}^\sharp}^2.
\end{align*}
In particular, for $f\in C^{\infty}(M_0/S^1)$ we have by Proposition \ref{Prop:CommDirOpFunct},
\begin{align*}
[D_{M_0/S^1},c(\bar{\kappa})]f=&D_{M_0/S^1}(fc(\bar{\kappa}))-c(\bar{\kappa})D_{M_0/S^1}f\\
=&c(df)c(\bar{\kappa})+fD_{M_0/S^1}c(\bar{\kappa})-c(\bar{\kappa})c(df)-fc(\bar{\kappa})D_{M_0/S^1}\\
=&f[D_{M_0/S^1},c(\bar{\kappa})]+[c(df),c(\bar{\kappa})].
\end{align*}
\end{remark}

Even though we are not going to make use of it,  we want to finish this section by showing that the operator $\mathscr{D}'$ is also essentially self-adjoint  whenever $M$ is a complete (not necessarily compact) manifold on which $S^1$ acts by orientation preserving isometries. The strategy of the proof is inspired in the similar result for Dirac operators. 

\begin{lemma}\label{lemma:CommBf}
For $f\in C^\infty(M)$ we have $[B,f]=c(df)c(\chi)+\inner{\chi}{d f}$, where $[\cdot, \cdot]$ denotes the commutator. 
\end{lemma}
\begin{proof}
First observe from Proposition \ref{Prop:Chirl}(3) that $\star (df\wedge)\star=-\iota_{(df)^\sharp}$. Then we just compute as in the proof of Proposition \ref{Prop:OpB},
\begin{align*}
[B,f]=&-c(\chi)\circ (df\wedge)+(\star (df\wedge)\star)\circ c(\chi)\\
=&-c(\chi)\circ (df\wedge)-\iota_{(df)^\sharp}\circ c(\chi)\\
=&df\wedge c(\chi)+\inner{\chi}{d f}-\iota_{(df)^\sharp}\circ c(\chi)\\
=& c(df)c(\chi)+\inner{\chi}{d f}.
\end{align*}
\end{proof}

\begin{coro}\label{Coro:BComplete}
Let $M$ be a complete Riemannian manifold. Then the operator $B$ is essentially self-adjoint. 
\end{coro}
\begin{proof}
We are going to adapt the proof for Dirac operators given in \cite[Chapter 4]{F00}, \cite[Theorem II.5.7]{LM89} and \cite{W73}. Let $\beta\in\dom(B^*)$, we want to show that $\beta\in\dom (B)$. Choose a real-valued smooth function  $\varrho\in C^\infty(M)$ satisfying the following conditions:
\begin{itemize}
\item $0\leq \varrho\leq 1 $
\item $\varrho(t)=1$ for $t\leq 1$.
\item $\varrho(t)=0$ for $t\geq 2$.
\item $|\varrho'(t)|\leq 2$.
\end{itemize}
Fix an element $x_0\in M$ and let $d:M\longrightarrow \mathbb{R}$ be the distance function at $x_0$, which is a locally Lipschitz continuous function and therefore differentiable with $|d'(x)|\leq 1$ almost everywhere. Since $M$ is complete, for each $r>1$ the closure of the open ball $B_r\coloneqq \{x\in M\:|\: d(x)\leq r\}$  is compact. Define a sequence of functions on $M$ by 
\begin{align*}
\varrho_n(x):=\varrho\left(\frac{d(x)}{n}\right),
\end{align*}
whose support satisfies $\supp(\varrho_n)\subset \overline{B_{2n}}$. Thus, each $\varrho_n$ has compact support. Moreover, these functions are also locally Lipschitz and therefore differentiable almost everywhere. On these points  we have the estimate 
\begin{align*}
\norm{d\varrho_n(x)}^2\leq\frac{4}{n^2}.
\end{align*}
Define now $\beta_n:=\varrho_n \beta\in\dom(B^*)$, we want to show that $\beta_n\longrightarrow \beta $ and $B^*\beta_n\longrightarrow B^*\beta $. First note that 
\begin{align*}
\int_M \norm{\beta_n-\beta}^2dx=\int_M|\varrho_n(x)-1|^2\norm{\beta}^2 dx \leq \int_{M-B_{2n}}\norm{\beta}^2 \longrightarrow 0
\end{align*}
as $n\longrightarrow \infty$. Now we deal with the sequence $(B^*\beta_n)_n$. If follows from \cite[Remark II.5.6]{LM89} that the analogous of Lemma \ref{lemma:CommBf} holds also in the distributional sense, 
\begin{align*}
B^*\beta_n=c(d\varrho_n)c(\chi)\beta+\inner{\chi}{d\varrho_n}\beta +\varrho_n B^*\beta.
\end{align*}
Arguing as before we see that  $\varrho_n B^*\beta\longrightarrow B^*\beta$ as $n\longrightarrow\infty$. On the other hand,
\begin{align*}
\norm{(d\varrho_n)c(\chi)\beta}^2_{L^2(\wedge T^*M)}&\leq\frac{4}{n^2}\norm{\beta}^2_{L^2(\wedge T^*M)}\longrightarrow 0,\\
\norm{\inner{\chi}{d\varrho_n}\beta}^2_{L^2(\wedge T^*M)}&\leq\frac{4}{n^2}\norm{\beta}^2_{L^2(\wedge T^*M)}\longrightarrow 0,
\end{align*}
as $n\longrightarrow \infty$. Altogether, $B^*\beta_n\longrightarrow B^*\beta$ in $L^2$ and therefore  $\beta_n\longrightarrow \beta$ in the graph norm of $B^*$. This shows that it is enough to consider $\beta\in\dom(B^*)$ with compact support. For this case we can now use Lemma \ref{LemmaGL02} to conclude $\beta\in \dom(B)$. 
\end{proof}

\subsection{The basic signature operator}

The final section of this chapter is devoted to comparing our methods with the ones of Habib and Richardson on modified differentials in the context of Riemannian foliations \cite{HR13}. In their work they defined the basic signature operator as a Dirac-type operator on the space of basic forms. This operator has the important property that it anti-commutes with the chirality operator $\bar{\star}$ and therefore the basic signature can be defined as its index.\\

We will continue in the setting in which $S^1$ acts effectively and semi-freely by orientation preserving isometries on a closed $4k+1$ dimensional Riemannian manifold $M$. To begin let us define the space of $L^2$-basic forms, denoted by $L^2_\text{bas}(M_0)$, as the Hilbert space completion of $\Omega_\text{{bas},c}(M_0)$ with respect to the inner product
\begin{align*}
(\omega_0,\omega_1)_\text{bas}\coloneqq\int_M \chi\wedge\omega_0\wedge \bar{*}\omega_1,
\end{align*}
where $\bar{*}$ is the basic Hodge star operator of Definition \ref{Def:Bar*}. Note that 
\begin{align*}
(\omega_0,\omega_1)_\text{bas}\coloneqq \int_M \chi\wedge\omega_0\wedge \bar{*}\omega_1=\int_M \omega_0\wedge \bar{*}\omega_1\wedge\chi=\int_M \omega_0\wedge* \omega_1, 
\end{align*}
so $L^2_\text{bas}(M_0)$ is nothing else but the closed subspace of $L^2(\wedge T^*M_0)$ spanned by $\Omega_\text{{bas},c}(M_0)$. As we have seen before, the exterior derivative maps basic forms to basic forms, so it is natural to ask which is the associated   formal adjoint with respect to the basic $L^2$-inner product described above. Observe from Remark \ref{Rmk:AdjointNotPresBas} that $d^\dagger $ is not such an adjoint because it does not preserve the space of basic forms. Nevertheless, one can correct this problem by composing it with the projection $P_\text{bas}:L^2(\wedge T^*M)\longrightarrow L^2_\text{bas}(M)$ described in \cite{PR96}. That is, the desired basic adjoint of $d$ is simply
\begin{align*}
\delta_b\coloneqq P_\text{bas} d^\dagger:\Omega^{r}_\text{bas}(M)\longrightarrow \Omega^{r-1}_\text{bas}(M). 
\end{align*}
Indeed, a simple computation shows that $\delta_b$ satisfies the required condition,
\begin{align*}
(\omega_0,\delta_b \omega_1)_\text{bas}=(\omega_0,P_\text{bas}d^{\dagger}\omega_1)_\text{bas}=(P_\text{bas}\omega_0,d^{\dagger}\omega_1)_\text{bas}=(\omega_0,d^{\dagger}\omega_1)_\text{bas}=(d\omega_0,\omega_1)_\text{bas}.
\end{align*}
Form Theorem \ref{Thm:OpInv} we see that this operator is given explicitly by (\cite[Theorem 7.10]{T97})
\begin{align}\label{Eqn:Deltab}
\delta_b=-\bar{\star}d\bar{\star}+\iota_H,
\end{align}
where $H$ is the mean curvature vector field. One would be tempted to define the basic Dirac operator as $d+\delta_b$, but this turns out to be problematic since  it is clear that  $d+\delta_b$ does not anti-commute with $\bar{\star}$. To overcome this issue, Habib and Richardson introduced in  \cite{HR13} the twisted differential 
\begin{align*}
\widetilde{d}\coloneqq d-\frac{1}{2}\kappa\wedge,
\end{align*}
where $\kappa$ is the mean curvature form. The first thing to check is that $\widetilde{d}$ is really a differential, i.e. that it satisfies $(\widetilde{d})^2=0$. This follows by Corollary \ref{Coro:kappa} since
\begin{align*}
(\widetilde{d})^2=-d \circ \left(\frac{1}{2}\wedge \kappa \right)-\frac{1}{2}\kappa\wedge d=-\frac{1}{2}d\kappa=0.
\end{align*}
Let us denote by  $\widetilde{H}^{*}_{\textnormal{bas},c}(M-M^{S^1})$ the cohomology with compact support of the twisted differential $\widetilde{d}$. Then we have the following result \footnote{I would like to thank Ken Richardson for pointing me out this fact.}.

\begin{proposition}
The symmetric pairing
\begin{align*}
\xymatrixcolsep{2cm}\xymatrixrowsep{0.01cm}\xymatrix{
\widetilde{H}^{2k}_{\textnormal{bas},c}(M-M^{S^1})\times \widetilde{H}^{2k}_{\textnormal{bas},c}(M-M^{S^1}) \ar[r] & \mathbb{R}\\
(\omega,\omega') \ar@{|->}[r] & \displaystyle{\int_M\chi\wedge \omega\wedge\omega'},
}
\end{align*}
is well defined.
\end{proposition}

\begin{proof}
Let $\beta\in\Omega^{2k-1}_{\text{bas},c}(M-M^{S^1})$ and $\omega\in\Omega_{\text{bas},c}^{2k}(M-M^{S^1})$ such that
\begin{align*}
d\omega-\frac{1}{2}\kappa\wedge\omega=&0.
\end{align*}
We need to show that 
\begin{align*}
\int_{M}\chi \wedge\left(d\beta-\frac{1}{2}\kappa\wedge\beta\right)\wedge \omega=0.
\end{align*}
Using \eqref{Eqn:dchi} we compute the exterior derivative
\begin{align*}
d(\chi\wedge\beta\wedge \omega)=&(d\chi)\wedge\beta\wedge\omega-\chi\wedge d(\beta\wedge\omega)\\
=&(d\chi)\wedge\beta\wedge\omega-\chi\wedge d\beta\wedge\omega+\chi\wedge \beta\wedge d\omega\\
=&\varphi_0\wedge\beta\wedge \omega-\kappa\wedge\chi\wedge\beta\wedge\omega-\chi\wedge d\beta\wedge\omega+\chi\wedge \beta\wedge d\omega\\
=&\varphi_0\wedge\beta\wedge \omega-\chi\wedge\beta\wedge\kappa\wedge\omega-\chi\wedge d\beta\wedge\omega+\chi\wedge \beta\wedge d\omega\\
=&\varphi_0\wedge\beta\wedge \omega-\frac{1}{2}\chi\wedge\beta\wedge\kappa\wedge\omega-\chi\wedge d\beta\wedge\omega\\
=&\varphi_0\wedge\beta\wedge \omega+\frac{1}{2}\chi\wedge\kappa\wedge\beta\wedge\omega-\chi\wedge d\beta\wedge\omega\\
=&\varphi_0\wedge\beta\wedge \omega-\chi \wedge\left(d\beta-\frac{1}{2}\kappa\wedge\beta\right)\wedge \omega.
\end{align*}
Now, on the one hand 
\begin{align*}
\int_M d(\chi\wedge\beta\wedge \omega)=0,
\end{align*}
by Stokes' theorem, and on the other hand
\begin{align*}
\int_M \varphi_0\wedge\beta\wedge\omega=0,
\end{align*}
since $\varphi_0\wedge\beta\wedge\omega$ is a top-degree basic form, which therefore should be zero. 
\end{proof}

The adjoint operator of the twisted differential $\widetilde{d}$ with respect to the $L^2$-inner product on basic forms is 
\begin{align*}
\widetilde{\delta}\coloneqq{(\widetilde{d})}^\dagger=\left(d-\frac{1}{2}\kappa\wedge\right)^\dagger=\delta_b-\frac{1}{2}\iota_H.
\end{align*}
It is natural to study the operator $\widetilde{D}:=\widetilde{d}+\widetilde{d}^\dagger$ acting on smooth basic forms, which by construction is symmetric with respect to $(\cdot, \cdot)_\text{bas}$. 

\begin{lemma}[{{\cite[Proposition 5.1]{HR13}}}]\label{Lemma:AntiCommBasDirOpHR}
The operator $\widetilde{D}$ satisfies $\bar{\star}\widetilde{D}+\widetilde{D}\bar{\star}=0$.
\end{lemma}
\begin{proof}
First observe from \eqref{Eqn:Deltab} that
\begin{align*}
\widetilde{D}=\widetilde{d}+\widetilde{d}^\dagger=d-\frac{1}{2}\kappa\wedge+\delta_b-\frac{1}{2}\iota_H=d-\bar{\star}d^\dagger\bar{\star}-\frac{1}{2}c(\kappa).
\end{align*}
Then, since
\begin{align*}
\bar{\star}\left(D_{M_0/S^1}-\frac{1}{2}c(\bar{\kappa})\right) +\left(D_{M_0/S^1}-\frac{1}{2}c(\bar{\kappa})\right)\bar{\star}=0,
\end{align*}
the result follows from the commutative diagram
\begin{align*}
\xymatrixrowsep{2.5cm}\xymatrixcolsep{2.5cm}\xymatrix{
\Omega_{\textnormal{bas},c}(M_0)  \ar[r]^-{\widetilde{D}} & \Omega_{\textnormal{bas},c} (M_0)\\
\Omega_c(M_0/S^1) \ar[u]^-{\pi_{S^1}^*} \ar[r]^-{D_{M_0/S^1}-\frac{1}{2}c(\bar{\kappa})}& \Omega_c(M_0/S^1). \ar[u]_-{\pi_{S^1}^*}
}
\end{align*}
\end{proof}
From this lemma we see that we can decompose $\widetilde{D}$  with respect to $\bar{\star}$ as
\begin{align*}
\widetilde{D}=
\left(
\begin{array}{cc}
0 & \widetilde{D}^-\\
\widetilde{D}^+ & 0
\end{array}
\right).
\end{align*}
\begin{definition}
The {\em basic signature operator} is defined as the component
$$\widetilde{D}^+:\Omega^+_{\text{bas},c}(M_0)\longrightarrow\Omega^-_{\text{bas},c}(M_0).$$
\end{definition}
In view of the proof of Lemma \ref{Lemma:AntiCommBasDirOpHR} we would like to study the operator $\widetilde{D}$ as an operator on $\Omega_c(M_0/S^1)$ via the pullback of the orbit map and compare it with $\mathscr{D}'$. First of all note that the operator
$$D_{M_0/S^1}-\frac{1}{2}c(\bar{\kappa})$$
is not symmetric on $\Omega_c(M_0/S^1)$ with respect to the quotient metric since $c(\bar{\kappa})^\dagger=-c(\bar{\kappa})$. As we did in the sections above we need to implement the transformation \eqref{Eqn:U}. Recall from the proof of Theorem \ref{Thm:THat} the formula
\begin{align*}
h^{1/2}D_{M_0/S^1}h^{-1/2}=D_{M_0/S^1}+\frac{1}{2}c(\bar{\kappa}). 
\end{align*}
Hence, if we set $\psi\coloneqq h^{-1/2}\circ \pi^*_{S^1}$, then we obtain the following commutative diagram 
\begin{align*}
\xymatrixrowsep{2cm}\xymatrixcolsep{2cm}\xymatrix{
\Omega_{\textnormal{bas},c}(M_0)  \ar[r]^-{\widetilde{D}} & \Omega_{\textnormal{bas},c} (M_0)\\
\Omega_c(M_0/S^1) \ar[u]^-{\psi} \ar[r]^-{D_{M_0/S^1}}& \Omega_c(M_0/S^1). \ar[u]_-{\psi}
}
\end{align*}
This means that the basic signature operator $\widetilde{D}^+$, regarded as a symmetric operator on the quotient space through $\psi$, is precisely the signature operator $D^+_{M_0/S^1}$ on the quotient space. 

\section{Examples}\label{Sect:Examples}

In this chapter we provide some examples of the theory presented so far, with particular focus on the construction of the operator $\mathscr{D}'$. The main objective is to illustrate how explicit computation are carried out in practice. We begin with some low dimensional examples where all the geometric quantities of Chapter \ref{Section:InducedOpConstruction}  can be easily visualized. In the second half of this chapter we present concrete examples of semi-free circle actions on the $5$-sphere and verify Theorem \ref{Thm:S1SignatureThm} for such concrete cases. In particular, we see the different nature of the terms involved in this formula.\\

To motivate the techniques, we consider the following generic case. In view of the proof of Theorem \ref{Thm:S1SignatureThm}, it can be seen that locally one can always study the induced linear action as $k$ copies of the standard $S^1$ action on $\mathbb{C}$ plus the trivial representation. Consequently, this linear model is a rich source to construct such examples. Concretely, consider $M\coloneqq S^{2n-1} \subset \mathbb{C}^n$ with the $S^1$-action given, for $\lambda\in S^1\subset \mathbb{C}$, by 
\begin{align*}
\lambda\cdot(z_1, z_2,\dots,z_{k-1},z_k, z_{k+1}, \dots z_n)\coloneqq (\lambda z_1, \lambda z_2,\dots,\lambda z_{k-1},\lambda z_k, z_{k+1}, \dots z_n). 
\end{align*}
It is straightforward to see that the fixed point set is 
\begin{align*}
M^{S^1}=\{(0, 0,\dots,0,0, z_{k+1}, \dots z_n)\in S^{2n+1}\}=S^{2(n-k)-1}. 
\end{align*}
Locally, the quotient space can therefore modeled as a mapping cone of a $\mathbb{C}^{k-1}$-Riemannian fibration. In particular,  $S^{2n-1}/S^1$ is satisfies the Witt condition if, and only if, $k$ is even. \\
Some examples that we present in this section are specific instances of this general construction. 

\subsection{Low dimensional examples}

\subsubsection{$2$-Torus}
In this first example we will study the embedded $2$-torus $\BB{T}^2\subset\BB{R}^3$ regarded as a surface of revolution obtained by rotating the circle $(x-R)^2+z^2=r^2$ around the $z$-axis for $0<r<R$. This surface can be locally parametrized by the equations
\begin{align*}
x&=(R+r\cos u)\cos v,\\
y&=(R+r\cos u)\sin v,\\
z&=r\sin u,
\end{align*}
where $0<u,v<2\pi$. The metric induced on $\BB{T}^2$ by the Euclidean metric of $\BB{R}^3$ can be written in these coordinates as
\begin{equation}\label{Eqn:MetricTorus}
g^{T\mathbb{T}^2}=r^2du^2+(R+r\cos u)^2dv^2.
\end{equation}
If we choose the orientation with respect to the outward normal vector field on $\mathbb{T}^2$ then the associated volume element is $\vol_{\mathbb{T}^2}=r(R+r\cos u)du\wedge dv$. \\

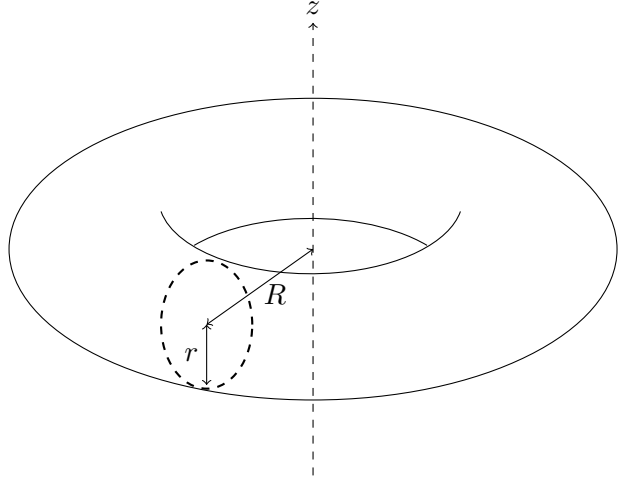
\begin{figure}[h]
\begin{center}
\begin{tikzpicture}
\draw (0,0) ellipse (4cm and 2cm);
\draw (-2,0.5) ellipsearc (190:350:2cm and 1cm);
\draw (1.5,0.05) ellipsearc (40:140:2cm and 1cm);
\draw [dashed,thick] (-1.4,-1) ellipse (0.6cm and 0.85cm);
\draw [<->](-1.4,-1)--(-1.4,-1.8);
\node at (-1.6,-1.4) {$r$};
\draw [<->](0,0)--(-1.4,-1);
\node at (-0.5,-0.6) {$R$};
\draw [dashed, ->](0,-3)--(0,3);
\node at (0,3.2) {$z$};
\end{tikzpicture}
\caption{The $2$-torus $\BB{T}\subset \mathbb{R}^3$ obtained by rotating the circle defined by the equation $(x-R)^2+z^2=r^2$, where $0<r<R$, around the $z$-axis.}
\end{center}
\end{figure}

We now consider the action of $S^1$ on $\mathbb{T}^2$ by rotations around the $z$-axis. Explicitly, with respect to the $(u,v)$-coordinates, we define $e^{it}(u,v)\coloneqq (u,v+t)$ for $e^{it}\in S^1$. This action is orientation and metric preserving. In spite of the fact that this action is free, it is still worth considering it in order to illustrate some results discussed in the previous chapters. Let us compute first the generating vector field $V$. For a smooth function $f\in C^\infty(\mathbb{T}^2)$ the action of $V$ on $f$ is by definition
\begin{align*}
V(f)(u,v)\coloneqq \frac{d}{dt}f(e^{it}(u,v))\bigg{|}_{t=0}=\frac{d}{dt}f(u,v+t)\bigg{|}_{t=0}=\partial_v f(u,v),
\end{align*}
therefore $V=\partial_v$. From \eqref{Eqn:MetricTorus} we see that the corresponding unit vector field is 
$$X=\frac{\partial_v}{R+r\cos u},$$
and the associated characteristic form is $\chi=X^\flat=(R+r\cos u)dv$. The mean curvature vector field is by definition
\begin{align*}
H=\nabla_X X=\frac{1}{(R+r\cos u)^2}\nabla_{\partial_v}\partial_v=\frac{\sin u \partial_u}{r(R+r\cos u)},
\end{align*}
where we have used $\nabla_{\partial_v}\partial_v=r^{-1}(R+r\cos u)\sin u\partial_u$ for the Levi-Civita connection, which can be easily derived from \eqref{Eqn:MetricTorus}. The mean curvature $1$-from is then
\begin{align*}
\kappa=\frac{r\sin u du}{R+r\cos u}.
\end{align*}
Note that $\kappa$ can be also computed from Proposition \ref{Prop:NormV}(2),
\begin{align}\label{Eqn:KappaT2}
\kappa=-d\log(\norm{V})=-d\log(R+r\cos u)=\frac{r\sin u du}{R+r\cos u}. 
\end{align}
Next we calculate the $2$-form $\varphi_0$. Observe from \eqref{Eqn:dchi},
\begin{align*}
d\chi=-r\sin u du\wedge dv=-\frac{r\sin u du}{R+r\cos u}\wedge (R+r\cos u)dv=-\kappa\wedge\chi.
\end{align*}
This implies that $\varphi_0=0$. This is of course not surprising since any top degree basic form associated to an effective action must vanish. \\

Let us now describe the volume of the orbit function $h$.  The volume of the orbit passing trough the point parametrized by the coordinates $(u,v)$ is the perimeter of a circle of radius $R+r\cos u$, i.e. $h(u,v)=2\pi(R+r\cos u)$. Note in particular, 
\begin{align*}
d(h^{-1/2})
=\frac{1}{2}(2\pi)^{-1/2}(R+r\cos u)^{-3/2}r\sin u du
=-\frac{1}{2}h^{-1/2}\kappa,
\end{align*}
which shows the idea behind the proof of  Lemma \ref{Lemmadh}.\\

The quotient space $\mathbb{T}^2/S^1$ of this free action can be identified with a circle $S^1(r)$ of radius $r>0$. If we equip this circle with the metric $g^{TS^1(r)}=r^2 du^2$  we easily see that the orbit map $\pi_{S_1}:\mathbb{T}^2\longrightarrow \mathbb{T}^2/S^1$ becomes a Riemannian submersion. \\

Let us illustrate Theorem \ref{Thm:OpInv} for the invariant form $\omega= f(u)du$. On the one hand 
\begin{align*}
(d_{\mathbb{T}^2}+d^\dagger_{\mathbb{T}^2})(f(u)du)=&-*d*f(u)du\\
=&-*d\left(\frac{(R+r\cos u)}{r}f(u)dv\right)\\
=&-*\left(\left(-\sin u f(u)+\frac{(R+r\cos u)}{r}f'(u)\right)du\wedge dv\right)\\
=&\frac{\sin u}{r(R+r\cos u)}f(u)-\frac{1}{r^2}f'(u),
\end{align*}
and on the other, if we regard $f(u)du$ as a $1$-form on $\mathbb{T}^2/S^1= S^1(r)$ we get (using the notation $\bar{*}\coloneqq *_{S^1(r)}$)
\begin{align*}
(d_{S^1(r)}+d^\dagger_{S^1(r)})(f(u)du)=&-\bar{*}d\bar{*}f(u)du
=-\bar{*}d\left(\frac{1}{r}f(u)\right)
=-\frac{1}{r^2}f'(u).
\end{align*}
Thus, in view of \eqref{Eqn:KappaT2}, we see how Theorem \ref{Thm:OpInv} is verified for $\omega$. \\

Finally, the operator $\mathscr{D}$ of Theorem \ref{Thm:InducedDiracOp} in this case is just
\begin{equation*}
\mathscr{D}'=\mathscr{D}=D_{S^1(r)}+\frac{1}{2}\left(\frac{r\sin u}{R+r\cos u}\right)c(du)\varepsilon, 
\end{equation*}
where $D_{S^1(r)}$ denotes the Hodge-de Rham operator of $S^1(r)$ with respect to the metric $d\bar{s}^2=r^2 du^2$. As the zero order term of $\mathscr{D}$ is bounded then, by Remark \ref{Rmk:GBInv} and the Gau\ss-Bonnet theorem, we obtain the index formula
\begin{align*}
\ind(\mathscr{D}^{\text{ev}})=\ind(D^{\text{ev}}_{S^1(r)})=\chi(S^1(r))=0.
\end{align*}

\subsubsection{Euclidean plane} 
Let us consider now the real plane $M=\mathbb{R}^2$ equipped with the usual orientation and with the Euclidean metric, which can be written in polar coordinates as $g^{T\mathbb{R}^2}= dr^2+r^2 d\theta^2$ for $r\geq 0$, $\theta\in [0,2\pi]$. Even tough this manifold is not compact, the Hodge-de Rham operator is essentially self-adjoint since the metric is complete.   We  study the action of $S^1$ on $\mathbb{R}^2$ by counter-clockwise rotations around the origin. This action is semi-free and the fixed point set consists only of the origin. Consequently, the principal orbit is $M_0=\mathbb{R}^2-\{(0,0)\}$. It is clear that $M_0/S^1$ can be identified with the open interval $\mathbb{R}_>\coloneqq (0,\infty)$ and the quotient metric is simply $g^{T\mathbb{R}_>}=dr^2$. In particular, note that $(\mathbb{R}_>, g^{T\mathbb{R}_>})$ is not complete.\\

We start by computing the main geometric quantities. The generating vector field of the action is $V=\partial_{\theta}$. Its associated unit vector field is $X=r^{-1}\partial_\theta$ and the corresponding characteristic form is $\chi=rd\theta$. The mean  curvature vector field is computed from the Levi-Civita connection as
\begin{align*}
H=\nabla_{X}X=\frac{1}{r^2}\nabla_{\partial_\theta}\partial_\theta=-\frac{1}{r}\partial_r,
\end{align*}
and the its associated mean curvature $1$-form is $\kappa=X^\flat=-dr/r$. In this case we have, as in the example above, 
\begin{align*}
\varphi_0=d\chi+\kappa\wedge\chi=dr\wedge d\theta+\left(-\frac{dr}{r}\right)\wedge rd\theta=0.
\end{align*}

These computations show that the operator $\mathscr{D}'$ of Theorem \ref{Thm:InducedDiracOp} is 
\begin{align*}
\mathscr{D}'=\mathscr{D}=D_{\mathbb{R}_>}-\frac{1}{2r}c(dr)\varepsilon,
\end{align*}
where $D_{\mathbb{R}_>}=d_{\mathbb{R}_>}+d^\dagger_{\mathbb{R}_>}$ is the Hodge-de Rham operator of $(\mathbb{R}_>, g^{T\mathbb{R}_>})$.
\begin{remark}[$D_{\mathbb{R}_>}$ is not essentially self-adjoint]
With respect to the decomposition of forms by degree ($0$-forms and $1$-forms) we can express
\begin{align*}
D_{\mathbb{R}_>}=
\left(
\begin{array}{cc}
0 & -\partial_r\\
\partial_r & 0
\end{array}
\right).
\end{align*}
Observe now the following facts:
\begin{itemize}
\item $e^{-r}\in L^2(\mathbb{R}_>)$.
\item The vector
\begin{align*}
\left(\begin{array}{c}
e^{-r}\\
i e^{-r}
\end{array}
\right)\in\Dom((D_{\mathbb{R}_>})_\text{max})
\end{align*}
\item The operator $D_{\mathbb{R}_>,\text{max}}$ has $i=\sqrt{-1}$ as an eigenvalue, i.e. 
\begin{align*}
\left(\begin{array}{cc}
 0&-\partial_r\\
\partial_r&0
\end{array}
\right)
\left(\begin{array}{c}
e^{-r}\\
i e^{-r}
\end{array}
\right)
=i\left(\begin{array}{c}
e^{-r}\\
i e^{-r}
\end{array}
\right).
\end{align*}
\end{itemize}
This shows that the deficiency indices are not zero and therefore $D_{\mathbb{R}_>}$ is not essentially self-adjoint when defined on 
the core $\Omega_c(\mathbb{R}_>)$. 
\end{remark}

With respect to degree decomposition, as in the remark above, we can express 
\begin{align}\label{Eqn:OpDPlane}
\mathscr{D}=
\left(\begin{array}{cc}
 0&-\partial_r-\frac{1}{2r}\\
\partial_r- \frac{1}{2r}&0
\end{array}
\right).
\end{align}
If we define the $2\times 2$ matrix
$$\gamma\coloneqq 
\left(
\begin{array}{cc}
0 & -1\\
1 & 0
\end{array}
\right),
$$
then we can write this operator as
\begin{align*}
\mathscr{D}=\gamma\left(\frac{\partial}{\partial r}+
\left(
\begin{array}{cc}
1 & 0\\
0 & -1
\end{array}
\right)
\otimes \left(-\frac{1}{2r}\right)\right).
\end{align*}
Here we explicitly see that we can write $\mathscr{D}$ as a regular singular first order differential operator in the sense of Br\"uning and Seeley (see Appendix \ref{App:RSO}). In particular, since the cone coefficient satisfies
\begin{align*}
\spec\left\{-\frac{1}{2}\left(
\begin{array}{cc}
1 & 0\\
0 & -1
\end{array}
\right)\right\}=\{\pm 1/2\}\notin (-1/2,1/2),
\end{align*}
we verify that $\mathscr{D}$ is indeed essentially self-adjoint by Theorem \ref{BS88Thm3.2} (we will go deeper into this argument in the next chapters). This result of course follows directly from \eqref{Diag:B} and Corollary \ref{Coro:BComplete}.
\begin{remark}[Deficiency indices]
For $\lambda\in\mathbb{C}$ and $\bar{\omega}=f_0(r)+f_1(r)dr$ consider the eigenvalue problem for the operator $\mathscr{D}$,
\begin{align*}
-f'_1-\frac{1}{2r}f_1&=\lambda f_0,\\
f'_0-\frac{1}{2r}f_0&=\lambda f_1.
\end{align*}
We can uncouple this system to obtain two independent second order differential equations
\begin{align}\label{Eqn:RD0}
f''_1+\left(\lambda^2-\frac{3}{4r^2}\right)f_1=0\quad \text{and}\quad f''_0+\left(\lambda^2+\frac{1}{4r^2}\right)f_0=0.
\end{align} 
In order to compute the deficiency indices we need to study these equations for $\lambda=\pm i$. From \cite[Lemma 4.2]{L93} it follows that in this case there exist no solutions for \eqref{Eqn:RD0}. Hence, the corresponding deficiency indices for $\mathscr{D}$ are zero. This also shows that the operator $\mathscr{D}$ is essentially self-adjoint by \cite[Proposition 3.7]{S12}.
\end{remark}

\begin{remark}
It is easy to verify that the functions $f_{\pm}(r)=r^{\pm 1/2}$ satisfy the relation 
$$f'_{\pm}(r)=\pm\frac{1}{2r} f_{\pm}(r),$$
and therefore $\mathscr{D}(f_{+}+f_{-}dr)=0$. However, this does not show that the kernel of $\mathscr{D}$ is zero since $f_{\pm}\notin L^2(\mathbb{R}_>)$. 
\end{remark}

To study the kernel of $\mathscr{D}$ we begin recalling a fundamental inequality.
\begin{lemma}[{Hardy's inequality, {\cite{H20}}}]\label{Lemma:Hardy}
Suppose $1<p<\infty$, $f\in L^p(\mathbb{R}_>)$ and 
\begin{align*}
F(x)=\frac{1}{x}\int_{0}^{x}f(r)dr. 
\end{align*}
Then,
\begin{align*}
\norm{F}_{L^p(\mathbb{R}_>)}\leq \left(\frac{p}{p-1}\right)\norm{f}_{L^p(\mathbb{R}_>)}. 
\end{align*}
Moreover, the equality holds if, and only if, $f=0$ almost everywhere. 
\end{lemma}

\begin{lemma}
The operator $\mathscr{D}^{\textnormal{ev}}$ satisfies $\ker(\mathscr{D}^{\textnormal{ev}})=\{0\}$. 
\end{lemma}

\begin{proof}
Let $f\in \ker(\mathscr{D}^{\textnormal{ev}})$. Since $C^\infty_c(\mathbb{R}_>)$ is a core for $\mathscr{D}^{\textnormal{ev}}$ then there exists a sequence $(f_n)_n\subset C^\infty_c(\mathbb{R}_>)$ such that $f_n\longrightarrow f$ and $\mathscr{D}^{\textnormal{ev}}f_n\longrightarrow 0$ in  $L^2(\mathbb{R}_>)$. Our aim is to estimate the $L^2$-norm of $\mathscr{D}^{\textnormal{ev}}f_n$. Observe that 
\begin{align*}
\left(f_n'-\frac{1}{2r}f_n\right)^2=& (f'_n)^2+\frac{1}{4r^2}f_n^2-\frac{1}{r}f_nf_n'.
\end{align*}
Using the relation 
\begin{align*}
\frac{d}{dr}\left(\frac{1}{2r}f_n^2\right)=-\frac{1}{2r^2}f_n^2+\frac{1}{r}f_nf_n',
\end{align*}
and using that $f_n$ has compact support we arrive to the expression
\begin{align*}
\norm{\mathscr{D}^{\textnormal{ev}}f_n}^2_{L^2(\mathbb{R}_>)}=&\int_0^\infty\left(f_n'-\frac{1}{2r}f_n\right)^2 dr=\int_0^\infty\left((f_n')^2-\frac{1}{4r^2}f_n^2\right)dr. 
\end{align*}
Finally, using Lemma \ref{Lemma:Hardy}, we conclude that $f_n\longrightarrow 0$. 
\end{proof}

\begin{coro}
We have $\ind(\mathscr{D}^\textnormal{ev})=0$. 
\end{coro}
\begin{proof}
First observe that $\mathscr{D}^{\textnormal{odd}}$ is Fredholm by Theorem \ref{BS88Thm3.1}. Hence, it suffices to show that $\ker(\mathscr{D}^{\textnormal{odd}})=0$. Using the notation of the previous proof and using Lemma \ref{Lemma:Hardy}, we can estimate for $f_n\in C^\infty(\mathbb{R}_>)$
\begin{align*}
\norm{\mathscr{D}^{\textnormal{odd}}f_n}^2_{L^2(\mathbb{R}_>)}=&\int_0^\infty\left(f_n'+\frac{1}{2r}f_n\right)^2 dr\\
=&\int_0^\infty\left((f_n')^2-\frac{1}{4r^2}f_n^2\right)dr+\int_{0}^\infty \frac{1}{r^2}f^2_n dr.
\end{align*}
If we assume that $\mathscr{D}^{\textnormal{odd}}f_n\longrightarrow 0$, then these two integrals should also converge to zero (as they are both positive). In particular, from the convergence to zero of the first integral we see, again from Lemma \ref{Lemma:Hardy}, that $f_n\longrightarrow 0$. 
\end{proof}

On the other hand, observe that the Euler characteristic  $\chi(\mathbb{R}_>)=0$ because $\mathbb{R_>}$ is contractible. Hence, we see that 
$\ind(\mathscr{D}^{\text{ev}})=\chi(\mathbb{R}_>)=0$. \\

We conclude this example by verifying Lemma \ref{Lemma:SquareD}. First, a straightforward computation shows that 
\begin{align*}
\mathscr{D}^2=
\left(
\begin{array}{cc}
-\partial^2_r -\frac{1}{4r^2} & 0\\
0 &-\partial^2_r +\frac{3}{4r^2} 
\end{array}
\right).
\end{align*}
On the other hand we compute, 
\begin{align*}
\nabla^{\mathbb{R}_>}_{\bar{\kappa}^\sharp}=&-\frac{1}{r}\nabla^{\mathbb{R}_>}_{\partial_r},\\
d^{\dagger}_{\mathbb{R}_>}(\bar{\kappa})=&-\frac{1}{r^2},\\
\norm{\bar{\kappa}^\sharp}^2=&\frac{1}{r^2}.
\end{align*}

Hence, the right hand side of Lemma \ref{Lemma:SquareD} applied to a smooth function $f_0$ is
\begin{align*}
&\left(\Delta_{\mathbb{R}_>}-\nabla_{\bar{\kappa}^\sharp}^{\mathbb{R}_>}\varepsilon-c(\bar{\kappa})D_{\mathbb{R}_>}\varepsilon+\frac{1}{2}d_{\mathbb{R}_>}^\dagger(\bar{\kappa})\varepsilon+\frac{1}{4}\norm{\bar{\kappa}^\sharp}^2\right)f_0\\
&=-f''_0+\frac{1}{r}f'_0-\frac{1}{r}\iota_{\partial_r}(f'_0 dr)-\frac{1}{2r^2}f_0+\frac{1}{4r^2}f_0\\
&=-f''_0-\frac{1}{4r^2}f_0,
\end{align*}
and similarly for a $1$-form $f_1dr$ we have
\begin{align*}
&\left(\Delta_{\mathbb{R}_>}-\nabla_{\bar{\kappa}^\sharp}^{\mathbb{R}_>}\varepsilon-c(\bar{\kappa})D_{\mathbb{R}_>}\varepsilon+\frac{1}{2}d_{\mathbb{R}_>}^\dagger(\bar{\kappa})\varepsilon+\frac{1}{4}\norm{\bar{\kappa}^\sharp}^2\right)f_1 dr\\
&=-f''_1dr-\frac{1}{r}f'_1dr+\frac{1}{r}dr\wedge(f'_1 )+\frac{1}{2r^2}f_1dr+\frac{1}{4r^2}f_1 dr\\
&=\left(-f''_1+\frac{3}{4r^2}f_1\right)dr.
\end{align*}
Thus, the desired statement is verified. In  particular observe that we can express
\begin{align*}
\mathscr{D}^2=\Delta_{\mathbb{R}_>}+\frac{1}{2r^2}\left(\frac{1}{2}-\varepsilon\right), 
\end{align*}
where $\Delta_{\mathbb{R}_>}=-\partial^2_r$ is the Laplacian. 

\subsubsection{$2$-sphere}
The aim of this example is to illustrate the whole procedure to obtain the operators of Theorem \ref{Thm:THat} and Theorem \ref{Thm:InducedDiracOp}. This is intended to get a better understatement of the theory. We are going to consider the semi-free circle action on the unit $2$-sphere $M=S^2\subset\mathbb{R}^3$ by rotations along the $z$-axis. The fixed point set is $M^{S^1}=\{\mathcal{N},\mathcal{S}\}$, where $\mathcal{N}$ and $\mathcal{S}$ denote the north and south pole respectively. On its complement $M_0=S^2-\{\mathcal{N},\mathcal{S}\}$ the action is free. We equip $S^2$ with the induced metric coming from the Euclidean inner product of $\mathbb{R}^3$. As this metric is rotational invariant we see that $S^1$ acts on $S^2$ by isometries. Let us consider the local parametrization of $S^2$ by spherical coordinates 
\begin{align*}
x&=\sin\theta\cos\phi,\\
y&=\sin\theta\sin\phi,\\
z&=\cos\theta,
\end{align*}
where $0<\theta<\pi$ and the  $0<\phi<2\pi$.  With respect to this parametrization the metric on $S^2$ takes the form  
\begin{equation}\label{Eq:metricS2}
g^{TS^2}=d\theta^2+\sin^2\theta d\phi^2.
\end{equation}
From this it follows that the induced inner products on $1$-forms are
\begin{align*}\label{Eqn:Inner1fS2}
\langle d\theta, d\theta\rangle =1, \quad\quad \langle d\theta, d\phi \rangle =0, \quad\quad \langle d\phi, d\phi \rangle =(\sin\theta)^{-2}.
\end{align*}
If we choose the orientation on $S^2$ using the outward normal vector then the associated Riemannian volume element is $\vol_{S^2}=\sin\theta d\theta\wedge d\phi$.\\

Now let us describe the $S^1$-action concretely. An element $e^{i\varphi}\in S^1$ acts on a point represented by the pair $(\theta,\phi)$ by $e^{i\varphi}(\theta,\phi)\longmapsto(\theta,\phi+\varphi). $ In particular we see, in view of \eqref{Eq:metricS2}, that $S^1$ acts on $S^2$ effectively by orientation preserving isometries. The quotient manifold $M_0/S^1$ can be identified with the open interval $I\coloneqq (0,\pi)$, which we equip with the flat metric $g^{TI}=d\theta^2$ so that the orbit map $\pi_{S^1}:M_0\longrightarrow M_0/S^1$ becomes a Riemannian submersion. \\

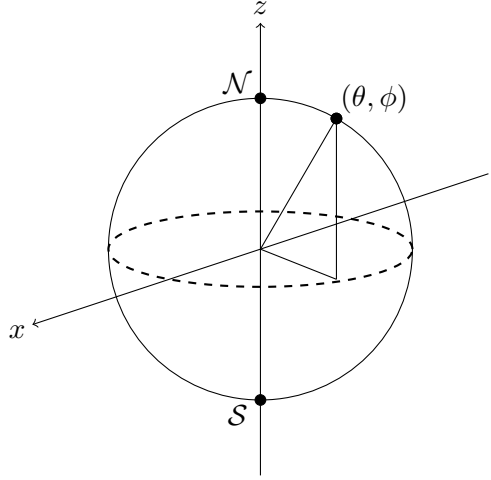
\begin{figure}[h]
\begin{center}
\begin{tikzpicture}
\draw (0,0) circle (2cm);
\draw[thick,dashed] (0,0) ellipse (2cm and 0.5cm);
\draw [<-](0,3)--(0,-3);
\draw[<-] (-3,-1)--(3,1);
\draw (0,0)--(1,1.732);
\draw (1,1.732)--(1,-0.4);
\draw (0,0)--(1,-0.4);
\node  at (1.5,2) {$(\theta,\phi)$};
\node at (1,1.732) {\pgfuseplotmark{*}};
\node at (-0.3,2.2) {$\mathcal{N}$};
\node at (0,2) {\pgfuseplotmark{*}};
\node at (-0.3,-2.2) {$\mathcal{S}$};
\node at (0,-2) {\pgfuseplotmark{*}};
\node at (1,1.732) {\pgfuseplotmark{*}};
\node at (-3.2,-1.1) {$x$};
\node at (0,3.2) {$z$};
\end{tikzpicture}
\caption{Local chart of $S^2$ defined by the polar angle $0<\theta<\pi$ and the azimuthal angle $0<\phi<2\pi$.}
\end{center}
\end{figure}

The generating vector field of the action is clearly $V=\partial_\phi$ and its associated unit vector field is 
\begin{equation*}\label{Eqn:GenVF}
X=\frac{1}{\sin\theta}\partial_\phi.
\end{equation*}
From \eqref{Eq:metricS2} we see that the corresponding characteristic form is $\chi=\sin\theta d\phi$. The mean curvature vector field is by definition
\begin{align*}
H=\frac{1}{\sin^2\theta}\nabla_{\partial_\phi}\partial_\phi=\frac{1}{\sin^2\theta}(-\sin\theta\cos\theta\partial_\theta)=-\cot\theta\partial_\theta,
\end{align*}
and its associated dual $1$-form is then $\kappa=-\cot\theta d\theta$. This $1$-form can be also computed using Proposition \ref{Prop:NormV}(2)
\begin{align*}
\kappa=-d\log(\norm{V})=-d\log(\sin\theta)=-\frac{\cos\theta d\theta}{\sin\theta}=-\cot\theta d\theta.
\end{align*}
Note in particular that it satisfies $d\kappa=0$ as expected. Again using \eqref{Eqn:dchi} we find that $\varphi_0=0$ since
\begin{align*}
d\chi+\kappa\wedge\chi=d(\sin\theta d\phi)-\cot\theta d\theta \wedge \sin\theta d\phi=0.
\end{align*}

Having computed the relevant geometric quantities we are going to start by verifying Theorem \ref{Thm:OpInv}. From Proposition \ref{Prop:SESBasic} we know that any $S^1$-invariant differential form on $M_0$ can be written as
\begin{equation}\label{Eqn:DecompInvS2}
\omega=\underbrace{(f_0(\theta)+f_1(\theta)d\theta)}_{\eqqcolon\omega_0}+\underbrace{(f_2(\theta)+f_3(\theta)d\theta))}_{
\eqqcolon\omega_1}\wedge\chi,
\end{equation}
where $\omega_0,\omega_1\in\Omega_\text{bas}(M_0)$.
First we compute the exterior derivative of such an $S^1$-invariant form
\begin{align*}
d(\omega_0+\omega_1\wedge\chi)=f'_0(\theta) d\theta +df_2(\theta)\wedge\chi+f_2(\theta)d\chi=f'_0(\theta)d\theta+(f'_2(\theta)+\cot\theta f_2(\theta))d\theta\wedge\chi.
\end{align*}
Similarly we calculate
\begin{align*}
d^\dagger(\omega_0+\omega_1\wedge\chi)=&-*d*(\omega_0+\omega_1\wedge\chi)\\
=&-*d(f_1(\theta)\chi-f_2(\theta)d\theta+f_3(\theta))\\
=&-*[(f'_1(\theta)+\cot\theta f_1(\theta))d\theta\wedge\chi+f'_3(\theta)d\theta]\\
=&-f'_1(\theta)-\cot\theta f_1(\theta)-f'_3(\theta)\chi.
\end{align*}
On the other hand the action of $D_I\coloneqq d_I+d_I^\dagger$, the Hodge-de Rham operator associated to $(I,d\theta^2)$, on a form $\bar{\omega}=g_0(\theta)+g_1(\theta)d\theta\in\Omega(I)$ is 
\begin{align*}
D_I(g_0(\theta)+g_1(\theta)d\theta)=-g'_1(\theta)+g'_0(\theta)d\theta.
\end{align*}
Here we have used $\bar{\star}(g_0+g_1d\theta)= i(-g_1+g_0 d\theta)$. Hence, from the decomposition \eqref{Eqn:DecompInvS2} we explicitly see that 
\begin{equation*}
S(D_{S^2})\coloneqq D_{S^2}\bigg{|}_{\Omega(M_0)^{S^1}}
=\left(
\begin{array}{cc}
d+\bar{\star}d\bar{\star}-\cot\theta\iota_{\partial_\theta} & 0\\
0 & d+\bar{\star}d\bar{\star}+\cot\theta d\theta\wedge
\end{array}
\right).
\end{equation*}
This illustrates the statement of Theorem \ref{Thm:OpInv} with $n=1$, $\kappa=-\cot\theta d\theta$ and $\varphi_0=0$. The induced operator $T(D_{S^2})$  on $\Omega_c(I)\otimes \mathbb{C}^2$ of Theorem \ref{Theorem:OpT} is then
\begin{equation*}
T(D_{S^2})\coloneqq
\left(
\begin{array}{cc}
D_I-\cot\theta\iota_{\partial_\theta} & 0\\
0 & D_I+\cot\theta d\theta\wedge
\end{array}
\right).
\end{equation*}
Now we will explicitly show that this operator is symmetric in $L^2(F,h)$. As we saw in Section \ref{Section:ConstrF}, the vector bundle $F\longrightarrow I$ is given by $F=\wedge_\mathbb{C}T^*I\oplus \wedge_\mathbb{C}T^*I$ and $\dom(T(D_{S^2}))=\Omega_c(I)\otimes\mathbb{C}^2$.  The volume of the orbit containing the point $(\theta,\phi)$ can be computed using  \eqref{Eqn:h}, 
\begin{align*}
h(\theta)=\int_{\pi^{-1}_{S^1}(\theta)} \chi =\int_0^{2\pi} \sin\theta d\phi=2\pi\sin\theta.
\end{align*}
This formula was expected since the orbits are circles of radius $\sin\theta$. As a consequence,  the $L^2(F,h)$-norm of a pair of differential forms 
with compact support $f_0(\theta)+f_1(\theta)d\theta$ and $f_2(\theta)+f_3(\theta)d\theta)$ is 
\begin{align*}
\norm{(f_0+f_1d\theta,f_2+f_3d\theta)}^2_{L^2(F,h)}=\int_0^\pi (|f_0|^2+|f_1|^2+|f_2|^2+|f_3|^2)2\pi\sin\theta d\theta.
\end{align*} 
Observe that
\begin{align*}
\frac{d}{d\theta}(f_0(\theta)f_1(\theta)\sin\theta )=f'_0(\theta)f_1(\theta)\sin\theta+f_0(\theta)f_1(\theta)\sin\theta+f_0(\theta)f_1(\theta)\cot\theta\sin\theta,
\end{align*}
so since $f_0,f_1$ have compact support in $I$ we obtain 
\begin{equation*}
\int_0^\pi f'(\theta)g(\theta)\sin\theta d\theta=-\int_0^\pi f(\theta)g'(\theta)\sin\theta d\theta-\int_0^\pi f(\theta)g(\theta)\cot\theta\sin\theta d\theta.
\end{equation*}
Using this relation we calculate the first component of the $L^2(F,h)$-inner product
\begin{align*}
((D_I-\cot\theta\iota_{\partial_\theta})(f_0+f_1 d\theta), & g_0+g_1 d\theta)_{L^{2}(\wedge T^*I,h)}\\
&=(-f'_1 -f_1\cot\theta+f'_0 d\theta,g_0+g_1 d\theta)_{L^2(\wedge T^*I,h)}\\
&=-2\pi \int_0^\pi( f'_1+f_1\cot\theta)g_0\sin\theta d\theta+2\pi \int_0^\pi f'_0g_1\sin\theta d\theta\\
&=2\pi \int_0^\pi f_1 g'_0\sin\theta d\theta-2\pi \int_0^\pi f_0( g'_1+g_1\cot\theta)\sin\theta d\theta\\
&=(f_0+f_1 d\theta,(D_I-\cot\theta\iota_{\partial_\theta})(g_0+g_1 d\theta))_{L^{2}(\wedge T^*I,h)}. 
\end{align*}
Similarly for the other component of $T(S^2)$,
\begin{align*}
((D_I+\cot\theta d\theta\wedge)(f_0+f_1 d\theta),&g_0+g_1 d\theta)_{L^{2}(\wedge T^*I,h)}\\
&=(-f'_1 +(f_0\cot\theta+f'_0 )d\theta,g_0+g_1 d\theta)_{L^2(\wedge T^*I,h)}\\
&=-2\pi \int_0^\pi  f'_1g_0\sin\theta d\theta+2\pi \int_0^\pi (f_0\cot\theta +f'_0)g_1\sin\theta d\theta\\
&=2\pi \int_0^\pi f_1(g_0\cot\theta+ g'_0)\sin\theta d\theta-2\pi \int_0^\pi f_0g'_1\sin\theta d\theta\\
&=(f_0+f_1 d\theta,(D_I+\cot\theta d\theta\wedge)(g_0+g_1 d\theta))_{L^{2}(\wedge T^*I,h)}.
\end{align*}
These relations illustrate that the operator $T(S^2)$ is indeed symmetric with respect to the $L^2(F,h)$-metric. \\

Now we want to conjugate this operator  with the unitary transformation \eqref{Eqn:U}, which in this example just given by
\begin{align*}
U:
\xymatrixcolsep{2cm}\xymatrixrowsep{0.01cm}\xymatrix{
\Omega_c(I)\ar[r] & \Omega_c(I)\\
\bar{\omega} \ar@{|->}[r]  & (2\pi\sin\theta)^{-1/2}\bar{\omega}.
}
\end{align*}
First we verify Lemma \ref{Lemmadh},
\begin{equation*}
\frac{d}{d\theta}(h^{-1/2})=\frac{d}{d\theta}(2\pi \sin \theta)^{-1/2}=-\frac{1}{2}(2\pi \sin \theta)^{-3/2}2\pi\cos\theta=-\frac{1}{2}h^{-1/2}\cot\theta.
\end{equation*}
Then we compute
\begin{align*}
(D_I-\cot\theta&\iota_{\partial_\theta})U(f_0+f_1 d\theta)\\
&=-(h^{-1/2}f_1)'-\cot\theta(h^{-1/2}f_1)+(h^{-1/2}f_0)'d\theta\\
&=\frac{1}{2}h^{-1/2}\cot\theta f_1 -h^{-1/2}f'_1-h^{-1/2}\cot\theta f_1-\frac{1}{2}h^{-1/2}f_0\cot\theta d\theta+h^{-1/2}f'_0 d\theta\\
&=h^{-1/2}\left((-f'_1+f'_0 d\theta)-\frac{1}{2}\cot\theta(f_1+f_0 d\theta)\right)\\
&=h^{-1/2}\left(D_I-\frac{1}{2}\cot\theta \widehat{c}(d\theta)\right)(f_0+f_1 d\theta). 
\end{align*}
A similar computation for $(D_I+\cot\theta d\theta)U(f_0+f_1 d\theta)$ shows that the operator from Theorem \ref{Thm:THat} takes the form
\begin{equation*}
\widehat{T}(D_{S^2})
\coloneqq
U^{-1}T(D_{S^2})U=
\left(
\begin{array}{cc}
D_I -\frac{1}{2}\cot\theta \widehat{c}(d\theta)& 0\\
 0&  D_I +\frac{1}{2}\cot\theta \widehat{c}(d\theta)
\end{array}
\right).
\end{equation*}
Finally note that the operator $\mathscr{D}'$ of Theorem \ref{Thm:InducedDiracOp}
is 
\begin{align*}
\mathscr{D}'=\mathscr{D}=D_I-\frac{1}{2}\cot\theta c(d\theta)\varepsilon.
\end{align*}
If we write a form $\omega_0=f_0(\theta)+f_1(\theta)d\theta\in\Omega_c(I)$ in column-vector notation as
\begin{align*}
\omega_0=
\left(
\begin{array}{c}
f_0\\
f_1
\end{array}
\right)
\end{align*}
then when can express $\mathscr{D}$ as
\begin{equation}\label{Eqn:D0S2}
\mathscr{D}
=
\left(
\begin{array}{cc}
0 & -\partial_\theta-\frac{1}{2}\cot\theta\\
\partial_\theta-\frac{1}{2}\cot\theta &0
\end{array}
\right).
\end{equation}
The zero order part of the operator $\mathscr{D}$ is proportional to $\pm \frac{1}{2}\cot\theta$. This potential blows up when $\theta\longrightarrow 0$ and $\theta\longrightarrow \pi$, i.e. when approaching the singular stratum. When $\theta\longrightarrow 0$ the Taylor expansion of $\cot\theta$ is 
\begin{align*}
\cot\theta\sim\frac{1}{\theta}-\frac{\theta}{3}+o(\theta),
\end{align*}
and so
\begin{align*}
\mathscr{D}\sim 
\left(
\begin{array}{cc}
0 & -\partial_\theta-\frac{1}{2\theta}+O(\theta)\\
\partial_\theta-\frac{1}{2\theta}+O(\theta)&0
\end{array}
\right). 
\end{align*} 
Note that, up to $O(\theta)$-terms, this operator coincides with \eqref{Eqn:OpDPlane} in the example discussed above (up to a sign this also holds for  $\theta\longrightarrow \pi$). Hence, in view of Theorem \ref{BS88Thm3.2}, we verify that the operator $\mathscr{D}$ with core $\Omega_c(I)$ is essentially self-adjoint. \\

We want to end this example by verifying Lemma \ref{Lemma:SquareD}. The purpose of this is to give a feeling of how concrete computations can be done. It is easy to see that
\begin{align*}
\nabla^I_{\bar{H}}(f_0+f_1d\theta)=&-f'_0\cot\theta-f'_1\cot\theta d\theta,\\
d_I^\dagger(-\cot\theta d\theta)=&-\csc^2\theta,\\
\norm{\bar{\kappa}^\sharp}^2=&\cot^2\theta.
\end{align*}
Hence, for a smooth function $f_0$ we have
\begin{align*}
 &\left(\Delta_{I}-\nabla_{\bar{\kappa}^\sharp}^{I}\varepsilon-c(\bar{\kappa})D_{I}\varepsilon +\frac{1}{2}d_{I}^\dagger(\bar{\kappa})\varepsilon+\frac{1}{4}\norm{\bar{\kappa}^\sharp}^2\right)f_0\\
&=-f''_0+f'_0\cot\theta-f'_0\cot\theta-\frac{1}{2}\csc^2\theta f_0+\frac{1}{4}f_0\cot^2{\theta}\\
&=-f''_0+\frac{1}{2}\left(\frac{1}{2}\cot^2{\theta}-\csc^2\theta\right)f_0,
\end{align*}
and for a $1$-form $f_1d\theta$ we compute similarly
\begin{align*}
&\left(\Delta_{I}-\nabla_{\bar{\kappa}^\sharp}^{I}\varepsilon-c(\bar{\kappa})D_{I}\varepsilon+\frac{1}{2}d_{I}^\dagger(\bar{\kappa})\varepsilon+\frac{1}{4}\norm{\bar{\kappa}^\sharp}^2\right)f_1d\theta\\
&=-f''_1d\theta-f'_1\cot\theta d\theta+\cot\theta d\theta\wedge (f'_1)+\frac{1}{2}\csc^2\theta f_1d\theta+\frac{1}{4}f_1\cot^2{\theta}d\theta\\
&=\left(-f''_1+\frac{1}{2}\left(\frac{1}{2}\cot^2{\theta}+\csc^2\theta\right)f_1\right) dr. 
\end{align*}

Therefore, with respect to the degree decomposition, we can write 
\begin{align*}
\mathscr{D}^2
=
\left(
\begin{array}{cc}
-\partial^2_\theta+\frac{1}{2}\left(\frac{1}{2}\cot^2{\theta}-\csc^2\theta\right) &0\\
0 & -\partial^2_\theta+\frac{1}{2}\left(\frac{1}{2}\cot^2{\theta}+\csc^2\theta\right) 
\end{array}
\right).
\end{align*}
Using the relation for the Laplacian $\Delta_I=-\partial^2_\theta$, we obtain
\begin{align*}
\mathscr{D}^2=\Delta_I+\frac{1}{2}\left(\frac{1}{2}\cot^2{\theta}-\csc^2\theta\varepsilon\right).
\end{align*}

\begin{remark}
The examples of the $2$-torus and the  $2$-sphere can be easily generalized to general surfaces of revolution. 
\end{remark}

\subsubsection{The Hopf fibration}\label{Sect:Hopf}

In this example we again consider a free action. However, the interesting feature is that the corresponding orbit map is not trivial so it is worth studying it. Moreover, it will serve as a preparation for a later example of a semi-free $S^1$-action on a $5$-dimensional closed manifold. Let us consider an action of $S^1=\{\lambda\in\mathbb{C}\: :\:|\lambda|=1\}\subset \mathbb{C}$ on 
$S^3=\{(z_0,z_1)\in\mathbb{C}^2\: : \: |z_0|^2+|z_1|^2=1\}\subset\mathbb{C}^2$,
 equipped with the induced metric, defined as follows: For $\lambda\in S^1$ and $(z_0,z_1)\in S^3$ set $\lambda (z_0,z_1)\coloneqq (\lambda z_0,\lambda z_1).$ It is clear that this action is orientation and metric preserving. In addition, we see that this action is free so the quotient space $S^3/S^1$ is a smooth manifold of dimension two. Actually we will see that we can identify $S^3/S^1\cong S^2$.  To begin with, we claim that the orbit map is given explicitly by the {\em Hopf map} (\cite[Section III.17]{BT82})
\begin{equation}\label{Def:HopfMap}
\pi_{S^1}:
\xymatrixcolsep{2cm}\xymatrixrowsep{0.01cm}\xymatrix{
S^3 \ar[r] & S^2\\
(z_0,z_1) \ar@{|->}[r] & (a,b):=(2z_0\bar{z}_1,|z_0|^2-|z_1|^2).
}
\end{equation}
First of all note that the condition $|z_0|^2+|z_1|^2=1$ implies
$$|a|^2+|b|^2=4|z_0|^2|z_1|^2+(|z_0|^2-|z_1|^2)^2=(|z_0|^2+|z_1|^2)^2=1,$$
which shows that so the map $\pi$ is well-defined. Now let us prove the claim. For en element $\lambda\in S^1$ we obviously have $\pi(\lambda z_0, \lambda z_1)=\pi(z_0,z_1)$. On the other hand, let us assume that $\pi_{S^1}(z_0,z_1)=\pi_{S^1}(w_0,w_1)$, we want to show that $(z_0,z_1)$ and $(w_0,w_1)$ belong to the same orbit, i.e. there exists $\lambda\in S^1$ such that $(z_0,z_1)=(\lambda w_0,\lambda w_1)$. Let us write each point in polar coordinates as $z_k=r_k e^{i\varphi_k}$ and $w_k=s_ke^{i\psi_k}$ with $r_k,s_k\geq 0$ for $k=0,1$. Then the following conditions must hold:
\begin{align*}
e^{i(\varphi_0-\varphi_1)}r_0r_1&=e^{i(\psi_0-\psi_1)}s_0s_1,\\
r^2_0-r^2_1&=s^2_0-s^2_1,\\
r^2_0+r^2_1&=s^2_0+s^2_1=1.
\end{align*}
The last two equations imply that $r_0=s_0$ and $r_1=s_1$ and thus, by the first equation, the claim follows.\\

Let us consider the now the local parametrization of $S^3\subset \BB{C}^2$ by 
\begin{align*}
z_0&=e^{i\xi_1}\cos\eta,\\
z_1&=e^{i\xi_2}\sin\eta,
\end{align*}
where $0<\xi_1,\xi_2<2\pi$ and $0<\eta<\pi/2$. With respect to these coordinates the induced metric from $\BB{C}^2$ on $S^3$
is
\begin{equation}\label{Eq:MetricHopf}
g^{TS^3}=\cos^2\eta d\xi^2_1+\sin^2\eta d\xi^2_2+d\eta^2.
\end{equation} 
The associated volume form, with respect to the orientation induced by outer normal vector field, is 
\begin{align}\label{VolFormS3}
\vol_{S^3}=-\cos\eta\sin\eta d\xi_1\wedge d\xi_2\wedge d\eta.
\end{align}

 It is straightforward to see that the Hopf map can be written in terms of these coordinates as 
\begin{equation}\label{Eqn:HopfMapHopfCoord}
\pi_{S^1}(e^{i\xi_1}\cos\eta,e^{i\xi_2}\sin\eta)=(e^{i(\xi_1-\xi_2)}\sin(2\eta), \cos(2\eta)),
\end{equation}
which shows that we can identify topologically $S^3/S^1$ with $S^2$. Observe that in these coordinates the action of an element  $e^{it}\in S^1$ is 
\begin{align*}
\xi_1&\longmapsto \xi_1+t,\\
\xi_2&\longmapsto \xi_2+t,\\
\eta &\longmapsto \eta. 
\end{align*}

As in previews examples, we start by computing $V$, $\chi$, $\kappa$ and $\varphi_0$. To find the generating vector field $V$ we compute its action on $f\in C^\infty(S^3)$, 
\begin{align*}
\frac{d}{dt}f(e^{it}(\xi_1,\xi_2,\eta))\bigg{|}_{t=0}=\frac{d}{dt}f(\xi_1+t,\xi_2+t,\eta)\bigg{|}_{t=0}=\partial_{\xi_1}f(\xi_1,\xi_2,\eta)+\partial_{\xi_2}f(\xi_1,\xi_2,\eta).
\end{align*}
Thus $X=V=\partial_{\xi_1}+\partial_{\xi_2}$, since $\norm{V}=1$. In particular Proposition \ref{Prop:NormV}(2) shows that $\kappa=0$, i.e. the $S^1$-fibers are totally geodesic. The corresponding characteristic form of the action is $\chi=\alpha=(\partial_{\xi_1}+\partial_{\xi_2})^\flat=\cos^2\eta d\xi_1+\sin^2\eta d\xi_2$. Finally using 
\begin{equation*}
d\chi=-2\sin\eta\cos\eta d\eta\wedge(d\xi_1-d\xi_2),
\end{equation*}
and  \eqref{Eqn:dchi} we find that
\begin{align}\label{Eqn:VarphiHopf}
\varphi_0=-2\sin\eta\cos\eta d\eta\wedge(d\xi_1-d\xi_2).
\end{align}
In particular $d\varphi_0=0$, which illustrates Proposition \ref{Prop:varphi0}(2) since $\kappa=0$. \\

Our next aim is to describe the quotient metric on  $S^2$ so that the Hopf map becomes a Riemannian submersion. With respect to the coordinates described above and motivated by \eqref{Eqn:HopfMapHopfCoord} we define the functions 
\begin{align*}
\theta(\xi_1,\xi_2,\eta)\coloneqq &2\eta,\\
\phi(\xi_1,\xi_2,\eta)\coloneqq &\xi_1-\xi_2,
\end{align*}
which parametrize the image of the Hopf map, i.e. $\pi(\xi_1,\xi_2,\eta)=(e^{i\phi}\sin\theta,\cos\theta)$. With respect to the local basis $\{\partial_{\xi_1},\partial_{\xi_2},\eta\}$ and $\{\partial_\theta,\partial_\phi\}$ the derivative of the Hopf map is 
\begin{align*}
d\pi_{S^1}(\xi_1,\xi_2,\eta)=\left(
\begin{array}{ccc}
0 & 0 & 2\\
1 & -1 & 0
\end{array}
\right).
\end{align*}
In particular, 
\begin{align*}
d\pi_{S^1}(\partial_{\xi_1}+\partial_{\xi_2})&=0,\\
d\pi_{S^1}(\partial_{\xi_1}-\partial_{\xi_2})&=2\partial_\phi,\\
d\pi_{S^1}(\partial_\eta)&=2\partial_\theta.
\end{align*}
The first of these relations is not surprising since $X=\partial_{\xi_1}+\partial_{\xi_2}$ is the generating vector field of the action. The  condition that characterizes the quotient metric $\inner{\cdot}{\cdot}_{S^2}$ on $S^2$ is that $\langle Y,Y\rangle_{S^3}=\langle d\pi_{S^1}(Y),d\pi_{S^1}(Y)\rangle_{S^2}$ for all vector fields $Y$ orthogonal to $\partial_{\xi_1}+\partial_{\xi_2}$, i.e. horizontal vector fields. For example, two linearly independent horizontal vector fields on $S^3$ are
\begin{align*}
\underline{e}_1\coloneqq \frac{\sin^2\eta\partial_{\xi_1}-\cos^2\eta\partial_{\xi_2}}{\sin\eta\cos\eta}\quad \text{and}\quad
 \underline{e}_2\coloneqq\partial_\eta. 
\end{align*}
Indeed, using \eqref{Eq:MetricHopf} is easy to verify the conditions 
\begin{itemize}
\item $\norm{\underline{e}_1}_{S^3}=\norm{\underline{e}_2}_{S^3}=1$.
\item $\inner{\underline{e}_1}{\underline{e}_2}_{S^3}=0$.
\item $\inner{\partial_{\xi_1}+\partial_{\xi_2}}{\underline{e}_1}_{S^3}=\inner{\partial_{\xi_1}+\partial_{\xi_2}}{\underline{e}_2}_{S^3}=0$.
\end{itemize}
The image of these vector fields under $d\pi_{S^1}$ is 
\begin{align*}
d\pi_{S^1}(\underline{e}_1)=\frac{2}{\sin\theta}{\partial_\phi}\quad \text{and}\quad d\pi_{S^1}(\underline{e}_1)=2\partial_\theta.
\end{align*}
Thus, if we want the Hopf map $\pi_{S^1}$ to be a Riemannian submersion we need to equip $S^2$ with the Riemannian metric 
\begin{equation}\label{Eqn:MetricFSS2}
g^{TS^2(1/2)}\coloneqq\frac{1}{4}d\theta^2+\frac{1}{4}\sin^2\theta d\phi^2.
\end{equation}
Hence, geometrically, we see that the image of the Hopf map is a $2$-sphere of radius $1/2$, denoted by $S^2(1/2)$, with the usual round metric (see \eqref{Eq:metricS2}). As a matter  of fact, the metric \eqref{Eqn:MetricFSS2} is precisely the Fubini-Study metric
under the identification $S^2\cong\mathbb{C}P^1$. 

\begin{remark}\label{Rmk:BasicFormsS3}
The dual forms of $\underline{e}_1$ and $\underline{e}_2$ are
\begin{align*}
\underline{e}^1=\sin\eta\cos\eta(d\xi_1-d\xi_2)\quad \text{and}\quad \underline{e}^2=d\eta,
\end{align*}
which are basic $1$-forms. It is easy to check that 
\begin{align*}
\underline{e}^1=\pi^*_{S^1}\left(\frac{1}{2}\sin\theta d\phi \right)\quad\text{and}\quad  \underline{e}^2=\pi^*_{S^1}\left(\frac{1}{2}d\theta\right).
\end{align*}
Observe in particular that the $1$-forms
\begin{align*}
e^\phi\coloneqq \frac{1}{2}\sin\theta d\phi \quad\text{and}\quad e^\theta\coloneqq \frac{1}{2}d\theta
\end{align*}
form a local orthonormal basis for $T^*S^2(1/2)$.
\end{remark}

\begin{remark}
Let us see how obtain $\varphi_0$ directly by computing it as a curvature form (Proposition \ref{Prop:varphi0}(1)). First we compute the commutator of the two horizontal vector fields $\sin^2\eta\partial_{\xi_1}-\cos^2\eta\partial_{\xi_2}$ and $\partial_\eta$,
\begin{equation*}
[\partial_\eta,\sin^2\eta\partial_{\xi_1}-\cos^2\eta\partial_{\xi_2}]=2\sin\eta\cos\eta(\partial_{\xi_1}+\partial_{\xi_2}).
\end{equation*}
Next we apply the characteristic form,
\begin{align*}
\chi([\partial_\eta,\sin^2\eta\partial_{\xi_1}-\cos^2\eta\partial_{\xi_2}])=&\inner{\partial_{\xi_1}+\partial_{\xi_2}}{[\partial_\eta,\sin^2\eta\partial_{\xi_1}-\cos^2\eta\partial_{\xi_2}]}\\
=&\inner{\partial_{\xi_1}+\partial_{\xi_2}}{2\sin\eta\cos\eta(\partial_{\xi_1}+\partial_{\xi_2})}\\
=&2\sin\eta\cos\eta.
\end{align*}
Finally observe
\begin{align*}
2\sin\eta\cos\eta d\eta\wedge(d\xi_1-d\xi_2)(\partial_\eta,\sin^2\eta\partial_{\xi_1}-\cos^2\eta\partial_{\xi_2})=2\sin\eta\cos\eta.
\end{align*}
This shows that $\varphi_0=-2\sin\eta\cos\eta d\eta\wedge(d\xi_1-d\xi_2)$ as expected.
\end{remark}

In order to obtain the operator $\mathscr{D}'$ of Theorem \ref{Thm:InducedDiracOp} we need to compute the form $\bar{\varphi}_0\in\Omega(S^2)$ satisfying the condition $\varphi_0=\pi_{S^1}^*(\bar{\varphi}_0)$. In view of Remark \ref{Rmk:BasicFormsS3}, we easily see
\begin{align}\label{Eqn:PullbackVarphiHopf}
\pi_{S^1}^*\left(-\frac{1}{2}\sin\theta d\theta\wedge d\phi\right)=-\frac{1}{2}\sin 2\eta(2d\eta)\wedge(d\xi_1-d\xi_2)=\varphi_0,
\end{align}
thus $\bar{\varphi}_0=-(1/2)\sin\theta d\theta\wedge d\phi$. Note however that, by the dimensional constraint, in this case the action of the operator $\widehat{c}(\bar{\varphi}_0)$ on $1$-forms is zero so $\mathscr{D}'=D_{S^2}$.

\begin{remark}[Euler class]
If we integrate the $1$-form $\alpha=\cos^2\eta d\xi_1+\sin^2d\xi_2$ over each orbit we get
\begin{align*}
\int_{S^1x}\alpha =\int_0^{2\pi}\int_0^{2\pi}\cos^2\eta d\xi_1+\sin^2d\xi_2 =2\pi,
\end{align*}
for all $x\in S^3$. Hence, we can use $\alpha$ to construct the Euler class of the Hopf map following the procedure of \cite[Section 6.2(d)]{M00}. Recall from the discussion above that 
\begin{align*}
d\alpha=d\chi=\varphi_0=\pi^*_{S^1}(\bar{\varphi_0}),
\end{align*}
where $\bar{\varphi}_0=-(1/2)\sin\theta d\theta\wedge d\phi$. The Euler class of $\pi_{S^1}$ is then (\cite[Definition 6.27]{M00}), 
\begin{align*}
e(\pi_{S^1})=-\left[\frac{\bar{\varphi}_0}{2\pi}\right]\in H^2(S^2). 
\end{align*}
Finally we want to compute the Euler number integrating this class over the $S^2$ with respect to the induced orientation. Form \eqref{Eqn:TransHosgeStar1} we first compute
\begin{align*}
\bar{*}\varphi_0=&*(\varphi_0\wedge \chi)\\
=&-*(2\sin\eta\cos\eta d\eta\wedge(d\xi_1-d\xi_2)\wedge(\cos^2\eta d\xi_1+\sin^2d\xi_2))\\
=&-2*(\sin\eta\cos\eta d\xi_1\wedge d\xi_2\wedge d\xi_3)\\
=&2,
\end{align*}
where we have used \eqref{VolFormS3}. This means, in view of Lemma \ref{Lemma:CommStar}, that the induced orientation on $S^2$ is such that 
\begin{align*}
\int_{S^2(1/2)}\left(-\frac{1}{4}\sin\theta d \theta\wedge d\phi\right)=\vol(S^2(1/2))=4\pi\left(\frac{1}{2}\right)^2=\pi.
\end{align*}
Here we have used the fact that the quotient $2$-sphere has radius $1/2$.  In particular, this shows that the Euler number of the Hopf map is 
\begin{align*}
\int_{S^2(1/2)}e(\pi_{S^1})=\int_{S^2(1/2)}\frac{1}{4\pi}\sin\theta d\theta\wedge d\phi=-1,
\end{align*}
which agrees with \cite[Example 6.29]{M00}.
\end{remark}

\begin{remark}\label{Rmk:SpectrumLaplFSS2}
Form \cite[Theorem 4.2]{IT78} and Remark \ref{Rmk:ScaleHodgeStar} we see that the eigenvalues of the Laplacian corresponding to the metric \eqref{Eqn:MetricFSS2} are of the form $4k(k+1)$ for $k\in \mathbb{N}_0$. 
\end{remark}

\subsection{Example: The 5-sphere}\label{Sect:S5}

Motivated by Section \ref{Sect:Hopf} we study in this section two semi-free $S^1$-actions on the $5$-sphere. We describe explicitly the geometric quantities discussed in Section \ref{Sect:MeanCurvFrom} in order to compute the zero order terms of the operator $\mathscr{D}'$ of Theorem  \ref{Thm:InducedDiracOp}. In addition, we verify Theorem \ref{Thm:S1SignatureThm} and Theorem \ref{Thm:S1EC}. For the first theorem we compute the $L$-polynomial of the quotient metric explicitly. These computations illustrate the ideas behind the proof discussed above in Section \ref{Sect:PfoofSignatureFormula}. We then compute  $\sigma_{S^1}(M)$ as the signature of the induced pairing in intersection homology (Corollary \ref{Coro:IH}). For the second theorem we also compute the Euler form explicitly from the expression of the curvature and then we verify that its integral coincides with the $\chi_{S^1}(M)$, computed as the relative Euler characteristic of the quotient space.  \\

Let us begin with a geometric description of the $5$-sphere.  As a submanifold of $\mathbb{C}^3$ it is described by the condition
$$M= S^5=\{(z_0,z_1,z_2)\in \mathbb{C}^3\: : \:|z_0|^2+|z_1|^2+|z_2|^2=1\}\subset\mathbb{C}^3.$$
We equip $S^5$ with the induced Riemannian metric from $\mathbb{C}^3$ and with the orientation induced by the outer normal vector field. Consider the following local parametrization,
\begin{align}\label{Eqn:CoordS5}
z_0&=e^{i\xi_1}\cos\eta\cos\beta, \notag\\\
z_1&=e^{i\xi_2}\sin\eta\cos\beta,\\
z_2&=e^{i\xi_3}\sin\beta,\notag
\end{align}
where $0<\xi_1,\xi_2, \xi_3<2\pi$ and $0<\eta,\beta<\pi/2$. With respect to these coordinates the induced metric is 
\begin{equation}\label{Eqn:MetricS5}
g^{TS^5}=\cos^2\beta(\cos^2\eta d\xi_1+\sin^2\eta d\xi_2+d\eta^2)+\sin^2\beta d\xi_3+d\beta^2.
\end{equation}

\subsubsection*{Codim $M^{S^1}$=4}
The first semi-free $S^1$-action that we are going to study is defined as
$$\lambda (z_0,z_1,z_2)\coloneqq (\lambda z_0,\lambda z_1, z_2)\quad\text{for} \: \lambda\in S^1.$$
It is easy to see that this action preserves the metric and the orientation. The fixed point set of the action is 
$$M^{S^1}=\{(0,0,z)\in\mathbb{C}^3\: : \:|z|=1\}\cong S^1,$$
and the principal orbit is
$$M_0=\{(z_0,z_1,z_2)\in S^5 \: : |z_2|<1\}.$$
Inspired in the description of the Hopf map \eqref{Eqn:HopfMapHopfCoord} one verifies that the orbit map is
\begin{equation*}
\pi_{S^1}:
\xymatrixcolsep{2cm}\xymatrixrowsep{0.01cm}\xymatrix{
M_0\subset \mathbb{C}^3 \ar[r] & M_0/S^1\subset \mathbb{R}\times \mathbb{C}^2\\
(z_0,z_1,z_2) \ar@{|->}[r] & (t,w,z)\coloneqq (|z_0|^2-|z_1|^2,2z_0\bar{z}_1,z_2).
}
\end{equation*}
With respect to the coordinates \eqref{Eqn:CoordS5} this map takes the form 
\begin{equation}\label{Eqn:QuotMapHopfCoordS5}
\pi_{S^1}(\xi_1,\xi_2,\xi_3,\eta, \beta)=(\cos^2\beta \cos(2\eta),\cos^2\beta e^{i(\xi_1-\xi_2)}\sin(2\eta),e^{i\xi_3}\sin\beta).
\end{equation}

Similarly, as we did before, we can define two auxiliary functions $\theta=\theta(\xi_1,\xi_2,\xi_3,\eta, \beta)$ and $\phi=\phi(\xi_1,\xi_2,\xi_3,\eta, \beta)$  by the relations
\begin{align*}
\theta \coloneqq &2\eta,\\
\phi \coloneqq &\xi_1-\xi_1,
\end{align*}
so that
$$\pi_{S^1}(\xi_1,\xi_2,\xi_3,\eta, \beta)=(t(\theta,\phi, \xi_3,\beta),w(\theta,\phi, \xi_3,\beta),z(\theta,\phi, \xi_3,\beta)),$$
where 
\begin{align*}
t\coloneqq &\cos^2\beta \cos\theta,\\
w\coloneqq &\cos^2\beta e^{i\phi}\sin\theta,\\
z\coloneqq &e^{i\xi_3}\sin\beta.
\end{align*}

We begin by providing a concrete description of the singular space $M/S^1$. The key observation is to realize that the components of the orbit map satisfy the relation 
\begin{align*}
z^2+\sqrt{t^2+w^2}=\sin^2\beta+\sqrt{\cos^4\beta\cos^2\theta+\cos^4\beta\sin^2\theta}=1.
\end{align*}
Hence, we see that we can explicitly describe the quotient space $M/S^1$ as the zero locus
\begin{align*}
M/S^1=\{(t,w,z)\in\mathbb{R}\times\mathbb{C}^2\:|\:z^2+\sqrt{t^2+w^2}-1=0\}.
\end{align*}
In Figure \ref{Fig:S51} we show a plot of the points satisfying the equation $z^2+\sqrt{t^2+w^2}-1=0$, where we take $t,w,z\in\mathbb{R}$ for the sake of visualization. From this model we see that there are two singular points, which actually represent the singular stratum $M^{S^1}=S^1$ (the two points in the plot are just the $0$-sphere). 
\begin{figure}[h]
\begin{center}
\includegraphics[scale=0.4]{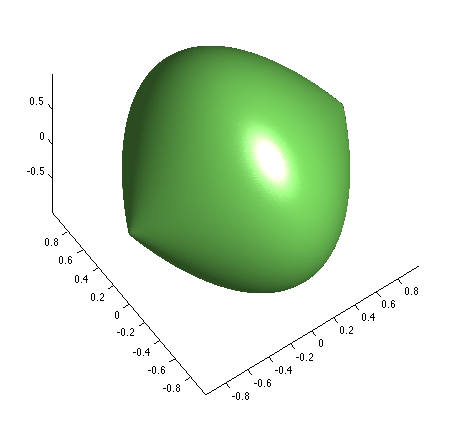} 
\caption{Plot of $z^2+\sqrt{t^2+w^2}-1=0$ for $t,z,w\in \mathbb{R}$.}\label{Fig:S51}
\end{center}
\end{figure}

Let us decompose the complex variables $w$ and $z$ into their real and imaginary parts
\begin{align*}
w=&w_1+iw_2,\\
z=&z_3+iz_4.
\end{align*}
Consider now the function $f:\mathbb{R}^5\longrightarrow \mathbb{R}$ defined by 
\begin{align*}
f(t,w_1,w_2,z_3,z_4)\coloneqq z_3^2+z_4^2+\sqrt{t^2+w_1^2+w_2^2}-1,
\end{align*}
so that $M/S^1=f^{-1}(0)$. If we compute the gradient of this function we get 
\begin{align*}
\nabla f=2z_3\partial_{z_3}+2 z_4 \partial_{z_4}+\frac{t\partial_t +w_1\partial_{w_1}+w_3\partial_{w_3}}{\sqrt{t^2+w_1^2+w_2^2}},
\end{align*}
and immediately see that it is not well defined in the limit $t+w_1+w_2\longrightarrow 0$, i.e. when approaching the fixed point set.  
\begin{remark}
Observe that a point $(t,w_1,w_2,z_3,z_4)\in f^{-1}(0)=M/S^1$ satisfies
\begin{align*}
t^2+w_1^2+w_2^2=(1-(z_3^2+z_4^2))^2.
\end{align*}
This shows that $M/S^1$ is a zero locus of a real polynomial of degree $4$.
\end{remark}

Our next aim is to compute the induced quotient metric on $M_0/S^1$ following the same strategy as for the Hopf fibration. Using the variables $\{\theta,\phi,\xi_3,\beta\}$ we can write the orbit map as
$$\pi_{S^1}(\xi_1,\xi_2,\xi_3,\eta, \beta)=(2\eta, \xi_1-\xi_2,\xi_3,\beta).$$ 
The derivative of $\pi_{S^1}$ with respect to the local basis $\{\partial_{\xi_1},\partial_{\xi_2},\partial_{\xi_2},\partial_{\eta},\partial_{\beta}\}$ and $\{\partial_{\theta},\partial_{\phi},\partial_{\xi_3},\partial_{\beta}\}$ is therefore
\begin{align*}
d\pi_{S^1}=\left(
\begin{array}{ccccc}
0 & 0 & 0 & 2 & 0\\
1 & -1 & 0 & 0 & 0\\
 0 & 0 & 1  & 0 & 0\\
 0 & 0 &  0 & 0 & 1
\end{array}
\right).
\end{align*}
In particular we obtain the relations
\begin{align}\label{Eqn:HorVFS51}
d\pi_{S^1}(\partial_{\xi_1}+\partial_{\xi_2})=&0,\notag\\
d\pi_{S^1}(\partial_{\xi_1}-\partial_{\xi_2})=&2\partial_{\phi},\notag\\
d\pi_{S^1}(\partial_{\beta})=&\partial_{\beta},\\
d\pi_{S^1}(\partial_{\xi_3})=&\partial_{\xi_3}, \notag \\
d\pi_{S^1}(\partial_{\eta})=&2\partial_{\theta}.\notag
\end{align}
Here we have identified the variables $\xi_3$ and $\beta$ both in $M_0$ and in $M_0/S^1$. Note from \eqref{Eqn:MetricS5} and \eqref{Eqn:HorVFS51} that the vector fields
\begin{align}\label{Eqns:ONBS5}
\widehat{\underline{e}}_1\coloneqq &\frac{\sin^2\eta\partial_{\xi_1}-\cos^2\eta\partial_{\xi_2}}{\cos\beta\sin\eta\cos\eta},\notag\\
\widehat{\underline{e}}_2\coloneqq &\frac{\partial_{\xi_3}}{\sin\beta},\\
\widehat{\underline{e}}_3\coloneqq &\frac{\partial_\eta}{\cos\beta},\notag\\
\widehat{\underline{e}}_4 \coloneqq &\partial_\beta\notag .
\end{align}
satisfy the relations $\inner{\widehat{\underline{e}}_i}{\widehat{\underline{e}}_j}_{S^5}=\delta_{ij}$ and $\inner{\partial_{\xi_1}+\partial_{\xi_2}}{\widehat{\underline{e}}_i}_{S^5}=0$ for $i,j=1,2,3,4$. Observe from \eqref{Eqn:CoordS5} that $V\coloneqq \partial_{\xi_1}+\partial_{\xi_2}$ is the generating vector field of the action, hence $\{\widehat{\underline{e}}_1,\widehat{\underline{e}}_2,\widehat{\underline{e}}_3,\widehat{\underline{e}}_4\}$ are mutually orthogonal unit horizontal vector fields on $S^5$.
\begin{align*}
d\pi_{S^1}(\widehat{\underline{e}}_1)=&\frac{2\partial_{\phi}}{\cos\beta\sin\theta},\\
d\pi_{S^1}(\widehat{\underline{e}}_2)=&\frac{\partial_{\xi_3}}{\sin\beta},\\
d\pi_{S^1}(\widehat{\underline{e}}_3)=&\frac{2\partial_\theta}{\cos\beta},\\
 d\pi_{S^1} (\widehat{\underline{e}}_4)=&\partial_\beta.
\end{align*}
From these relations we deduce the quotient metric on $M_0/S^1$ must be (for $\xi\coloneqq \xi_3$),
\begin{align}\label{Eqn:MetricSigma}
g^{T(M_0/S^1)}=\cos^2\beta\left(\frac{1}{4}d\theta^2+\frac{1}{4}\sin^2\theta d\phi^2\right)+\sin^2\beta d\xi^2+d\beta^2.
\end{align}

\begin{figure}[h]
\begin{center}
\includegraphics[scale=0.5]{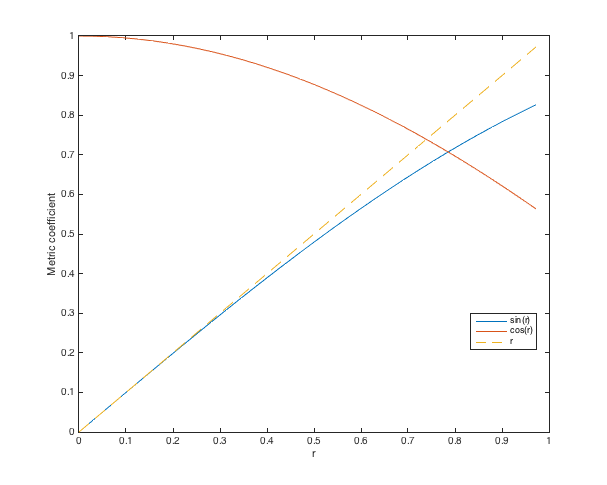} 
\caption{Plot of the functions $\sin(r), \cos(r)$ and $r$ for $0<r<1$.}\label{Fig:ApproxMetric}
\end{center}
\end{figure}

Note that we approach to the fixed point set as the coordinate $\beta\longrightarrow \pi/2$ or equivalently as $r\coloneqq \pi/2-\beta\longrightarrow 0^+$. Since $\sin \beta=\cos r$ and $\cos\beta=\sin r$, then the metric \eqref{Eqn:MetricSigma} close to the fixed point set $M^{S^1}$ can be approximated by the wedge-metric
\begin{align}\label{Eqn:CloseFixPtS5}
r^2\left(\frac{1}{4}d\theta^2+\frac{1}{4}\sin^2\theta d\phi^2\right)+ d\xi_3^2+d r^2,
\end{align}
for $0<r\ll 1$ (see Figure \ref{Fig:ApproxMetric}). This shows that close to $M^{S^1}$ the orbit space is isometric to a cone bundle over $M^{S^1}$ with link $S^2$ (see Figure \ref{Fig:S5CloseFI}).

\begin{remark}\label{Rmk:S5/S1=S4}
Topologically, a cone over a sphere is homeomorphic to a disc, so we see that $M/S^1$ is in fact homeomorphic to a $4$-dimensional closed manifold. 
\end{remark}

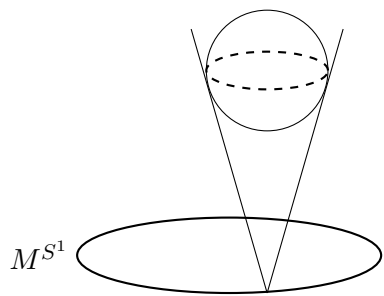
\begin{figure}[h]
\begin{center}
\begin{tikzpicture}
\draw[thick] (0,0) ellipse (2cm and 0.5cm);
\draw (0.5,-0.5)--(-0.5,3);
\draw (0.5,-0.5)--(1.5,3);
\draw (0.5,2.45) circle (0.8cm);
\draw[thick, dashed] (0.5,2.45) ellipse (0.8cm and 0.25cm);
\node at (-2.5,0) {$M^{S^1}$};
\end{tikzpicture}
\caption{Description of $M/S^1$ around the fixed point set $M^{S^1}$.}\label{Fig:S5CloseFI}
\end{center}
\end{figure}

With respect to this local description of the fixed point set we find that $M/S^1$ is a Witt space since $S^2\cong \mathbb{C}P^1$ (see Definition \ref{Def:Witt}). 

\begin{remark}
To visualize better the behavior close to the singular stratum we plot in Figure \ref{Fig:S52} the level surface  $z_3^2+ z_4^2+t-1=0$ in $\mathbb{R}^3$. In this figure we can see the fixed point set $M^{S^1}=S^1$ since we have fixed the condition $w^2_1+w^2_2=0$ which collapses the link to a point. 
\begin{figure}[h]
\begin{center}
\includegraphics[scale=0.5]{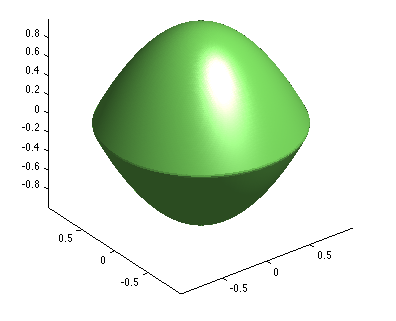} 
\caption{Plot of $z_3^2+ z_4^2+t-1=0$ for $t,z_3,z_4\in \mathbb{R}$.}\label{Fig:S52}
\end{center}
\end{figure}
\end{remark}

Now that we have a better picture of the quotient space $M/S^1$ we continue towards the description of the operator $\mathscr{D}'$ from Theorem \ref{Thm:InducedDiracOp}.  Recall that $V\coloneqq \partial_{\xi_1}+\partial_{\xi_2}$ is the generating vector field of the action. The corresponding unit vector field and characteristic $1$-form are 
$$X=\frac{1}{\cos\beta}(\partial_{\xi_1}+\partial_{\xi_2})\quad\text{and}\quad\chi=\cos\beta(\cos^2\eta d\xi_1+\sin^2\eta d\xi_2),$$
respectively. As in the previous examples, we can calculate the mean curvature $\kappa$  on $S^5$ from the norm of the vector field $V$ (Proposition \ref{Prop:NormV}(2)),
\begin{align*}
\kappa=-d\log(\norm{V})=-d\log(\cos\beta)=\tan\beta d\beta.
\end{align*}
Using the musical isomorphism we obtain the corresponding mean curvature vector field $H=\kappa^\sharp =\tan\beta \partial_\beta$.

\begin{remark}
Let us compute the volume of the orbit containing a point parametrized by the coordinates $(\xi_1,\xi_2,x_3\eta,\beta)$ using \eqref{Eqn:h},
\begin{align*}
h(\xi_1,\xi_2,x_3\eta,\beta)=\int_0^{2\pi}\int_0^{2\pi}\cos\beta(\cos^2\eta d\xi_1+\sin^2\eta d\xi_2)=2\pi\cos\beta. 
\end{align*}
Then $dh=-2\pi\sin\beta=-h\tan\beta=-h\kappa$, which verifies Lemma \ref{Lemma:dh}. 
\end{remark}

We now compute the form $\varphi_0$ using \eqref{Eqn:dchi}, 
\begin{align*}
d\chi=& d(\cos\beta)\wedge(\cos^2\eta d\xi_2+\sin^2d\xi_2)+ \cos\beta d(\cos^2\eta d\xi_2+\sin^2d\xi_2)\\
=&-\tan\beta d\beta\wedge \cos\beta(\cos^2\eta d\xi_2+\sin^2d\xi_2)-2\cos\beta\sin\eta\cos\eta d\eta\wedge(d\xi_1-d\xi_2)\\
=&-\kappa\wedge\chi-\cos\beta \sin(2\eta)d\eta\wedge(d\xi_1-d\xi_2),
\end{align*}
thus, $\varphi_0=-\cos\beta \sin(2\eta)d\eta\wedge(d\xi_1-d\xi_2)$. In view of the analogous expression for  $\varphi_0$ in the Hopf fibration case (see \eqref{Eqn:VarphiHopf}) and \eqref{Eqn:PullbackVarphiHopf}, we verify
\begin{align*}
\pi^*_{S^1}\left(-\frac{1}{2}\cos\beta\sin\theta d\theta\wedge d\phi\right)=-\frac{1}{2}\cos\beta\sin(2\eta)d(2\eta)\wedge d(\xi_1-\xi_2)=\varphi_0.
\end{align*}
Hence, the operator $\mathscr{D}'$ is given explicitly by
\begin{align*}
\mathscr{D}'=D_{M_0/S^1}+\frac{1}{2}\tan\beta c(d\beta)\varepsilon+\frac{1}{2}\:\widehat{c}(\cos\beta\sin\theta d\theta\wedge d\phi)\left(\frac{1-\varepsilon}{2}\right). 
\end{align*}
Here $D_{M_0/S^1}$ is the Hodge-de Rham operator  with respect to the metric \eqref{Eqn:MetricSigma}. \\

Recall from \eqref{Eqn:CloseFixPtS5} that as $r=\pi/2-\beta\longrightarrow 0^+$ the zero order part of $\mathscr{D}'$ becomes
\begin{align}\label{Eqn:PotS5}
-\frac{1}{2r}c(dr)\varepsilon-r\:\widehat{c}\left(-\frac{1}{2}\sin\theta d\theta\wedge d\phi\right)\left(\frac{1-\varepsilon}{2}\right).
\end{align}
Observe in particular from \eqref{Eqn:PotS5} that the term containing $\widehat{c}(\bar{\varphi}_0)$ remains bounded. Thus, using the Kato-Rellich Theorem (\cite[Theorem V.4.3]{KATO}), we see concretely that the operator $\mathscr{D}$ of Remark \ref{Rmk:OpD}, 
\begin{align*}
\mathscr{D}=D_{M_0/S^1}+\frac{1}{2}\tan\beta c(d\beta)\varepsilon
\end{align*}
is also essentially self-adjoint.\\

Out final goal is to verify Theorem \ref{Thm:S1SignatureThm}. The purpose is to present explicit computations, which are often not easy to find in the literature. To begin with, we need the Christoffel symbols associated to the metric \eqref{Eqn:MetricSigma}. Straightforward computations shows that they are given by (\cite[Chapter 2.3]{dC92})
\begin{align*}
\cris{\phi}{\phi}{\theta}=&-\cos\theta\sin\theta, & \cris{\xi}{\beta}{\xi}=&\cot\beta,\\
\cris{\theta}{\beta}{\theta}=&-\tan\beta, & \cris{\theta}{\theta}{\beta}=&\frac{1}{4}\cos\beta\sin\beta,\\
\cris{\theta}{\phi}{\phi}=&\cot\theta, & \cris{\phi}{\phi}{\beta}=&\frac{1}{4}\cos\beta\sin\beta\sin^2\theta,\\
\cris{\phi}{\beta}{\phi}=&-\tan\beta, & \cris{\xi}{\xi}{\beta}=&-\cos\beta\sin\beta.
\end{align*}

For example,
\begin{align*}
\cris{\theta}{\beta}{\theta}=&\frac{1}{2}g^{\theta \theta}(\partial_\beta g_{\theta\theta}+\partial_\beta g_{\theta\beta}-\partial_\theta g_{\theta\beta})\\
=&\frac{1}{2}\left(\frac{4}{\cos^2\beta}\right)\left(-\frac{2\cos\beta\sin\beta}{4}\right)\\
=&-\tan\beta.
\end{align*} 

Next we are going to compute the components of the connection $1$-form of the Levi-Civita connection with respect to a local orthonormal basis for $T^*(M_0/S^1)$ induced from $T^*M$.  First we construct such a basis. The dual $1$-forms of vector fields \eqref{Eqns:ONBS5} are
\begin{align*}
\widehat{\underline{e}}^1\coloneqq &\cos\beta\sin\eta\cos\eta(d\xi_1-d\xi_2),\\
\widehat{\underline{e}}^2\coloneqq &\sin\beta d\xi_3,\\
\widehat{\underline{e}}^3\coloneqq &\cos\beta d\eta,\\
\widehat{\underline{e}}^4\coloneqq &d\beta.
\end{align*}
Note that these forms are basic and satisfy the relations
\begin{align*}
\widehat{\underline{e}}^1\coloneqq &\pi^*_{S^1}\left(\frac{1}{2}\cos\beta \sin\theta d\phi\right),\\
\widehat{\underline{e}}^2\coloneqq &\pi^*_{S^1}(\sin\beta d\xi_3),\\
\widehat{\underline{e}}^3\coloneqq &\pi^*_{S^1}\left(\frac{1}{2}\cos\beta d\theta\right),\\
\widehat{\underline{e}}^4\coloneqq &\pi^*_{S^1} d\beta.
\end{align*}
This motivates the choice of basis 
\begin{align*}
\widehat{e}^\theta\coloneqq &\frac{1}{2}\cos\beta d\theta \eqqcolon \cos\beta e^\theta,\\
\widehat{e}^\phi\coloneqq &\frac{1}{2}\cos\beta \sin\theta d\phi\eqqcolon\cos\beta e^\phi,\\
\widehat{e}^\xi\coloneqq &\sin\beta d\xi\eqqcolon\sin\beta e^\xi,\\
\widehat{e}^\beta\coloneqq &d\beta. 
\end{align*}

Observe that the forms $\{e^\theta, e^\phi\}$ and $e^\xi$ form a local orthonormal basis for the vertical an horizontal tangent bundle respectively in the sense of Section \ref{Sect:PfoofSignatureFormula}.\\

Using the expressions for the Christoffel symbols above one can calculate the non-zero components of the connection $1$-form $\widehat{\omega}^I_J$ with respect to this basis,
\begin{align*}
\widehat{\omega}^\theta_\phi=&-2\sec\beta\cot\theta\widehat{e}^\phi, & \widehat{\omega}^\phi_\beta=&-\tan\beta\widehat{e}^{\phi},\\
\widehat{\omega}^\theta_\beta=&-\tan\beta\widehat{e}^\theta, & \widehat{\omega}^\xi_\beta=&\cot\beta\widehat{e}^{\xi}.
\end{align*}
One can also compute these components using \eqref{Eqns:StructureEquations}. For example,
\begin{align*}
d\widehat{e}^\theta=&-\frac{1}{2}\sin\beta  d\beta\wedge d\theta=\tan\beta \widehat{e}^{\theta}\wedge\widehat{e}^\beta,\\
d\widehat{e}^\phi=&-\frac{1}{2}\sin\beta \sin\theta d\beta\wedge d\phi+\frac{1}{2}\cos\beta\cos\theta d\theta\wedge d\phi=\tan\beta \widehat{e}^{\phi}\wedge\widehat{e}^\beta+2\sec\beta\cot\theta\widehat{e}^\theta\wedge\widehat{e}^\phi,\\
d\widehat{e}^\xi=&\cos\beta d\beta\wedge d\xi=-\cot\beta \widehat{e}^{\xi}\wedge\widehat{e}^\beta.
\end{align*}

We now want to calculate the components of the curvature form $\widehat{\Omega}$ using the structure equations \eqref{Eqns:ComponentsCurvatureForm}. We begin by computing the exterior derivative of the components of the connection $1$-form. For example for $\widehat{\omega}^\theta_\phi$ we have
\begin{align*}
d\widehat{\omega}^\theta_\phi=&d(-2\sec\beta\cot\theta\widehat{e}^\phi )\\
=&-2\sec\beta\tan\beta\cot\theta d\beta\wedge\widehat{e}^\phi+2\sec\beta\csc^2\theta d\theta\wedge\widehat{e}^\phi-2\sec\beta\cot\theta d\widehat{e}^\phi\\
=&2\sec\beta\tan\beta\cot\theta\widehat{e}^\phi \wedge\widehat{e}^\beta+4\sec^2\beta\csc^2\theta \widehat{e}^\theta\wedge\widehat{e}^\phi\\
&-2\sec\beta\cot\theta (\tan\beta \widehat{e}^{\phi}\wedge\widehat{e}^\beta+2\sec\beta\cot\theta\widehat{e}^\theta\wedge\widehat{e}^\phi)\\
=&4\sec^2\beta\widehat{e}^\theta\wedge\widehat{e}^\phi.
\end{align*}
Analogously for $\widehat{\omega}^\phi_\beta$,
\begin{align*}
d\widehat{\omega}^\phi_\beta=&d(-\tan\beta\widehat{e}^\phi )\\
=&-\sec^2\beta d\beta\wedge\widehat{e}^\theta-\tan\beta d\widehat{e}^\phi\\
=&\sec^2\beta\widehat{e}^\phi\wedge\widehat{e}^\beta-\tan\beta(\tan\beta \widehat{e}^{\phi}\wedge\widehat{e}^\beta+2\sec\beta\cot\theta\widehat{e}^\theta\wedge\widehat{e}^\phi)\\
=&\widehat{e}^\phi\wedge\widehat{e}^\beta-2\sec\beta\tan\beta\cot\theta\widehat{e}^\theta\wedge\widehat{e}^\phi. 
\end{align*}
Similar calculations show that these differentials are
\begin{align*}
d\widehat{\omega}^\theta_\phi=&4\sec^2\beta\widehat{e}^\theta\wedge\widehat{e}^\phi,\\
d\widehat{\omega}^\phi_\beta=&\widehat{e}^\phi\wedge\widehat{e}^\beta-2\sec\beta\tan\beta\cot\theta\widehat{e}^\theta\wedge\widehat{e}^\phi,\\
d\widehat{\omega}^\theta_\beta=&\widehat{e}^\theta\wedge\widehat{e}^\beta,\\
d\widehat{\omega}^\xi_\beta=&\widehat{e}^\xi\wedge\widehat{e}^\beta.
\end{align*}
Altogether, the components of the curvature form are computed using the structure equations,
\begin{align*}
\widehat{\Omega}^{\theta}_\phi=&d\widehat{\omega}^\theta_\phi+\widehat{\omega}^\theta_\beta\wedge\widehat{\omega}^\beta_\phi=4\sec^2\beta\widehat{e}^\theta\wedge\widehat{e}^\phi-\tan^2\beta\widehat{e}^\theta\wedge\widehat{e}^\phi
=(3\sec^2\beta+1)\widehat{e}^\theta\wedge\widehat{e}^\phi,\\
\widehat{\Omega}^{\theta}_\xi=&d\widehat{\omega}^\theta_\xi+\widehat{\omega}^\theta_\beta\wedge\widehat{\omega}^\beta_\xi=(-\tan\beta\widehat{e}^\theta)\wedge(-\cot\beta\widehat{e}^\xi)=\widehat{e}^\theta\wedge\widehat{e}^\xi,\\
\widehat{\Omega}^{\theta}_\beta=&d\widehat{\omega}^\theta_\beta+\widehat{\omega}^\theta_\phi\wedge\widehat{\omega}^\phi_\beta=\widehat{e}^\theta\wedge\widehat{e}^\beta,&\\
\widehat{\Omega}^{\phi}_\xi=&d\widehat{\omega}^\phi_\xi+\widehat{\omega}^\phi_\beta\wedge\widehat{\omega}^\beta_\xi
=(-\tan\beta\widehat{e}^\phi)\wedge(-\cot\beta\widehat{e}^\xi)=\widehat{e}^\phi\wedge\widehat{e}^\xi,\\
\widehat{\Omega}^{\phi}_\beta=&d\widehat{\omega}^\phi_\beta+\widehat{\omega}^\phi_\theta\wedge\widehat{\omega}^\theta_\beta=\widehat{e}^\phi\wedge\widehat{e}^\beta-2\sec\beta\tan\beta\cot\theta\widehat{e}^\theta\wedge\widehat{e}^\phi
+(2\sec\beta\cot\theta\widehat{e}^\phi)\wedge(-\tan\beta\widehat{e}^\theta),\\
\widehat{\Omega}^\xi_\beta=&d\widehat{\omega}^\xi_\beta=\widehat{e}^\xi\wedge\widehat{e}^\beta.
\end{align*}
We can write these components in matrix form,
\begin{align*}
\widehat{\Omega}=&\left(
\begin{array}{cccc}
0 & -(3\sec^2\beta+1)\widehat{e}^\theta\wedge\widehat{e}^\phi & -\widehat{e}^\theta\wedge\widehat{e}^\xi & -\widehat{e}^\theta\wedge\widehat{e}^\beta \\
(3\sec^2\beta+1)\widehat{e}^\theta\wedge\widehat{e}^\phi & 0 & - \widehat{e}^\phi\wedge\widehat{e}^\xi &  -\widehat{e}^\phi\wedge\widehat{e}^\beta\\
\widehat{e}^\theta\wedge\widehat{e}^\xi & \widehat{e}^\phi\wedge\widehat{e}^\xi & 0 & -\widehat{e}^\xi\wedge\widehat{e}^\beta \\
\widehat{e}^\theta\wedge\widehat{e}^\beta &  \widehat{e}^\phi\wedge\widehat{e}^\beta & \widehat{e}^\xi\wedge\widehat{e}^\beta & 0
\end{array}\right).
\end{align*}

In order to study the explicit dependence on $\beta$ we write this curvature matrix in terms of the $1$-forms $\{e^\theta,e^\phi, e^\xi, d\beta\}$,
\begin{align*}
\widehat{\Omega}(\beta)\coloneqq
&\left(
\begin{array}{cccc}
0 & - (3+\cos^2\beta){e}^\theta\wedge {e}^\phi & -\sin\beta\cos\beta{e}^\theta\wedge{e}^\xi & -\cos\beta{e}^\theta\wedge d\beta \\
(3+\cos^2\beta){e}^\theta\wedge {e}^\phi & 0 & - \sin\beta\cos\beta{e}^\phi\wedge {e}^\xi &  -\cos\beta{e}^\phi\wedge d\beta\\
\sin\beta\cos\beta {e}^\theta\wedge {e}^\xi & \sin\beta\cos\beta{e}^\phi\wedge{e}^\xi & 0 & -\sin\beta {e}^\xi\wedge d\beta \\
\cos\beta {e}^\theta\wedge d\beta & \cos\beta {e}^\phi\wedge d\beta& \sin\beta{e}^\xi\wedge d\beta & 0
\end{array}\right).
\end{align*}
In particular, the limit close to the fixed point set 
\begin{align*}
\lim_{\beta\rightarrow \pi/2}\widehat{\Omega}(\beta)
=\left(
\begin{array}{cccc}
0 & - 3 {e}^\theta\wedge {e}^\phi & 0 & 0 \\
3 {e}^\theta\wedge {e}^\phi & 0 &  0 & 0\\
0 &0 & 0 & -{e}^\xi\wedge{e}^\beta \\
0 & 0& {e}^\xi\wedge{e}^\beta & 0
\end{array}\right),
\end{align*}
is well defined. This shows explicitly that (compare with \eqref{Eqns:Curv})
\begin{align*}
\lim_{\beta\rightarrow \pi/2}\int_{M_0/S^1}L(\widehat{\Omega}(\beta))< \infty.
\end{align*}
Moreover, since $\widehat{\Omega}\wedge\widehat{\Omega}=0$ we actually have
\begin{align*}
\int_{M_0/S^1} L(\widehat{\Omega})=0.
\end{align*}

\begin{remark}[Euler class]
Using the explicit form of the curvature $\widehat{\Omega}(\beta)$ we want to compute the Euler class of $T(M_0/S^1)$, defined by (\cite[Section 5.6]{M00})
\begin{align*}
e(\widehat{\Omega}(\beta))\coloneqq\left[\frac{\Pf(\widehat{\Omega}(\beta))}{(2\pi)^2}\right]\in H^4(M_0/S^1),
\end{align*}
where $\Pf(\widehat{\Omega}(\beta))$ denotes the Pfaffian of $\widehat{\Omega}(\beta)$, which is characterized by the relation $\Pf(A)^2=\det(A)$ for any skew-symmetric matrix $A$. It is not hard to verify for a $4\times 4$ skew-symmetric matrix the Pfaffian formula
\begin{align*}
\Pf
\left[
\left(
\begin{array}{cccc}
0 & a & b & c \\
-a & 0 & d & e \\
-b & -d & 0 & f \\
-c & -e & -f & 0 
\end{array}
\right)
\right]=af-be+dc.
\end{align*}
As a consequence, using the explicit form of $\widehat{\Omega}(\beta)$ we obtain
\begin{align*}
\text{Pf}(\widehat{\Omega}(\beta))=& 3(1+\cos^2\beta)\sin\beta{e}^\theta\wedge {e}^\phi \wedge{e}^\xi\wedge d\beta.
\end{align*}
Now we integrate this form over the open manifold $M_0/S^1$,
\begin{align*}
\int_{M_0/S^1}\frac{\text{Pf}(\widehat{\Omega}(\beta))}{(2\pi)^2}
=&\int_{M_0/S^1}\frac{3}{(2\pi)^2}(1+\cos^2\beta)\sin\beta \left(\frac{1}{2}d\theta\right)\wedge \left(\frac{1}{2}\sin\theta d\phi\right)  \wedge d\xi\wedge d\beta\\
=&\int_0^{\pi/2}\int_0^{2\pi}\int_0^{2\pi}\int_0^{\pi}\frac{3}{(2\pi)^2}\left(\frac{1}{4}\right)(1+\cos^2\beta)\sin\beta  d\theta (\sin\theta d\phi) d\xi d\beta\\
=&\frac{3}{(2\pi)^2}\left(\frac{1}{4}\right)(4\pi)(2\pi)\frac{4}{3} \\
=&2.
\end{align*}
Hence, 
\begin{align}\label{Eqn:IntECExapmpleS5}
\int_{M/S^1}e(T(M_0/S^1))=\int_{M_0/S^1}\frac{\text{Pf}(\widehat{\Omega}(\beta))}{(2\pi)^2}=2. 
\end{align}
It is important to remark that this integral is indeed finite. 
\end{remark}

Now we deal with the eta invariant. As mentioned above, $M/S^1$ is a Witt space since the link is topologically a $2$-sphere and $H^1(S^2)$ vanishes.  It follows from Lemma \ref{Lemma:VanishingEtaWitt} that $\eta(M^{S^1})=0$. This can be explicitly seen as follows: the even part of the tangential signature operator \eqref{Eqn:DefOddSignOp} of the fixed point set $M^{S^1}=S^1$, which acts on smooth functions, is 
\begin{align*}
A^\text{ev}=\star_{S^1}(d_{S^1}+d^\dagger_{S^1})\left(\frac{1+\varepsilon_{S^1}}{2}\right) =-i\frac{d}{d\xi}.
\end{align*}
Observe that the spectrum of this operator is symmetric with respect to $0$ (Section \ref{Section:VanishEta}). Indeed, if $i\partial_\xi f_\lambda=\lambda f_\lambda$ for some $\lambda\in\mathbb{R}$ then 
$$i\partial_\xi \bar{f}_\lambda=\overline{-i\partial_\xi f_\lambda}=\overline{-\lambda f_\lambda}=-\lambda \bar{f}_\lambda.$$
This implies that $\eta_{A^{\text{ev}}}(0)=0$. Hence, both terms of the right hand side of Theorem \ref{Thm:S1SignatureThm} vanish. \\

Next we are going to compute the $S^1$-equivariant signature using purely topological methods.  In view of the isomorphisms \eqref{IsomsBasic} we will first compute the relative cohomology groups $H^*(M/S^1,M^{S^1})$. Then we will use this to compute $\sigma_{S^1}(M)$ from the intersection pairing in intersection homology (Corollary \ref{Coro:IH}). The main ingredient for these computations is an appropriate decomposition of $M/S^1$. We claim this space can be obtained from the following pushout diagram, 
\begin{align}\label{Diag:Pushout}
\xymatrixcolsep{1.5cm}\xymatrixrowsep{1.5cm}\xymatrix{
S^1\times S^2 \ar[d]_-{p_1} \ar[r]^-{j} & D^2\times S^2 \ar[d]\\
S^1 \ar[r] & M/S^1.
}
\end{align}
Here $D^2$ denotes the $2$-disk, $p_1$ is the projection onto the first component and $j$ is the natural inclusion, i.e. $S^1$ is included as the boundary of $D^2$ and on $S^2$ is just the identity. The claim can be seen as follows: From  the expression \eqref{Eqn:QuotMapHopfCoordS5} of the quotient map it is easy to see that $M_0/S^1$ is homeomorphic to $S^2\times\mathring{D}^2$, where $\mathring{D}^2\coloneqq D^2-S^1$ is the open $2$-disk. More precisely, for $\beta<\pi/2$ fixed, the first two components in \eqref{Eqn:QuotMapHopfCoordS5} describe a $2$-sphere of radius $\cos^2\beta$  and the third component describes a circle of radius $\sin\beta$. When $\beta\longrightarrow \pi/2$ we see that the radius of the $2$-sphere collapses to zero. Hence, the space $M/S^1$ can be obtained by collapsing $S^1\times S^{2}\subset D^2\times S^2$ to a circle $S^1$. This is precisely what the pushout diagram \eqref{Diag:Pushout} represents.\\

The advantage of this description  is the existence of an associated Mayer-Vietoris sequence (\cite[Proposition 6.2.6, Remark E.4(3)]{A11}), which enable us to compute the desired cohomology groups\footnote{I would like to thank Peter Patzt for discussions and references around this topic.}. Concretely, \eqref{Diag:Pushout} has an associated long exact sequence
\begin{center}
\begin{tikzpicture}[descr/.style={fill=white,inner sep=1.5pt}]
        \matrix (m) [
            matrix of math nodes,
            row sep=1em,
            column sep=2.5em,
            text height=1.5ex, text depth=0.25ex
        ]
        { 0 & H^0(M/S^1;\mathbb{R}) & H^0(A;\mathbb{R})\oplus  H^0(B;\mathbb{R}) & H^0(C;\mathbb{R}) \\
            & H^1(M/S^1;\mathbb{R}) & H^1(A;\mathbb{R})\oplus  H^1(B;\mathbb{R}) & H^1(C;\mathbb{R}) \\
            & H^2(M/S^1;\mathbb{R}) & H^2(A;\mathbb{R})\oplus  H^2(B;\mathbb{R}) & H^2(C;\mathbb{R}) \\
            & H^3(M/S^1;\mathbb{R}) & H^3(A;\mathbb{R})\oplus  H^3(B;\mathbb{R}) & H^3(C;\mathbb{R}) \\
            &\cdots \mbox{}         &                 & \mbox{}         \\
        };

        \path[overlay,->, font=\scriptsize,>=latex]
        (m-1-1) edge (m-1-2) 
        (m-1-2) edge (m-1-3)
        (m-1-3) edge (m-1-4)
        (m-1-4) edge[out=355,in=175] node {} (m-2-2)
        (m-2-2) edge (m-2-3)
        (m-2-3) edge (m-2-4)
        (m-2-4) edge[out=355,in=175] node {} (m-3-2)
        (m-3-2) edge (m-3-3)
        (m-3-3) edge (m-3-4)
        (m-3-4) edge[out=355,in=175] node {} (m-4-2)
        (m-4-2) edge (m-4-3)
        (m-4-3) edge (m-4-4)
        (m-4-4) edge[out=355,in=175] node {} (m-5-2);
\end{tikzpicture}
\end{center}
where $A=S^1$, $B=D^2\times S^2$ and $C=S^1\times S^2$. As $H^j(S^k;\mathbb{R})=\mathbb{R}$ whenever $j\in\{0,k\}$ and zero otherwise, 

\begin{center}
\begin{tikzpicture}[descr/.style={fill=white,inner sep=1.5pt}]
        \matrix (m) [
            matrix of math nodes,
            row sep=1em,
            column sep=2.5em,
            text height=1.5ex, text depth=0.25ex
        ]
        { 0 & H^0(M/S^1;\mathbb{R}) & \mathbb{R}\oplus  \mathbb{R}& \mathbb{R} \\
            & H^1(M/S^1;\mathbb{R}) & \mathbb{R}\oplus  \{0\} & \mathbb{R} \\
            & H^2(M/S^1;\mathbb{R}) & \{0\}\oplus \mathbb{R} & \mathbb{R} \\
            & H^3(M/S^1;\mathbb{R}) & 0 &  \mathbb{R}\\
            & H^4(M/S^1;\mathbb{R})        &     0.            & \mbox{}         \\
        };

        \path[overlay,->, font=\scriptsize,>=latex]
        (m-1-1) edge (m-1-2) 
        (m-1-2) edge (m-1-3)
        (m-1-3) edge (m-1-4)
        (m-1-4) edge[out=355,in=175] node {} (m-2-2)
        (m-2-2) edge (m-2-3)
        (m-2-3) edge (m-2-4)
        (m-2-4) edge[out=355,in=175] node {} (m-3-2)
        (m-3-2) edge (m-3-3)
        (m-3-3) edge (m-3-4)
        (m-3-4) edge[out=355,in=175] node {} (m-4-2)
        (m-4-2) edge (m-4-3)
        (m-4-3) edge (m-4-4)
         (m-4-4) edge[out=355,in=175] node {} (m-5-2)
         (m-5-2) edge (m-5-3);
\end{tikzpicture}
\end{center}
Form the construction of the Mayer-Vietoris sequence we know that the map in cohomology $H^i(A;\mathbb{R})\oplus H^i(B;\mathbb{R})\longrightarrow H^i(C;\mathbb{R})$ is $p^*_1-j^*$, so one can deduce
\begin{align}\label{Eqn:CohoQuotS5}
H^i(M/S^1;\mathbb{R})=
\begin{cases}
\mathbb{R}, & \text{if }i=0,4 \\
0, & \text{if }i\neq 0.
\end{cases}
\end{align}

Similar computations in homology allow us to conclude that 
\begin{align*}
H_i(M/S^1;\mathbb{R})=
\begin{cases}
\mathbb{R}, & \text{if }i=0,4 \\
0, & \text{if }i\neq 0.
\end{cases}
\end{align*}
This is of course consistent with Remark \ref{Rmk:S5/S1=S4}.\\

Now we use the long exact sequence in cohomology of the pair $(M/S^1,M^{S^1})$,
\begin{center}
\begin{tikzpicture}[descr/.style={fill=white,inner sep=1.5pt}]
        \matrix (m) [
            matrix of math nodes,
            row sep=1em,
            column sep=2.5em,
            text height=1.5ex, text depth=0.25ex
        ]
        { 0 & H^0(M/S^1,M^{S^1};\mathbb{R}) & H^0(M/S^1;\mathbb{R}) & H^0(M^{S^1};\mathbb{R}) \\
            & H^1(M/S^1,M^{S^1};\mathbb{R})& H^1(M/S^1;\mathbb{R}) & H^1(M^{S^1};\mathbb{R}) \\
            & H^2(M/S^1,M^{S^1};\mathbb{R}) & H^2(M/S^1;\mathbb{R}) & H^2(M^{S^1};\mathbb{R}) \\
            & H^3(M/S^1,M^{S^1};\mathbb{R}) & H^3(M/S^1;\mathbb{R}) & H^3(M^{S^1};\mathbb{R}) \\
            &\cdots \mbox{}         &                 & \mbox{}         \\
        };

        \path[overlay,->, font=\scriptsize,>=latex]
        (m-1-1) edge (m-1-2) 
        (m-1-2) edge (m-1-3)
        (m-1-3) edge (m-1-4)
        (m-1-4) edge[out=355,in=175] node {} (m-2-2)
        (m-2-2) edge (m-2-3)
        (m-2-3) edge (m-2-4)
        (m-2-4) edge[out=355,in=175] node {} (m-3-2)
        (m-3-2) edge (m-3-3)
        (m-3-3) edge (m-3-4)
        (m-3-4) edge[out=355,in=175] node {} (m-4-2)
        (m-4-2) edge (m-4-3)
        (m-4-3) edge (m-4-4)
        (m-4-4) edge[out=355,in=175] node {} (m-5-2);
\end{tikzpicture}
\end{center}
Using \eqref{Eqn:CohoQuotS5} and the fact that $M^{S^1}=S^1$, the above sequence reduces to
\begin{center}
\begin{tikzpicture}[descr/.style={fill=white,inner sep=1.5pt}]
        \matrix (m) [
            matrix of math nodes,
            row sep=1em,
            column sep=2.5em,
            text height=1.5ex, text depth=0.25ex
        ]
        { 0 & H^0(M/S^1,M^{S^1};\mathbb{R}) & \mathbb{R} & \mathbb{R} \\
            & H^1(M/S^1,M^{S^1};\mathbb{R})& 0 & \mathbb{R} \\
            & H^2(M/S^1,M^{S^1};\mathbb{R}) & 0 & 0 \\
            & H^3(M/S^1,M^{S^1};\mathbb{R}) & 0 & 0 \\
            & H^4(M/S^1,M^{S^1};\mathbb{R}) & \mathbb{R} & 0, \\
        };

        \path[overlay,->, font=\scriptsize,>=latex]
        (m-1-1) edge (m-1-2) 
        (m-1-2) edge (m-1-3)
        (m-1-3) edge (m-1-4)
        (m-1-4) edge[out=355,in=175] node {} (m-2-2)
        (m-2-2) edge (m-2-3)
        (m-2-3) edge (m-2-4)
        (m-2-4) edge[out=355,in=175] node {} (m-3-2)
        (m-3-2) edge (m-3-3)
        (m-3-3) edge (m-3-4)
         (m-3-4) edge[out=355,in=175] node {} (m-4-2)
        (m-4-2) edge (m-4-3)
        (m-4-3) edge (m-4-4)
         (m-4-4) edge[out=355,in=175] node {} (m-5-2)
        (m-5-2) edge (m-5-3)
        (m-5-3) edge (m-5-4);
\end{tikzpicture}
\end{center}
from where it follows that
\begin{align*}
H^i(M/S^1,M^{S^1};\mathbb{R})=
\begin{cases}
\mathbb{R}, & \text{if }i=2,4 \\
0, & \text{if }i\quad \text{otherwise}.
\end{cases}
\end{align*}
As a consequence, we see that in this example $\chi_{S^1}(M)=2$ which, in view of \eqref{Eqn:IntECExapmpleS5}, verifies Theorem \ref{Thm:S1EC}. \\

Finally, let us study the intersection paring to compute $\sigma_{S^1}(M)$. Let $C_*(M/S^1)$ denote the singular chain complex of $M/S^1$ and $\partial_i: C_i(M/S^1)\longrightarrow  C_{i-1}(M/S^1)$ be the associated boundary operator so that 
$H^i(M/S^1)\coloneqq \ker \partial_i/\text{ran} \:\partial_{i+1}$.  Intersection homology, introduced by Goresky and McPhearson in \cite{GMcP80} as a theory to recover Poincar\'e duality for singular spaces, can also be defined though a certain chain complex. In fact, there are various intersection homology groups labeled by certain {\em perversity}, which controls how transversally we allow the admissible chains to intersect. For our example it is enough to compute the intersection homology group $IH^{\bar{m}}_2(M/S^1)$, where $\bar{m}$ is the lower middle perversity, which is given explicitly by 
\begin{center}
\begin{tabular}{c | c | c | c | c | c | c | c | c  c}
  $i$ & $2$ & $3$ & $4$ & $5$ & $6$ & $7$ & $8$ & $\cdots$ \\ \hline
 $\bar{m}(i)$ & $0$ & $0$ & $1$ & $1$ & $2$ & $2$ & $3$ & $\cdots$
\end{tabular}
\end{center}
We are going to compute this group using the stratification $\emptyset \subset  M^{S^1} \subset M/S^1$. Following \cite[Definition 4.1.9]{B07}, the complex of {\em admissible intersection chains} is 
\begin{align*}
IC^{\bar{m}}_i(M/S^1)\coloneqq
\{ \upsilon\in C_i(X)\:|\: &\dim( |\upsilon|\cap M^{S^1})\leq i-3+\bar{m}(3),\\
&\dim( |\partial \upsilon|\cap M^{S^1})\leq i-4+\bar{m}(3)\},
\end{align*}
where the notation $|\upsilon |$ refers to the support of $\upsilon$ in $M/S^1$. We start the computation of $IH^{\bar{m}}_2(M/S^1)$ from its definition,
\begin{align*}
&IH^{\bar{m}}_2(M/S^1)\coloneqq \frac{\ker(\partial_{2}: IC^{\bar{m}}_{2}(M/S^1)\longrightarrow IC^{\bar{m}}_{1}(M/S^1))}
{\text{ran}(\partial_{3}: IC^{\bar{m}}_{3}(M/S^1)\longrightarrow IC^{\bar{m}}_{2}(M/S^1))}\\
&=\frac{\{ \upsilon\in C_2(M/S^1)\:|\: \partial\upsilon=0,\dim( |\upsilon|\cap M^{S^1})\leq 2-3+\bar{m}(3)\}}
{\partial {\{ \upsilon\in C_3(M/S^1)\:|\: \dim(\upsilon\cap M^{S^1})\leq 3-3+\bar{m}(3),
\dim(|\partial \upsilon| \cap M^{S^1})\leq 3-4+\bar{m}(3)\}}}\\
&=\frac{\{ \upsilon\in C_2(M/S^1)\:|\: \partial\upsilon=0, |\upsilon|\cap M^{S^1}=\emptyset\}}
{\partial {\{ \upsilon\in C_3(M/S^1)\:|\: \dim (|\upsilon|\cap M^{S^1})\leq 0}, |\partial \upsilon|\cap M^{S^1}=\emptyset\}}.
\end{align*}
Note that 
\begin{align*}
\{ \upsilon\in C_2(M/S^1)\:|\: \partial\upsilon=0, |\upsilon|\cap M^{S^1}=\emptyset\}=
\ker(\partial_2:C_2(M_0/S^1)\longrightarrow C_1(M_0/S^1)).
\end{align*}
On the other hand, chains in $\{ \upsilon\in C_3(M/S^1)\:|\: \dim (|\upsilon|\cap M^{S^1})\leq 0, |\partial \upsilon|\cap M^{S^1}=\emptyset\}$
are either in $C_3(M_0/S^1)$ or their interior intersect the fixed point in at most one point (see Figure \ref{Fig:S5CloseFI}).  Hence, in view of \eqref{Eqn:CohoQuotS5},  one can verify that $IH^{\bar{m}}_2(M/S^1)=0$. As a consequence of Corollary \ref{Coro:IH} we see then that $\sigma_{S^1}(M)=0$. This completes the verification of Theorem \ref{Thm:S1SignatureThm} in this concrete example. 

\subsubsection*{Codim $M^{S^1}$=2}
Now we want to study another semi-free $S^1$-action on $S^5$ for which the fixed point set has dimension $3$. The action is defined by 
$$\lambda (z_0,z_1,z_2):=(z_0,z_1,\lambda z_2)\quad\text{for} \: \lambda\in S^1.$$
The fixed point set of the action is 
$M^{S^1}=\{(z_0,z_1,0)\in S^5\}\cong S^3$,
and the principal orbit is therefore 
$$M_0=\{(z_0,z_1,z_2)\in S^5 \: : |z_2|>0\}.$$
It is straightforward to see that the orbit map on the principal orbit is
\begin{equation*}
\pi_{S^1}:
\xymatrixcolsep{2cm}\xymatrixrowsep{0.01cm}\xymatrix{
M_0\subset \mathbb{C}^3 \ar[r] & S^4_{+}\subset\mathbb{C}^2\times\mathbb{R}\\
(z_0,z_1,z_2) \ar@{|->}[r] & (z_0,z_1,t):=(z_0,z_1,|z_2|).
}
\end{equation*}
Here $S^4_{+}\coloneqq\{(z_0,z_1,t)\in S^4\: : \: t>0\}$ denotes upper hemisphere of the $4$ dimensional sphere $S^4\subset \mathbb{C}^2\times\mathbb{R}$. 
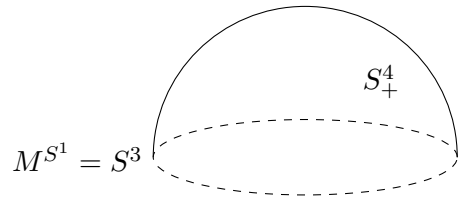
\begin{figure}[h]\label{Fig S4plus}
\begin{center}
\begin{tikzpicture}
\draw[dashed] (0,0) ellipse (2cm and 0.5cm);
\draw (2,0) arc (0:180:2 and 2);
\node at (-3,0) {$M^{S^1}=S^3$};
\node at (1,1) {$S^4_+$};
\end{tikzpicture}
\caption{Description of the orbit space $S^4_+\cup S^3$.}
\end{center}
\end{figure}
The metric that we need to equip $S^4_+$ with so that $\pi_{S^1}$ becomes a Riemannian submersion is
\begin{equation}\label{Eqn:MetricS4plus}
g^{T{S^4_+}}=\cos^2\beta(\cos^2\eta d\xi_1+\sin^2\eta d\xi_2+d\eta^2)+d\beta^2.
\end{equation}
We are going to describe now the operator $\mathscr{D}'$. The generating vector field of the action is $V=\partial_{\xi_3}$ and has a norm $\norm{V}=\sin\beta$. Computations as in the previous examples allow us to obtain
\begin{align*}
\chi=&\sin\beta d\xi_3,\notag\\
\kappa=&-\cot\beta d\beta,\\
d\chi=&\cot\beta d\beta\wedge \sin\beta d\xi_3,\\
 \varphi_0=&0.
\end{align*}
Form these relations and Theorem \ref{Thm:InducedDiracOp} we get
\begin{equation*}
\mathscr{D}'=\mathscr{D}=D_{S^4_+}-\frac{1}{2}\cot\beta c(d\beta)\varepsilon, 
\end{equation*}
where $D_{S^4_+}=d_{S^4_+}+d^\dagger_{S^4_+}$ is the Hodge-de Rham operator with respect to the metric \eqref{Eqn:MetricS4plus}.\\

For this example we approach the fixed point set as $\beta\longrightarrow 0^+$. In this limit the the metric \eqref{Eqn:MetricS4plus} takes the form
\begin{align*}
(\cos^2\eta d\xi_1+\sin^2\eta d\xi_2+d\eta^2)+dr^2,
\end{align*}
for $0<r<t$ with $t$ sufficiently small. Hence, its is approximated by a product metric on $S^3\times (0,t)$. \\

In view of the description of the quotient space, we see from Example \ref{Ex:SpiningMwB} that the equivariant $S^1$-signature is  $\sigma_{S^1}(S^5)=\sigma(S^4_+)$. The topological signature of this manifold with boundary can be calculated, for example,  using Proposition \ref{Prop:PropSignatureClosed}(2) and Novikov's additivity formula (Proposition \ref{Prop:Novikov}),
\begin{align*}
0=\sigma(S^4)=2\sigma(S^4_+).
\end{align*}
This also shows, by \eqref{Eqn:SignTheoClosed} applied to to $S^4$, that the integral of the $L$-polynomial over $S^4_+$ vanishes. Thus, $\eta_{A^{\text{ev}}(S^3)}(0)=0$ by the APS signature  Theorem \ref{Thm:SignThmMBound}.

\section{Local description}\label{Sect:LocalDesc}

This chapter is intended to describe the Dirac-Schr\"odinger operator constructed in Section \ref{Sect:ConstDiracSchrOp},
\begin{equation*}
\mathscr{D}'=D_{M_0/S^1}+\frac{1}{2}c(\bar{\kappa})\varepsilon+\varepsilon\widehat{c}(\bar{\varphi}_0)\left(\frac{1-\varepsilon}{2}\right),
\end{equation*}
near to the fixed point set $M^{S^1}$. We use the geometric description of a neighborhood of a connected component $F\subset M^{S^1}$ discussed in Section \ref{Sect:PfoofSignatureFormula} to compute $\bar{\kappa}$ and $\bar{\varphi}_0$. Recall that we decomposed the quotient space as $M_0/S^1=Z_t\cup U_t$ where $Z_t\coloneqq M_0/S^1-N_t(F)$ is a compact manifold with boundary and $U_t\coloneqq N_t(F)$ is the $t$-neighborhood of $F$ in $M/S^1,$ as schematically visualized in Figure \ref{Fig:DecompQuotSpace}. Moreover, following the proof presented in \cite{L00}, we model $U_t$ as the mapping cylinder of a Riemannian fibration. Under this setting, following  \cite[Section 2 and 3]{B09}, we compute the complete operator  $\mathscr{D}'$ in this local model. In particular, we takeover the claim of Remark \ref{Rmk:OpD} and show that we can restrict ourselves to $\mathscr{D}$ without losing essential self-adjointness. Finally we take care of the spectral decomposition of the associated cone coefficient. This last step, in combination with the parametrix construction of next chapter, will explicitly show how the potential of the operator $\mathscr{D}$ shifts the spectrum of the cone coefficient of $D_{M_0/S^1}$ ensuring the essential self-adjointness regardless of the Witt/non-Witt condition. 

\begin{figure}[H]
\begin{center}
\begin{tikzpicture}
\draw[rounded corners=32pt](7,-1)--(4,-1)--(2,-2)--(0,0) -- (2,2)--(4,1)--(7,1);
\draw (1.5,0.2) arc (175:315:1cm and 0.5cm);
\draw (3,-0.28) arc (-30:180:0.7cm and 0.3cm);
\draw (7.5,0) arc (0:360:0.5cm and 1cm);
\node (a) at (20:2.5) {$Z_t$};
\node (a) at (7,-1.5) {$\partial Z_t=\mathcal{F}_{t}$};
\draw (11,-0.6)--(11,0.6);
\draw (7,1)--(11,0.6);
\draw (7,-1)--(11,-0.6);
\node (a) at (9,0) {$U_t$};
\node (a) at (11.5,0) {$M^{S^1}$};
\end{tikzpicture}
\caption{Decomposition of $M/S^1$. }\label{Fig:DecompQuotSpace}
\end{center}
\end{figure}
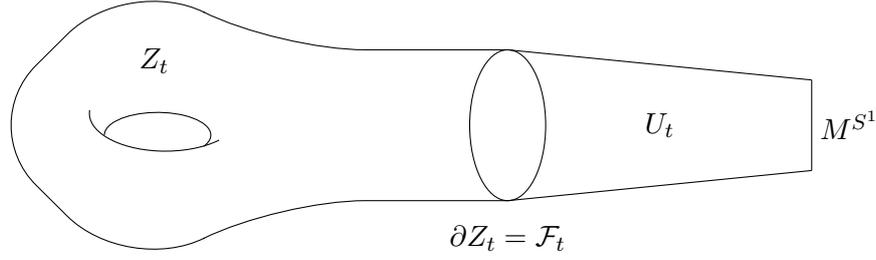

\subsection{Mean curvature $1$-form on the normal bundle}

Recall from the proof of Theorem \ref{Thm:S1SignatureThm} presented in  Section \ref{Sect:PfoofSignatureFormula}, that we considered a connected component $F\subset M^{S^1}$ of the fixed point set of dimension $4k-2N-1$. We studied the induced $S^1$-action on the normal bundle $NF\longrightarrow F$ and saw that the quotient space of the  corresponding sphere bundle $\mathcal{S}/S^1$ could be identified with the total space of a Riemannian fibration $\pi_\mathcal{F}:\mathcal{F}\longrightarrow F$ with fiber $\mathbb{C}P^N$. This allowed us to model $M/S^1$ close to $F$ as the mapping cylinder  $C(\mathcal{F})$ of the projection $\pi_\mathcal{F}$. 
Let $\pi_\mathcal{S}:\mathcal{S}\longrightarrow F$  be the projection of the sphere bundle in $NF$ and consider the decomposition of the tangent bundle $T\mathcal{S}=T_V\mathcal{S}\oplus T_H\mathcal{S}$ on which the metric decomposes as $g^{T\mathcal{S}}=g^{T_V\mathcal{S}}\oplus g ^{T_H\mathcal{S}}$, where we identify $T_H\mathcal{S}\cong TF$ via $\pi_\mathcal{S}$. Using the conventions of Section \ref{Sect:RiemFibr} we choose a local oriented orthonormal basis for $T\mathcal{S}$ of the form 
\begin{align}\label{Eqn:ONBTSN}
\{{\underline{e}}_i\}_{i=0}^{v}\cup\{{\underline{f}}_\alpha\}_{\alpha=1}^{h},
\end{align} 
where  $v\coloneqq 2N $ and $h\coloneqq 4k-2N-1$.  Here $\underline{e}_i$ and $\underline{f}_\alpha$ are vertical and horizontal vector fields respectively. Let $\{{\underline{e}}^i\}_{i=0}^{v}\cup\{{\underline{f}}^\alpha\}_{\alpha=1}^{h}$
denote the associated dual basis. In view of Proposition \ref{Prop:S1actionNF} we can assume with out loss of generality that the generating vector field of the free $S^1$-action on $\mathcal{S}$ is $V_\mathcal{S}\coloneqq{\underline{e}}_0\in C^\infty(\mathcal{S}, T_V \mathcal{S})$. This implies that the corresponding mean curvature form $\kappa_\mathcal{S}$ vanishes by Proposition \ref{Prop:NormV}(2) since $\norm{{\underline{e}}_0}=1$.

\begin{remark}\label{Rmk:BasisBasicForms}
We can assume that the differential $1$-forms $\{{\underline{e}}^i\}_{i=1}^{2N}\cup\{{\underline{f}}^\alpha\}_{\alpha=1}^{4k-2N-1}$ are basic (see Remark \ref{Rmk:BasicFormsS3}).
\end{remark}

Let $\widehat{\underline{\omega}}$ be the connection $1$-form of the Levi-Civita connection of the metric $g^{T\mathcal{S}}$ associated to the orthonormal frame above. Recall that its components satisfy the structure equations (see \eqref{Eqns:StructureEquations})
\begin{align*}
d{\underline{e}}^i+{\underline{\omega}}^i_j\wedge {\underline{e}}^j+{\underline{\omega}}^i_\alpha\wedge {\underline{f}}^\alpha=0,\\
d{\underline{f}}^\alpha+{\underline{\omega}}^\alpha_j\wedge {\underline{e}}^j+{\underline{\omega}}^\alpha_\beta\wedge {\underline{f}}^\beta=0.
\end{align*}
Since the characteristic $1$-form is $\chi_{\mathcal{S}}={\underline{e}}^0$ and $\kappa_\mathcal{S}=0$ we can use these structure equations to compute the $2$-form $\varphi_{0,\mathcal{S}}$  from \eqref{Eqn:dchi}, namely
\begin{equation}\label{Eqn:Varphi0S}
\varphi_{0,\mathcal{S}}=d\chi_{\mathcal{S}}=d\underline{e}^0=-{\underline{\omega}}^0_i\wedge {\underline{e}}^i-{\underline{\omega}}^0_\alpha\wedge {\underline{f}}^\alpha.
\end{equation}
In particular observe that
\begin{align*}
\iota_{V_{\mathcal{S}}} \varphi_{0,\mathcal{S}}=-\iota_{\underline{e}_0}(\underline{\omega}^0_i\wedge \underline{e}^i+\underline{\omega}^0_\alpha\wedge \underline{f}^\alpha)
=-(\underline{\omega}^0_i(e_0) e^i+\underline{\omega}^0_\alpha(\underline{e}_0) \underline{f}^\alpha)
=-(\nabla_{\underline{e}_0}\underline{e}_0)^\flat=-\kappa_{\mathcal{S}}=0,
\end{align*}
as we expected.\\

Let $\mathring{\mathcal{D}}\coloneqq \mathcal{D}-F$ denote the disk bundle of the normal bundle without the zero section, where the induced $S^1$-action is still free. Now we calculate the corresponding mean curvature $1$-form $\kappa_{\mathring{\mathcal{D}}}$ and the $2$-form $\varphi_{0,\mathring{\mathcal{D}}}$ for the metric  
\begin{equation}\label{Eqn:metricTD}
g^{T\mathring{\mathcal{D}}}=dr^2 \oplus r^2 g^{T_V\mathcal{S}}\oplus g^{T_V\mathcal{S}},
\end{equation} 
where $r>0$ denotes the radial direction. From the orthonormal basis of $T^*\mathcal{S}$ described above we can construct an orthonormal basis for $T^*\mathring{\mathcal{D}}$ as
\begin{align}\label{Def:BasisTD}
\{dr\}\cup\{\widehat{\underline{e}}^i\}_{i=0}^{v}\cup\{\widehat{\underline{f}}^\alpha\}_{\alpha=1}^{h}, 
\end{align} 
setting $\widehat{\underline{e}}^i\coloneqq r\underline{e}^i$ and $\widehat{\underline{f}}^\alpha \coloneqq \underline{f}^\alpha$. Here we regard $\underline{e}^i$ and $\underline{f}^\alpha$ as $1$-forms on $\mathring{\mathcal{D}}$ by pulling them back from $\mathcal{S}$ along the projection $\mathring{\mathcal{D}}\longrightarrow \mathcal{S}$. The generating vector field of the $S^1$-action on $\mathring{\mathcal{D}}$ is still $V_{\mathring{\mathcal{D}}}=\underline{e}_0=r\widehat{\underline{e}}_0$ and therefore, by Proposition \ref{Prop:NormV}(2), we get
 \begin{align*}
\kappa_{\mathring{\mathcal{D}}}=-d\log(\norm{V_{\mathring{\mathcal{D}}}})=-d(\log r)=-\frac{dr}{r}.
\end{align*}

Now we want to compute the $2$-form $\varphi_{0,\mathring{\mathcal{D}}}$. First note that the associated characteristic $1$-form is $\chi_{\mathring{\mathcal{D}}}=\widehat{\underline{e}}^0=r\underline{e}^0$, thus we need to calculate $d\widehat{\underline{e}}^0$ in order to use \eqref{Eqn:dchi}.  Let $\widehat{\underline{\omega}}$ be the connection $1$-form corresponding to the Levi-Civita connection of the metric \eqref{Eqn:metricTD} associated with the basis \eqref{Def:BasisTD}. Using the structure equations 
\begin{align*}
d\widehat{\underline{e}}^i+\widehat{\underline{\omega}}^i_j\wedge \widehat{\underline{e}}^j+\widehat{\underline{\omega}}^i_\alpha\wedge \widehat{\underline{f}}^\alpha+\widehat{\underline{\omega}}^i_r\wedge dr=0,\\
d\widehat{\underline{f}}^\alpha+\widehat{\underline{\omega}}^\alpha_j\wedge \widehat{\underline{e}}^j+\widehat{\underline{\omega}}^\alpha_\beta\wedge \widehat{\underline{f}}^\beta+\widehat{\underline{\omega}}^\alpha_r\wedge dr=0,\\
\widehat{\underline{\omega}}^r_j\wedge\widehat{\underline{e}}^j+\widehat{\underline{\omega}}^r_\alpha\wedge\widehat{\underline{f}}^\alpha=0,
\end{align*}
 we can proceed as in Section \ref{Sect:PfoofSignatureFormula} to obtain the components of $\widehat{\omega}$ (see \eqref{Eqns:Conn1From}), 
\begin{align*}
\widehat{\underline{\omega}}^i_j=&\underline{\omega}^i_j,\\
\widehat{\underline{\omega}}^i_r=&\underline{e}^i,\\
\widehat{\underline{\omega}}^i_\alpha=& r\underline{\omega}^i_\alpha,\\
\widehat{\underline{\omega}}^\alpha_\beta=&r^2\underline{\omega}^\alpha_{\beta i}\underline{e}^i+\underline{\omega}^\alpha_{\beta\gamma}\underline{f}^\gamma,\\
\widehat{\underline{\omega}}^\alpha_r=&0.
\end{align*}
From these equations we find
\begin{align*}
d\chi_{\mathring{\mathcal{D}}}=& d\widehat{\underline{e}}^0\\
=&-\widehat{\underline{\omega}}^0_i\wedge\widehat{\underline{e}}^i-\widehat{\underline{\omega}}^0_\alpha\wedge\widehat{\underline{f}}^\alpha-\widehat{\underline{\omega}}^0_r\wedge dr\\
=&-\underline{\omega}^0_i\wedge(r \underline{e}^i)-(r\underline{\omega}^0_\alpha)\wedge \underline{f}^\alpha-\underline{e}^0\wedge dr\\
=&r\left(-\underline{\omega}^0_i\wedge\underline{e}^i-\underline{\omega}^0_\alpha\wedge \underline{f}^\alpha\right)-\left(-\frac{dr}{r}\right)\wedge (r\underline{e}^0)\\
=&r\varphi_{0,\mathcal{S}}-\kappa_{\mathring{\mathcal{D}}}\wedge \chi_{\mathring{\mathcal{D}}}. 
\end{align*}
As a result, we conclude from \eqref{Eqn:dchi} that $\varphi_{0,\mathring{\mathcal{D}}}=r\varphi_{0,\mathcal{S}}$. We summarize these results in the following proposition. 
\begin{proposition}\label{Prop:KappaPhiDCirc}
Let $F$ be a connected component of the fixed point set of an effective semi-free $S^1$-action on $M$. Then, for the induced action on $\mathring{\mathcal{D}}$, the mean curvature $1$-form is 
\begin{align*}
\kappa_{\mathring{\mathcal{D}}}=-\frac{dr}{r},
\end{align*}
and the $2$-form $\varphi_{0,\mathring{\mathcal{D}}}$ defined by \eqref{Eqn:dchi} is $\varphi_{0,\mathring{\mathcal{D}}}=r\varphi_{0,\mathcal{S}}$. 
\end{proposition}

Now we proceed to study the $S^1$-quotient. Recall from above that we have the following commutative diagram 
\begin{align*}
\xymatrixrowsep{2cm}\xymatrixcolsep{2cm}\xymatrix{
\mathcal{S} \ar[rr]^-{\pi_{S^1}} \ar[dr]_-{\pi_\mathcal{S}}  && \mathcal{F} \ar[dl]^-{\pi_\mathcal{F}} \\
& F & 
}
\end{align*}
where $\pi_{S^1}:\mathcal{S}\longrightarrow\mathcal{F}$ is the orbit map and $\pi_\mathcal{S}$, $\pi_\mathcal{F}$ are corresponding projections.  The decomposition $T\mathcal{S}=T_V\mathcal{S}\oplus T_H\mathcal{S}$ induces a decomposition  $T\mathcal{F}=T_V\mathcal{F}\oplus T_H\mathcal{F}$ via the orbit map $\pi_{S^1}$. Consequently, as in Section \ref{Sect:RiemFibr},  there is an induced splitting of the exterior algebra,
\begin{align}\label{DecompBundleE}
\wedge T^*\mathcal{F}=\wedge T_H^*\mathcal{F}\otimes\wedge T^*_V\mathcal{F} =\bigoplus_{p+q}\wedge^p T^*_H\mathcal{F}\otimes \wedge^q T^*_V\mathcal{F}
\eqqcolon
\bigoplus_{p,q}\wedge^{p,q}T^*\mathcal{F}. 
\end{align} 
We denote the space of sections by $\Omega^{p,q}(\mathcal{F})\coloneqq C^{\infty}(\mathcal{F},\wedge^{p,q}T^*\mathcal{F})$ and the degree operators by $\text{hd}|{\wedge^{p,q} T^*\mathcal{F}}\coloneqq p$ and  $\text{vd}|{\wedge^{p,q} T^*\mathcal{F}}\coloneqq q$ as in Section \ref{Sect:RiemFibr}.\\

As mentioned in Remark \ref{Rmk:BasisBasicForms}, the forms $\{\underline{e}^i\}_{i=1}^{v}\cup\{\underline{f}^\alpha\}_{\alpha=1}^{h}$ are basic, and therefore there exist $1$-forms $\{{e}^i\}_{i=1}^{v}\cup\{{f}^\alpha\}_{\alpha=1}^{h}$, where ${e}^i\in T^*_V\mathcal{F}$ and ${f}^\alpha\in T^*_H\mathcal{F}$ such that $\pi^*_{S^1}{e}^i=\underline{e}^i$ and $\pi^*_{S^1}{f}^\alpha=\underline{f}^\alpha$. The set $\{{e}^i\}_{i=1}^{v}\cup\{{f}^\alpha\}_{\alpha=1}^{h}$ forms a local orthonormal basis for $T^*\mathcal{F}$ and can be regarded as the basis considered in Section \ref{Sect:PfoofSignatureFormula} (Step 3) . We choose an orientation of $\mathcal{F}$ so that  $\{{f_1},\cdots, {f}_h,{e_1},\cdots,{e}_v\}$ is an oriented orthonormal basis. The following result is a direct consequence of Proposition \ref{Prop:KappaPhiDCirc}. 

\begin{proposition}
 Close to a connected component $F\subset M^{S^1}$ of the fixed point set:
 \begin{enumerate}
\item The $1$-form $\bar{\kappa}$ of diagram \ref{Diag:KappaAction} is given by
\begin{align*}
\bar{\kappa}=-\frac{dr}{r}.
\end{align*}
\item There exists $\varphi_{0,\mathcal{F}}\in\Omega^2(\mathcal{F})$ such that the $2$-form $\bar{\varphi}_0$ defined by diagram \eqref{Diag:Varphi0Action} can be expressed as $\bar{\varphi}_{0}= r\varphi_{0,\mathcal{F}}$. In particular, the operator $\widehat{c}(\bar{\varphi}_0)$
of Proposition \ref{Prop:CliffMultVarphi} is bounded.
\end{enumerate}
 \end{proposition}

Combining this proposition with Theorem \ref{Thm:InducedDiracOp} and the Kato-Rellich Theorem (\cite[Theorem V.4.3]{KATO}) we can prove the claim of Remark \ref{Rmk:OpD}.

\begin{coro}
The operator $\mathscr{D}:\Omega_c(M_0/S^1)\longrightarrow\Omega_c(M_0/S^1)$ defined by 
\begin{align*}
\mathscr{D}\coloneqq D_{M_0/S^1}+\frac{1}{2}c(\bar{\kappa})\varepsilon,
\end{align*}
is essentially self-adjoint. 
\end{coro}

\subsection{Local description of the operators}\label{Section:DecConeOp}
Now we analyze how the complete operator $\mathscr{D}$ can be written near the fixed point set. For the Hodge-de Rham operator this was done by Br\"uning in \cite[Section 2]{B09}. We follow closely his treatment and provide some omitted calculations for the sake of completeness. The main ingredient is the decomposition of the exterior derivative of Section \ref{Sect:RiemFibr}.

\subsubsection{Local description: Hodge-de Rham operator} \label{Section:LocalDescrHdROp}
Let us define, for $t>0$, the model space $U_t\coloneqq \mathcal{F}\times (0,t)$ and let $\pi: U_t\longrightarrow \mathcal{F}$ be the  projection onto the first factor. Equip $U_t$ with the metric
\begin{align}\label{Eqn:MetricUt}
g^{TU_t}\coloneqq dr^2\oplus  g^{T_H\mathcal{F}}\oplus r^2 g^{T_V\mathcal{F}},
\end{align}
for $0<r<t$. We choose the orientation on $U_t$ defined by the oriented orthonormal basis (see Remark \ref{Rmk:OrBound})
\begin{equation}\label{Eqn:OrioentationUt}
\{-\partial_r, \widehat{f_1},\cdots, \widehat{f}_h,\widehat{e_1},\cdots,\widehat{e}_v\}.
\end{equation}

Consider the unitary transformation introduced in \cite[Equation (2.12)]{B09},
\begin{equation}\label{Eqn:TransfPsi1}
\Psi_1:
\xymatrixcolsep{3pc}\xymatrixrowsep{0.5pc}\xymatrix{
L^2((0,t),\Omega(\mathcal{F})\otimes\mathbb{C}^2) \ar[r] & L^2(U_t)\\
(\sigma_1,\sigma_2) \ar@{|->}[r] & \pi^* r^{\nu}\sigma_1(r)+ dr\wedge\pi^* r^\nu \sigma_2(r),
}
\end{equation}
where $\nu$ is the operator defined by
$$\nu:=\text{vd}-\frac{v}{2}=\text{vd}-N.$$
For example $\nu e_i=(1-N)e_i$ and $\nu f_\alpha=-Nf_\alpha$. To see that $\Psi_1$ is indeed a unitary transformation we verify, using Remark \ref{Rmk:ScaleHodgeStar}, that the volume forms are related by the equation $\vol_{U_t}=r^{2\nu} \vol_\mathcal{F}\wedge dr$. Hence,
\begin{align*}
\norm{\Psi_1(\sigma_1,\sigma_2)}^2_{L^2(U_t)}=\int_0^\infty\left(\norm{\sigma_1(r)}^2_{L^2(\mathcal{F})}+\norm{\sigma_2(r)}^2_{L^2(\mathcal{F})}\right) dr.
\end{align*}

Now we want to understand how the Hodge-de Rahm operator $D_{U_t}=d_{U_t}+d_{U_t}^\dagger$ on $U_t$, associated to the metric \eqref{Eqn:MetricUt}, transforms under $\Psi_1$. First we use the decomposition \eqref{Eqn:d} of the exterior derivative to compute its action on $\Psi_1(\sigma_1,\sigma_2)$. For the first component we have
\begin{align*}
d_{U_t}(\pi^{*}r^{\nu}\sigma_1(r))=&(d^{(1)}_H+d^{(2)}_H+d_V+dr\wedge \partial_r)(\pi^{*}r^{\nu}\sigma_1(r))\\
=&\pi^{*}r^{\nu}d_H^{(1)}\sigma_1(r)+\pi^{*}r^{\nu}r d_H^{(2)}\sigma_1(r)+\pi^{*}r^{\nu}r^{-1}d_V\sigma_1(r)+dr\wedge\partial_r\pi^{*}r^{\nu}\sigma_1(r)\\
=&\pi^*r^\nu ((d^{(1)}_H+rd^{(2)}_H+r^{-1}d_V)\sigma_1(r))+dr\wedge \pi^*r^\nu(r^{-1}\nu+\partial_r)\sigma_1(r).
\end{align*}
For the second one we can use the result above,
\begin{align*}
d_{U_t}(dr\wedge \pi^{*}r^{\nu}\sigma_2(r))=&-dr\wedge d_{U_t}(\pi^*r^{\nu}\sigma_2(r))
=-dr\wedge \pi^*r^\nu ((d^{(1)}_H+rd^{(2)}_H+r^{-1}d_V)\sigma_2(r)).
\end{align*}
These computations show that  $\Psi^{-1}_1d_{U_t}\Psi_1$ can be written in matrix form as 
\begin{align*}
\Psi^{-1}_1d_{U_t}\Psi_1=
\left(
\begin{array}{cc}
d^{(1)}_H+rd^{(2)}_H+r^{-1}d_V & 0\\
\partial_r+r^{-1}\nu &- d^{(1)}_H-rd^{(2)}_H-r^{-1}d_V
\end{array}
\right). 
\end{align*}
Observe that, since $\norm{dr}=1$, we can compute $\Psi^{-1}_1d^\dagger_{U_t}\Psi_1$ by just taking the formal adjoint, 
\begin{align*}
\Psi^{-1}_1d^\dagger_{U_t}\Psi_1=
\left(
\begin{array}{cc}
(d^{(1)}_H+rd^{(2)}_H+r^{-1}d_V)^\dagger & -\partial_r+r^{-1}\nu\\
0 &- (d^{(1)}_H-rd^{(2)}_H-r^{-1}d_V)^\dagger
\end{array}
\right). 
\end{align*}
As a result we get the following description of the of the operator $\Psi^{-1} D_{U_t}\Psi_1$.
\begin{proposition}[{\cite[Corollary 2.3]{B09}}]\label{Prop:TransDPsi1}
The Hodge-de Rham operator $D_{U_t}$ transforms under $\Psi_1$ as
\begin{align*}
\tilde{D}_{U_t}\coloneqq
\Psi_1^{-1}D_{U_t}\Psi_1=
&\gamma\left(\frac{\partial}{\partial r}+
\left(
\begin{array}{cc}
0 & -\tilde{A}_H(r)\\
-\tilde{A}_H(r) & 0
\end{array}
\right)
+\frac{1}{r}
\left(
\begin{array}{cc}
\nu & -\tilde{A}_{0V}\\
-\tilde{A}_{0V} & -\nu 
\end{array}
\right)
\right)\\
=&
\gamma\left(\frac{\partial}{\partial r}+
\left(
\begin{array}{cc}
0 & -\tilde{A}_0(r)\\
-\tilde{A}_0(r) & 0
\end{array}
\right)
+\frac{1}{r}
\left(
\begin{array}{cc}
\nu & 0\\
0 & -\nu 
\end{array}
\right)
\right)
\end{align*}
where 
\begin{align*}
\gamma \coloneqq 
\left(
\begin{array}{cc}
0 & -I\\
I & 0
\end{array}
\right).
\end{align*}
and
\begin{align*}
\tilde{A}_H(r)\coloneqq &(d^{(1)}_H+rd^{(2)}_H)+(d^{(1)}_H+rd^{(2)}_H)^\dagger,\\
\tilde{A}_{0V}\coloneqq &d_V+d^\dagger_V,\\
\tilde{A}_0(r)\coloneqq &\tilde{A}_H(r)+r^{-1}\tilde{A}_{0V}.
\end{align*}
\end{proposition}

Next we want to calculate how the chirality operator $\bar{\star}$ transforms under \eqref{Eqn:TransfPsi1}. With respect to the orientation \eqref{Eqn:OrioentationUt} we have by definition
\begin{align*}
\bar{\star}=&i^{2k}c(-dr)c({f}^1)\cdots c({f}^h)c(r{e}^1)\cdots c(r{e}^v)\\
=&i^{2k}(-1)^{h+v+1}c({f}^1)\cdots c({f}^h)c(r{e}^1)\cdots c(r{e}^v)c(dr)\\
=&i^{2k}c({f}^1)\cdots c({f}^h)c({e}^1)\cdots c({e}^v)r^{-2\nu}c(dr)\\
=&\bar{\star}_H\bar{\star}_Vr^{-2\nu}c(dr),
\end{align*}
where 
\begin{align*}
\bar{\star}_H\coloneqq &i^{[(h+1)/2]}c({f}^1)\cdots c({f}^h),\\
\bar{\star}_V\coloneqq &i^{[(v+1)/2]}c({e}^1)\cdots c({e}^v),
\end{align*}
are the horizontal and vertical chirality operators respectively.
\begin{remark}
The factor $r^{-2\nu}$  in the expression for $\bar{\star}$ arises from the relation
\begin{align*}
c(r{e}^{j})e^{k}=(re^j\wedge-\iota_{e_j/r})e^k=re^j\wedge e^k-\frac{1}{r}e^k(e_j),
\end{align*}
which shows that the total power of the $r$ factor is  $(v-\text{vd})-\text{vd}=-2\nu$.
\end{remark}

Using the relation $\bar{\star}_V r^{-\nu}=r^\nu\bar{\star}_V$ we compute,
\begin{align*}
\bar{\star}\Psi_1(\sigma_1,\sigma_2)
=&\bar{\star}_H\bar{\star}_Vr^{-2\nu}c(dr)(\pi^* r^{\nu}\sigma_1(r)+ dr\wedge\pi^*r^\nu \sigma_2(r))\\
=&\bar{\star}_H\bar{\star}_Vr^{-2\nu}(dr \wedge\pi^* r^{\nu}\sigma_1(r)-\pi^* r^\nu \sigma_2(r))\\
=&(-1)^{h+v}dr \wedge \bar{\star}_H\bar{\star}_Vr^{-2\nu}\pi^*r^{\nu}\sigma_1(r)- \bar{\star}_H\bar{\star}_V r^{-2\nu}\pi^*r^\nu \sigma_2(r)\\
=&-dr \wedge \pi^*r^{\nu} (\bar{\star}_H\bar{\star}_V \sigma_1(r))- \pi^* r^\nu (\bar{\star}_H\bar{\star}_V\sigma_2(r)).
\end{align*}
In matrix notation we can write the transformed chirality operator as (\cite[Lemma 2.4]{B09})
\begin{align}\label{TransTauPsi1}
\widehat{\star}\coloneqq 
\Psi_1^{-1}\bar{\star}\Psi_1
=\left(\begin{array}{cc}
0 & -\bar{\star}_H\bar{\star}_V \\
-\bar{\star}_H\bar{\star}_V  & 0
\end{array}
\right)
=\left(\begin{array}{cc}
0 & I\\
I & 0
\end{array}
\right)\otimes(-\alpha),
\end{align}
where $\alpha \coloneqq \varepsilon_H^v \bar{\star}_H\otimes\bar{\star}_V=\bar{\star}_H\bar{\star}_V $.

\begin{remark}\label{Rmk:CommEpsStarH}
 Note that the operators $\varepsilon_H^v$ and $\bar{\star}_H$ commute. This follows from the relation
 $\varepsilon_H\bar{\star}_H=(-1)^h\bar{\star}_H\varepsilon_H$. Concretely, $\varepsilon_H^v\bar{\star}_H=(-1)^{vh}\bar{\star}_H\varepsilon_H^{v}=\bar{\star}\varepsilon_H^{v},$
where the sign $(-1)^{vh}$  disappears since $vh=v(4k-1-v)=4kv-v(v+1)$ is always even.
 \end{remark}

Now that we have described the Hodge-de Rham operator and the involution $\bar{\star}$ we can use Proposition  \ref{Prop:TransDPsi1} and \eqref{TransTauPsi1} to compute explicitly the signature operator
\begin{equation}\label{Eqn:tildeDTrans}
\tilde{D}_{U_t}^{+}\coloneqq \frac{1}{2}(1-\widehat{\star})\tilde{D}_{U_t}\frac{1}{2}(1+\widehat{\star}).
\end{equation}
To do so we need the commutation relations between $\tilde{A}_0(r)$, $\nu$ and $\alpha$. These can be derived from the equation  $\tilde{D}_{U_t}\widehat{\star}+\widehat{\star}\tilde{D}_{U_t}=0$. More precisely, since
\begin{align*}
\tilde{D}_{U_t}\widehat{\star}=\Psi_1^{-1}D_{U_t}\bar{\star}\Psi_1 
=&\left(\begin{array}{cc}
\partial_r\alpha-r^{-1}\nu\alpha &-\tilde{A}_0(r)\alpha   \\
 \tilde{A}_0(r)\alpha &-\partial_r\alpha-r^{-1}\nu\alpha 
\end{array}
\right),\\
\widehat{\star}\tilde{D}_{U_t}=\Psi_1^{-1}\bar{\star} D_{U_t}\Psi_1 
=&\left(\begin{array}{cc}
-\alpha\partial_r-\alpha r^{-1}\nu &\alpha \tilde{A}_0(r)  \\
-\alpha  \tilde{A}_0(r)& \alpha\partial_r-\alpha r^{-1}\nu
\end{array}
\right),
\end{align*}
we conclude that
\begin{align}
\alpha \tilde{A}_0(r)=&\tilde{A}_0(r)\alpha\label{Eqn:Aalpha},\\
\nu\alpha=&-\alpha\nu.\label{Eqn:alphanu}
\end{align}
\begin{remark}
Observe that the last relation can be derived explicitly from the definition of $\alpha$ since $\bar{\star}_V\nu=-\nu\bar{\star}_V$ and 
$\bar{\star}_H\nu=\nu\bar{\star}_H$ .
\end{remark}
Instead of computing \eqref{Eqn:tildeDTrans} directly we are going to implement a further unitary transformation which makes the involution $\widehat{\star}$ diagonal. Let us consider the projection $P^+(\widehat{\star})$ onto the $+1$-eigenspace of $\widehat{\star}$,
\begin{align*}
P^+(\widehat{\star})\coloneqq 
\frac{1}{2}(1+\widehat{\star})=\frac{1}{2}
\left(\begin{array}{cc}
I & -\alpha\\
-\alpha & I
\end{array}
\right).
\end{align*}
We can diagonalize this projection using the unitary transformation
\begin{align*}
\mathcal{U}\coloneqq 
\frac{1}{\sqrt{2}}
\left(\begin{array}{cc}
I & \alpha\\
-\alpha & I
\end{array}
\right). 
\end{align*}
Indeed, a straightforward computation shows 
\begin{align*}
P\coloneqq 
\mathcal{U}^{-1}P^+(\widehat{\star})\mathcal{U}
=\left(\begin{array}{cc}
I & 0\\
0 & 0
\end{array}
\right),
\end{align*}
from where we deduce
\begin{align*}
\mathcal{U}^{-1}\:\widehat{\star}\:\mathcal{U}=
\left(\begin{array}{cc}
I & 0\\
0& -I
\end{array}
\right).
\end{align*}
As we want to diagonalize the operator of Proposition \ref{Prop:TransDPsi1} using $\mathcal{U}$, we see that we need to take into account the similarity relations, both derived using \eqref{Eqn:Aalpha} and \eqref{Eqn:alphanu}, 
\begin{align*}
\mathcal{U}^{-1}\left(\begin{array}{cc}
\nu & 0\\
0 &-\nu
\end{array}\right) \mathcal{U}=
\left(\begin{array}{cc}
\nu & 0\\
0 &-\nu
\end{array}\right)
\end{align*}
and 
\begin{align*}
\mathcal{U}^{-1}\left(\begin{array}{cc}
0 & -\tilde{A}_0(r) \\
-\tilde{A}_0(r) & 0
\end{array}\right) \mathcal{U}=
\left(\begin{array}{cc}
\tilde{A}_0(r)\alpha & 0\\
0 &-\tilde{A}_0(r)\alpha
\end{array}\right).
\end{align*}
It is therefore convenient to define
\begin{align*}
A_H(r)\coloneqq &\tilde{A}_H(r)\alpha,\\
A_{0V}\coloneqq &\tilde{A}_{0V}\alpha,\\
A_{V}\coloneqq &A_{0V}+\nu,\\
A_0(r)\coloneqq &A_H(r)+r^{-1}A_{0V}\\
A(r)\coloneqq &A_H(r)+r^{-1}A_V.
\end{align*}
Using the relation $\mathcal{U}\gamma=\gamma \mathcal{U}$ we obtain from Proposition \ref{Prop:TransDPsi1} and the computations above the following result. 

\begin{theorem}[{\cite[Theorem 2.5]{B09}}]\label{Thm:TransfDiracOps}
Under the unitary transformation  $\Psi\coloneqq \Psi_1 \mathcal{U}$ the Hodge-de Rham operator $D_{U_t}$ is transformed as
\begin{equation*}
\Psi^{-1}D_{U_t}\Psi=\gamma\left(\frac{\partial}{\partial r}+
\left(\begin{array}{cc}
I& 0\\
0 &-I
\end{array}\right)\otimes A(r)
\right). 
\end{equation*}
In particular the  transformed signature operator is
\begin{equation*}
\Psi^{-1}D^{\textnormal{+}}_{U_t}\Psi=\frac{\partial}{\partial r}+A(r). 
\end{equation*}
\end{theorem}

We now discuss some consequences of this theorem. First, let us consider the slice $\mathcal{F}_r\coloneqq \mathcal{F}\times\{r\}\subset U_t$ for $0<r<t$ with metric  $g^{T\mathcal{F}_r}\coloneqq g^{T_H \mathcal{F}_r}\oplus r ^2 g^{T_V \mathcal{F}_r}$. If we denote by  $D_{\mathcal{F}_r}$ and $\bar{\star}_r$ the corresponding Hodge-de Rham  and chirality operator on $\mathcal{F}_r$, then from the discussion above we can obtain the transformation under $\Psi$, restricted to the slice, of the odd signature operator $\bar{\star}_r D_{\mathcal{F}_r}$ (see\eqref{Eqn:DefOddSignOp}) of $\mathcal{F}_r$. More precisely, consider the unitary transformation

\begin{align*}
\Psi_{1,r}:
\xymatrixcolsep{2cm}\xymatrixrowsep{0.01cm}\xymatrix{
L^2((0,t),\Omega(\mathcal{F}))  \ar[r] & L^2(\mathcal{F}_r)\\
\sigma \ar@{|->}[r] & \pi_r\circ \Psi_1(\sigma,0),
}
\end{align*}
where $\pi_r: L^2(U_t)\longrightarrow L^2(\mathcal{F}_r)$ is the projection.
\begin{coro}[{\cite[Equation (2.37)]{B09}}]\label{Coro:BoundaryOp}
For the slice $\mathcal{F}_r$ we have
\begin{align*}
\Psi_{1,r}^{-1}D_{\mathcal{F}_r}\bar{\star}_r\Psi_{1,r}= A_0(r).
\end{align*}
\end{coro}
\begin{proof}
From the discussion above we easily see that 
\begin{align*}
\Psi^{-1}_1D_{\mathcal{F}_r}\Psi_1=& \tilde{A}_0(r),\\
\Psi^{-1}_1\bar{\star}_r \Psi_1=&\alpha,
\end{align*}
from where the result follows.
\end{proof}

\begin{definition}\label{Def:ConeOpD}
The {\em cone operator} associated to $D_{M_0/S^1}$ is defined by 
\begin{equation*}
D_\textnormal{cone}\coloneqq 
\gamma\left(\frac{\partial}{\partial r}+\frac{1}{r}
\left(\begin{array}{cc}
I& 0\\
0 &-I
\end{array}\right)\otimes A_V
\right),
\end{equation*}
and $A_V$ is called {\em cone coefficient}. 
\end{definition}

\begin{remark}\label{Rmk:ConeCoef}
Using the definition of $d_V$ and $\alpha$ we can derive a more explicit expression of the cone coefficient which separates the horizontal and vertical contributions,
\begin{align*}
A_{V}=&(d_V+d_V^\dagger )\alpha+\nu\\
=&(\varepsilon_H\otimes(d_Y+d_Y^\dagger))(\varepsilon_H^v\bar{\star}_H\otimes \bar{\star}_V)+\nu\\
=&\varepsilon_H^{v+1}\bar{\star}_H\otimes (d_Y\bar{\star}_V +d_Y^\dagger\bar{\star}_V )+\nu\\
=&\varepsilon_H^{v+1}\bar{\star}_H\otimes (d_Y\bar{\star}_V +(-1)^{v+1}\bar{\star}_V d_Y)+\nu.\\
\end{align*}
\end{remark}
Based on the analysis of \cite[Section 1]{B09}, we can deduce the following result, which is of fundamental importance for later purposes. We present an explicit proof for completeness. 
\begin{lemma}[{\cite[Theorem 2.5]{B09}}]\label{Lemma:VertOp1}
The operator $A_{HV}\coloneqq A_H(0)A_V+A_VA_H(0)$ is a first order vertical operator, i.e. it only differentiates with respect to the vertical coordinates. If $A_V$ is invertible then, for $r$ small enough, there exists a constant $C>0$ such that $A(r)^2\geq Cr^{-2}A_V^2$, in particular $A(r)$ is also invertible.
\end{lemma}

\begin{proof}
From \eqref{Eqn:alphanu} it follows that $\tilde{A}_H(0)\alpha\nu=-\nu\tilde{ A}_H(0)\alpha$. In addition, using \eqref{Eqn:Aalpha} we then compute
\begin{align*}
A_{HV}=&A_H(0)A_V+A_VA_H(0)\\
=&\tilde{A}_H(0)\alpha(\tilde{A}_{0V}\alpha+\nu)+(\tilde{A}_{0V}\alpha+\nu)\tilde{A}_H(0)\alpha\\
=&\tilde{A}_H(0)\tilde{A}_{0V}+\tilde{A}_{0V}\tilde{A}_H(0),
\end{align*}
which we know is first order vertical from \cite[Theorem 1.2]{B09}. For the second claim recall that $A(r)=A_H(r)+r^{-1}A_V$, and so for a section $\sigma$, 
\begin{align*}
(A(r)^2\sigma,\sigma)=(A_H(r)^2\sigma,\sigma)+r^{-2}(A_V^2\sigma,\sigma)+r^{-1}(A_{HV}(r)\sigma,\sigma), 
\end{align*}
where $(\cdot,\cdot)$ denotes the $L^2$-inner product in $L^2(\wedge T^* \mathcal{F})$. From the comments in the proof of  \cite[Theorem 2.5(2)]{B09} and the first statement of this lemma one can verify that $A_{HV}(r)\coloneqq A_H(t)A_V+A_VA_{H}(r)$ is also a vertical first order differential operator. Hence there exists a constant $C_1\geq 0$ such that 
\begin{align*}
\norm{A_{HV}(r)A_V^{-1}}\leq C_1\quad\text{for $r\in(0,t_0]$}.
\end{align*}
Thus, if we set $T\coloneqq A_{HV}(r)A_V^{-1}$, then for an appropriate $\tilde{C}_1>0$,
\begin{align}\label{Eqn:proofLemma:VertOp1}
|(A_{HV}(r)\sigma,\sigma)|=|(TA_V\sigma,\sigma)|=|(A_V\sigma,T^*\sigma)|\leq  \tilde{C}_1\norm{A_V \sigma}\norm{\sigma}
\end{align}
On the other hand, since $A_H(r)$ is self-adjoint, then $(A_H(r)^2\sigma,\sigma)=\norm{A_H(r)\sigma}^2\geq 0$, so 
\begin{align*}
(A(r)^2\sigma,\sigma)\geq \frac{1}{r^2}(A_V^2\sigma,\sigma)+\frac{1}{r}(A_{HV}(r)\sigma,\sigma). 
\end{align*}
Hence, from the estimate \ref{Eqn:proofLemma:VertOp1} we conclude that, for  $r$ sufficiently small, $A(r)^2\geq Cr^{-2}A_V^2$
for any constant $0<C<1.$
\end{proof}

\subsubsection{Local description of the Operator $\mathscr{D}$} 
Now we study how the zero order part of the operator $\mathscr{D}_{U_t}$ transforms under $\Psi=\Psi_1\mathcal{U}$. First we calculate the action of  $c(dr)\varepsilon$ on the image of $\Psi_1$,
\begin{align*}
c(dr)\varepsilon\Psi_1(\sigma_1,\sigma_2)=&c(dr)\varepsilon (\pi^*r^{\nu}\sigma_1(r)+ dr\wedge\pi^*r^\nu \sigma_2(r))\\
=&c(dr )(\pi^*r^{\nu}\varepsilon\sigma_1(r)- dr\wedge\pi^*r^\nu \varepsilon\sigma_2(r))\\
=& dr\wedge\pi^*r^{\nu}\varepsilon\sigma_1(r)+\pi^*r^\nu \varepsilon\sigma_2(r).
\end{align*}
Therefore, we obtain
\begin{equation}\label{Trans1TermPsi1}
\Psi^{-1}_1\left(-\frac{1}{2r}c(dr)\varepsilon\right)\Psi_1=
-\frac{1}{2r}\left(
\begin{array}{cc}
 0& \varepsilon \\
\varepsilon & 0
\end{array}
\right)=
\gamma\left(
\frac{1}{2r}\left(
\begin{array}{cc}
 -\varepsilon & 0\\
0&\varepsilon 
\end{array}
\right)\right).
\end{equation}
Now we want to include the transformation $\mathcal{U}$. To begin with, we need the commutation relation between $\varepsilon=\varepsilon_H\otimes\varepsilon_V$ and $\alpha$. To do so we exploit the relation $\bar{\star}\varepsilon=\varepsilon\bar{\star}$ on $U_t$.  Using the transformation
\begin{align*}
\Psi_1^{-1}\varepsilon\Psi_1=\left(
\begin{array}{cc}
\varepsilon & 0 \\
0 &-\varepsilon 
\end{array}
\right)
\end{align*}
 and \eqref{TransTauPsi1} we compute
\begin{align*}
\Psi_1^{-1}\varepsilon\bar{\star}\Psi_1=&\left(
\begin{array}{cc}
\varepsilon & 0 \\
0 &-\varepsilon 
\end{array}
\right)\left(
\begin{array}{cc}
 0& -\alpha\\
-\alpha & 0
\end{array}
\right)
=
\left(
\begin{array}{cc}
 0& -\varepsilon\alpha \\
\varepsilon\alpha& 0
\end{array}
\right),\\
\Psi_1^{-1}\bar{\star}\varepsilon\Psi_1=&
\left(
\begin{array}{cc}
 0& -\alpha\\
-\alpha & 0
\end{array}
\right)\left(
\begin{array}{cc}
\varepsilon & 0 \\
0 &-\varepsilon 
\end{array}
\right)
=
\left(
\begin{array}{cc}
 0& \alpha\varepsilon \\
-\alpha\varepsilon& 0
\end{array}
\right).
\end{align*}
Hence, we must have
\begin{equation}\label{Eqn:EpsilonAlpha}
\alpha\varepsilon=-\varepsilon\alpha. 
\end{equation}

\begin{remark}
This can also be derived from the concrete expression of $\alpha$, namely
\begin{align*}
\varepsilon\alpha=&(\varepsilon_H\otimes\varepsilon_V)(\varepsilon_H^v \bar{\star}_H\otimes\bar{\star}_V)
=\varepsilon_H^{v+1} \bar{\star}_H\otimes \varepsilon_V\bar{\star}_V=\varepsilon^v_H(-1)^h\bar{\star}_H\varepsilon_H\otimes(-1)^v\bar{\star}\varepsilon_V=(-1)^{h+v}\alpha\varepsilon.
\end{align*}
The claim then follows since $h+v=4k-1$.
\end{remark}

\begin{remark}\label{Rmk:AEpsilon}
Observe that $\varepsilon$ commutes with $A(r)$. Indeed, 
\begin{align*}
\varepsilon A(r)=&\varepsilon \tilde{A}_H(r)\alpha +\varepsilon \tilde{A}_{0V}\alpha+\frac{1}{r}\varepsilon\nu\\
=&-\tilde{A}_H(r)\varepsilon \alpha -\tilde{A}_{0V}\varepsilon \alpha+\frac{1}{r}\nu\varepsilon\\
=&\tilde{A}_H(r) \alpha\varepsilon +\tilde{A}_{0V} \alpha\varepsilon+\frac{1}{r}\nu\varepsilon \\
=&A(r)\varepsilon. 
\end{align*}
\end{remark}

Finally we can implement the transformation $\Psi$ using $\gamma\mathcal{U}=\mathcal{U}\gamma$, \eqref{Trans1TermPsi1} and  \eqref{Eqn:EpsilonAlpha}, 
\begin{align*}
\Psi^{-1}\left(-\frac{1}{2r}c(dr)\varepsilon\right)\Psi
=&\gamma\left(\frac{1}{4r}
\left(
\begin{array}{cc}
 I& -\alpha\\
\alpha & I
\end{array}
\right)
\left(
\begin{array}{cc}
 -\varepsilon& 0\\
0 & \varepsilon
\end{array}
\right)
\left(
\begin{array}{cc}
 I& \alpha\\
-\alpha & I
\end{array}
\right)
\right)\\
=&\gamma\left(\frac{1}{4r}
\left(
\begin{array}{cc}
 I& -\alpha\\
\alpha & I
\end{array}
\right)
\left(
\begin{array}{cc}
 -\varepsilon& -\varepsilon\alpha\\
-\varepsilon\alpha & \varepsilon
\end{array}
\right)
\right)\\
=&\gamma\left(\frac{1}{4r}
\left(
\begin{array}{cc}
 I& -\alpha\\
\alpha & I
\end{array}
\right)
\left(
\begin{array}{cc}
 -\varepsilon& \alpha\varepsilon\\
\alpha\varepsilon & \varepsilon
\end{array}
\right)
\right)\\
=&\gamma\left(\frac{1}{2r}
\left(
\begin{array}{cc}
 -\varepsilon& 0\\
0 & \varepsilon
\end{array}
\right)
\right).
\end{align*}

Combining this result with Theorem \ref{Thm:TransfDiracOps} we get the following representation of $\mathscr{D}_{U_t}$. 
\begin{theorem}\label{Thm:LocalDesD1}
Under the unitary transformation $\Psi$ the operator $\mathscr{D}_{U_t}$ transforms as 
\begin{equation*}
\Psi^{-1}\mathscr{D}_{U_t}\Psi=\gamma\left(\frac{\partial}{\partial r}+
\left(\begin{array}{cc}
I& 0\\
0 &-I
\end{array}\right)\otimes \mathscr{A}(r)
\right),
\end{equation*}
where 
\begin{align*}
\mathscr{A}(r)\coloneqq A(r)-\frac{1}{2r}\varepsilon=A(r)-\frac{1}{2r}\varepsilon_H\otimes\varepsilon_V.
\end{align*}
In particular,
\begin{equation*}
\mathscr{D}_{U_t}^+=\frac{\partial}{\partial r}+\mathscr{A}(r). 
\end{equation*}
\end{theorem}

\begin{definition}\label{Def:DVPot}
We define the {\em cone operator} associated to $\mathscr{D}$ by 
\begin{align*}
\mathscr{D}_\textnormal{cone}
\coloneqq\gamma\left(\frac{\partial}{\partial r}+
\frac{1}{r}\left(\begin{array}{cc}
I& 0\\
0 &-I
\end{array}\right)\otimes \mathscr{A}_V\right),
\end{align*}
and we call 
\begin{align}
\mathscr{A}_V\coloneqq A_V-\frac{1}{2}\varepsilon.
\end{align}
its {\em cone coefficient}. 
\end{definition}

\begin{remark}\label{Rmk:VertCondOP}
In contrast with Lemma \ref{Lemma:VertOp1}, the operator $\mathscr{A}_{HV}\coloneqq A_H(0)\mathscr{A}_V+\mathscr{A}_VA_H(0)$ is not a first order vertical. Indeed, if we compute using Remark \ref{Rmk:AEpsilon}, 
\begin{align*}
\mathscr{A}_{HV}=&A_H(0)\mathscr{A}_V+\mathscr{A}_VA_H(0)\\
=&A_H(0)\left(A_V-\frac{\varepsilon}{2}\right)+\left(A_V-\frac{\varepsilon}{2}\right)A_H(0)\\
= &A_{HV}-\frac{1}{2}(\varepsilon A_H(0)+A_H(0)\varepsilon)\\
=&A_{HV}-\varepsilon A_H(0),
\end{align*}
we see that the first order horizontal operator $\varepsilon A_H(0)$ appears in the formula. 
\end{remark}

\subsection{Spectral decomposition of cone coefficient}\label{Section:SpectralDecomp}

Our next objective is to describe the spectrum of the cone coefficient $\mathscr{A}_V$. To begin we recall how to calculate the spectrum for the cone coefficient $A_V$ following \cite[Section 3]{B09}. We do this with two purposes: to fill out some details in the computations and because the methods used can be adapted to find an analogous result for the operator $\mathscr{A}_V$. 

\subsubsection{Spectral decomposition of $A_V$}
For a fixed point $x\in F$ the corresponding fiber $Y_x$ of the fibration $\pi_{\mathcal{F}}:\mathcal{F}\longrightarrow F$ is a closed oriented Riemannian manifold with metric $g^{T_V\mathcal{F}}$ (more precisely, its restriction to $TY_x$). From Hodge theory we know that the fiber Laplacian (the dependence of $x$ is suppressed from the notation) 
$\Delta_Y\coloneqq d_Yd_Y^\dagger+d_Y^\dagger d_Y,$
induces a decomposition $\Omega^j(Y)=\mathcal{H}^j(Y)\oplus\Omega^j_\text{cl}(Y)\oplus \Omega^j_\text{ccl}(Y)$ where 
\begin{align*}
\mathcal{H}^j(Y)\coloneqq &\ker(\Delta^j_Y)\coloneqq \ker\left(\Delta|_{\Omega^j(Y)}\right),\\
\Omega^j_\text{cl}(Y)\coloneqq &\{\beta\in\Omega^j(Y)\:|\:d_Y\beta=0\},\\
\Omega^j_\text{ccl}(Y)\coloneqq &\{\beta\in\Omega^j(Y)\:|\:d^\dagger_Y\beta=0\}.
\end{align*}
Let us denote by $\Delta^j_{Y,\text{cl/ccl}}\coloneqq \Delta^j_Y|_{\Omega^j_\text{cl/ccl}(Y)}$ the restriction of $\Delta_Y$ to closed and co-closed $j$-forms respectively and by $E^j_\text{cl/ccl}(\lambda)$ the eigenspace of $\Delta^j_{Y,\text{cl/ccl}}$ for the eigenvalue $\lambda>0$. Recall from Remark \ref{Rmk:ConeCoef} that the cone coefficient $A_V$ is explicitly given by 
\begin{align*}
A_V=\varepsilon_H^{v+1}\bar{\star}_H\otimes (d_Y\bar{\star}_V +(-1)^{v+1}\bar{\star}_V d_Y)+\nu.
\end{align*}
In particular note that $(d_Y\bar{\star}_V +(-1)^{v+1}\bar{\star}_V d_Y)^2=\left((d_Y+d^\dagger_Y)\bar{\star}_V\right)^2=\Delta_Y$.\\

We now introduce the following spaces which, from the expression above, are invariant under the action of the operator $A_V$,  
\begin{align*}
\widetilde{\Omega}^j_\text{h}\coloneqq &\mathcal{H}^j(Y)\oplus\mathcal{H}^{v-j}(Y),\\
\widetilde{\Omega}^j_{\text{cl}}(Y)\coloneqq &\Omega^j_{\text{cl}}(Y)\oplus \Omega^{v+1-j}_{\text{cl}}(Y),\\
\widetilde{\Omega}^j_{\text{ccl}}\coloneqq &\Omega^{j-1}_{\text{ccl}}(Y)\oplus \Omega^{v-j}_{\text{ccl}}(Y),\\
F^j_\text{cl}(\lambda)\coloneqq &E^j_\text{cl}(\lambda)\oplus E^{v+1-j}_\text{cl}(\lambda),\\
F^j_\text{ccl}(\lambda)\coloneqq &E^{j-1}_\text{ccl}(\lambda)\oplus E^{v-j}_\text{ccl}(\lambda).
\end{align*}
This allow us to consider the restriction operators:
\begin{itemize}
\item On harmonic $j$-(vertical) forms:
\begin{align*}
A_{V,\text{h}}^j\coloneqq A_V\bigg{|}_{\widetilde{\Omega}^j_\text{h}}
=\left(
\begin{array}{cc}
j-\frac{v}{2} & 0\\
0 & -\left(j-\frac{v}{2}\right).
\end{array}
\right)
\end{align*}
\item On closed (vertical) forms:
\begin{align*}
A_{V,\text{cl}}-\frac{1}{2}\coloneqq \left(A_V-\frac{1}{2}\right)\bigg{|}_{\widetilde{\Omega}^j_\text{cl}}
=\left(
\begin{array}{cc}
j-\frac{v+1}{2} & \varepsilon_H^{v+1}\bar{\star}_H\otimes d_Y \bar{\star}_V \\
 \varepsilon_H^{v+1}\bar{\star}_H\otimes d_Y\bar{\star}_V & -\left(j-\frac{v+1}{2}\right)
\end{array}
\right),
\end{align*}
where we have used
\begin{align*}
(v+1-j)-\frac{v}{2}-\frac{1}{2}=-\left(j-\frac{v+1}{2}\right).
\end{align*}
\item On co-closed (vertical) forms:
\begin{align*}
A_{V,\text{ccl}}+\frac{1}{2}\coloneqq \left(A_V+\frac{1}{2}\right)\bigg{|}_{\widetilde{\Omega}^j_\text{ccl}}
=\left(
\begin{array}{cc}
j-\frac{v+1}{2} & (-1)^{v+1}\varepsilon_H^{v+1}\bar{\star}_H\otimes \bar{\star}_V d_Y \\
(-1)^{v+1}\varepsilon_H^{v+1}\bar{\star}_H\otimes \bar{\star}_V d_Y& -\left(j-\frac{v+1}{2}\right)
\end{array}
\right).
\end{align*}
\end{itemize}
It is important to emphasize  the fact that if we take the square of these restriction operators we will obtain, in this matrix form representation, diagonal operators. This is because diagonal entries have opposite signs and the off-diagonal terms are equal. In order to compute the diagonal operators obtained in the squaring process it is necessary to calculate $( \varepsilon_H^{v+1}\bar{\star}_H\otimes d_Y \bar{\star}_Y)^2$ and $( \varepsilon_H^{v+1}\bar{\star}_H\otimes\bar{\star}_Y d_Y )^2$. 
For the first one we expand
\begin{align*}
( \varepsilon_H^{v+1}\bar{\star}_H\otimes d_Y \bar{\star}_Y)^2= \varepsilon_H^{v+1}\bar{\star}_H\varepsilon_H^{v+1}\bar{\star}_H\otimes d_Y \bar{\star}_V d_Y\bar{\star}_V.
\end{align*}
Let us study the horizontal and the vertical contributions separately. For the horizontal part we use Remark  \ref{Rmk:CommEpsStarH} and $\varepsilon_H\bar{\star}_H=(-1)^h\bar{\star}_H \varepsilon_H$ to calculate
\begin{align*}
\varepsilon_H^{v+1}\bar{\star}_H\varepsilon_H^{v+1}\bar{\star}_H=\varepsilon_H^{v+1}\bar{\star}_H\varepsilon_H \varepsilon_H^{v}\bar{\star}_H
=\varepsilon^{v+1}_H\bar{\star}_H\varepsilon_H\bar{\star}_H\varepsilon^{v}_H=(-1)^h.
\end{align*}
On the other hand $d_Y \bar{\star}_V d_Y\bar{\star}_V=(-1)^{v+1}d_Yd_Y^\dagger$, and so
\begin{align*}
( \varepsilon_H^{v+1}\bar{\star}_H\otimes d_Y \bar{\star}_Y)^2=(-1)^{h+v+1}d_Yd_Y^\dagger=d_Yd_Y^\dagger.
\end{align*}
An analogous computation shows that 
\begin{align*}
( \varepsilon_H^{v+1}\bar{\star}_H\otimes\bar{\star}_Y d_Y )^2=d_Y^\dagger d_Y.
\end{align*}
Hence, we obtain on closed forms
\begin{align}\label{Enq:Acl}
\left(A^j_{V,\text{cl}}-\frac{1}{2}\right)^2
=\left(
\begin{array}{cc}
\Delta^j_{Y,\text{cl}}+\left(j-\frac{v+1}{2}\right)^2 & 0\\
0 & \Delta^{v+1-j}_{Y,\text{cl}}+\left(j-\frac{v+1}{2}\right)^2
\end{array}\right).
\end{align}
Similarly for co-closed forms,
\begin{align}\label{Enq:Accl}
\left(A^j_{V,\text{ccl}}+\frac{1}{2}\right)^2
=\left(
\begin{array}{cc}
\Delta^{j-1}_{Y,\text{ccl}}+\left(j-\frac{v+1}{2}\right)^2 & 0\\
0 & \Delta^{v-j}_{Y,\text{ccl}}+\left(j-\frac{v+1}{2}\right)^2
\end{array}\right).
\end{align}
Finally on harmonic forms we obviously have
\begin{align}\label{Enq:Ah}
\left(A_{V,\text{h}}^j\right)^2
=\left(
\begin{array}{cc}
\left(j-\frac{v}{2}\right)^2 & 0\\
0 & \left(j-\frac{v}{2}\right)^2
\end{array}\right).
\end{align}
Using these expressions one can give a complete description of the spectrum of $A_V$. 
\begin{theorem}[{\cite[Theorem 3.1]{B09}}]\label{Thm:SpectralDecomp}
The spectral resolution of $A_V$ is given by:
\begin{enumerate}
\item $A^j_{V,\textnormal{h}}$ has eigenspaces $\mathcal{H}^j(Y)$ and $\mathcal{H}^{v-j}(Y)$ with eigenvalues
\begin{align*}
\pm\left(j-\frac{v}{2}\right).
\end{align*}
\item For each $\lambda\in\spec(\Delta^j_{\textnormal{cl}})\backslash\{0\}$ the operator $A^j_{V,\textnormal{cl}}$ has two eigenspaces in $F^j_\textnormal{cl}(\lambda)$ with eigenvalues
\begin{align*}
{\mu}^j_{\textnormal{cl},\pm}(\lambda)\coloneqq &\frac{1}{2}\pm\sqrt{\lambda+\left(j-\frac{v+1}{2}\right)^2},
\end{align*}
and multiplicities $m^j_{\textnormal{cl},\pm}(\lambda)$.
\item For each $\lambda\in\spec(\Delta^j_{\textnormal{ccl}})\backslash\{0\}$ the operator $A^j_{V,\textnormal{ccl}}$ has two eigenspaces in $F^j_\textnormal{ccl}(\lambda)$ with eigenvalues
\begin{align*}
{\mu}^j_{\textnormal{ccl},\pm}(\lambda)\coloneqq &-\frac{1}{2}\pm\sqrt{\lambda+\left(j-\frac{v+1}{2}\right)^2},
\end{align*}
and multiplicities $m^j_{\textnormal{ccl},\pm}(\lambda)$.
\item For $\lambda>0$, the four eigenvalues of $A_V$ in $F^j_\textnormal{cl}(\lambda)\oplus F^j_\textnormal{ccl}(\lambda)$ have the common multiplicity $2\:\textnormal{dim} (E^j_{\textnormal{cl}}(\lambda))$.
\end{enumerate}
\end{theorem}

\begin{remark}[Vertical harmonic eigenvalues]\label{Rmk:ZeroEigenValueAV}
Note that since in our particular case of study the fiber is $Y=\mathbb{C}P^N$  and 
\begin{align*}
H^i(\mathbb{C}P^N;\mathbb{R})=
\begin{cases}
0,\quad \text{for $i$ odd},\\
\mathbb{R},  \quad \text{for $i=2j$ with $0\leq j\leq N$},
\end{cases}
\end{align*}
the eigenvalues of $A_V$ on vertical harmonic forms are $2j-N$. If $N$ is odd (Witt case) we see that the zero eigenvalue can never occur, whereas for $N$ even (non-Witt case) the zero eigenvalue appears on vertical harmonic forms of degree $2j=N$.
\end{remark}

\begin{remark}[{\cite[Theorem 3.1(4)]{B09}}]
In \cite{B09}, the fibration can have arbitrary vertical dimension $v$. If this dimension is  odd then there are two more eigenspaces for the operator 
\begin{align*}
\left(A^{(v+1)/2}_{V,\textnormal{cl}}-\frac{1}{2}\right)\oplus \left(A^{(v+1)/2}_{V,\textnormal{ccl}}+\frac{1}{2}\right):
\Omega^{\frac{v+1}{2}}_{\text{cl}}(Y)\oplus \Omega^{\frac{v-1}{2}}_{\text{ccl}}(Y)
\longrightarrow
\Omega^{\frac{v+1}{2}}_{\text{cl}}(Y)\oplus \Omega^{\frac{v-1}{2}}_{\text{ccl}}(Y),
\end{align*}
in $E^{(v+1)/2}_{\textnormal{cl}}\oplus E^{(v-1)/2}_{\textnormal{ccl}}$ with eigenvalues $\pm\sqrt{\lambda}$, as can be easily seen from \eqref{Enq:Acl} and \eqref{Enq:Acl}. 
\end{remark}

\subsubsection{Spectral decomposition of $\mathscr{A}_V$}
Now we want to compute in a similar manner the eigenvalues of the cone coefficient $\mathscr{A}_V$. We define the restriction operators analogously 
\begin{align*}
\mathscr{A}^j_{V,\text{h}}\coloneqq &\mathscr{A}_V|_{\widetilde{\Omega}_{\text{h}}^j},\\
\mathscr{A}^j_{V,\text{cl}}-\frac{1}{2}\coloneqq &\left(\mathscr{A}_V-\frac{1}{2}\right)\bigg{|}_{\widetilde{\Omega}_{\text{cl}}^j},\\
\mathscr{A}^j_{V,\text{ccl}}+\frac{1}{2}\coloneqq &\left(\mathscr{A}_V+\frac{1}{2}\right)\bigg{|}_{\widetilde{\Omega}_{\text{ccl}}^j}.
\end{align*}
The additional terms of these operators, with respect to the $A_V$, is the potential
\begin{align*}
-\frac{1}{2}\varepsilon=-\frac{1}{2}\varepsilon_H\otimes \varepsilon_V. 
\end{align*}
The corresponding restrictions are
\begin{align*}
-\frac{1}{2}{\varepsilon}\bigg{|}_{\widetilde{\Omega}^j_\text{h}}=&
\left(\begin{array}{cc}
-\frac{1}{2}\varepsilon_H\otimes (-1)^j & 0\\
0 &-\frac{1}{2}\varepsilon_H\otimes (-1)^{v-j}
\end{array}\right),\\
-\frac{1}{2}{\varepsilon}\bigg{|}_{\widetilde{\Omega}^j_\text{cl}}=&
\left(\begin{array}{cc}
-\frac{1}{2}\varepsilon_H\otimes (-1)^j & 0\\
0 &-\frac{1}{2}\varepsilon_H\otimes (-1)^{v+1-j}
\end{array}\right),\\
-\frac{1}{2}{\varepsilon}\bigg{|}_{\widetilde{\Omega}^j_\text{ccl}}=&
\left(\begin{array}{cc}
-\frac{1}{2}\varepsilon_H\otimes (-1)^{j-1} & 0\\
0 &-\frac{1}{2}\varepsilon_H\otimes (-1)^{v-j}
\end{array}\right).
\end{align*}
Using \eqref{Enq:Acl}, \eqref{Enq:Accl}, \eqref{Enq:Ah} and the fact that for our case of interest $v=2N$, we obtain the following expressions for restriction operators
\begin{align*}
\mathscr{A}^j_{V,\text{h}}
\coloneqq  A_V\bigg{|}_{\widetilde{\Omega}^j_\text{h}}&=
\left(
\begin{array}{cc}
j-N-\frac{1}{2}\varepsilon_H\otimes(-1)^j & 0\\
0 & -\left(j-N+\frac{1}{2}\varepsilon_H\otimes(-1)^j\right)
\end{array}
\right),\\
\mathscr{A}^j_{V,\text{cl}}-\frac{1}{2}\coloneqq 
\left(\mathscr{A}_V-\frac{1}{2}\right)\bigg{|}_{\widetilde{\Omega}^j_\text{cl}}
&=\left(
\begin{array}{cc}
j-\frac{2N+1}{2}-\frac{1}{2}\varepsilon_H\otimes(-1)^j & \varepsilon_H\bar{\star}_H\otimes d_Y \bar{\star}_V \\
 \varepsilon_H\bar{\star}_H\otimes d_Y \bar{\star}_V & -\left(j-\frac{2N+1}{2}-\frac{1}{2}\varepsilon_H\otimes(-1)^j\right)
\end{array}
\right),\\
\mathscr{A}^j_{V,\text{ccl}}+\frac{1}{2}\coloneqq 
\left(\mathscr{A}_V+\frac{1}{2}\right)\bigg{|}_{\widetilde{\Omega}^j_\text{ccl}}
&=\left(
\begin{array}{cc}
j-\frac{2N+1}{2}-\frac{1}{2}\varepsilon_H\otimes(-1)^{j-1} & -\varepsilon_H\bar{\star}_H\otimes \bar{\star}_V d_Y \\
-\varepsilon_H\bar{\star}_H\otimes  \bar{\star}_V d_Y & -\left(j-\frac{2N+1}{2}-\frac{1}{2}\varepsilon_H\otimes(-1)^{j-1}\right)
\end{array}
\right).
\end{align*}
As before, from the structure of these formulas, we can easily calculate their squares,
\begin{align*}
\left(\mathscr{A}^j_{V,\text{h}}\right)^2
&=\left(
\begin{array}{cc}
\left(j-N-\frac{1}{2}\varepsilon_H\otimes(-1)^j \right)^2& 0\\
0 & \left(j-N+\frac{1}{2}\varepsilon_H\otimes(-1)^j\right)^2
\end{array}
\right),\\
\left(\mathscr{A}^j_{V,\text{cl}}-\frac{1}{2}\right)^2
&=\left(
\begin{array}{cc}
\Delta^j_\text{cl}+\left(j-\frac{2N+1}{2}-\frac{1}{2}\varepsilon_H\otimes (-1)^j\right)^2 & 0\\
0 & \Delta^{2N+1-j}_\text{cl}+\left(j-\frac{2N+1}{2}-\frac{1}{2}\varepsilon_H\otimes (-1)^j\right)^2
\end{array}\right),\\
\left(\mathscr{A}^j_{V,\text{ccl}}+\frac{1}{2}\right)^2
&=\left(
\begin{array}{cc}
\Delta^{j-1}_\text{ccl}+\left(j-\frac{2N+1}{2}-\frac{1}{2}\varepsilon_H\otimes (-1)^{j-1}\right)^2 & 0\\
0 & \Delta^{2N-j}_\text{ccl}+\left(j-\frac{2N+1}{2}-\frac{1}{2}\varepsilon_H\otimes (-1)^{j-1}\right)^2
\end{array}\right).
\end{align*}

\begin{theorem}\label{Thm:SpecDec}
The operator $\mathscr{A}_V$ preserves the eigenspaces of $A_V$ and it has the eigenvalues are described as follows:   
\begin{enumerate}
\item For harmonic $2j$-(vertical) forms it has eigenvalues
\begin{align*}
2j-N\pm\frac{1}{2}.
\end{align*}
\item For each $\lambda\in\spec(\Delta^j_{\textnormal{cl}})\backslash\{0\}$, $\mathscr{A}^j_{V,\textnormal{cl}}$ has four eigenspaces in $F^j_{\textnormal{cl}}(\lambda)$, with eigenvalues 
\begin{align*}
\tilde{\mu}^j_{\textnormal{cl},\pm,\pm}(\lambda)\coloneqq &\frac{1}{2}\pm\sqrt{\lambda+\left(j-N-\frac{1\pm 1}{2}\right)^2}.
\end{align*}
\item For each $\lambda\in\spec(\Delta^j_{\textnormal{ccl}})\backslash\{0\}$, $\mathscr{A}^j_{V,\textnormal{ccl}}$ has four eigenspaces in $F^j_{\textnormal{ccl}}(\lambda)$, with eigenvalues 
\begin{align*}
\tilde{\mu}^j_{\textnormal{ccl},\pm,\pm}(\lambda)\coloneqq &-\frac{1}{2}\pm\sqrt{\lambda+\left(j-N-\frac{1\pm 1}{2}\right)^2}.
\end{align*}
\end{enumerate}
\end{theorem}

Observe that for each $j=0,...,2N$ the quantity 
\begin{align*}
n_{j,\pm}\coloneqq \left(j-N-\frac{1\pm 1}{2}\right)^2
\end{align*}
is always an integer. Consequently, if $n_{j,\pm}\neq 0$, then $\sqrt{\lambda+n_{j,\pm}^2}>1$ for all $\lambda>0$ and therefore 
\begin{align*}
\pm \frac{1}{2}\pm n_{j,\pm}\not\in\left(-\frac{1}{2},\frac{1}{2}\right).
\end{align*}
On the other and, observe from Theorem \ref{Thm:SpecDec}(1),
\begin{align*}
2j-N\pm\frac{1}{2}\notin \left(-\frac{1}{2},\frac{1}{2}\right).
\end{align*}
Thus, if $\spec(\Delta_Y)\backslash \{0\} \subset (1,\infty)$, then $|\mathscr{A}_V|\geq 1/2$. This condition can always be achieved, in view of Remark \ref{Rmk:ScaleHodgeStar}, by rescaling the vertical metric.

\begin{example}[Spectrum of the link]
Let us illustrate these results in the first example treated in Section \ref{Sect:S5}. Recall that we studied a semi-free circle action on $S^5$ whose fixed point set was $S^1$. Close to the singular stratum we saw that $M/S^1$ is isometric to a mapping cylinder of a Riemannian fribration with fiber $S^2\cong \mathbb{C}P^1$. In view of \eqref{Eqn:MetricFSS2} and Remark \ref{Rmk:SpectrumLaplFSS2} we see that the spectrum of the induced vertical Laplacian $\Delta_Y$ on the link is of the form $4k(k+1)$ for $k\in\mathbb{N}_0$ (\cite[Proposition III.C.1]{BGM}). In particular, the condition $\spec(\Delta_Y)\backslash \{0\} \subset (1,\infty)$ is satisfied.  
\end{example}

\section{The parametrix construction}\label{Sect:Param}

In this chapter we describe how to construct a parametrix for the operator $\mathscr{D}$ by adapting the analogous construction for the Hodge-de Rham operator $D$ treated in detail in \cite[Section 4]{B09}. In this work Br\"uning  implements a variant of the parametrix construction developed in the sequence of seminal papers \cite{BS85}, \cite{BS87}, \cite{BS88} and \cite{BS91}. We show that the additional potential entering in the operator $\mathscr{D}$ adapts well to this construction because of its particular form. Concretely, in the first part of this chapter we revise the main ideas concerning the parametrix construction as a Neumann series described in \cite[Section 4]{B09}. A motivating example for this procedure can be found in \cite[Section 4]{BS85}. As in that case, the parametrix allows us to construct a self-adjoint extension of the operator $D_\text{min}$. In order to obtain the desired parametrix one need to deal with the large and small eigenvalues of the cone coefficient $A_V$ separately. Finally we show that the same procedure can be applied to the operator $\mathscr{D}$ since the corresponding cone coefficient $\mathscr{A}_V$ processes the same (anti)-commutation relations as $A_V$ which are required in the construction.
Of course, we do not intend to give all the details of the treatment in \cite[Section 4]{B09}, but rather explain the strategy and really understand why these methods apply also to the operator $\mathscr{D}$. 

\subsection{A self-adjoint extension for the signature operator}

Let $D_\text{min}$ denote the minimal extension of the Hodge-de Rham operator $D=D_{M_0/S^1}$ defined on $\Omega_c(M_0/S^1)$ and let $D_\text{max}=(D_\text{min})^*$ denote its maximal extension. In general $D_\text{min} \neq D_\text{max}$, that is, $D$ is in general not essentially self-adjoint. Nevertheless, we will describe how to construct a self-adjoint extension of $D_\text{min}$ by means of an operator family $\mathcal{G}(\mu,D)$ defined for $\mu\in\mathbb{R}$ with $|\mu|$ large enough. More precisely assume the operator family $\mathcal{G}(\mu,D)$ satisfies the  conditions:
\begin{align}\label{CondExt}
&\bullet\text{The range of the operators $\mathcal{G}(\mu,D)$ is contained in $\dom(D_\text{max})$.}\notag \\
&\bullet\text{For all $\sigma\in L^2(\wedge T^*(M_0/S^1))$ we have $(D-i\mu)\mathcal{G}(\mu,D)\sigma=\sigma$.}\\
&\bullet\text{All the operators map into a common domain on which $D$ is symmetric.}\notag
\end{align}
This domain defines a self-adjoint extension of $D$  with resolvent $G(\mu,D)$ by \cite[Problem V.3.17]{KATO} and \cite[Proposition 3.11]{S12}. We denote such domain by $\dom({D_\delta})$. Our aim is therefore to understand how construct such an operator family. It turns out that this family is constructed as a as a pseudo-differential operator on the base $F$ with operator valued symbol.

\subsubsection{General description}
To begin we need to slightly modify our point of view on the geometry near to the fixed point set $F$. It is convenient to regard differential forms on $M_0/S^1$ close to $F$ as sections of an infinite dimensional Hilbert bundle $\mathcal{E}$ over $F$ (see \cite[Section 1]{B09}). For each $x\in F$, the fiber $\mathcal{E}_x$ is the Hilbert space $\mathcal{E}_x\coloneqq L^2(Y_x,(\wedge T^*\mathcal{F})|_x)$.
This Hilbert bundle is locally trivial, i.e for a fixed $x_0\in F$ there exist a small open neighborhood $W_{x_0}\subset F$ and a trivialization map 
$\mathcal{E}{|}_{W_{x_0}}\longrightarrow B^{\mathbb{R}^h}_\delta(0)\times \mathcal{E}_{x_0}$,
where $B^{\mathbb{R}^h}_\delta(0)$ denotes the open ball of radius $\delta>0$ around $0\in\mathbb{R}^h$. Under this trivialization we can identify $C_c(W_{x_0},\mathcal{E}{|}_{W_{x_0}})$ with $C_c(B^{\mathbb{R}^h}_\delta(0), \mathcal{E}_{x_0})$. In view of the construction of a pseudo-differential operator from a symbol (see \cite[Chapter III.3]{LM89}), we define a {\em local parametrix} in terms of an operator symbol $G(\mu,x,\beta)$, still to be defined, as
\begin{align}\label{Eqn:G1}
\mathcal{G}_1(\mu,D,x_0)\coloneqq
\int_{\mathbb{R}^h}e^{i\inner{x}{\beta}}G(\mu,x,\beta)\frac{d\beta}{(2\pi)^h},
\end{align}
where $(x,\beta)\in T^*W_{x_0}$. The requirements to be satisfied by the symbols $G(\mu,x,\beta)$ are:
\begin{itemize}
\item One can patch these local parametrices together using a well-adapted  partition of unity to obtain a global parametrix $\mathcal{G}_1(\mu,D)$.
\item The operator norm of $\mathcal{R}\coloneqq I-(D-i\mu)\mathcal{G}_1$ decays like $|\mu|^{-1}$. 
\end{itemize}
The operators $G(\mu,x,\beta)$ will be constructed as resolvent kernels of an operator $\widehat{D}$ which differs from $D$ by zero order horizontal terms, as described below. This means that we will seek for the defining  relation $(\widehat{D}-i\mu)\mathcal{G}_1(\mu,D)=I$.
If we achieve obtaining such  symbols then, by the decaying property above, we would obtain the desired operator family as a Neumann series 
\begin{align}\label{Eqn:NeumSer}
\mathcal{G}(\mu,D)\coloneqq \mathcal{G}_1\sum_{j=0}^{\infty}\mathcal{R}^j.
\end{align}
Indeed, 
$
(D-i\mu)\mathcal{G}(\mu,D)=(D-i\mu)\mathcal{G}_1\sum_{j=0}^{\infty}\mathcal{R}^j=(I-\mathcal{R})\sum_{j=0}^{\infty}\mathcal{R}^j=I.
$

\subsubsection{Procedure to construct $G(\mu,x,\beta)$}
Motivated by Definition \ref{Def:ConeOpD} we consider, for $(x,\beta)\in T^*W_{x_0}$ fixed , the operator 
\begin{align}\label{Eqn:DefDV}
{D}_V(x)\coloneqq &\gamma\left(\frac{d}{dr}+\frac{1}{r}\widetilde{{A}}(x)\right)\\
\coloneqq &\varepsilon_H\otimes
\left(
\begin{array}{cc}
0 & - I\\
I & 0
\end{array}
\right)
\left(
\frac{d}{dr} +\frac{1}{r}
\left(
\begin{array}{cc}
D_{Y_x}\alpha_{Y_x}+\nu & 0\\
0 & -(D_{Y_x}\alpha_{Y_x}+\nu)
\end{array}
\right)
\right)\notag,
\end{align}
where $D_{Y_x}$ and $\alpha_{Y_x} $ are the Hodge-de Rham and chirality operators of the fiber $Y_x$, discussed in Section \ref{Section:LocalDescrHdROp}. This operator is initially defined on
$
C^1_c((0,\infty), H)\subset L^2((0,\infty),H),
$
where $H\coloneqq H_{x_0}\coloneqq \wedge T^*_{x_0}F\otimes\mathbb{C}^2\otimes L^2(\wedge T^* Y_{x_0})$. Observe that this operator is a regular singular operator in the sense of Br\"uning and Seeley (see Appendix \ref{App:RSO}). For this kind of operators there is a complete characterization of their closed extensions of in terms of the spectrum of the corresponding cone coefficient. 

\begin{lemma}[Lemma \ref{BS88Lemma3.2}, {\cite[Lemma 4.1(3)]{B09}}]\label{Lemma:DVmin}
For the operator ${D}_V(x)$ the following conditions hold:
\begin{enumerate}
\item Any $\sigma\in \dom({D}_{V,\textnormal{max}}(x))$ has a representation of the form
\begin{align*}
\sigma(r)=\sum_{\substack{\lambda\in\spec(\widetilde{A}(x)),\\|\lambda|<1/2}} r^{-\lambda}c_\lambda(\sigma)+O_\sigma(r^{1/2}|\log r|),
\end{align*} 
as $r\longrightarrow 0$ for certain linear functionals $c_\lambda$.
\item Each closed extension of ${D}_{V,\textnormal{min}}(x)$ is determined by a linear relations between the coefficients $c_\lambda$ for $|\lambda|<1/2$.
\item An element $\sigma\in \dom({D}_{V,\textnormal{min}}(x))$ if, and only if, $\norm{\sigma(r)}_H=O_\sigma(r^{1/2}|\log r|)$ as $r\longrightarrow 0$.
\end{enumerate}
\end{lemma}

\begin{remark}
By Theorem \ref{BS88Thm3.2} we know that $\dim\left(\dom(D_{V,\text{max}})(x)/\dom(D_{V,\text{min}})(x)\right)$ is finite dimensional for each fixed $x$ . However, for non-isolated singularities the dimension of $\dim\left(\dom(D_{\text{max}})/\dom(D_{\text{min}})\right)$ can be infinitely dimensional (\cite{AG16}).
\end{remark}

An important algebraic property of ${D}_V(x)$ is given by the following lemma.
\begin{lemma}[{\cite[Lemma 1.1(2)]{B09}}]\label{Lemma 1.1(2)}
For each $(x,\beta)\in T^* W_{x_0}$, the operator $D_V(x)$ satisfies
\begin{align*}
{D}_V(x) c(\beta)+c(\beta){D}_V(x)=0.
\end{align*}
\end{lemma}

\begin{coro}[{\cite[Equation (4.4)]{B09}}]\label{Coro:GraphNormDV}
For all $\sigma(x)\in\dom(D_{V,\delta}(x))$,
\begin{align*}
\norm{({D}_V(x)+ic(\beta)-i\mu)\sigma(x)}^2_{{D}_V}=&\norm{{D}_V(x)\sigma(x)}^2_{{D}_V} +(|\mu|^2+\norm{\beta}_x^2)\norm{\sigma(x)}^2_{{D}_V},
\end{align*}
where $\norm{\beta}^2_x= g^{T^*F}(x)(\beta,\beta)$.
\end{coro}

\begin{proof}
The claim follows from \cite[Lemma1.1(2)]{B09} and from the relation 
\begin{align*}
\inner{{D}_V(x)\sigma(x)}{(ic(\beta)-i\mu)\sigma(x)}=&\inner{(ic(\beta)+i\mu){D}_V(x)\sigma(x)}{\sigma(x)}\\
=&-\inner{{D}_V(x)(ic(\beta)-i\mu)\sigma(x)}{\sigma}\\
=&-\inner{(ic(\beta)-i\mu)\sigma(x)}{{D}_V(x)\sigma(x)}.
\end{align*}
\end{proof}

The operator $\widehat{D}$ mentioned above is construed using the operator symbol
\begin{align*}
D_{V}(b)+ic(\beta)-i\mu.
\end{align*}
In view of Lemma \ref{Lemma:DVmin}, if the spectrum of $\widetilde{A}(x)$ satisfies the relation 
\begin{align}\label{Eqn:CondSpectTildeAx}
\spec(\widetilde{A}(x))\cap \left(-\frac{1}{2},\frac{1}{2}\right)=\emptyset,
\end{align}
then $D_V(x)$ is essentially self-adjoint and  $G(\mu, x,\beta)$ can be defined just as its resolvent. Observe that this spectral condition can vary as the point $x\in F$ varies. If \eqref{Eqn:CondSpectTildeAx} is not satisfied then one needs to work harder: The operator $G(\mu,x,\beta)$ is constructed as the resolvent of $D_{V,\text{max}}(x)+ic(\beta)$ restricted to certain domain. Concretely, we want $\ran(G(\mu,x,\beta))\subset \dom(D_{V,\text{max}}(x))$ and  
\begin{align}\label{Cond:ExtDV}
(D_{V,\text{max}}(x)+ic(\beta)-i\mu)G(\mu,x,\beta)\sigma=\sigma, \quad \text{for all $\sigma\in L^2((0,\infty),H)$}.
\end{align}
As before, these relations define a self-adjoint extension of $D_V(x)$. We denote the domain of this extension by $\dom(D_{V,\delta}(x))$.\\

From this discussion we see that it is necessary to split the problem into large and small eigenvalues. The following splitting lemma makes this approach feasible. 
 
\begin{lemma}[{Splitting Lemma}, {\cite[Lemma 1.1]{BS91}}]\label{Lemma:Spliting}
For a sufficiently small neighborhood $W_{x_0}$ of $x\in F$, the spectral projection $Q_> \coloneqq Q_{|\lambda|\geq \Lambda}(\widetilde{{A}}(x))$, does not depend on $x\in W_{x_0}$ for some $\Lambda\geq 1$ with the property that $\Lambda\notin\spec(\widetilde{{A}}(x))$ for all $x\in W_{x_0}$.
\end{lemma}

\begin{remark}
Let us comment on the notation used above. Given a self-adjoint operator $A:\dom(A)\subset H\longrightarrow H$ on a separable Hilbert space $H$ we can consider, for each Borel subset $J\subseteq\mathbb{R}$, the corresponding spectral projection  $Q_J\coloneqq Q_J(A)$, see for instance to \cite[Section VI.5.1]{KATO}, \cite[Section 4]{S12}. For example $Q_{\{0\}}$ is precisely the projection onto the kernel of $A$. These spectral projections will be discussed again in the subsequent chapters. 
\end{remark}

\subsubsection{Large eigenvalues}
First we handle the large eigenvalues. In view of Lemma \ref{Lemma:DVmin} we define the operator 
\begin{align*}
{D}_V(x)_>\coloneqq &\gamma\left(\frac{d}{dr}+\frac{1}{r}\widetilde{{A}}(x) _>\right),
\end{align*}
where $\widetilde{{A}}(x) _>\coloneqq \widetilde{{A}}(x) Q_>$. Since $|\widetilde{{A}}(x) _>|\geq 1$ then, as discussed above, ${D}_V(x)_>$ is essentially selj-adjoint on compactly supported forms because condition \eqref{Eqn:CondSpectTildeAx} holds. We will still denote this self-ajoint extension by ${D}_V(x)_>\coloneqq {D}_{V,\text{min}}(x)_>$. In this case we can just define
$G(\mu,x,\beta)_>\coloneqq ({D}_V(x)_>+ic(\beta)-i\mu)^{-1}.$
Using Corollary \ref{Coro:GraphNormDV} and Lemma \ref{Lemma:DVmin} it is possible to prove the following estimates. 
\begin{proposition}[{\cite[Equations (4.9), (4.10)]{B09}}]\label{Prop:G>}
The operator $G(\mu,x,\beta)_>$ satisfies
\begin{align*}
\bnorm{\frac{\partial^j}{\partial \mu ^j}\frac{\partial ^{|p|}}{\partial x^p}\frac{\partial ^{|q|}}{\partial \beta^q}G(\mu,x,\beta)_>}_\mathcal{L(\mathcal{E})}\leq C_{j,p,q}|\mu|^{-1-j}
\end{align*}
and $G(\mu,x,\beta)_>\sigma(r)=O_\sigma(r^{1/2}|\log r|)$, as $r\longrightarrow 0$.
\end{proposition}

The following fundamental result shows that the estimates of Proposition \ref{Prop:G>} serve as a bound of the operator associated with the symbol $G(\mu,x,\beta)>$. In particular we see why it is valuable to have an estimates for all derivatives. 
\begin{theorem}[Calder\'on-Vaillancourt,{ \cite{CV71}}]\label{Thm:CV}
There exists a constant $N_{CV}>0$ such that 
\begin{align*}
\norm{\mathcal{G}_{1}(\mu, D, x_0)}\leq \sup_{|p|,|q|\leq N_{CV}}
\bnorm{\frac{\partial ^{|p|}}{\partial x^p}\frac{\partial ^{|q|}}{\partial \beta^q}G(\mu,x,\beta)_>}_\mathcal{L(\mathcal{E})}.
\end{align*}
\end{theorem}

\subsubsection{Small eigenvalues}

Let $Q_<\coloneqq I-Q_>$ be the complementary spectral projection of $Q_>$, specifically set $Q_< \coloneqq Q_{|\lambda|< \Lambda}(\widetilde{{A}}(x))$. Analogously as before, consider the corresponding reduced operator 
\begin{align*}
{D}_V(x)_<\coloneqq &\gamma\left(\frac{d}{dr}+\frac{1}{r}\widetilde{{A}}(x) _<\right),
\end{align*}
where $\widetilde{{A}}(x) _<:=\widetilde{{A}}(x) Q_<$. Note that condition \eqref{Eqn:CondSpectTildeAx} might not be satisfied, which implies  that ${D}_V(x)_<$ is not necessarily  essentially self-adjoint when defined on forms with compact support. Thus, a choice of ``boundary conditions" is needed. We are going to define this self-adjoint extension by constructing the parametrix $G(\mu, x,\beta)$ explicitly and then defining the domain by \eqref{Cond:ExtDV}, which in this case it reads as
\begin{align}
&(D_{V,\text{max}}(x)_<+ic(\beta)-i\mu)G(\mu,x,\beta)_<\sigma_<=\sigma_<, \quad \text{for all $\sigma_<\in L^2((0,\infty),H_<)$},\label{Eqn:(4.12)}\\
 &\text{$D_{V,\text{max}}(x)_<$ is symmetric on  $\ran(G(\mu,x,\beta)_<)$},\label{Eqn:(4.13)}
\end{align}
where $H_>\coloneqq \ran (Q_>)$. The construction of $G(\mu, x,\beta)$ will involve Bessel functions which will allow us to obtain estimates of the desired form 
\begin{align}\label{Eqn:(4.14)}
\bnorm{\frac{\partial^j}{\partial \mu ^j}\frac{\partial ^{|p|}}{\partial x^p}\frac{\partial ^{|q|}}{\partial \beta^q}G(\mu,x,\beta)_<}_\mathcal{L(\mathcal{E})}\leq C_{j,p,q}|\mu|^{-1-j},
\end{align}
and enable us to make sense of the Neumann series \eqref{Eqn:NeumSer}. This strategy to handle the small eigenvalues can be implemented essentially because the space $H_<$ is finite dimensional, as we will see in the following subsection.  
\subsubsection{A model operator}
Let $V$ be a finite dimensional complex Hilbert space,  $A\in\mathcal{L}(V)$ be a Hermitian operator and $\alpha_1,\alpha_2$ be two self-adjoint involutions  such that the following relations hold 
\begin{align*}
\alpha_1\alpha_2+\alpha_2\alpha_1=&0,\\
\alpha_1 A-A\alpha_1=&0,\\
\alpha_2 A+A\alpha_2=&0.
\end{align*}
We want to study the equation
\begin{equation}\label{Eqn(4.18)}
L(A)\sigma(r)\coloneqq \left(\frac{d}{dr}+\frac{1}{r}A+\mu\alpha_2\right)\sigma(r)=\tau(r), \:r>0,
\end{equation}
in $L^2(\mathbb{R}_+,H)$ for $\mu\in\mathbb{R}-\{0\}$. We can use the involution $\alpha_1$ to split $V$ as $V=V^+\oplus V^-$ so that the operator $A$, commuting with $\alpha_1$,  decomposes accordingly as
\begin{align*}
\left(
\begin{array}{cc}
A^+ & 0\\
0 &A^-
\end{array}
\right).
\end{align*}
Under the isomorphism 
\begin{align*}
\xymatrixrowsep{0.2pc}\xymatrix{
\mathbb{C}^2\otimes V^+ \ar[r] & V=V^+\oplus V^-\\
x_+\otimes y_+ \ar@{|->}[r] & x_+ + \alpha_2 y_+,
}
\end{align*} 
we see that the action of $A$ is given by
\begin{align*}
A(x_+ + \alpha_2 y_+)=A^+ x_+ + A^-\alpha_2 y_+=A^+ x_+ -\alpha_2 A^+ y_+.
\end{align*}
On the other hand, $\mu\alpha_2(x_+ + \alpha_2 y_+)=\mu\alpha_2x_+ +\mu y_+=\mu( y_+ +\alpha_2  x_+)$. Hence, we can transform \eqref{Eqn(4.18)} under this isomorphism to obtain the equation
\begin{align*}
\left(\frac{d}{dr}+\frac{1}{r}
\left(
\begin{array}{cc}
A^+ & 0\\
0 &-A^+
\end{array}
\right)
+\mu
\left(
\begin{array}{cc}
0 & I\\
I & 0
\end{array}
\right)
\right)
\left(
\begin{array}{c}
\sigma_+\\
\sigma_-
\end{array}
\right)(r)
=\left(
\begin{array}{c}
\tau_+\\
\tau_-
\end{array}
\right)(r).
\end{align*}
The following theorem, regarding the solution operator of \eqref{Eqn(4.18)}, is one of the most fundamental results of \cite{B09}.

\begin{theorem}[{\cite[Theorem 4.2]{B09}}]\label{Thm:B09MainSec4}
For $\mu>0$, the equation \eqref{Eqn(4.18)} admits the solution
\begin{align*}
&G(\mu,A)\left(
\begin{array}{c}
\tau_+\\
\tau_-
\end{array}
\right)(r)\\
&=\int_0^r\mu(rs)^{1/2}
\left(
\begin{array}{cc}
K_{A+1/2}(\mu r)I_{A-1/2}(\mu s) & K_{A+1/2}(\mu r)I_{A+1/2}(\mu s)\\
K_{A-1/2}(\mu r)I_{A-1/2}(\mu s) & K_{A-1/2}(\mu r)I_{A+1/2}(\mu s)
\end{array}
\right)
\left(
\begin{array}{c}
\tau_+\\
\tau_-
\end{array}
\right)(s)ds\\
&\quad -\int_r^\infty\mu(rs)^{1/2}
\left(
\begin{array}{cc}
I_{A+1/2}(\mu r)K_{A-1/2}(\mu s) & -I_{A+1/2}(\mu r)K_{A+1/2}(\mu s)\\
-I_{A-1/2}(\mu r)K_{A-1/2}(\mu s) & I_{A-1/2}(\mu r)K_{A+1/2}(\mu s)
\end{array}
\right)
\left(
\begin{array}{c}
\tau_+\\
\tau_-
\end{array}
\right)(s)ds\\
&\eqqcolon G_0(\mu,A)\tau(r)+G_\infty(\mu,A)\tau(r). 
\end{align*}
The operators $G_{0/\infty}(\mu,A)\tau(r)$ are bounded in $L^2(\mathbb{R}_+,H)$ and are smooth functions of the variables $\mu\in[1,\infty)$ and $A\in\mathcal{L}_s(V)$, the space of Hermitian matrices on $H$, such that for $p,q\in\mathbb{Z}_+$,
\begin{align}\label{Eqn:(4.21)}
\bnorm{D_A^p\left(\frac{\partial}{\partial \mu}\right)^qG_{0/\infty}(\mu,A)}_{L^2(\mathbb{R}_+,H)}\leq C_{p,q,A}|\mu|^{-1}. 
\end{align}
Moreover, for $\sigma\in \ran(G(\mu,A))$ and $r$ is sufficiently small we have the estimates 
\begin{align}
\norm{\sigma_+(r)}_H &\leq C_\epsilon r^{1/2-\epsilon}\norm{\tau}_{L^2(\mathbb{R}_+,H)} \textnormal{ for every $\epsilon>0$},\label{Eqn:(4.22)}\\
\norm{\sigma_-(r)}_H &\leq C r^{-1/2+\delta}\norm{\tau}_{L^2(\mathbb{R}_+,H)} \textnormal{ for some $\delta>0$}. \label{Eqn:(4.23)}
\end{align}
If $|A|\geq 1/2$, then we have the better estimate 
\begin{align}\label{Eqn:(4.24)}
\norm{\sigma(r)}_H\leq C r^{1/2}\norm{\tau}_{L^2(\mathbb{R}_+,H)}. 
\end{align}
\end{theorem}

\subsubsection{Construction of the operator symbol $G(\mu,x,\beta)_<$}

Now we show how to apply Theorem \ref{Thm:B09MainSec4} to our concrete case of interest. First of all, the finite dimensional complex Hilbert space is $V\coloneqq H_<=\wedge T^*_{x_0}F\otimes\mathbb{C}^2\otimes Q_< (L^2(\wedge T^* Y_{x_0}). $
In addition, in view of \eqref{Eqn:DefDV}, the role of the operator $A$ in Theorem \ref{Thm:B09MainSec4} is 
\begin{align*}
\widetilde{{A}}(x)_{<}=
\left(
\begin{array}{cc}
{A}(x)_{<} & 0\\
0& - {A}(x)_{<} 
\end{array}
\right),
\end{align*}
where ${A}(x)_{<} \coloneqq Q_{<\Lambda} \left(D_{Y_x}\alpha_{Y_x}+\nu\right)$. It remains to define the involutions. As before set
\begin{align*}
\gamma\coloneqq \varepsilon_H\otimes
\left(
\begin{array}{cc}
0 & - I\\
I & 0
\end{array}
\right),
\end{align*}
$\widetilde{\gamma}\coloneqq i\gamma$ and $\zeta\coloneqq \zeta(\mu,\beta)\coloneqq \mu\widetilde{\gamma}-\widetilde{\gamma}c(\beta)$, for $\beta\in T^*F$. 
\begin{lemma}\label{Lemma:RelZeta}
The following relations hold true
\begin{enumerate}
\item $\zeta^\dagger=\zeta$.
\item $\zeta ^2=(|\mu|^2+|\beta|^2_x)\eqqcolon\tilde{\mu}(x,\beta)I$.
\item $\zeta\widetilde{A}(x)_<+\widetilde{A}(x)_<\zeta=0$.
\end{enumerate}
\end{lemma}

\begin{proof}
The proof relies on a careful control of signs. We include the details for completeness. 
\begin{enumerate}
\item First observe that 
$c(\beta)\varepsilon_H=-\varepsilon_Hc(\beta)$ 
which implies that $c(\beta)\widetilde{\gamma}=-\widetilde{\gamma}c(\beta)$
and shows $\zeta^\dagger=(\mu\widetilde{\gamma}-\widetilde{\gamma}c(\beta))^\dagger=\mu\widetilde{\gamma}+c(\beta)\widetilde{\gamma}=\mu\widetilde{\gamma}-\widetilde{\gamma}c(\beta)=\zeta$.
\item Similarly we compute, using that $\widetilde{\gamma}^2=I$,
\begin{align*}
\zeta^2=(\mu\widetilde{\gamma}-\widetilde{\gamma}c(\beta))^2=|\mu|^2+\widetilde{\gamma}c(\beta)\widetilde{\gamma}c(\beta)-\mu c(\beta)-\widetilde{\gamma}c(\beta)\mu\widetilde{\gamma}=(|\mu|^2+|\beta|^2_x).
\end{align*}
\item To begin, note the relation
\begin{align*}
(\varepsilon_H\otimes I)\left(I\otimes\left(D_{Y_x}\alpha_{Y_x}+\nu\right)\right)=\left(I\otimes \left(D_{Y_x}\alpha_{Y_x}+\nu\right)\right)(\varepsilon_H\otimes I).
\end{align*}
We now calculate
\begin{align*}
\widetilde{\gamma}\widetilde{{A}}(x)_<=&
\varepsilon_H\left(
\begin{array}{cc}
0 & -I\\
I & 0
\end{array}
\right)
\left(
\begin{array}{cc}
{A}(x)_{<} & 0\\
0& - {A}(x)_{<} 
\end{array}
\right)\\
=&\left(
\begin{array}{cc}
0 & \varepsilon_H{A}(x)_{<} \\
\varepsilon _H{A}(x)_{<} & 0
\end{array}
\right)\\
=&\left(
\begin{array}{cc}
0 &{A}(x)_{<}\varepsilon_H \\
{A}(x)_{<}\varepsilon _H & 0
\end{array}
\right)\\
=&
-
\left(
\begin{array}{cc}
{A}(x)_{<} & 0\\
0& - {A}(x)_{<} 
\end{array}
\right)
\left(
\begin{array}{cc}
0 & -\varepsilon_H\\
\varepsilon_H & 0
\end{array}
\right)\\
=&-\widetilde{A}(x)_<\widetilde{\gamma}. 
\end{align*}
This takes care of the first term of $\zeta$. For the one second we argue similarly using 
\begin{align*}
(c(\beta)\otimes I)\left(I\otimes\left(D_{Y_x}\alpha_{Y_x}+\nu\right)\right)=\left(I\otimes \left(D_{Y_x}\alpha_{Y_x}+\nu\right)\right)(c(\beta)\otimes I).
\end{align*}
\end{enumerate}
\end{proof}
We now define the involutions $\alpha_1$ and $\alpha_2$ by the relations
\begin{align*}
\alpha_1\coloneqq I\otimes
\left(
\begin{array}{cc}
I& 0\\
0 & -I
\end{array}
\right)
\quad\text{and}\quad
\widetilde{\mu}\alpha_2\coloneqq\zeta.
\end{align*}
The following proposition, which follows immediately from Lemma \ref{Lemma:RelZeta}, shows how to use Theorem \ref{Thm:B09MainSec4} in order to construct the operator symbol $G(\mu,x,\beta)_<$. 

\begin{proposition}[{\cite[Lemma 4.4]{B09}}]\label{Prop:Lemma4.4}
The operator ${D}_V(x)_{<}+ic(\beta)-i\mu$ can be expressed, with respect to the notation above, as
\begin{align*}
{D}_V(x)_{<}+ic(\beta)-i\mu=\gamma\left(\frac{d}{dr}+\frac{1}{r}\widetilde{{A}}_<(x)+\tilde{\mu}\alpha_2\right),
\end{align*} 
and the following relations hold:
\begin{align*}
\alpha_1\alpha_2+\alpha_2\alpha_1=&0,\\
\alpha_1 \widetilde{{A}}_<(x)-\widetilde{{A}}_<(x)\alpha_1=&0,\\
\alpha_2 \widetilde{{A}}_<(x)+\widetilde{{A}}_<(x)\alpha_2=&0.
\end{align*}
\end{proposition}

\subsubsection{Properties of the self-adjoint extension}

In this section we are going to synthesize the construction and results discussed in this chapter so far. To give a visual explanation of the general strategy we provide a diagram which describes the construction of the self-adjoint extension $D_\delta$ of the operator $D$ defined by the conditions \eqref{CondExt}. 

\begin{align*}
\xymatrixcolsep{1cm}\xymatrixrowsep{1cm}\xymatrix{
\text{Large eingevalues: $G(\mu,x,\beta)_>$}\ar[dr]^-{\text{ Lemma \ref{Lemma:Spliting}}} &  & \ar[dl]_-{\text{ Lemma \ref{Lemma:Spliting}}} \text{Small eingevalues: $G(\mu,x,\beta)_>$}\\
&G(\mu,x,\beta) \ar[dr]^-{\eqref{Cond:ExtDV}} \ar[dl]_-{\eqref{Eqn:G1}}& \\
 \mathcal{G}_1(\mu,D,x_0) \ar[d]_-{\text{partition of unity}}  & & \widehat{D}_V(x) \ar[d]^-{\int_{\mathbb{R}^h}e^{i\inner{x}{\beta}}\cdot\frac{d\beta}{(2\pi)^h}}  \\
  \mathcal{G}_1(\mu,D) \ar[dr]_{\eqref{Eqn:NeumSer}} \ar@{<->}[rr]^-{\mathcal{G}_1(\mu,D)\widehat{D}=I}  &&   \widehat{D} \ar[dl]^{\eqref{Eqn:NeumSer}}\\
 &\mathcal{G}(\mu,D) \ar[d]_{\eqref{CondExt}}  & \\
& D_\delta & \\
}
\end{align*}
In view of Proposition \ref{Prop:G>}, Theorem \ref{Thm:B09MainSec4} and Proposition \ref{Prop:Lemma4.4} one can explicitly write the  domain of such extension (\cite[Equation (0.16)]{B09}), 
\begin{align*}
\dom(D_\delta)\coloneqq
\left\{
\begin{array}{ll}
\sigma\in \dom(D_\text{max})\::\: &\norm{\sigma^+}=O(r^{1/2-\epsilon})\:\text{for every $\epsilon>0$},\\
&\norm{\sigma^-}=O(r^{-1/2+\delta}) \:\text{for some $\delta>0$, as $r\longrightarrow 0$}
\end{array}
\right\}.
\end{align*}

\begin{theorem}[{\cite[Theorem 0.1(1)]{B09}}]
The operator $D_\delta$ is self-adjoint, discrete and anti-commutes with the cirality operator $\bar{\star}$. 
\end{theorem}

\begin{proof}
We just describe a sketch of the proof. We first need to verify that the operator symbol
\begin{align*}
G(\mu,x,\beta)_<\coloneqq\left(\gamma\left(\frac{d}{dt}+\frac{1}{t}\widetilde{{A}}_<(x)+\tilde{\mu}\alpha_2\right)\right)^{-1}
\end{align*}
obtained using Proposition \ref{Prop:Lemma4.4} satisfies the conditions \eqref{Eqn:(4.12)}, \eqref{Eqn:(4.13)} and \eqref{Eqn:(4.14)}. It is important to note that Theorem \ref{Thm:B09MainSec4} does not apply directly in this situation because $\tilde{\mu}$ and $\alpha_2$ depend both on $x$ and $\beta$. One can overcome this subtlety by means of the  Splitting Lemma (Lemma \ref{Lemma:Spliting}). Next one proceeds to verify the required conditions: \eqref{Eqn:(4.12)}  follows by construction and \eqref{Eqn:(4.14)} follows from \eqref{Eqn:(4.21)}. In order to prove \eqref{Eqn:(4.13)} we compute  for $\sigma_1,\sigma_2\in\ran(G(\mu,x,\beta)_<)$ using integration by parts,
\begin{align*}
(D_{V,\text{max}}(x)\sigma_1,\sigma_2)-(\sigma_1,D_{V,\text{max}}(x)\sigma_2)=&\lim_{r\longrightarrow \infty}\left(\inner{\sigma^+_1(r)}{\sigma^-_2(r)}-\inner{\sigma^-_1(r)}{\sigma^+_2(r)}\right)\\
&-\lim_{r\longrightarrow 0}\left(\inner{\sigma^+_1(r)}{\sigma^-_2(r)}-\inner{\sigma^-_1(r)}{\sigma^+_2(r)}\right). 
\end{align*}
The first limit is zero because of the $L^2$ condition. For the second one we use the boundary conditions \eqref{Eqn:(4.22)} and  \eqref{Eqn:(4.23)}, 
\begin{align*}
|\inner{\sigma^-_1(r)}{\sigma^+_2(r)}|\leq \norm{{\sigma^-_1(r)}}\norm{\sigma^+_2(r)}\leq \tilde{C}_{\epsilon} r^{\delta-\epsilon}\longrightarrow 0, \quad \text{as $r\longrightarrow 0$},
\end{align*}
since we can always choose $\epsilon<\delta$. Therefore, 
$(D_{V,\text{max}}(x)\sigma_1,\sigma_2)=(\sigma_1,D_{V,\text{max}}(x)\sigma_2)$. This shows that the desired conditions for $G(\mu,x,\beta)_<$ are satisfied. Hence, in combination with Proposition \ref{Prop:G>} and Theorem \ref{Thm:CV} we conclude that the induced operator $\mathcal{G}_1(\mu,D)$ satisfies the desired properties. As a result, the Neumann series \eqref{Eqn:NeumSer} is well-defined and it satisfies \eqref{Cond:ExtDV}. \\

Now we prove that $D_\delta$ is discrete, which is equivalent to the compactness of its resolvent $\mathcal{G}(\mu,D)$. Note that since
\begin{align*}
\mathcal{G}(\mu, D)=\lim_{N\rightarrow \infty} \mathcal{G}_1(\mu,D)\sum_{j=0}^N\mathcal{R}^j,
\end{align*} 
it is enough to show that $\mathcal{G}_1(\mu,D)$ is compact by \cite[Theorem III.4.7]{KATO} .  Let  us choose a cut-off function $\psi_r\in C_c(M_0/S^1)$ with $\psi_r=1$ on $M_0/S^1-N_r(F)$ (see Section \ref{Sect:PfoofSignatureFormula}) and which vanishes close to $F\subset M^{S^1}$. By interior regularity  the operator  $\psi_r\mathcal{G}_1(\mu,D)$ is compact (\cite[Theorem 2.1]{BS87}). Now we want to show that the norm of the remainder $(1-\psi_r)\mathcal{G}_1(\mu,D)$ goes to zero as $r\longrightarrow 0$. Using \eqref{Eqn:(4.22)} and \eqref{Eqn:(4.23)} for small eigenvalues and \eqref{Eqn:(4.24)} for large eigenvalues (Proposition \ref{Prop:G>}), we estimate
\begin{align*}
\norm{(1-\psi_r)\mathcal{G}_1}^2\leq C\int_{0}^r t^{-1+2\delta} dt =C t^{2\delta}\bigg{|}_0^r= C r^{2\delta}\longrightarrow 0, 
\end{align*}
as $r\longrightarrow 0$ since $\delta>0$. Thus, 
\begin{align*}
\mathcal{G}_1(\mu,D)=\lim _{N\longrightarrow \infty} \psi_{1/N}\mathcal{G}_1(\mu,D), 
\end{align*}
which shows that $\mathcal{G}_1(\mu,D)$ is compact again by \cite[Theorem III.4.7]{KATO}. \\
Finally observe that $\bar{\star}$ coincides with $\alpha_1$ in Proposition \ref{Prop:Lemma4.4}, so it is clear that $\bar{\star}$ preserves the domain of $D_\delta$ and anti-commutes with it. 
\end{proof}

\begin{coro}[{\cite[Theorem 0.1(2)]{B09}}]\label{Coro:DESA}
If $|A_V|\geq 1/2$, then $D_\textnormal{min}$ is essentially self-adjoint.
\end{coro}

\begin{proof}
Let $\sigma\in \dom(D_\delta)$, we want to show that there exists a sequence $(\sigma_n)_n$ of compactly supported forms such that $\sigma_n\longrightarrow \sigma$ in the graph norm of $D$. Actually is enough to show that $(\sigma_n)_n$ is a Cauchy sequence. The desired sequence is constructed using the cut-off functions \eqref{Eqn:TildePsin} of Appendix \ref{App:Seq}. For $n\in\mathbb{N}$ and $n\geq 2$ define
$\sigma_n\coloneqq \widetilde{\psi}_n\sigma$. By the construction of $\widetilde{\psi}_n$ we see that each $\sigma_n$ has compact support. Moreover, since $\widetilde{\psi}_n$ is uniformly bounded and converges to pointwise to $1$ it is easy to verify, as in the proof of Lemma \ref{Lemma:PropPsin}, that $\sigma_n\longrightarrow \sigma$ in $L^2$. Therefore, it remains to show that $D \sigma_n\longrightarrow  D\sigma$ in $L^2$. Observe from Proposition \ref{Prop:CommDirOpFunct} that $D(\widetilde{\psi}_n\sigma)=c(d\widetilde{\psi}_n)\sigma+\widetilde{\psi}n D\sigma$. Thus, it is enough to show that $c(d\widetilde{\psi}_n)\longrightarrow 0$ in $L^2$. Moreover, by the properties of the Clifford multiplication, this is equivalent to showing $\norm{d\widetilde{\psi}_n}^2\sigma\longrightarrow 0$ in $L^2$. This last statement can be proven using \eqref{Eqn:EstimDTildePsi}, Lemma \ref{LemmaB2} and Proposition \ref{Prop:AppSeqMain}. 
\end{proof}

Let us now assume the Witt condition (Definition \ref{Def:Witt}). For each $a>0$ define a metric on $M_0/S^1$ such that
\begin{align}\label{Eqn:ScaledMetr}
g^{T(M_0/S^1)}(a)\coloneqq 
\begin{cases}
dr^2\oplus g^{T_H\mathcal{F}}\oplus a^2 r^2g^{T_H\mathcal{F}},\quad\text{on $N_{t/2}(F)$}\\
g^{T(M_0/S^1)},\quad\text{on $M_0/S^1- N_t(F)$.}\\
\end{cases}
\end{align}
The factor $a$ just rescales the metric in the vertical direction. Denote the by $D_\delta(a)$ its corresponding self-adjoint operator. Observe from Remark \ref{Rmk:ScaleHodgeStar} and the spectral decomposition of Theorem \ref{Thm:SpectralDecomp} that in the Witt case we can always find $a_0>0$ such that $|A_V|\geq 1/2$. Indeed, in the complement of the space of vertical harmonic forms the spectral gap can be controlled entirely with the parameter $a$. The rescaling however, has no effect over the eigenvalues on the space of vertical harmonic forms. Here is where the Witt condition comes in as, by Remark \ref{Rmk:ZeroEigenValueAV}, these eigenvalues are of the form $2j-N$, for $j=0,\cdots, N$. In particular $|2j-N|>1$ if $N$ odd.  Consequently the operator $D_\delta(a_0)$ is essentially self-adjoint by Corollary \ref{Coro:DESA}.

\begin{coro}[{\cite[Theorem 0.1(2)]{B09}}]\label{Coro:WittimpL2}
In the Witt case, the $L^2$-Stokes theorem holds. 
\end{coro}

\begin{proof}
As all the metrics $g^{T(M_0/S^1)}(a)$ are mutually quasi-isometric, it follows that the $L^2$-Stokes theorem is independent of the choice of $a>0$ by Remark \ref{L2StokesQI}. In particular, we can choose $a_0>0$ as above so that $D_\delta(a_0)$ is essentially self-adjoint.   The claim follows now from Proposition \ref{Prop:Lemma2.3BL}. 
\end{proof}
This corollary makes it possible to prove that the signature theorem considered in Section \ref{Sec:L2} is defined if the Witt condition is satisfied (and the space $M/S^1$ is compact).
\begin{proposition}[{\cite[Theorem 0.1(4)]{B09}}]
In the Witt case, the operator $D^+_{\textnormal{max}}$ (see  \eqref{Def:Dmax+}) is well-defined and it coincides with $D^+_\delta$. 
\end{proposition}
Moreover, form the proof of the last proposition one obtains the following remarkable result. 
\begin{coro}[{\cite[Theorem 0.1(4)]{B09}}]\label{Coro:IndexInvScal}
The index of $\widetilde{D}(\alpha)^+$ is independent of $\alpha$. 
\end{coro}
\begin{remark}
As pointed out in \cite{B09}, for the special case of isolated conical singularities ($h=0$) the results above are due to Cheeger (\cite{Ch79}, \cite{Ch83}). For general horizontal dimension $h$, Corollary \ref{Coro:DESA} and Corollary \ref{Coro:WittimpL2} and can also be deduced from \cite{Ch83}.
\end{remark}

\subsection{A parametrix for $\mathscr{D}$}

In this section we are going to describe how to adapt the construction described above for the operator $\mathscr{D}$. In view of Definition \ref{Def:DVPot} and \eqref{Eqn:DefDV} we define
\begin{align*}
\mathscr{D}_V(x)\coloneqq
&\gamma\left(\frac{d}{dr}+\frac{1}{r}\widetilde{\mathscr{A}}(x)\right)\\
\coloneqq &\varepsilon_H\otimes
\left(
\begin{array}{cc}
0 & - I\\
I & 0
\end{array}
\right)
\left(
\frac{d}{dr} +\frac{1}{r}
\left(
\begin{array}{cc}
D_{Y_x}\alpha_{Y_x}+\nu-\frac{\varepsilon_V}{2} & 0\\
0 & -(D_{Y_x}\alpha_{Y_x}+\nu-\frac{\varepsilon_V}{2})
\end{array}
\right)
\right).
\end{align*}

\begin{lemma}
For each $(x,\beta)\in T^* W_{x_0}$, the operator $D_V(x)$ satisfies
\begin{align*}
\mathscr{D}_V(x) c(\beta)+c(\beta)\mathscr{D}_V(x)=0
\end{align*}
\end{lemma}
\begin{proof}

This follows from Lemma \ref{Lemma 1.1(2)} and the relation
\begin{align*}
(I\otimes c(\beta))(\varepsilon_V\otimes\varepsilon_H)+(\varepsilon_V\otimes\varepsilon_H)(I\otimes c(\beta))=0.
\end{align*}
\end{proof}
This lemma shows that we have an analogous result of Corollary \ref{Coro:GraphNormDV}.
\begin{coro}
For all $\sigma(x)\in\dom(\mathscr{D}_{V}(x))$,
\begin{align*}
\norm{(\mathscr{D}_V(x)+ic(\beta)-i\mu)\sigma(x)}^2_{\mathscr{D}_V}=&\norm{\mathscr{D}_V(x)\sigma(x)}^2_{\mathscr{D}_V} +(|\mu|^2+|\beta|_x^2)\norm{\sigma(x)}^2_{\mathscr{D}_V}.
\end{align*}
\end{coro}
Using the Splitting Lemma \ref{Lemma:Spliting} we can define, similarly as above, 
\begin{align*}
\widetilde{\mathscr{A}}(x)_{<}=
\left(
\begin{array}{cc}
\mathscr{A}(x)_{<} & 0\\
0& - \mathscr{A}(x)_{<} 
\end{array}
\right),
\end{align*}
where
\begin{align*}
\mathscr{A}(x)_{<} \coloneqq Q_{<\Lambda} \left(D_{Y_x}\alpha_{Y_x}+\nu-\frac{\varepsilon_V}{2}\right).
\end{align*}
Then we have the following extension of Proposition \ref{Prop:Lemma4.4}. 
\begin{proposition}
The operator $\mathscr{D}_V(x)_{<}+ic(\beta)-i\mu$ can be expressed as
\begin{align*}
\mathscr{D}_V(x)_{<}+ic(\beta)-i\mu=\gamma\left(\frac{d}{dr}+\frac{1}{r}\widetilde{\mathscr{A}}_<(x)+\tilde{\mu}\alpha_2\right),
\end{align*} 
and the following relations hold:
\begin{align*}
\alpha_1\alpha_2+\alpha_2\alpha_1=&0,\\
\alpha_1 \widetilde{\mathscr{A}}_<(x)-\widetilde{\mathscr{A}}_<(x)\alpha_1=&0,\\
\alpha_2 \widetilde{\mathscr{A}}_<(x)+\widetilde{\mathscr{A}}_<(x)\alpha_2=&0.
\end{align*}
\end{proposition}

\begin{proof}
This proposition follows just as  Proposition \ref{Prop:Lemma4.4} follows from Lemma \ref{Lemma:RelZeta}. We need to study the commutation relations with the potential. Observe first 
\begin{align*}
(\varepsilon_H\otimes I)(I\otimes\varepsilon_V)=(I\otimes\varepsilon_V)(\varepsilon_H\otimes I)=\varepsilon_H\otimes\varepsilon_V.
\end{align*}
Hence, for the first term of $\zeta$ we calculate 
\begin{align*}
\left(
\begin{array}{cc}
0 & -\varepsilon_H\\
\varepsilon_H & 0 
\end{array}
\right)
\left(
\begin{array}{cc}
-\varepsilon_V & 0\\
0 & \varepsilon_V
\end{array}
\right)
=&
\left(
\begin{array}{cc}
0 & -\varepsilon_H\varepsilon_V\\
-\varepsilon_H\varepsilon_V & 0 
\end{array}
\right)\\
=&
\left(
\begin{array}{cc}
0 & -\varepsilon_V\varepsilon_H\\
-\varepsilon_V\varepsilon_H & 0 
\end{array}
\right)\\
=&
-\left(
\begin{array}{cc}
-\varepsilon_V & 0\\
0 & \varepsilon_V
\end{array}
\right)
\left(
\begin{array}{cc}
0 & -\varepsilon_H\\
\varepsilon_H & 0 
\end{array}
\right).
\end{align*}
For the second term of $\zeta$ we can argue similarly using the relation
\begin{align*}
(c(\beta)\otimes I)(I\otimes \varepsilon_V)= (I\otimes \varepsilon_V)(c(\beta)\otimes I).
\end{align*}
\end{proof}

Following the same ideas as of last section and in view of Theorem \ref{Thm:SpecDec}, we can prove the following theorem.
\begin{theorem}
The operator $\mathscr{D}$ is discrete, and the index of $\mathscr{D}^+$ is invariant under the vertical rescaling of the vertical metric. 
\end{theorem}

\section{An index theorem}\label{Sect:Index}

In this chapter we compute the index of the operator $\mathscr{D}^+$, in the Witt case, following the techniques used in \cite[Section 5]{B09}. The analytic tools required to treat the index computation are based on the work \cite{BBC08} of Ballmann, Br\"uning and Carron on {\em Dirac-Schr\"odinger Systems}. In this chapter we plan to collect the main results of their work in order to apply them to our concrete case. We omit the proofs of these theorems but nevertheless we present an example on which we illustrate the theory: we derive the Novikov additivity formula of signature (Proposition \ref{Prop:Novikov}) as a gluing index theorem. After this brief summary of the required analytic tools we first concentrate on the index computation for the signature operator $D^+\coloneqq D^+_{\text{max}}$ defined by \eqref{Def:Dmax+} in the Witt case, following closely the strategy of \cite[Section 5]{B09}. This will serve as a model to study the index of operator $\mathscr{D}^+$. 

\subsection{Perturbation of regular projections}\label{Sec:Perturbations}

The first section of this chapter describes some notions and results on regular projections which, as we will see later, are important prototypes for elliptic boundary conditions. This section is based in \cite[Sections 1.2, 1.6]{BBC08} and \cite[Section 5]{B09}.\\

Let $(H, \inner{\cdot}{\cdot})$ be a separable complex Hilbert space and let $A:\dom(A)\subset H\longrightarrow H$ be a discrete self-adjoint operator. For each Borel subset $J\subseteq\mathbb{R}$ we denote by $Q_J\coloneqq Q_J(A)$ the associated spectral projection of A in $H$. For $\Lambda\in\mathbb{R}$ we will use the notation $Q_{<\Lambda}\coloneqq Q_{(-\infty,\Lambda)}$, $Q_{\geq\Lambda }\coloneqq Q_{[\Lambda,\infty)}$, etc. For $\Lambda=0$ simply write $Q_{>}\coloneqq Q_{>0}$, etc. In particular $Q_0$ denotes the projection onto the kernel of $A$. 

\begin{definition}[Sobolev chain]
For $s\geq 0$ we define the space $H^s\coloneqq H^s(A)$ to be the completion of $\dom(A)\subset H$ with respect to the inner product 
\begin{align*}
\inner{\sigma_1}{\sigma_2}_s\coloneqq \inner{(I+A^2)^{s/2}\sigma_1}{(I+A^2)^{s/2}\sigma_2}. 
\end{align*}
Note for example that $H^0=H$ and $H^1=\dom(A)$. For $s<0$ we define $H^s$  to be the strong dual space of $H^{-s}$.
\end{definition}
For a Borel subbset $J\subseteq\mathbb{R}$ we denote by $H^s_J\coloneqq Q_J(H^s)$ the image of the Sobolev space $H^s$  under the projection $Q_J$. In particular we will use the notation
$H^s_<\coloneqq Q_{<0}(H^s)$, etc.

\begin{remark}
By the discreteness of $A$, if $s<t$ then the embedding $H^t\longrightarrow H^s$ is compact. 
\end{remark}

\begin{remark}
For each $s\in\mathbb{R}$ the pairing
\begin{align}\label{PairingB}
B_s:
\xymatrixcolsep{2cm}\xymatrixrowsep{0.01cm}\xymatrix{H^s\times H^{-s}\ar[r] & \mathbb{C}\\
(\sigma_1,\sigma_2)\ar@{|->}[r] & \inner{(I+A^2)^{s/2}\sigma_1}{(I+A^2)^{-s/2}\sigma_2}
}
\end{align}
is non-degenerate so it can be used to identify $H^s$ with its dual space $H^{-s}$.
\end{remark}

\begin{definition}
A bounded operator $S\in\mathcal{L}(H)$ is called {\em $1/2$-smooth} if it restricts to an operator $\hat{S}:H^{1/2}\longrightarrow H^{1/2}$ and extends to an operator $\tilde{S}:H^{-1/2}\longrightarrow H^{-1/2}$. The operator $S$ will be called {\em $(1/2)$-smoothing}, or simply {\em smoothing}, if $\ran(\tilde{S})\subset H^{1/2}$.
\end{definition}

In addition to the operator $A$ let us assume that we are given  $\gamma\in\mathcal{L}(H)$ such that 
\begin{enumerate}
\item $\gamma^*=\gamma^{-1}=-\gamma$.
\item $\gamma A+A\gamma=0$.
\end{enumerate}
Note in particular, if $A\sigma=\lambda \sigma$ then $A(\gamma \sigma)=-\gamma A\sigma=-\lambda(\gamma \sigma)$, thus  $Q_{>}\gamma=\gamma Q_{<}$.
\begin{definition}
A $1/2$-smooth orthogonal projection $P$ in $H$ is called {\em regular} (with respect to A) if for some, or equivalently, for any $\Lambda\in\mathbb{R}$ we have 
\begin{align*}
\sigma\in H^{-1/2}, \tilde{P}\sigma=0, Q_{\leq \Lambda}(A)\sigma\in H^{1/2}\:\Rightarrow\: \sigma\in H^{1/2}.
\end{align*}
A regular projection $P$ is called {\em elliptic} (with respect to $A$) if $P_\gamma\coloneqq \gamma^*(1-P)\gamma$ is also a regular projection. 
\end{definition}

\begin{example}[Spectral projections]\label{Example:SpectralProjections}
Let $\Lambda\in\mathbb{R}$, then the spectral projection $Q_{>\Lambda}$ is an elliptic projection. To see this first note from the definition of $H^s$ and from the discreteness of $A$ that  $Q_{>\Lambda}(H^s)\subset H^s $ for any $s\in\mathbb{R}$, in particular  this shows that $Q_{>\Lambda}$ is $1/2$-smooth. Moreover since $Q^*_{>\Lambda}=Q_{>\Lambda}$ then $\tilde{Q}_{>\Lambda}=\hat{Q}_{>\Lambda}$. Let us now see why $Q_{>\Lambda}$ is regular: assume that $\sigma\in H^{-1/2}, \tilde{Q}_{>\Lambda}\sigma=0$ and $Q_{\leq \Lambda}\sigma\in H^{1/2}$ then  we decompose $\sigma$ as $\sigma= Q_{>\Lambda} \sigma+ Q_{\leq\Lambda}\sigma=Q_{\leq\Lambda}\sigma\in H^{1/2}$, which is what we wanted to prove.  Finally note that 
$$(Q_{>\Lambda})_\gamma=\gamma^*(I-Q_{>\Lambda})\gamma=I-\gamma^*Q_{>\Lambda}\gamma=I-Q_{<\Lambda}=Q_{\geq\Lambda},$$
which is easy to see that is also regular. Thus, we see indeed that $Q_{>\Lambda}$ an elliptic projection. 
\end{example}

Let us now recall the notion of a Fredholm pair. In \cite[Chapter IV]{K58} there is a detailed an extensive study on the subject. 
\begin{definition}[Fredholm pair]
A {\em Fredholm pair} $(X,Y)$ in the Hilbert space $H$ consists of two closed linear subspaces $X,Y\subseteq H$ such that
\begin{align*}
\text{null}(X,Y)\coloneqq &\dim(X\cap Y)<\infty,\\
\text{def}(X,Y)\coloneqq &\text{codim}(X+ Y)<\infty.
\end{align*}
For such a pair its {\em Kato index} is defined by $\ind(X,Y)\coloneqq\text{null}(X,Y)-\text{def}(X,Y)$.
\end{definition}

\begin{remark}
There is a weaker notion of the definition above. A pair of closed subspaces $(X,Y)$ in $H$ is called a {\em left/right-Fredholm pair} if the sum $X+Y$ is closed in $H$ and $\text{null}(X,Y)\coloneqq \dim(X\cap Y)<\infty$ or $\text{def}(X,Y)\coloneqq \text{codim}(X+ Y)<\infty$ respectively.
\end{remark}

\begin{example}\label{Ex:H>FP}
The pair $(H_{\leq}, H_{\geq})$ is a Fredholm pair and its index is 
\begin{align*}
\ind(H_{\leq}, H_{\geq})=\dim(H_{\leq}\cap H_{\geq})-\text{codim}(H_{\leq}+ H_{\geq})=\dim(\ker A).
\end{align*}
Moreover, $(H_{<}, H_{\geq})$ is also a Fredholm pair and $\ind(H_{<}, H_{\geq})=0$.
\end{example}

\begin{example}
If $T:H\longrightarrow\ H$ is a closed Fredholm operator, then its Fredholm index can be regarded as a Kato index using the relation 
\begin{align*}
\ind(T)=\ind(\text{Graph}(T),H\times \{0\}),
\end{align*}
where $\text{Graph}(T)\coloneqq \{(x,Tx)\:|\: x\in\dom(T)\subseteq H \}\subset H\times H$.
\end{example}

We will be strongly interested in how the ellipticity condition behaves under perturbations. The following lemma gives a first answer in this direction.  
\begin{lemma}[{\cite[Lemma 5.7]{B09}}]\label{BLemma5.7}
Let $P$ be a $1/2$-smooth orthogonal projection in $H$ such that $P=Q_{>}+R_1+R_2$, where $R_1$ is smoothing and $\textnormal{max}\{\norm{\hat{R}_2},\norm{\tilde{R}_2}\}<1$, then $P$ is elliptic with respect to $A$ and $(\textnormal{\ran}(I-P),\textnormal{\ran}( Q_{>}))$ is a Fredholm pair in $H$.
\end{lemma}

The following result is concerned with compact perturbations.  
\begin{lemma}[{\cite[Proposition A.13]{BBC08}}]\label{Prop:A.13}
Let $P,Q$ be projections in $H$ such that $P-Q$ is compact. Then $(\ran P,\ker Q)$ is a Fredholm pair. If $E\subset H$ is a closed subset of $H$ then $(E,\ran P)$ is a Fredholm pair if and only if $(E,\ran Q)$ is a Fredholm pair, and then 
\begin{align*}
\ind(E,\ran P)=\ind(E,\ran Q)+\ind(\ker Q, \ran P).
\end{align*}
Moreover, we also have $\textnormal{ind}(E,\ran P)=\textnormal{\ind}((I-P):E\longrightarrow \ker P)$.
\end{lemma}

\begin{example}\label{Example:KatoIndexComputation}
Let us consider the case where $P=Q_\geq$ and $Q=Q_>$, so that $P-Q=Q_0$ is the orthogonal projection onto $\ker A$. By the discreteness of $A$ we know that $Q_0$ is a finite-rank operator and hence compact. Let $B=H_<$, then from Example \ref{Ex:H>FP} it is easy to verify the relations
\begin{align*}
\ind(B,\ran P)=&\ind(H_<, H_\geq)=0,\\
\ind(B,\ran Q)=&\ind(H_<, H_>)=-\dim\ker A,\\
\ind(\ker Q, \ran P)=&\ind(H_\leq, H_\geq)=\dim \ker A,
\end{align*}
which verify the first index formula of Lemma \ref{Prop:A.13}. Additionally, we can compute the $\ind(H_<, H_\geq)$ using the second index formula, 
\begin{align*}
\ind(H_<, H_\geq)=&\ind(B,\ran P)\\
=&\textnormal{\ind}((I-P):B\longrightarrow \ker P)\\
=&\textnormal{\ind}((I-Q_\geq):H_<\longrightarrow \ker (Q_\geq))\\
=&\textnormal{\ind}(Q_<:H_<\longrightarrow  H_<)\\
=&0,
\end{align*}
where the last equality follows because is just the index of the identity map.  
\end{example}

In view of Lemma \ref{BLemma5.7} we want to consider perturbations of Kato-type.
\begin{definition}
Let $B$ be a symmetric operator in $H$ defined on $\dom(A)$ such that 
$\norm{B\sigma}\leq a\norm{\sigma}+b\norm{A\sigma}$ for all $\sigma\in\dom(A)$ and $a,b\in\mathbb{R}_+$ with $b<1$. Then we call the operator $A+B$, defined on $\dom(A)$, a {\em Kato perturbation of $A$}. 
\end{definition}

The Kato-Rellich Theorem states that the operator $A+B$ is again self-adjoint and discrete (\cite[Theorem V.4.3]{KATO}).  Thus, we can consider the corresponding Sobolev chain $H^s(A+B)$. 

\begin{remark}\label{Rmk:SobolevKatoPert}
Observe that, for a Kato perturbation $A+B$, we have the norm estimate
\begin{align*}
\norm{(A+B)\sigma }\leq \norm{A\sigma}+\norm{B\sigma}\leq (1+b)\norm{A\sigma}+b\norm{\sigma}. 
\end{align*}
Similarly, $\norm{A\sigma}\leq\norm{(A+B)\sigma}+\norm{B\sigma}\leq\norm{(A+B)\sigma}+a\norm{\sigma}+b\norm{A\sigma}$, and therefore
\begin{align*}
\norm{A\sigma}\leq\left(\frac{1}{1-b}\right)\norm{(A+B)\sigma}+\left(\frac{a}{1-b}\right)\norm{\sigma}. 
\end{align*}
These inequalities show that the graph norms of $A$ and $A+B$ are equivalent. Consequently we can identify as Hilbert spaces $H^s(A)\cong H^s(A+B)$ for all $s\in \mathbb{R}$.
\end{remark}

\begin{remark}\label{Rmk:DiffKatoProj}
 Let $A+B$ be a Kato perturbation of $A$. If $A$ and $B$ are invertible then using \cite[Lemma VII.5.6]{KATO}  one can express (see \cite[Equation (5.6)]{B09}),
\begin{align*}
{Q_>(A+B)-Q_>(A)}=\frac{1}{2\pi i}\int_{\text{Re}\: z=0}(A+B-z)^{-1}B(A-z)^{-1}dz. 
\end{align*}
\end{remark}

The following result describes how the ellipticity condition of a spectral projection behaves under a Kato perturbation. 
\begin{theorem}[{\cite[Theorem 5.9]{B09}}]\label{BThm5.9}
Assume that $A+B$ is a Kato perturbation of $A$ with $b<2/3$. Then $Q_{>}(A+B)$ is an elliptic projection with respect to $A$ and the subspaces $Q_{\leq }(A)(H)\coloneqq \ran(Q_\leq (A))$ and $Q_{> }(A+B)(H)\coloneqq \ran(Q_>(A+B))$ from a Fredholm pair.
\end{theorem}

The following consequence will be of particular interests for later purposes. We provide a detailed proof since we need to have control on the constants involved. 
\begin{coro}\label{Coro:VanishingKatoIndexPert}
 If $B$ is bounded and $|A|\geq \mu$ where $\mu>\sqrt{2}\norm{B}$, then 
$$\textnormal{\ind}(Q_{\leq }(A)(H),Q_{>}(A+B)(H))=0.$$
\end{coro}

\begin{proof}
First of all observe that since $B$ is bounded then the condition of Theorem \ref{BThm5.9} is trivially satisfied with $b=0$.  The strategy of the proof is to use \cite[Lemma A.1]{BB01} to show the norm estimate
\begin{align}\label{Eqn:EstimateQAAB}
\norm{Q_>(A)-Q_>(A+B)}< 1.
\end{align}
This condition implies $I-(Q_>(A)-Q_>(A+B))=-Q_>(A)+Q_\leq(A+B)$ is an invertible operator and so we could use Lemma \ref{Prop:A.13} to compute the Kato index,
\begin{align*}
\ind(Q_{\leq }(A)(H), & Q_{>}(A+B)(H))\\
=&\ind((I-Q_{>}(A+B)):Q_{\leq }(A)(H)\longrightarrow  \ker(Q_{>}(A+B)))\\
=&\ind(Q_{\leq}(A+B):Q_{\leq }(A)(H)\longrightarrow  Q_{\leq}(A+B)(H))\\
=&\ind(-Q_>(A)+Q_{\leq}(A+B):Q_{\leq }(A)(H)\longrightarrow  Q_{\leq}(A+B)(H))\\
=& 0.
\end{align*}
Hence, it just remains to show  the estimate \eqref{Eqn:EstimateQAAB}. In view of Remark \ref{Rmk:DiffKatoProj}, we can apply \cite[Lemma A.1]{BB01} with $A_1=A+B$, $A_2=A$, $B(z)=B$ and $\alpha_1=\alpha_2=0$ 
to obtain the estimate
\begin{align*}
\norm{Q_>(A)-Q_>(A+B)}\leq \frac{\norm{B}}{2}\beta(A,0)^{-1/2}\beta(A+B,0)^{-1/2}\leq \frac{\norm{B}}{2}\mu^{-1/2}|\mu-\norm{B}|^{-1/2},
\end{align*}
where we have used \cite[Theorem V.4.10, pg. 291]{KATO} for the last inequality. Therefore, in order to satisfy \eqref{Eqn:EstimateQAAB} we must require
\begin{align*}
\norm{B}^2<4\mu|\mu-\norm{B}|.
\end{align*}
Let us assume $\mu\geq \norm{B}$ first.  The condition above is then 
$$0<-\norm{B}^2-4\mu\norm{B}+4\mu^2.$$
To fulfil this we need to find the values $\xi(=\norm{B})$ for which the parabola defined by $f(\xi)=-\xi^2-4\mu\xi+4\mu^2$ is strictly positive. 
\begin{center}
\begin{tikzpicture}[domain=0:2]
\draw[->] (-1.5,0) -- (6.5,0)
node[below right] {$\xi$};
\draw[->] (0,-3) -- (0,4)
node[left] {$f(\xi)$};
\draw (2,3) parabola (-1,-2);
\draw (2,3) parabola (5,-2);
\draw (-0.3,0) node {$\bullet$};
\draw (4.3,0) node {$\bullet$};
\draw (-0.8,-0.3) node {$\xi_-$};
\draw (4.2,-0.3) node {$\xi_+$};
\end{tikzpicture} 
\end{center}
The roots of $f(\xi)$ are computed as
\begin{align*}
\xi_{\pm}\coloneqq \frac{4\mu\pm \sqrt{16\mu^2-4(-1)(4\mu^2)}}{-2}=(-2\mp 2\sqrt{2})\mu. 
\end{align*}
Thus, we require $(-2-2\sqrt{2})\mu<\xi<(-2+ 2\sqrt{2})\mu$, i.e. 
\begin{align*}
\norm{B}\left(\frac{1}{2}(\sqrt{2}+1)\right)<\mu. 
\end{align*}
Observe that 
\begin{align*}
\left(\frac{1}{2}(\sqrt{2}+1)\right)\simeq 1.207<\sqrt{2}\simeq 1.414. 
\end{align*}
Finally note that if $\mu<\norm{B}$ , then the associated parabola $\tilde{f}(\xi)=-\xi^2+4\mu\xi-4\mu^2$ would just have one root
\begin{align*}
\tilde{\xi}_{\pm}\coloneqq \frac{-4\mu\pm \sqrt{16\mu^2-4(-1)(-4\mu^2)}}{-2}=2\mu,
\end{align*}
so the required condition would never hold true. 
\end{proof}

\subsection{Dirac systems}\label{App:DiracSystems}\label{Section:DiracSystems}

In this section we provide a brief summary of the main results on index theory for Dirac systems developed in \cite{BBC08}. Other important and extensive references around boundary value problems for Dirac-type operators are, for example, \cite{BB12}, \cite{BB13}  and \cite{BW}.\\

As before let $H$ be a separable complex Hilbert space with inner product denoted by $\inner{\cdot}{\cdot}$. For $r\in\mathbb{R}_+$ let $\inner{\cdot}{\cdot}_r$ be a family of scalar products on $H$ compatible with the Hilbert space structure and such that $\inner{\cdot}{\cdot}_0=\inner{\cdot}{\cdot}$. Let us denote by $H_r$ the Hilbert space associated with the inner product $\inner{\cdot}{\cdot}_r$ and by $\mathcal{H}\coloneqq (H_r)_{r\ge 0}$ the family of Hilbert spaces $H_r$. In this context, we define Dirac systems axiomatically following closely  \cite[Section 2]{BBC08}:
\begin{axiom}{}
For all $T\in \mathbb{R}_+\coloneqq [0,\infty)$ there exists a constant $C_T$ such that 
\begin{align*}
|\inner{\sigma_1}{\sigma_2}_s-\inner{\sigma_1}{\sigma_2}_r| \leq \norm{\sigma_1}_r\norm{\sigma_2}_r|r-s|, 
\end{align*}
for all $s,t\in[0,T]$ and $\sigma_1,\sigma_2\in H$.
\end{axiom}
Let $G_r\in\mathcal{L}(H)$ be the bounded operator defined by the relation $\inner{G_r\sigma_1}{\sigma_2}_0=\inner{\sigma_1}{\sigma_2}_r$ for all $\sigma_1,\sigma_2\in H$. By Axiom 1. the map $G:\mathbb{R}_+\longrightarrow \mathcal{L}(H)$ defined by $G(r)\coloneqq G_r$ is in $\text{Lip}_{\text{loc}}(\mathbb{R}_+,\mathcal{L}(H))$, the space of sections which are locally Lipschitz. This implies that it has a weak derivative  $G'\in L_{\text{loc}}^\infty (\mathbb{R}_+,\mathcal{L}(H))$, which is symmetric on $H_0$ for almost all $r\in\mathbb{R}_+$ (\cite[Lemma 2.2]{BBC08}). We can therefore consider the operator
\begin{align*}
\Gamma\coloneqq\frac{1}{2}G^{-1}G'\in L^\infty_\text{loc}(\mathbb{R}_+,\mathcal{L}(H)).
\end{align*}
We now define a {\em continuous metric connection} by
\begin{align*}
\partial\coloneqq\left(\partial_r+\Gamma\right):\text{Lip}_{\text{loc}}(\mathbb{R}_+,H)\longrightarrow L^\infty_\text{loc}(\mathbb{R}_+,H).
\end{align*}
The metric condition means that for $\sigma_1,\sigma_2\in \text{Lip}_{\text{loc}}(\mathbb{R}_+,H)$ the function defined by 
$\inner{\sigma_1}{\sigma_2}(r)\coloneqq \inner{\sigma_1(r)}{\sigma_2(r)}_r$, for $\sigma_1,\sigma_2\in \text{Lip}_{\text{loc}}(\mathbb{R}_+,H)$, satisfies the relation 
\begin{equation}\label{Eqn:MetricConnDS}
\partial_r\inner{\sigma_1}{\sigma_2}=\inner{\partial\sigma_1}{\sigma_2}+\inner{\sigma_1}{\partial \sigma_2}.
\end{equation}
Indeed, it is isntructive to verify this property, 
\begin{align*}
\partial_r\inner{\sigma_1}{\sigma_2}(r)=&\partial_r\inner{\sigma_1(r)}{\sigma_2(r)}_r\\
=&\partial_r\inner{ G_r\sigma_1(r)}{\sigma_2(r)}_0\\
=&\inner{ G'_r\sigma_1(r)}{\sigma_2(r)}_0+\inner{ G_r\sigma'_1(r)}{\sigma_2(r)}_0+\inner{ G_r\sigma_1(r)}{\sigma'_2(r)}_0\\
=&\inner{ 2G_r\Gamma_r \sigma_1(r)}{\sigma_2(r)}_0+\inner{ G_r\sigma'_1(r)}{\sigma_2(r)}_0+\inner{ G_r\sigma_1(r)}{\sigma'_2(r)}_0\\
=&2\inner{\Gamma_r \sigma_1(r)}{\sigma_2(r)}_r+\inner{\sigma'_1(r)}{\sigma_2(r)}_r+\inner{\sigma_1(r)}{\sigma'_2(r)}_r\\
=&\inner{\partial\sigma_1}{\sigma_2}(r)+\inner{\sigma_1}{\partial \sigma_2}(r).
\end{align*}
\begin{axiom}{}
There is a family $\mathcal{A}\coloneqq (A_r)_{r\geq 0}$ of self-adjoint operators $A_r$ on $H_r$ for $r\in\mathbb{R}_+$ with common domain $\dom(A)$ such that 
\begin{enumerate}
\item If $\norm{\cdot}_{A_r}$  denotes the graph norm of $A_r$, then the embedding 
$$(\dom(A),\norm{\cdot}_{A_0})\longrightarrow (H,\norm{\cdot}),$$ is compact. 
\item For all $T\in\mathbb{R}_+$, there is a constant $C_T$ such that 
\begin{align*}
|\inner{A_s \sigma_1}{\sigma_2}_s-\inner{A_r\sigma_1}{\sigma_2}_r|\leq C_T\norm{\sigma_1}_{A_r}\norm{\sigma_2}_r|r-s|
\end{align*}
for all $\sigma_1\in\dom(A)$, $\sigma_2\in H$ and $s,r\in[0,T]$.
\end{enumerate}
\end{axiom}

\begin{axiom}{}
There is a section $\gamma\in \text{Lip}_{\text{loc}}(\mathbb{R}_+,\mathcal{L}(H))\cap L^\infty_\text{loc}(\mathbb{R}_+,\mathcal{L}(\dom(A)))$ such that 
\begin{enumerate}
\item $-\gamma_r=\gamma_r^*=\gamma^{-1}_r$ on $H_r$.
\item $A_r \gamma_r+\gamma_rA_r=0$ on $\dom(A)$.
\item $\partial \gamma=\gamma\partial$ on $ \text{Lip}_{\text{loc}}(\mathbb{R}_+,H)$.
\end{enumerate}
Here $\mathcal{L}(\dom(A))$ denotes the space of all bounded operators on $\dom(A)$ (equipped with the graph norm).
\end{axiom}

\begin{definition}
The triple $\mathcal{D}\coloneqq (\mathcal{H},\mathcal{A},\gamma)$ as above satisfying Axioms 1-3, is called a {\em Dirac System} over $\mathbb{R}_+$.  
\end{definition}
For a Dirac system $\mathcal{D}=(\mathcal{H},\mathcal{A},\gamma)$ define the spaces
\begin{align*}
\mathcal{L}_\text{loc}(\mathcal{D})\coloneqq &\text{Lip}_{\text{loc}}(\mathbb{R}_+,H)\cap L^\infty_\text{loc}(\mathbb{R}_+,\dom(A)),\\
\mathcal{L}_c(\mathcal{D})\coloneqq&\{\sigma\in \mathcal{L}_\text{loc}(\mathcal{D})\:|\:\supp(\sigma)\:\:\text{is compact in $\mathbb{R}_+$}\},\\
\mathcal{L}_{0,c}\coloneqq &\{\sigma\in\mathcal{L}_c(\mathcal{D})\:|\:\sigma(0)=0\},\\
\mathcal{L}_{cc}(\mathcal{D})\coloneqq &\{\sigma\in \mathcal{L}_\text{loc}(\mathcal{D})\:|\:\supp(\sigma)\:\:\text{is compact in $(0,\infty)$}\}.
\end{align*}
On $\mathcal{L}_c(\mathcal{D})$ one defines the inner product
\begin{align}\label{Eqn:L2DiracSyst}
(\sigma_1,\sigma_2)\coloneqq\int_0^\infty \inner{\sigma_1}{\sigma_2}\coloneqq\int_0^\infty \inner{\sigma_1(r)}{\sigma_2(r))}_r dr.
\end{align}
Let $L^2(\mathcal{D})$ denote the Hilbert space completion of $\mathcal{L}_c(\mathcal{D})$ with respect to it. 
\begin{definition}
The {\em Dirac operator} $D$ associated to this Dirac system $\mathcal{D}=(\mathcal{H}, \mathcal{A},\gamma)$ is defined by 
\begin{align*}
D\coloneqq \gamma(\partial+A):\mathcal{L}_{\text{loc}}(\mathcal{D})\longrightarrow L^\infty_\text{loc}(\mathbb{R}_+,H).
\end{align*}
\end{definition}
\begin{remark}
Of course this definition is motivated by the form of the operator \eqref{Eqn:ModelOpCylinder} and Theorem \ref{Thm:TransfDiracOps}, but this theory of Dirac systems can be applied to numerous geometric settings (see \cite{BBC12}).
\end{remark}

It is easy to verify using \eqref{Eqn:MetricConnDS} and Axiom 3. the relation for sections $\sigma_1,\sigma_2\in\mathcal{L}_\text{loc}(\mathcal{D})$,
\begin{align}\label{Eqn:DefectSymD}
\int_0^\infty \inner{D\sigma_1}{\sigma_2}-\int_0^\infty\inner{\sigma_1}{D\sigma_2}=\inner{\sigma_1}{\gamma\sigma_2}\bigg{|}_0^\infty.
\end{align}
Observe that formally $D^\dagger=(-\partial+A)(-\gamma))=D$, so in order to compute indices we want to introduce a self-adjoint involution that splits this operator.
\begin{definition}
A {\em super-symmetry} for a Dirac system $\mathcal{D}$ is defined by an involution $\alpha\in \text{Lip}_\text{loc}(\mathbb{R}_+,\mathcal{L}(H))\cap L^\infty_\text{loc}(\mathbb{R}_+,\mathcal{L}(\dom(A)))$ satisfying
\begin{enumerate}
\item $\alpha_r=\alpha_{r}^* =\alpha_{r}^{-1}$ on $H_r$.
\item $\alpha_{r}\gamma_r+\gamma_r\alpha_{r}=0$ on $\dom(A)$.
\item $\partial\alpha=\alpha\partial$ on $\text{Lip}_\text{loc}(\mathbb{R}_+,\mathcal{L}(H))$.
\item $A_r\alpha_{r}=\alpha_{r}A_r$ on $\dom(A)$.
\end{enumerate}
\end{definition}
In the presence  of such a super-symmetry we define $H_r^{\pm}$ as the $\pm 1$-eigenspaces of $\alpha_{r}$ in $H_r$ so that $H_r=H^+_r\oplus H_r^-$. In a similar manner we can decompose 
$$\dom(A)=\dom(A_r)^{+}\oplus\dom(A_r)^{-},$$ 
where $\dom(A_r)^\pm\coloneqq \dom(A)\cap H^{\pm}_r$. With respect to this decomposition we can express the operator $A_r$ as
\begin{align*}
A_r=\left(
\begin{array}{cc}
A^+_r & 0\\
0 & A^-_r
\end{array}
\right).
\end{align*}
Here $A_r^\pm:\dom(A_r)^\pm\subset H_r^\pm\longrightarrow H_r^\pm$. Moreover, we also obtain the decomposition of 
\begin{align*}
\partial=\left(
\begin{array}{cc}
\partial^+ & 0\\
0 & \partial^-
\end{array}
\right)
\quad\quad\text{so that}\quad\quad
D=\left(
\begin{array}{cc}
0 &D^- \\
 D^+ & 0
\end{array}
\right).
\end{align*}
A super-symmetry also induces a decomposition $L^2(\mathcal{D})=L^2(\mathcal{D})^+\oplus L^2(\mathcal{D})^-$ where 
\begin{align*}
L^2(\mathcal{D})^{\pm}\coloneqq \{\sigma\in L^2(\mathcal{D})\:|\:\ran(\sigma)\in H^\pm\}. 
\end{align*}

We now want to discuss some closed extensions of the operator $D$. Observe from \eqref{Eqn:DefectSymD} that the restriction ${D}_{0,c}$ of $D$ to $\mathcal{L}_{0,c}$ is a symmetric operator, thus we can always consider (see Section \ref{Section:DiffOp}),
\begin{itemize}
\item The {\em minimal} extension defined by $D_\text{min}\coloneqq \overline{D_{0,c}}$.
\item The {\em maximal} extension defined by $D_\text{max}\coloneqq (D_{0,c})^*$.
\end{itemize}

In view of the APS index formula \eqref{Eqn:IndexDBoundary}, we know that relevant geometric closed extensions for Dirac-type operators on manifolds with boundary are obtained by imposing appropriate boundary conditions. In \cite{BW}, for example, elliptic boundary conditions for such operators are defined by requiring some compatibility condition with the principal symbol of the Calder\'on projector (\cite[Definition 18.1]{BW}). In the setting of Dirac systems this point of view is somehow modified and the boundary conditions are determined by closed subspaces of certain hybrid Sobolev space. Let us describe this more concretely. Consider the Sobolev chain $H^s=H^s(A_0)$ associated to the Hilbert space $H_0$ and the self-adjoint operator $A_0$ as in Section \ref{Sec:Perturbations}. One can define the space
\begin{align*}
\check{H}\coloneqq\check{H}(A_0)\coloneqq H^{-1/2}_{>}\oplus Q_0(H)\oplus H^{1/2}_<, 
\end{align*}
where $Q_0(H)$ denotes the kernel of $A$ in $H$. 

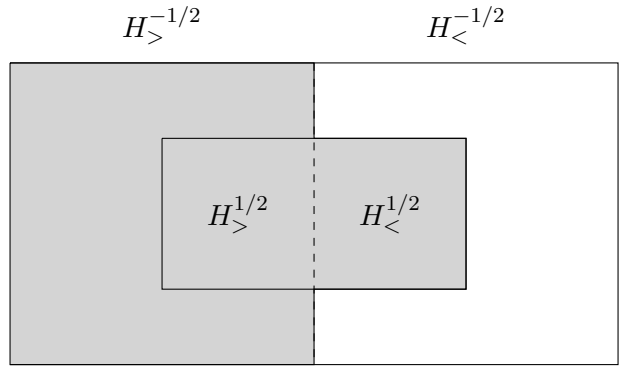
\begin{figure}[h]
\begin{center}
\begin{tikzpicture}
\draw [fill={rgb:black,1;white,5}] (-4,2)--(0,2)--(0,1)--(2,1)--(2,-1)--(0,-1)--(0,-2)--(-4,-2)--(-4,2);
\draw (-4,2)--(4,2)--(4,-2)--(-4,-2)--(-4,-2);
\draw[dashed] (0,2)--(0,-2);
\node at (-2,2.5) {$H^{-1/2}_>$};
\node at (2,2.5) {$H^{-1/2}_<$};
\draw (-2,1)--(2,1)--(2,-1)--(-2,-1)--(-2,-1)--(-2,1);
\node at (-1,0) {$H^{1/2}_>$};
\node at (1,0) {$H^{1/2}_<$};
\end{tikzpicture}
\caption{The shadowed area represents schematically the space $\check{H}$. It is important to point out that the norm in $H^{1/2}_<$ is $\norm{\cdot}_{1/2}$ and not the induced $\norm{\cdot}_{-1/2}$ from $H^{-1/2}_<$.}
\end{center}
\end{figure}

\begin{definition}
A closed subspace $B\subset \check{H}$ is called  a (linear) {\em boundary condition} for the Dirac system $\mathcal{D}$. For any such $B$ we define the domain of the corresponding operator $D_{B,\text{max}}$ by $\dom(D_{B,\text{max}})\coloneqq \{\sigma\in\dom(D_\text{max})\:|\:\sigma(0)\in B\}$. 
\end{definition}
The following result unravels the definition above. 
\begin{proposition}[{\cite[Proposition 1.50]{BBC08}}]\label{Prop:BBC08Prop1.50}
Closed extensions of $D_{0,c}$ between the minimal extension $D_\textnormal{min}$ and the maximal extension $D_\textnormal{max}$ correspond to closed linear subspaces $B$ of $\check{H}$.
\end{proposition}

Given a boundary condition $B$ we can associate to it yet another closed operator.
\begin{definition}
For a boundary condition $B$ we define the operator $D_B$ to be the restriction of $D_{B,\text{max}}$ to the domain
$\dom(D_B)\coloneqq\{\sigma\in \dom(D_{\text{max}})\:|\: \sigma(0)\in B\cap H^{1/2}\}$.
\end{definition}
Now we describe a special type of boundary conditions on which $D_{B,\text{max}}$ and $D_{B}$ agree. 
\begin{definition}
A boundary condition is called {\em regular} if $B\subseteq H^{1/2}\subset \check{H}$.
\end{definition}
In order to define adjoint boundary conditions we need to consider, in view of \eqref{Eqn:DefectSymD}, the the non-degenerate, continuous, skew-symmetric pairing $\omega:\check{H}\times\check{H}\longrightarrow\mathbb{C}$ defined by 
\begin{align*}
\omega(\sigma_1,\sigma_2)\coloneqq B_{1/2}(Q_{\leq }\sigma_1,\gamma Q_{\geq }\sigma_2)+B_{-1/2}(Q_{> }\sigma_1,\gamma Q_{<}\sigma_2),
\end{align*}
where $B_{\pm 1/2}$ is given by \eqref{PairingB}. Indeed, the right hand side of \eqref{Eqn:DefectSymD} is precisely $\omega(\sigma_1(0),\sigma_2(0))$. In view of this fact and by Proposition \ref{Prop:BBC08Prop1.50} the following definition arises naturally. 
\begin{definition}
For a boundary condition $B\subseteq \check{H}$ we define its {\em annihilator}
\begin{align*}
B^a\coloneqq \{\sigma_1\in \check{H}\:|\:\omega(\sigma_1,\sigma_2)=0,\:\forall \sigma_2\in B\}.
\end{align*}
We say that the boundary condition $B\subseteq \check{H}$ is {\em elliptic} if $B$ and $B^a$ are both regular. We say that 
$B$ is {\em self-adjoint} if $B=B^a$.
\end{definition}

\begin{remark}[{\cite[Lemma 1.46]{BBC08}}]\label{Rmk:SimplifyOmega}
If $\sigma_1,\sigma_2\in \check{H}$  with $\sigma_1\in H^{1/2}$ then $Q_>\sigma_1\in H^{1/2}$ and therefore 
\begin{align*}
B_{-1/2}(Q_{> }\sigma_1,\gamma Q_{<}\sigma_2)=&({(I+A^2)^{-1/4}Q_>\sigma_1},{(I+A^2)^{1/4}\gamma Q_<\sigma_2})\\
=&({(I+A^2)^{-1/4}Q_>\sigma_1},{(I+A^2)^{1/4}\gamma Q_<\sigma_2})\\
=&((I+A^2)^{-1/2}{(I+A^2)^{1/4}Q_>\sigma_1},{(I+A^2)^{1/4}\gamma Q_<\sigma_2})\\
=&({(I+A^2)^{1/4}Q_>\sigma_1},{(I+A^2)^{-1/4}\gamma Q_<\sigma_2})\\
=&B_{1/2}(Q_{> }\sigma_1,\gamma Q_{<}\sigma_2).
\end{align*} 
So in this case $\omega(\sigma_1,\sigma_2)$ simplifies to 
\begin{align*}
\omega(\sigma_1,\sigma_2)=&B_{1/2}(Q_{\leq }\sigma_1,\gamma Q_{\geq }\sigma_2)+B_{-1/2}(Q_{> }\sigma_1,\gamma Q_{<}\sigma_2)\\
=&B_{1/2}(Q_{\leq }\sigma_1,\gamma Q_{\geq }\sigma_2)+B_{1/2}(Q_{> }\sigma_1,\gamma Q_{<}\sigma_2)\\
=& B_{1/2}(\sigma_1,\gamma\sigma_2).
\end{align*}
\end{remark}
\begin{example}[APS boundary conditions]\label{Example:APSBoundCond}
The boundary condition introduced by Atiyah, Patodi and Singer in \cite{APSI} required to obtain Theorem \ref{Thm:SignThmMBound} (see Section \ref{Section:SigThmMB}) can be described in this setting by the closed subspace $B_{APS}\coloneqq \check{H}_<={H}^{1/2}_{<}$. It is immediate from the definition that  $B_{APS}$ is a regular boundary condition. Let us compute its associated adjoint boundary condition. For $\sigma_1\in B_{APS}$ and $\sigma_2\in\check{H}$ we have, by the last remark, $\omega(\sigma_1,\sigma_2)= B_{1/2}(\sigma_1,\gamma\sigma_2)$
so $\sigma_2\in B^a_{APS} \Leftrightarrow Q_< \gamma\sigma_2 =0 \Leftrightarrow Q_>\sigma_2 =0 $.
This shows that $B_{APS}^a\coloneqq H^{1/2}_{\leq }  $, which is regular as well. Hence, $B_{APS}=H^{1/2}_{<}$ is an elliptic boundary condition.
\end{example}

\begin{example}[{\cite[Example 1.85]{BBC08}}]\label{Ex:1.85}
Let $\beta:H\longrightarrow H$ be a $1/2$-smooth map, with restriction  $\hat{\beta}:H^{1/2}\longrightarrow H^{1/2}$ and extension $\tilde{\beta}:H^{-1/2}\longrightarrow H^{-1/2}$, 
such that 
\begin{enumerate}
\item $\beta^*=\beta^{-1}=\beta$,
\item $\gamma\beta+\beta\gamma=0$,
\item $A\beta+\beta A=0$.
\end{enumerate}
 Then 
$B\coloneqq \{\sigma\in H^{1/2}\:|\:\hat{\beta}\sigma=\sigma\}\subset \check{H}$ is a regular self-adjoint boundary condition, and therefore also elliptic. Indeed it is clear that $B$ is regular since $B\subset H^{1/2}$. On the other hand, if $\sigma_1\in B$ and $\sigma_2\in\check{H}$ then by Remark \ref{Rmk:SimplifyOmega} we have $\omega(\sigma_1,\sigma_2)=B_{1/2}(\sigma_1,\gamma\sigma_2)$. Moreover, from the commutation relations of $\gamma$ and $A$ with $\beta$ we see that 
\begin{align*}
B_{1/2}(\sigma_1,\gamma\sigma_2)=B_{1/2}(\hat{\beta}\sigma_1,\gamma\sigma_2)=B_{1/2}(\sigma_1,\tilde{\beta}\gamma\sigma_2)=-B_{1/2}(\sigma_1,\gamma\tilde{\beta}\sigma_2),
\end{align*}
which shows that $\tilde{\beta}(B^a)= B^a$ and $B_{1/2}(\sigma_1,\gamma(\sigma_2+\tilde{\beta}\sigma_2))=0$.
In particular we get $B\subset B^a$. To prove the opposite inclusion let us assume that $\sigma_2\in B^a\cap H_>^{-1/2}$, then 
$\tilde{\beta}\sigma_2=\tilde{\beta}Q_>\sigma_2=Q_<\tilde{\beta}\sigma_2\in B^a$.
This shows $\tilde{\beta}\sigma_2\in H^{1/2}_<\cap B^a$ which implies, by applying $\tilde{\beta}$, that $\sigma_2\in B^a\cap H_>^{1/2}$. As a result, since in this case  $0=B_{1/2}(\sigma_1,\gamma\sigma_2)=\inner{\sigma_1}{\gamma\sigma_2}$, we see that $\gamma \sigma_2$ is in the  orthogonal complement of the $+1$-eigenspace of $\beta$. Consequently, by property (2), $\sigma_2$ lies in the $+1$-eigenspace of $\beta$, which shows that $B^a\subset B$. 
\end{example}

In the presence of a super-symmetry $\alpha$ we are interested in boundary conditions $B$ which are $\alpha$-invariant, i.e. such that  $\alpha_0(B)= B$ in $\check{H}$. We call this class {\em super-symmetric boundary conditions}.  For such a boundary condition we have a decomposition $B=B^+\oplus B^-$ induced by $\alpha_0$.  A super-symmetric boundary condition $B$ is regular/elliptic if and only if $B^\pm$ are regular/elliptic in $\check{H}^{\pm}$.\\

We want to describe  boundary conditions induced by the regular projections discussed in the last section. Let $P$ be a $1/2$-smooth projection in $H$. Then $B_P\coloneqq \ker\tilde{P}\cap \check{H}$ is a closed subspace of $\check{H}$ and therefore it defines a boundary condition. 
\begin{proposition}[{\cite[Proposition 1.99]{BBC08}}]\label{Prop:BBCProp1.99}
For a $1/2$-smooth orthogonal projection $P$ in $H$ the following are equivalent:
\begin{enumerate}
\item $P$ is a regular orthogonal projection.
\item $B_P=\ker\hat{P}$.
\item $B_P$ is a regular boundary condition.
\end{enumerate}
\end{proposition}

\begin{example}[APS projection]\label{Example:APSproj}
The projection associated to the APS boundary condition from Example \ref{Example:APSBoundCond} is $P_{APS}\coloneqq Q_{\geq}$ since $\ker Q_{\geq}\cap\check{H}=\check{H}_<$. Moreover, observe that  Example \ref{Example:SpectralProjections} and Example \ref{Example:APSBoundCond} verify the statement of Proposition \ref{Prop:BBCProp1.99}.
\end{example}

We continue our short excursion describing some index theorems for Dirac operators associated with super-symmetric Dirac systems. To be able to state these results  we need to introduce some additional notion. Indeed,  even for an elliptic boundary condition $B$, the corresponding operator $D_{B,\text{max}}$ is in general not Fredholm. It turns out that, in order to obtain a Fredholm operator, we just need to impose a further condition: non-parabolicity (see \cite{C01a},\cite{C01}). To describe it we define for $\sigma\in\dom(D_\text{max})$ the norm
\begin{align*}
\norm{\sigma}^2_W\coloneqq \norm{\sigma(0)}^2_{\check{H}}+\norm{D_\text{max}\sigma}^2_{L^2(\mathcal{D})}.
\end{align*} 
Here $\norm{\cdot}_{\check{H}}$ denotes the norm of the hybrid space $\check{H}$ and $\norm{\cdot}^2_{L^2(\mathcal{D})}$ is induced by \eqref{Eqn:L2DiracSyst}.
\begin{definition}\label{Def:NonPara}
We say that $\mathcal{D}$ is {\em non-parabolic} if for each $T>0$ exists a constant $C_T>0$ such that 
\begin{align*}
\norm{\sigma}_{L^2([0,T],\mathcal{H})}\leq C_T\norm{\sigma}_{W}\quad \forall\sigma\in\mathcal{L}_{c}(\mathcal{D}),
\end{align*} 
where $\norm{\sigma}_{L^2([0,T],\mathcal{H})}^2\coloneqq \norm{1_{[0,T]}\sigma }^2_{L^2(\mathcal{D})}$ and $1_{[0,T]}$ is the characteristic of $[0,T]$. 
\end{definition}

\begin{remark}\label{Rmk:NonPar}
It is actually enough check the inequality of Definition \ref{Def:NonPara} for sections $\sigma\in\mathcal{L}_{0,c}(\mathcal{D})$ by \cite[Lemma 2.38]{BBC08}.
\end{remark}

For a non-parabolic Dirac system $\mathcal{D}$ define $W\subset L^2_\text{loc}(\mathcal{D})$ to be the completion of $\mathcal{L}_c(\mathcal{D})$ with respect to the norm $\norm{\cdot}_W$ defined above. In this case one can extend $D_\text{max}$ to a bounded operator
$D_\text{ext}:W\longrightarrow L^2(\mathcal{D})$. Additionally, if we are given a boundary condition $B\subseteq\check{H}$, we associate to it the closed operator $D_{B,\text{ext}}$ with domain 
$$W_B\coloneqq \dom(D_{B,\text{ext}})\coloneqq \{\sigma\in W\:|\:\sigma(0)\in B\}.$$

\begin{remark}[{\cite[Section 2.c]{C01}}]
The notion of non-parabolicity is motivated by the results of \cite{APSI}. Here the index of a Dirac-type operator on a compact manifold with boundary is interpreted as the $L^2$-index of an associated {\em elongated manifold} with cylindrical ends. 
\end{remark}

Now we state one on the most important result of Dirac systems. 

\begin{theorem}[{\cite[Theorem 2.43]{BBC08}}]\label{Thm:BBC2.43}
Let $\mathcal{D}$ be a non-parabolic Dirac system and let $B$ be a regular boundary condition. Then $D_{B,\textnormal{ext}}:W_B\longrightarrow L^2(\mathcal{D})$ is a left-Fredholm operator with $(\ran \: D_{B,\textnormal{ext}} )^\perp=\ker D_{B,\textnormal{max}}$. The extended index of $D_{B,\textnormal{ext}}$ is defined by 
\begin{align*}
\textnormal{\ind}\:D_{B,\textnormal{ext}}\coloneqq \dim\ker D_{B,\textnormal{ext}}-\dim\ker D_{B^a,\textnormal{max}}.
\end{align*}
\end{theorem}

\begin{coro}[{\cite[Corollary 2.44]{BBC08}}]
Let $\mathcal{D}$ be a non-parabolic Dirac system and $B$ an elliptic boundary condition. Then $D_{B}$ and $D_{B^a}$ have finite dimensional kernels and we can define the $L^2$-index of $D_{B}$ as
\begin{align*}
L^2-\textnormal{\ind}\:D_{B}\coloneqq \dim\ker D_{B}-\dim\ker D_{B^a}.
\end{align*}
\end{coro}
\begin{proof}
By Theorem \ref {Thm:BBC2.43} we know that $\dim\ker D_{B,\textnormal{ext}}$ is finite. In addition, since $B$ is elliptic then $\ker D_{B,\textnormal{max}}=\ker D_{B}$, so we see that the kernel of $D_B$ has finite dimension. We can argue similarly for $D_{B^a}$. 
\end{proof}

The above notions and results can be generalized to Dirac systems with potentials.
\begin{definition}[{\cite[Section 2.2]{BBC08}}]
A {\em Dirac-Schr\"odinger system} consists of a Dirac system $\mathcal{D}=(\mathcal{H},\mathcal{A},\gamma)$ and a {\em potential} $V\in L_\text{loc}^\infty(\mathbb{R}_+,\mathcal{L}({H}))$ with $V=V^*$. The associated  Dirac-Schr\"odinger operator is given by $D+V:\mathcal{L}_\text{loc}(\mathcal{D})\longrightarrow  L^\infty_\text{loc}(\mathbb{R}_+,H)$. In addition, a {\em super-symmetric Dirac-Schr\"odinger system} can be defined with an involution $\alpha$ as above with the additional condition that $\alpha_{r} V_t+V_t\alpha_{r}=0$ on $H_r$. 
\end{definition}

The next result describes how to compute the extended index in presence of a super-symmetry. 

\begin{proposition}[{\cite[Proposition 4.9]{BBC08}}]
Let $(\mathcal{D},V,\alpha)$ be a non-parabolic super-symmetric Dirac-Schr\"odinger system and $B$ be an $\alpha_0$-invariant elliptic boundary condition. Then
\begin{align*}
\textnormal{\ind}\:D_{B,\textnormal{ext}}=\textnormal{\ind}\:D^+_{B^+,\textnormal{ext}}+\textnormal{\ind}\:D^-_{B^-,\textnormal{ext}},
\end{align*}
where $D^\pm_{B^\pm,\textnormal{ext}}:W_B^\pm\longrightarrow L^2(\mathcal{D})^{\mp}$.
\end{proposition}
The following theorem describes how the index changes when the given boundary condition is changed to an APS boundary condition. These kind of formulas are known as Agranovi\v{c}-Dynin type (see \cite[Section 21]{BW}).

\begin{theorem}[{\cite[Theorem 4.14]{BBC08}}]\label{BBCThm4.14}
Let $(\mathcal{D},V,\alpha)$ be a non-parabolic super-symmetric Dirac-Schr\"odinger system and let $B$ be an $\alpha_0$-invariant elliptic boundary condition. Then 
\begin{align*}
\textnormal{\ind}\:D^+_{B^+,\textnormal{ext}}=\textnormal{\ind}\:D^+_{{H^+_{\leq}},\textnormal{ext}}+\textnormal{\ind}(\bar{B}^+,H^+_{>}),
\end{align*}
where $\bar{B}$ denotes the closure of $B$ in $H$ and  $\textnormal{\ind}(\bar{B}^+,H^+_{>})$ is the Kato index of the Fredholm pair $(\bar{B}^+,H^+_{>})$.
\end{theorem}

To end this section we describe how to glue two Dirac systems consistently and how to obtain a gluing formula for the index (\cite[Section 3.2]{BBC08}). We will illustrate this construction in the next section through an example. 
\begin{definition}\label{Def:CompDiracSys}
Let $(\mathcal{D}_1,V_1,\alpha_1)$ and $(\mathcal{D}_2,V_2,\alpha_2)$ be super-symmetric Dirac-Schr\"odinger systems. We say that they are {\em compatible} if
\begin{enumerate}
\item The initial Hilbert spaces $H_{1,0}$ and $H_{2,0}$ coincide. 
\item The initial self-adjoint operators satisfy $A_{1,0}=-A_{2,0}=:A$.
\item We have $\gamma_{1,0}=-\gamma_{2,0}=:\gamma$.
\item The involutions $\alpha_1$ and $\alpha_2$ coincide at $r=0.$
\end{enumerate}
\end{definition}
Under these assumptions we can define a new super-symmetric Dirac-Schr\"odinger system by $(\mathcal{D},V,\alpha)=(\mathcal{D}_1\oplus \mathcal{D}_2,V_1\oplus V_2,\alpha_1\oplus \alpha_2)$. In addition, for such a ``glued'' Dirac system we can always consider the elliptic and self-adjoint boundary condition $B\coloneqq \{(\sigma,\sigma)\:|\:\sigma\in H^{1/2}\}$. This  follows automatically from Example \ref{Ex:1.85} by setting  the map $\beta$ to be $\beta(\sigma_1,\sigma_2)=(\sigma_2,\sigma_1)$ in $(\mathcal{D},V,\alpha)$. For example,

\begin{align*}
\beta\gamma=
\left(
\begin{array}{cc}
0 & I\\
I & 0
\end{array}
\right)
\left(
\begin{array}{cc}
\gamma & 0\\
0 & -\gamma
\end{array}
\right)
=
\left(
\begin{array}{cc}
 0& -\gamma\\
\gamma & 0  
\end{array}
\right),\\
\gamma\beta=
\left(
\begin{array}{cc}
\gamma & 0\\
0 & -\gamma
\end{array}
\right)
\left(
\begin{array}{cc}
0 & I\\
I & 0
\end{array}
\right)
=
\left(
\begin{array}{cc}
 0& \gamma\\
-\gamma & 0  
\end{array}
\right).
\end{align*}
Similarly one verifies that $A\beta+\beta A=0$. We call $B$  the {\em transmission boundary condition}. Moreover, one directly sees by condition (4), that this boundary conditions is $\alpha$-invariant and $B=B^+\oplus B^{-}$  for  $B^{\pm}\coloneqq \{(\sigma,\sigma)\:|\:\sigma\in H^{\pm}\cap H^{1/2}\}$, where $H^{\pm}$ is induced by $\alpha_{1,0}=\alpha_{2,0}$. If $(\mathcal{D}_1,V_1,\alpha_1)$ and $(\mathcal{D}_2,V_2,\alpha_2)$  are both non-parabolic then $(\mathcal{D},V,\alpha)$ is also non-parabolic. 
\begin{remark}\label{Rmk:HybridSpDouble}
Observe that 
\begin{align*}
\check{H}(-A)=& H^{-1/2}(-A)_{>}\oplus Q_0(-A)(H)\oplus H^{1/2}(-A)_<\\
=& H^{-1/2}(A)_{<}\oplus Q_0(A)(H)\oplus H^{1/2}(A)_>.
\end{align*}
Thus, the hybrid Sobolev space of the sum of two compatible Dirac system is
\begin{align*}
\check{H}(A\oplus -A)= 
{
\begin{array}{c}
 H^{-1/2}(A)_{>}\oplus Q_0(A)(H)\oplus H^{1/2}(A)_<\\
\oplus\\
H^{-1/2}(A)_{<}\oplus Q_0(A)(H)\oplus H^{1/2}(A)_>.
\end{array}
}
=H^{-1/2}(A)\oplus H^{1/2}(A)
\end{align*}
\end{remark}

\begin{theorem}[{\cite[Theorem 4.17]{BBC08}}]\label{BBCThm4.17}
Let $(\mathcal{D}_1,V_1,\alpha_1)$ and $(\mathcal{D}_2,V_2,\alpha_2)$ be two compatible non-parabolic super-symmetric Dirac-Schr\"odinger systems as above. If $B_1$ is any $\alpha_1$-invariant elliptic boundary condition for $(\mathcal{D}_1,V_1,\alpha_1)$ and  $B_2$ is any $\alpha_2$-invariant elliptic boundary condition for $(\mathcal{D}_2,V_2,\alpha_2)$ then 
\begin{align*}
\textnormal{\ind}\:D^+_{B^+,\textnormal{ext}}=\textnormal{\ind}\:D^+_{1,B_{1}^+,\textnormal{ext}}+\textnormal{\ind}\:D^+_{2,B_{2}^+,\textnormal{ext}}-\ind(H^+_>,\bar{B}^+_1)-\ind(H^+_\leq,\bar{B}^+_2).
\end{align*}
In particular, if such $B_1$ is given and we set $B_2\coloneqq B^\perp_1\cap H^{1/2}$, then 
\begin{align*}
\textnormal{\ind}\:D^+_{B^+,\textnormal{ext}}=\textnormal{\ind}\:D^+_{1,B_{1}^+,\textnormal{ext}}+\textnormal{\ind}\:D^+_{2,B_{2}^+,\textnormal{ext}}.
\end{align*}
\end{theorem}

\subsubsection{Example: The Novikov additivity of signature}\label{Section:ExNovikov}

In order to illustrate the theory of Dirac systems discussed above we work out an explicit example: we derive the Novikov additivity  formula of the signature (Proposition \ref{Prop:Novikov}) using the APS index formula \eqref{Eqn:IndexDBoundary}. In addition, we verify Theorem \ref{BBCThm4.17} in this context using Theorem \ref{BBCThm4.14}.\\

 Let $X_1$ and $X_2$ be two $4k$-dimensional compact oriented smooth manifolds with boundary such that $\partial X_1=\partial X_2=Y$ as oriented closed manifolds (see Remark \ref{Rmk:OrBound}). Assume in addition that $X_1$ and $X_2$ carry Riemannian metrics such that close to their boundary they are product metrics. This is not a strong restriction since otherwise we can always glue a small cylinder and deform the metric as described in \cite[Chapter 9]{BW}.
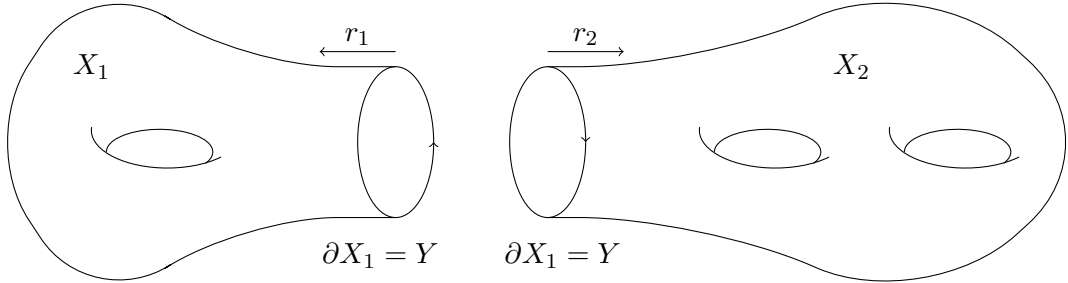
\begin{figure}[h]
\begin{center}
\begin{tikzpicture}
\draw[rounded corners=35pt](6,-1)--(4,-1)--(2,-2.3)--(0.5,0)--(2,2.3)--(4,1)--(6,1);
\draw (2,0.2) arc (175:315:1cm and 0.5cm);
\draw (3.5,-0.28) arc (-30:180:0.7cm and 0.3cm);
\draw[->](6.5,0) arc (0:360:0.5cm and 1cm);
\node at (5.8,-1.5) {$\partial X_1=Y$};
\node at (2,1) {$X_1$};
\draw[rounded corners=45pt](8,-1)--(10,-1)--(13,-2.3)--(15.5,0)--(13,2.3)--(10,1)--(8,1);
\draw[<-] (8.5,0) arc (0:360:0.5cm and 1cm);
\draw (12.5,0.2) arc (175:315:1cm and 0.5cm);
\draw (14,-0.28) arc (-30:180:0.7cm and 0.3cm);
\draw (10,0.2) arc (175:315:1cm and 0.5cm);
\draw (11.5,-0.28) arc (-30:180:0.7cm and 0.3cm);
\node at (8.2,-1.5) {$\partial X_1=Y$};
\node at (12,1) {${X_2}$};
\draw [<-] (5,1.2)--(6,1.2);
\draw [->] (8,1.2)--(9,1.2);
\node at (5.5,1.4) {$r_1$};
\node at (8.5,1.4) {$r_2$};
\end{tikzpicture}
\caption{Two smooth manifolds $X_1,X_2$ having the same oriented boundary $\partial X_1=\partial X_2=Y$.}\label{Fig:Add}
\end{center}
\end{figure}
Recall that if we denote by $-X_2$ the manifold $X_2$ with reversed orientation then we can glue $X_1$ and $-X_2$ along their common oriented boundary to obtain a closed oriented $4k$-dimensional Riemannian manifold $X\coloneqq X_1\cup (-X_2)$. The signature additivity formula states that
\begin{equation}\label{Eqn:Additivity}
\sigma(X)=\sigma(X_1)+\sigma(-X_2),
\end{equation}
where $\sigma(-X_2)=-\sigma(X_2)$ by Proposition \ref{Prop:PropSignatureClosed}. The main idea to derive \eqref{Eqn:Additivity} as a gluing index theorem is to construct two compatible Dirac and then compute their corresponding indices. \\

From Theorem \ref{Thm:TransfDiracOps} we know that the Hodge-de Rham operator on $X_1$ is unitary equivalent to the operator 
\begin{align*}
D_1=\gamma\left(\frac{\partial}{\partial r_1}+A_1\right),
\end{align*}
where
\begin{align*}
A_1=\left(
\begin{array}{cc}
I & 0\\
0 & -I
\end{array}
\right)\otimes
A_Y\quad\text{and}\quad
\gamma=\left(
\begin{array}{cc}
0& -I\\
I & 0
\end{array}
\right).
\end{align*} 
Here $A_Y$ denotes the odd signature operator on $Y$ and $r_1$ denotes the inward normal coordinate close to the boundary in $X_1$ (see Figure \ref{Fig:Add}). Note that $A_1\gamma_1+\gamma_1 A_1=0$, as required. 
Under the transformation $\Psi$ of Section \ref{Section:LocalDescrHdROp} the chirality operator becomes
\begin{align*}
\star_{X_1}=\left(
\begin{array}{cc}
I& 0\\
0 & -I
\end{array}
\right)=\gamma
\left(
\begin{array}{cc}
0& -I\\
-I & 0
\end{array}
\right),
\end{align*}
so the associated positive graded operator is 
\begin{align*}
D^+_1=&\frac{\partial}{\partial r_1}+ A_Y. 
\end{align*}
\begin{remark}\label{Rmk:NonParabNovikov}
Observe that this Dirac system is non-parabolic since $X_1$ is compact, as the non-parabolicity condition only involves the behavior of sections at infinity. More concretely, for  $T>0$ we compute for a compactly supported section $\sigma$ in $X_1$ vanishing at the boundary (see Remark \ref{Rmk:NonPar}),
\begin{align*}
\int_{0}^T \norm{D^+_1\sigma(r)}^{2}_{L^2(Y)} dr=& \int_{0}^T \norm{\sigma'(r)}^{2}_{L^2(Y)} dr+\int_{0}^T \norm{A_Y\sigma(r)}^{2}_{Y}
dr\\
&+\int_{0}^T 2\text{Re}(\sigma'(r),A_Y\sigma(r))_{L^2(Y)} dr.
\end{align*}
Note that 
\begin{align*}
\frac{d}{dr}(\sigma(r),A_Y\sigma(r))_{L^2(Y)}=&2\text{Re}(\sigma'(r),A_Y\sigma(r))_{L^2(Y)},
\end{align*}
and therefore
\begin{align*}
\norm{D^+_1\sigma}_{L^2(X_1)}^2=\norm{\sigma'}_{L^2(X_1)}^2+\norm{A_Y\sigma}_{L^2(X_1)}^2.
\end{align*}
Following the proof of \cite[Proposition 2.5]{C01} we estimate using Cauchy-Schwarz inequality
\begin{align*}
\norm{\sigma(r)}_{L^2(Y)}=\bnorm{\int_0^r\sigma'(s)ds}_{L^2(Y)}\leq \sqrt{r}\norm{\sigma'}_{L^2(X_1)},
\end{align*}
and conclude that 
\begin{align*}
\int_0^T \norm{\sigma(r)}_{L^2(Y)}^2 dr\leq \norm{\sigma'}^2_{L^2(X_1)} \int_0^T rdr = \frac{T^2}{2}\norm{\sigma'}_{L^2(X_1)}^2 \leq \frac{T^2}{2}\norm{D^+_1\sigma}_{L^2(X_1)}^2.
\end{align*}
This verifies that this Dirac system is non-parabolic. 
\end{remark}
In an analogous manner, the Hodge-de Rham operator on $X_2$ can be written as 
\begin{align*}
D_2=\gamma\left(\frac{\partial}{\partial r_2}+A_2\right),
\end{align*}
where $A_2=A_1$ since the orientation of $Y$ in both manifolds coincide. As before, the corresponding positive graded operator of $X_2$ is 
\begin{align*}
D^+_2=&\frac{\partial}{\partial r_2}+A_Y.
\end{align*}
Observe however that $r_2=-r_1\eqqcolon r$, thus we can write
\begin{align}\label{Eqn:D2+}
D_2=-\gamma\left(\frac{\partial}{\partial r}-A_1\right).
\end{align}
Hence, the Dirac systems $(D_1,\star_{X_1})$ and $(D_2,\star_{-X_2})$ are compatible in the sense of Definition \ref{Def:CompDiracSys}. It is important to see that $\star_{X_1}=\star_{-X_2}=-\star_{X_2}$. 
Th next step is to glue $D_1$ and $D_2$ to construct a global operator on $X$. Observe that since we have reversed the orientation of $X_2$ in $X$, then we actually need to glue $D^+_1$ and $D^-_2$. From \eqref{Eqn:D2+} we see that 
\begin{align*}
D^-_2=\frac{\partial}{\partial r}+A_Y,
\end{align*}
so indeed $D^+\coloneqq D^+_1\cup D^-_2$ defines a global operator on $X$.\\

Now we want to compute the index of $D^+$ in order to compute the signature of $X$. First we compute the index on each separate piece. From \eqref{Eqn:IndexDBoundary} we have
\begin{align*}
\ind(D^+_{1,Q_<(A_Y)(H)})=\int_{X_1}\alpha_1(x)dx-\frac{1}{2}(\dim \ker(A_Y)+\eta_{A_Y}),\\
\ind(D^+_{2,Q_<(A_Y)(H)})=\int_{X_2}\alpha_2(x)dx-\frac{1}{2}(\dim \ker(A_Y)+\eta_{A_Y}), 
\end{align*}
where $\alpha_1(x),\alpha_2(x)$ denote the local index densities on $X_1$ and $X_2$ respectively. Here the initial Hilbert space 
 is $H=L^2(Y)\coloneqq L^2(\wedge T^* Y)$ and $Q_<(A_Y)(H)$ is the APS boundary condition of Example \ref{Example:APSBoundCond}. Observe that in order to compute the index $D^-_2$ we can also use Equation  $\eqref{Eqn:IndexDBoundary}$ with $-A_Y$ instead of $A_Y$. As discussed in \cite[Section 3]{APSI} we would obtain
\begin{align*}
\ind(D^-_{2,Q_<(-A_Y)(H)})=&-\int_{X_2}\alpha_2(x)dx-\frac{1}{2}(\dim \ker(-A_Y)+\eta_{-A_Y})\\
=&\int_{-X_2}\alpha_2(x)dx-\frac{1}{2}(\dim \ker(A_Y)-\eta_{A_Y}).
\end{align*}
On the other hand, since  $\alpha_1(x),\alpha_2(x)$ are local terms, 
\begin{align*}
\ind(D^+)=\int_{X_1}\alpha_1(x)dx+\int_{-X_2}\alpha_2(x)dx. 
\end{align*}
Hence we see that (see \cite[Proposition 23.2]{BW})
\begin{align}\label{Eqn:GluingIndexSign}
\ind(D^+)=\ind(D^+_{1,Q_<(A_Y)(H)})+\ind(D^-_{2,Q_>(A_Y)(H)})+\dim\ker(A_Y),
\end{align}
where we have used the relation $Q_<(-A_Y)(H)=Q_>(A_Y)(H)$.\\

Now we use \eqref{Eqn:GluingIndexSign} to derive \eqref{Eqn:Additivity}. Recall fromt \eqref{Eqn:IndexVSSignatureAPS} that 
\begin{align*}
\ind(D^+_{1,Q_<(A_Y)(H)})=&\sigma(X_1)-\frac{1}{2}\dim\ker(A_Y),\\
\ind(D^+_{2,Q_<(A_Y)(H)})=&\sigma(X_2)-\frac{1}{2}\dim\ker(A_Y). 
\end{align*}
From the last relation we can deduce that 
\begin{align*}
\ind(D^-_{2,Q_>(A_Y)(H)})=&-\sigma(X_2)-\frac{1}{2}\dim\ker(A_Y). 
\end{align*}
Indeed, using the results above we compute 
\begin{align*}
\ind(D^-_{2,Q_>(A_Y)(H)})=&-\int_{X_2}\alpha_2(x)dx-\frac{1}{2}(\dim \ker(-A_Y)+\eta_{-A_Y})\\
=& - \ind(D^+_{2,Q_<(A_Y)(H)})-\dim \ker(A_Y)\\
=&-\sigma(X_2)-\frac{1}{2}\dim\ker(A_Y).
\end{align*}

Finally, from \eqref{Eqn:GluingIndexSign} and the index identity  $\ind(D^+)=\sigma(X_1\cup(-X_2))$ on the closed manifold $X$, we get
\begin{align*}
\sigma(X_1\cup(-X_2))=&\ind(D^+)\\
=&\ind(D^+_{1,Q_<(A_Y)(H)})+\ind(D^-_{2,Q_>(A_Y)(H)})+\dim\ker(A_Y)\\
=&\sigma(X_1)+\sigma(-X_2).
\end{align*}
This concludes the proof of \eqref{Eqn:Additivity}.\\

We end this example by illustrating Theorem \ref{BBCThm4.17}. First we need to change the boundary conditions since $Q_<(A_Y)(H)$ and $Q_>(A_Y)(H)$ are not complementary (unless $\ker A_Y=\{0\}$, in which case Theorem \ref{BBCThm4.17} is just given by \eqref{Eqn:GluingIndexSign}). To do this we use Example \ref{Example:KatoIndexComputation} and Theorem \ref{BBCThm4.14},
\begin{align*}
\ind(D^+_{1,Q_<(A_Y)(H)})=&\ind(D^+_{1,Q_\leq(A_Y)(H)})+\ind(Q_<(A_Y)(H),Q_>(A_Y)(H))\\
=&\ind(D^+_{1,Q_\leq(A_Y)(H)})-\dim\ker(A_Y). 
\end{align*}
Thus, combining this with \eqref{Eqn:GluingIndexSign} we get
\begin{align*}
\ind(D^+)=\ind(D^+_{1,Q_\leq(A_Y)(H)})+\ind(D^-_{2,Q_>(A_Y)(H)}),
\end{align*}
which is precisely the content of Theorem  \ref{BBCThm4.17}. 
\begin{remark}
The fact that in this last formula $D^-_2$ appears instead of $D^+_2$, in contrast with Theorem \ref{BBCThm4.17}, is due the fact that $\star_{X_1}=\star_{-X_2}$.
\end{remark}

\subsection{The index formula for the signature operator in the Witt case}\label{Section:IndexSigOpWitt}

The objective of this section is to study the techniques used in \cite[Section 5]{B09} to compute the index of the signature operator ${D}^+$ on $M_0/S^1$ in the Witt case, using the machinery of  Sections \ref{Sec:Perturbations} and \ref{App:DiracSystems}. In the subsequent section we will then adapt these same techniques to compute the index of the operator $\mathscr{D}^+$ also in the Witt case.\\

Let consider the local description close to a connected component $F\subset M^{S^1}$ discussed in Chapter \ref{Sect:LocalDesc}. Recall that  $M/S^1=(M_0/S^1)\cup M^{S^1}$ where $M_0$ is  principal orbit, on which the action is free. Following Section \ref{Sect:PfoofSignatureFormula}, consider the decomposition $M_0/S^1=Z_{t_0}\cup U_{t_0}$, where $t_0>0$, $Z_{t_0}$ is a compact manifold with boundary and $U_{t_0}$ is diffeomorphic  to the mapping cylinder $C(\mathcal{F})$ of a Riemannian fibration $\pi_\mathcal{F}:\mathcal{F}\longrightarrow F$ with typical fiber $\mathbb{C}P^N$, where $N$ is determined by the dimension $\dim F=4k-2N-1$.\\

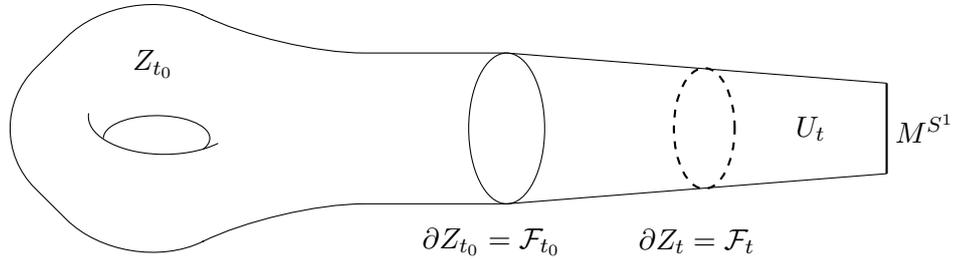
\begin{figure}[h]
\begin{center}
\begin{tikzpicture}
\draw[rounded corners=32pt](7,-1)--(4,-1)--(2,-2)--(0,0) -- (2,2)--(4,1)--(7,1);
\draw (1.5,0.2) arc (175:315:1cm and 0.5cm);
\draw (3,-0.28) arc (-30:180:0.7cm and 0.3cm);
\draw (7.5,0) arc (0:360:0.5cm and 1cm);
\draw[dashed,thick] (10,0) arc (0:360:0.4cm and 0.8cm);
\node (a) at (20:2.5) {$Z_{t_0}$};
\node (a) at (6.8,-1.5) {$\partial Z_{t_0}=\mathcal{F}_{t_0}$};
\node (a) at (9.5,-1.5) {$\partial Z_{t}=\mathcal{F}_{t}$};
\draw [thick](12,-0.6)--(12,0.6);
\draw (7,1)--(12,0.6);
\draw (7,-1)--(12,-0.6);
\node (a) at (11,0) {$U_t$};
\node (a) at (12.5,0) {$M^{S^1}$};
\end{tikzpicture}
\caption{The quotient space $M/S^1$ decomposed as $M/S^1=Z_{t_0}\cup U_{t_0}$ and $M/S^1=Z_{t}\cup U_{t}$ for $t<t_0$. }\label{Fig:Quot}
\end{center}
\end{figure}

Recall also that the Witt condition reads in this case $H^N(Y)=0$, which is equivalent to saying that $N$ is odd since $\mathbb{C}P^N$ has only non-vanishing cohomology groups in even degrees. In view Remark \ref{Rmk:ZeroEigenValueAV} and  Corollary \ref{Coro:IndexInvScal} we can assume, for the index computation,  that
\begin{align}\label{SpecAvWittCase}
|A_V|\geq 1
\end{align}
by rescaling the vertical metric. Indeed, from spectral decomposition of Theorem \ref{Thm:SpectralDecomp} we know that this scaling procedure will shift (arbitrarily) the spectrum of the operator $A_V$ in both, the space of closed and co-closed forms. In addition,  in the space of vertical harmonic forms the spectrum is given by $\pm(2j-N)$ for $0\leq j \leq N$. Consequently, in the Witt case we have $|j-N|\geq 1$, which verifies that \eqref{SpecAvWittCase} can always be achieved. \\
On the other hand, this condition combined with Corollary \ref{Coro:DESA} shows that, for the purpose of computing the index, we can assume that the operator $D$ is essentially self-adjoint.\\ 

Let us describe now the strategy for the index calculation. The main idea is to use the geometric decomposition of $M/S^1$ described above in order to apply Theorem \ref{BBCThm4.17}.  We compute the index of the signature operator ${D}^+$ by adding up the index contributions of $Z_t$ and $U_t$ (see Figure \ref{Fig:Quot}).  In order to do so, we need to construct  compatible super-symmetric Dirac systems on $Z_{t}$ and $U_{t}$, as in Definition \ref{Def:CompDiracSys}. These Dirac systems will be defined in a similar manner as in Section \ref{Section:ExNovikov}. Once we obtain these compatible Dirac systems, we just glue them and apply  Theorem \ref{BBCThm4.17}. The strategy is then, to compute each index contribution separately, at least when $t\longrightarrow 0$. \\

In view of Theorem \ref{Thm:TransfDiracOps} we define, for $0<t<t_0/2$ fixed, the operator
\begin{align*}
{D}_1\coloneqq \gamma\left(\frac{\partial}{\partial r}+\bigstar\otimes {A}(t+r)\right),
\end{align*}
where $0<r<t$. Here as before, $A$ is the operator of Theorem \ref{Thm:TransfDiracOps}, 
\begin{align*}
\gamma=
\left(
\begin{array}{cc}
0 & -I \\
I & 0
\end{array}
\right)
\quad\text{and}\quad
\bigstar=
\left(
\begin{array}{cc}
I & 0 \\
0 & -I
\end{array}
\right).
\end{align*}
From the relation $\gamma\bigstar +\bigstar\gamma =0$ we see that $({D}_1,\bigstar)$  defines a super-symmetric Dirac system on $Z_t$. The corresponding graded operator is 
\begin{align*}
{D}^+_1=\frac{\partial}{\partial r}+{A}(t+r). 
\end{align*}
\begin{remark}
The Hilbert space on which ${A}(t)$,the operator corresponding to $r=0$,  is defined is $H \coloneqq L^2(\wedge T^*\mathcal{F}_t)$ where the metric on $\mathcal{F}_t$ is $g^{T\mathcal{F}_t}=g^{T_H\mathcal{F}}\oplus t^2 g^{T_V\mathcal{F}}$. For $s\geq 0$ the associated Sobolev chain is $H^{s}\coloneqq \dom(|{A}(t)|^s)$. 
\end{remark}
Similarly, we define on $U_t$,
\begin{align*}
{D}_2\coloneqq -\gamma\left(\frac{\partial}{\partial r}-\bigstar\otimes {A}(t-r)\right),
\end{align*}
and we also see that $(D_2,\bigstar)$ is a super-symmetric Dirac system with associated graded operator 
\begin{align*}
{D}^+_2=-\left(\frac{\partial}{\partial_r}-{A}(t-r)\right).
\end{align*}

It is straightforward to verify that the super-symmetric Dirac systems ${D}_1=({D}_1,\bigstar)$ and ${D}_2=({D}_2,\bigstar)$ are compatible in the sense of Definition \ref{Def:CompDiracSys}. We should think of ${D}_1$ and ${D}_2$ as the restrictions of the operator ${D}$ to $Z_t$ and $U_t$ respectively. 
For the operator ${D}_1$ defined on the manifold with boundary $Z_t$, in view of Theorem \ref{Thm:SignThmMBound} and the proof of Theorem \ref{Thm:S1SignatureThm}, we impose the APS-type boundary condition (Example \ref{Example:APSBoundCond})
\begin{align*}
B_1\coloneqq &\bigstar Q_{<}({A}(t))(\dom(|{A}(t)|^{1/2})),
\end{align*}
which we know is elliptic. 
Obviously this boundary is invariant under $\bigstar$, so we can split it as $B_1=B^+_1\oplus B^-_1$ where
$B^+_1\coloneqq  Q_{<}({A}(t))(\dom(|{A}(t)|^{1/2}))$. In view of Theorem \ref{BBCThm4.17}, we choose for the operator ${D}^+_2$ the complementary boundary condition (see Remark \ref{Rmk:HybridSpDouble}) $B^+_2\coloneqq  Q_{\geq}({A}(t))(\dom(|{A}(t)|^{1/2}))$. 

\begin{remark}[Non-parabolicity]
Observe that both Dirac systems ${D}_1$ and ${D}_2$ are non-parabolic. For $D_1$ it is clear since $Z_t$ is compact (see Remark \ref{Rmk:NonParabNovikov}). To prove it for $D_2$ we would need to derive a bound, for $0<r<t$ fixed, of the form
\begin{align*}
\int_{t-r}^{t}\norm{\sigma(\tau)}_H^2 d\tau\leq C(r)\norm{D_2\sigma}^2_{L^2((0,t],H)}, 
\end{align*}
where $\sigma$ is any compactly supported section such that $\sigma(t)=0$ and $C(r)>0$ is a constant depending just on $r$. This can be achieved in using \cite[Proposition 2.5]{C01a}  and the computations in the proof of Lemma \ref{Lemma:5.12} below. The key observation is that the coefficients of $D_2$ are smooth an bounded in the compact interval $[t-r,t]$. 
\end{remark}
If we denote both operators with their compatible elliptic boundary conditions by
\begin{align*}
{D}^{+}_{Z_t,Q_{<}({A}(t))(H)}\coloneqq &{D}^+_{1,{B^+_1}},\\
{D}^{+}_{U_t,Q_{\geq}({A}(t))(H)}\coloneqq &{D}^+_{1,{B^+_2}},
\end{align*}
then from Theorem \ref{BBCThm4.17} we obtain the following decomposition result for the index.

\begin{theorem}[{\cite[Theorem 5.1]{B09}}]\label{Thm:5.1}
The operators ${D}^+_{U_t,Q_{\geq}({A}(t))(H)}$ and ${D}^+_{Z_t,Q_{<}({A}(t))(H)}$ are Fredholm, and we have the index identity
\begin{align*}
\textnormal{ind}({D}^+)=\textnormal{ind}\left({D}^{+}_{Z_t,Q_{<}({A}(t))(H)}\right)+\textnormal{ind}\left({D}^{+}_{U_t,Q_{\geq }({A}(t))(H)}\right).
\end{align*}
\end{theorem}

We will first study the index contribution on $U_t$. Set ${D}^+_t\coloneqq  {D}_{U_t,Q_{\geq}({A}(t))(H)}$ to lighten the notation. Note that as a consequence of condition \eqref{SpecAvWittCase},  we know that
\begin{align*}
\mathcal{C}({D}^+_t)\coloneqq \{\sigma\in C^{1}_c((0,t],H^1))\:|\:Q_<({A}(t))\sigma(t)=0\}
\end{align*}
is a core for ${D}^+_t$. One of the main ingredients of the index computation in \cite{B09} is the following remarkable vanishing result.
\begin{theorem}[{\cite[Theorem 5.2]{B09}}]\label{Thm:5.2}
For $t\in(0,t_0]$ sufficiently small we have 
\begin{align*}
\textnormal{ind}\left({D}^{+}_{U_t,Q_{\geq}({A}(t))(H)}\right)=0. 
\end{align*}
\end{theorem}

The key point of the proof is the fact that, by the rescaling argument and since in our particular case the dimension of the fibers is even,  we can always assume that 
\begin{align}\label{Eqn:1/2notinSpecAV}
\pm\frac{1}{2}\notin\spec(A_V),
\end{align}
in view of condition \eqref{SpecAvWittCase}.  This observation allow us to obtain an estimate which shows that $\ker(D^+_t)=\{0\}$.  Then, arguing analogously for the adjoint operator one verifies the vanishing of the index. The concrete form of the claimed estimate is discussed in the following lemma. We revise its proof because it will inspire techniques to derive a similar vanishing result for the operator $\mathscr{D}^+$. 

\begin{lemma}\label{Lemma:5.12}
For $t$ small enough and $\sigma\in\mathcal{C}({D}^+_t)$ we have the estimate
\begin{align*}
\norm{{D}^+_{t}\sigma}_{L^{2}((0,t],H)}\geq \frac{1}{2t}\bnorm{\left({A}_V+\frac{1}{2}\right)\sigma}_{L^{2}((0,t],H)}.
\end{align*}
\end{lemma}

\begin{proof}
Let $\sigma\in\mathcal{C}({D}^+_{t})$. From  the decomposition of $A$ into its horizontal and vertical component, we have
\begin{align*}
{D}^+_{t}\sigma(r)=\sigma'(r)+{A}_H(r)\sigma(r)+\frac{1}{r}{A}_V
\sigma(r).
\end{align*}
First, for $0<r<t$ fixed, we compute its norm in $H$,
\begin{align}\label{Eqn:Lemma5.12Norm}
\norm{{D}^+_{t}\sigma(r)}_{H}^2=&\norm{\sigma'(r)}^2_{H}+\norm{A_H\sigma(r)}^2_{H}+r^{-2}\norm{{A}_V\sigma(r)}^2_{H}\\
&+r^{-1}\inner{{A}_{HV}(r)\sigma(r)}{\sigma(r)}_{H}+2\text{Re}\inner{\sigma'(r)}{{A}\sigma(r)}_H,\notag
\end{align}
where $A_{HV}(r)\coloneqq A_H(r)A_V+A_VA_H(r)$. Next, differentiating the operator $A(r)$ with respect to $r$ we get 
\begin{align*}
\frac{d}{dr}{A}(r)=A'_H(0)-r^{-2}{A}_V,
\end{align*}
where we have used that $A'_H(r)=A_H'(0)$. This implies, 
\begin{align*}
\frac{d}{dr}\inner{\sigma(r)}{{A}(r)\sigma(r)}_H&=\inner{\sigma'(r)}{{A}(r)\sigma(r)}_H+\inner{\sigma(r)}{{A}'(r)\sigma(r)}_H+\inner{\sigma(r)}{{A}(t)\sigma'(r)}_H\\
&=2\text{Re}\inner{\sigma'(r)}{{A}(r)\sigma(t)}_H+\inner{\sigma(r)}{(A'_H(0)-t^{-2}{A}_V)\sigma(r)}_H.
\end{align*}
We can replace this expression in \eqref{Eqn:Lemma5.12Norm} to obtain 
\begin{align*}
\norm{{D}^+_{t}\sigma(r)}_{H}^2=&\norm{\sigma'(r)}^2_{H}+\norm{A_H\sigma(r)}^2_{H}+r^{-2}\norm{{A}_V\sigma(r)}^2_{H}+r^{-1}\inner{{A}_{HV}\sigma(r)}{\sigma(r)}_{H}\\
&-\frac{d}{dr}\inner{\sigma(r)}{{A}(r)\sigma(r)}_H+r^{-2}\inner{\sigma(r)}{{A}_V\sigma(r)}-\inner{\sigma(r)}{A'_H(0)\sigma(r)}_H. 
\end{align*}

From the proof of Lemma \ref{Lemma:VertOp1} we see that ${A}_{HV}(r)$ is a first order vertical operator, thus ${L}\coloneqq {A}_{HV}(r)({A}_V+1/2)^{-1}$ is a zero order operator and  therefore there exists a constant $C_1>0 $ such that,
\begin{align*}
\bnorm{{A}_{HV}(r)\left({A}_V+\frac{1}{2}\right)^{-1}}_H\leq C_1, \:\text{for $r\in(0,t]$}.
\end{align*}
Here we have used \eqref{Eqn:1/2notinSpecAV} to ensure the invertibility of  ${A}_V+1/2$. 
Using the relation
\begin{align*}
\bnorm{\left({A}_V+\frac{1}{2}\right)\sigma(r)}^2_{H}=
\norm{{A}_V\sigma(r)}^2_{H}+\frac{1}{4}\norm{\sigma(r)}^2_H+\inner{{A}_V\sigma(r)}{\sigma(r)},
\end{align*}
we  can express the norm $\norm{{D}^+_{t}\sigma(r)}_{H}^2$ as
\begin{align}\label{Eqn:AuxLemma5.12}
\norm{{D}^+_{t}\sigma(r) }_{H}^2=&\left(\norm{\sigma'(r)}^2_H-\frac{1}{4r^2}\norm{\sigma(r)}^2_H\right)+\frac{d}{dr}\inner{\sigma(r)}{{A}(r)\sigma(r)}_H\\
&+r^{-2}\bnorm{\left({A}_V+\frac{1}{2}\right)\sigma(r)}^2_H+r^{-1}\left\langle\left({A}_V+\frac{1}{2}\right)\sigma(r),{L}^*\sigma(r)\right\rangle_H\notag\\
&-\inner{\sigma(r)}{A'_H(0)\sigma(r)}_H+\norm{A_H(r)\sigma(r)}^2.\notag
\end{align}
Now we integrate \eqref{Eqn:AuxLemma5.12} between $0$ and $t$ in order to compute the $L^2$-norm. After integration, the first term in the right hand side of this equation is positive by Hardy's inequality,
\begin{align}\label{Eqn:Hardy}
\int_0^t\left(\frac{1}{r}\norm{\sigma(r)}_H\right)^2 dr\leq 4 \int_0^t \norm{\sigma'(r)}^2_H dr.
\end{align}
The second term in the right hand side of \eqref{Eqn:AuxLemma5.12} is also positive after integration as a result of the boundary condition at $t$, 
\begin{align*}
\int_0^t  \left(\frac{d}{dr}\inner{\sigma(r)}{{A}(r)\sigma(r)}_H\right)dr=\inner{\sigma(t)}{{A}(t)\sigma(t)}_H-\inner{\sigma(0)}{{A}(0)\sigma(0)}_H\geq 0.
\end{align*}
Thus, we arrive to the estimate 
\begin{align*}
\norm{{D}^+_{t}\sigma}^2_{L^2((0,t],H)}
\geq & \int_0^t \left(r^{-2}\bnorm{\left({A}_V+\frac{1}{2}\right)\sigma(r)}^2_H+r^{-1}\left\langle\left({A}_V+\frac{1}{2}\right)\sigma(r),{L}^*\sigma(r)\right\rangle_H \right. \\ 
&\left. +\norm{A_H(r)\sigma(r)}^2_H-\inner{\sigma(r)}{A'_H(0)\sigma(r)}_H \right)dt.
\end{align*}
Finally, choose $t$ small enough such that 
\begin{align*}
\norm{{D}^+_{t}\sigma}^2_{L^2((0,t],H)} \geq \frac{1}{4}\int_0^t r^{-2}\bnorm{\left({A}_V+\frac{1}{2}\right)\sigma(r)}^2_H dr\geq \frac{1}{4t^2}\bnorm{\left(A_V+\frac{1}{2}\right)\sigma}^2_{L^2((0,t],H)}.
\end{align*}
\end{proof}

\begin{remark}
In the proof of \cite[Lemma 5.12]{B09} the term $\norm{A_H(r)\sigma(r)}_H^2$ is missing, but this does not cause any harm. 
\end{remark}

\begin{remark}\label{Lemma:5.12R}
Observe that the adjoint operator (Remark \ref{Rmk:AdjointBoundCond})
\begin{align*}
({D}^+_{t})^*=-\frac{\partial}{\partial r}+{A}(r),
\end{align*}
has  a core $\mathcal{C}(({D}^+_{t})^*)\coloneqq \{\sigma\in C^1_c((0,t],H^1)\:|\:Q_{\geq}({A}(t))\sigma(t)=0\}$.
Hence,  we can compute similarly for $\sigma\in \mathcal{C}(({D}^+_{t})^*)$,
\begin{align*}
\norm{({D}^+_{t})^*\sigma(r)}_{H}^2=&\norm{\sigma'(r)}^2_{H}+\norm{A_H\sigma(r)}^2_{H}+r^{-2}\norm{{A}_V\sigma(r)}^2_{H}\\
&+r^{-1}\inner{{A}_{HV}\sigma(r)}{\sigma(r)}_{H}-2\text{Re}\inner{\sigma'(r)}{{A}\sigma(r)}_H. 
\end{align*}
Again, since we can assume that $1/2 \notin \spec({A}_V)$, then we can proceed as before and express this norm as
\begin{align*}
\norm{({D}^+_{t})^*\sigma(r)}_{H}^2=&\left(\norm{\sigma'(r)}^2_H-\frac{1}{4r^2}\norm{\sigma(r)}^2_H\right)-\frac{d}{dr}\inner{\sigma(r)}{{A}(r)\sigma(r)}_H\\
&+r^{-2}\bnorm{\left({A}_V-\frac{1}{2}\right)\sigma(r)}^2_H+r^{-1}\left\langle\left({A}_V-\frac{1}{2}\right)\sigma(r),\tilde{{L}}^*\sigma(r)\right\rangle_H\notag\\
&+\inner{\sigma(r)}{A'_H(0)\sigma(r)}_H+\norm{A_H(r)\sigma(r)}_H^2,\notag
\end{align*}
where $\tilde{{L}}\coloneqq {A}_{HV}({A}_V-1/2)$. All together, we get the analogous estimate
\begin{align*}
\norm{({D}^+_{t})^*\sigma}_{L^{2}((0,t],H)}\geq \frac{1}{2t}\bnorm{\left({A}_V-\frac{1}{2}\right)\sigma}_{L^{2}((0,t],H)}.
\end{align*}
\end{remark}

\begin{proof}[Proof of Theorem \ref{Thm:5.2}]
Let $\sigma\in\dom(D^+_t)$ such that $D^+_t\sigma=0$, we want to prove that $\sigma=0$. Since $\mathcal{C}({D}^+_t)$ is a core of $D^+_t$, then there exists a sequence $(\sigma_n)_n\subset \mathcal{C}({D}^+_t)$ such that $\sigma_n\longrightarrow\sigma$ and ${D}^+_t\sigma_n\longrightarrow {D}^+_t\sigma=0$.  From Lemma \ref{Lemma:5.12} we can therefore deduce that $\sigma_n\longrightarrow 0$ , which shows that $\sigma=0$. Hence,  $\ker(D^+_t)=\{0\}$. In a similar manner we can use Remark \ref{Lemma:5.12R} to show that $\ker((D^+_t)^*)=\{0\}$. These two conditions imply that $\textnormal{ind}\left({D}^{+}_t\right)=0$.
\end{proof}

\begin{remark}\label{Rmk:VanishingBaseEven}
When the dimension of the fiber is odd then the Witt condition is automatically satisfied. However, the condition $\pm1/2\notin\spec(A_V)$ need not be necessarily true. Nevertheless, one can still prove Theorem \ref{Thm:5.2} using Lemma \ref{Lemma:5.12} by a deformation argument (\cite[pg. 39]{B09}). We will actually describe this deformation procedure in the next section since it will be adapted to prove the similar vanishing result for the operator $\mathscr{D}^+$. 
\end{remark}

From Theorem \ref{Thm:5.2} we conclude that for $t>0$ small enough
\begin{align}\label{Eqn:Indexwoc}
\textnormal{ind}({D}^+)=\textnormal{ind}\left({D}^{+}_{Z_t,Q_{<}({A}(t))(H)}\right),
\end{align}
which is just the index of the signature operator on the manifold with boundary $Z_t$ with an APS-type boundary condition. 
In order to compute this index we would like to use \eqref{Eqn:IndexDBoundary}. Moreover, note that the left hand side of \eqref{Eqn:Indexwoc} does not depend on $t$, thus we can study the behavior of the right hand side in the limit $t\longrightarrow 0^+$. This is of course motivated  by the proof of Theorem \ref{Thm:S1SignatureThm} presented in Section \ref{Sect:PfoofSignatureFormula}.  Nevertheless, we need to be cautious when applying \eqref{Eqn:IndexDBoundary} because of the following observations:

\begin{enumerate}
\item The metric close to the boundary is not a product. However, as a result of Corollary \ref{Coro:IndexInvScal},  we can modify the metric  so that it becomes a product near  $r=t$ without changing the index.   To do so let $\psi:(0,\infty)\longrightarrow (0,\infty)$ be a smooth function such that
\begin{align*}
 \psi(r)=
\begin{cases} 
 r & \text{, if }r\in (0,1]\cup [4,\infty),\\
 1& \text{, if } r\in[2,3].
  \end{cases}
\end{align*}

Then, for $0<t< t_0/4$ we set
\begin{align}\label{Eqn:ModifMetric}
g_t^{TU_{t_0}}\coloneqq &dr^2\oplus g^{T_H \mathcal{F}}\oplus t^{2}\psi(r/ t)^2g^{T_V\mathcal{F}}, \notag\\
g_t^{T(M_0/S^1)}\big{|}_{Z_{t_0}}\coloneqq &g^{T(M_0/S^1)}\big{|}_{Z_{t_0}},\\
g_t^{T(M_0/S^1)}\big{|}_{U_{t_0}}\coloneqq &g_t^{TU_{t_0}}.\notag
\end{align}
In particular, close to $r=t$ the metric takes the form $dr^2\oplus g^{T_H\mathcal{F}}\oplus t^{2} g^{T_V\mathcal{F}}$
which is a product metric since $t$ is fixed.

\begin{figure}[h]
\begin{tikzpicture}[domain=0:2]
\draw[->] (-1,0) -- (6.5,0)
node[below right] {$r$};
\draw[->] (0,-1) -- (0,6)
node[left] {$y$};
\draw (1,-0.3) node {1};
\draw (2,-0.3) node {2};
\draw (3,-0.3) node {3};
\draw (4,-0.3) node {4};
\draw (5,-0.3) node {5};
\draw [rounded corners=15pt] (0,0) -- (1,1)--(3.5,1)--(3.8,3.8)--(5.5,5.5);
\draw[dashed] (0,0)--(6,6);
\draw (2.3,3) node {$y=r$};
\draw (4.5,2) node {$y=\psi(r)$};
\draw (1,0) node {$\bullet$};
\draw (2,0) node {$\bullet$};
\draw (3,0) node {$\bullet$};
\draw (4,0) node {$\bullet$};
\draw (5,0) node {$\bullet$};
\end{tikzpicture} 
\end{figure}

\item  The second observation is the boundary condition. Note that $Q_\geq({A}(t))$ is not the correct boundary condition (associated projection) that we need to impose in order to be able to apply Theorem \ref{Thm:SignThmMBound} since ${A}(t)$ does not correspond to the tangential signature operator on $\partial Z_t$ (see \eqref{Eqn:DefOddSignOp}). As a matter of fact, from Corollary \ref{Coro:BoundaryOp} we see that the right operator to consider is  $A_0(t)$.  Hence, the index formula  of Theorem \ref{Thm:SignThmMBound} only applies to the operator
$D^{\text{+}}_{Z_t,Q_{<}(A_0(t))(H)}$. Recall that the difference between these operators is 
\begin{align}\label{Eqn:DifferenceAA0}
{A}(t)-A_0(t)=\frac{\nu}{t}. 
\end{align}
Following the proof of Proposition \ref{Prop:TransDPsi1} it is easy to see that this difference term does not appear, close to $r=t$, if we consider the deformation metric \eqref{Eqn:ModifMetric}. Hence, we have the index relation
\begin{align}\label{Eqn:Thm5.3(I)}
\textnormal{ind}\left({D}^+_{Z_t,Q_{<}({A}_0(t))(H)}\right)=\textnormal{ind}\left(D^{+}_{(Z_t,g^{TZ_t}_t),Q_{<}({A}_0(t))(H)}\right).
\end{align}
\end{enumerate}

In principle, from these two observations we can compute the right hand side of \eqref{Eqn:Thm5.3(I)} using  \eqref{Eqn:IndexDBoundary}. Before doing that we want to relate this index with the $\sigma_{S^1}(M)$. To do so, observe from \eqref{Eqn:IndexVSSignatureAPS} and \eqref{Eqn:Thm5.3(I)} the relation
\begin{equation*}
\textnormal{ind}\left({D}^+_{Z_t,Q_{<}({A}_0(t))(H)}\right)=\sigma(Z_t)-\frac{1}{2}\dim\ker(A_0(t)),
\end{equation*}
where $\sigma(Z_t)$ denotes the topological signature of the manifold with boundary $Z_t$. Recall from Section \ref{Sect:PfoofSignatureFormula} the crucial observation that, for $t$ sufficiently small, $\sigma(Z_t)=\sigma_{S^1}(M)$. Thus, for $t$ small enough we have 
\begin{align}\label{IndS1SignKern}
\sigma_{S^1}(M)=\textnormal{ind}\left({D}^+_{Z_t,Q_{<}({A}_0(t))(H)}\right)+\frac{1}{2}\dim\ker(A_0(t)). 
\end{align}

In view of this formula and \eqref{Eqn:Indexwoc}, we see that  it is necessary to take care of the change of boundary condition between the spectral projections of $A(t)$ and $A_0(t)$. 

\begin{theorem}[{\cite[Theorem 5.3]{B09}}]\label{Thm:5.3}
For $t$ sufficiently small $(Q_{<}({A}(t))(H),Q_{\geq }({A}_0(t))(H))$ is a Fredholm pair in $H$ and 
\begin{align*}
\textnormal{ind}\left( {D}^+_{Z_t,Q_{<}({A}(t))(H)}\right)=\: \textnormal{ind}\left( {D}^+_{
Z_t,Q_{<}(A_0(t))(H)}\right)
+\ind(Q_{<}({A}(t))(H),Q_{\geq }({A}_0(t))(H)).
\end{align*}
\end{theorem}

\begin{proof}
First observe from \eqref{Eqn:DifferenceAA0} that for fixed $t>0$, $A_0(t)$ is a Kato perturbation of ${A}(t)$ since $\nu$ is a bounded operator. Thus, by Theorem \ref{BThm5.9} we see that  $Q_{>}({A}_0(t))$ is elliptic with respect to ${A}(t)$ and 
$(Q_{\leq }({A}(t))(H),Q_{>}(A_0(t))(H))$ is a Fredholm pair. Note that the operators
\begin{align*}
Q_{\leq }({A}(t))-Q_{< }({A}(t))=&Q_{0}({A}(t)),\\
Q_{\geq }({A}_0(t))-Q_{> }({A}_0(t))=&Q_{0}({A}_0(t)),
\end{align*}
are both finite rank, and therefore compact. It follows then from Proposition \ref{Prop:A.13} that $(Q_{<}({A}(t))(H),Q_{\geq}(A_0(t))(H))$ is also a Fredholm pair.\\
We now verify the index formula.  By the argument above we know that $Q_>(A_0(t))$ is an ellipic projection with respect to $A(t)$, so its kernel $Q_{\leq}(A_0(t))(H)$ defines an elliptic boundary condition (Proposition \ref{Prop:BBCProp1.99}). Now we use the  Agranovi\v{c}-Dynin type formula from Theorem \ref{BBCThm4.14} applied to $B^+\coloneqq Q_{\leq}(A_0(t))(H)$,
\begin{align*}
\textnormal{ind}\left( {D}^+_{Z_t,Q_{\leq }({A}_0(t))(H)}\right)=\:
\textnormal{ind}\left( {D}^+_{
Z_t,Q_{\leq }(A(t))(H)}\right)
+\ind(Q_\leq (A_0(t))(H),Q_{>}({A}(t))(H)). 
\end{align*}
Similarly, again by Theorem \ref{BBCThm4.14},  we obtain the index discontinuity formulas
\begin{align*}
\textnormal{ind}\left( {D}^+_{Z_t,Q_{\leq }({A}_0(t))(H)}\right)=&\textnormal{ind}\left( {D}^+_{Z_t,Q_{< }({A}_0(t))(H)}\right)+\dim\ker({A}_0(t)),\\
\textnormal{ind}\left( {D}^+_{Z_t,Q_{\leq }({A}(t))(H)}\right)=&\textnormal{ind}\left( {D}^+_{Z_t,Q_{< }({A}(t))(H)}\right)+\dim\ker({A}(t)).
\end{align*}
As we are in the Witt case we know that $A_V$ is invertible (see \eqref{SpecAvWittCase}) and therefore, by Lemma \ref{Lemma:VertOp1},  we know that $\ker(A(t))=\{0\}$ for $t$ sufficiently small. In particular, $Q_{>}({A}(t))(H)=Q_{\geq}({A}(t))(H)$. Altogether we obtain
\begin{align*}
\textnormal{ind}\left( {D}^+_{Z_t,Q_{< }({A}(t))(H)}\right)=&
\textnormal{ind}\left( {D}^+_{Z_t,Q_{< }({A}_0(t))(H)}\right)
-\ind(Q_<(A_0(t))(H),Q_{\geq}({A}(t))(H)),
\end{align*}
where we have used the relation 
\begin{align*}
\ind(Q_\leq(A_0(t))(H),Q_{>}({A}(t))(H))=\ind(Q_<(A_0(t))(H),Q_{>}({A}(t))(H))+\dim\ker({A}_0(t)),
\end{align*}
which is easily derived from Lemma \ref{Prop:A.13}. Finally note from \cite[Corollary IV.4.13]{KATO} that 
\begin{align*}
\ind(Q_<(A_0(t))(H),Q_{\geq}({A}(t))(H))=-\ind(Q_<(A(t))(H),Q_{\geq}({A}_0(t))(H)).
\end{align*}
This completes the proof of the desired formula. 
\end{proof}

This theorem, combined with \eqref{Eqn:Indexwoc} and \eqref{IndS1SignKern}, implies the following result. 
\begin{proposition}\label{Prop:IndexFormulaCasi}
For $t>0$ sufficiently small we have
\begin{align*}
\ind ({D}^+)=\sigma_{S^1}(M)-\frac{1}{2}\dim\ker(A_0(t))+\ind(Q_{<}({A}(t))(H),Q_{\geq }({A}_0(t))(H)).
\end{align*}
\end{proposition}

Next we describe how to compute the Kato index of the proposition above. The main ingredient is the {\em generalized Thom space} $T_\pi$ of the fibration  $\pi_\mathcal{F}:\mathcal{F}\longrightarrow F$, which is a stratified space constructed as follows: 
Topologically it is defined as $ T_{\pi}\coloneqq (0,2)\times \mathcal{F}$ and it is oriented with respect its product orientation (which is the opposite orientation chosen in \cite[Theorem 1.1]{CD09}). For each fixed $t\in(0,1/2)$ we consider the metric on $T_\pi$ defined as 
\begin{align}\label{Eqn:MetricGTS}
g^{TT_\pi}_t\coloneqq dr^2\oplus g^{T\mathcal{F}}_t(r),
\end{align}
where
\begin{align*}
g^{T\mathcal{F}}_t(r)=
\begin{cases}
g^{T_H\mathcal{F}}\oplus r^2 g^{T_V\mathcal{F}}, &\quad\text{if $0<r\leq 1/2,$}\\
(2-r)^2(t^{-2}g^{T_H\mathcal{F}}\oplus  g^{T_V\mathcal{F}}), &\quad\text{if $3/2< r <2.$}
\end{cases}
\end{align*}
Observe in particular the relation
\begin{align*}
g^{T\mathcal{F}}_t(t)=g^{T\mathcal{F}}_t(2-t)=g^{T_H\mathcal{F}}\oplus t^2 g^{T_V\mathcal{F}}.
\end{align*}

 \begin{figure}[h]
\begin{center}
\begin{tikzpicture}
\draw (-2,5)--(2,5);
\draw (-3,3)--(1,3);
\draw (-2,5)--(-3,3);
\draw (2,5)--(1,3);
\draw (-2.5,1)--(1.5,1);
\draw (-3,3)--(-2.5,1);
\draw (2,5)--(1.5,1);
\draw (1,3)--(1.5,1);
\draw [dashed](-2,5)--(-2.25,3);
\draw (-2.5,1)--(-2.25,3);
\draw (-1,3)--(-0.5,1);
\draw [dashed](0,5)--(-0.25,3);
\draw (-0.5,1)--(-0.25,3);
\draw (0,5)--(-1,3);
\node (b) at (0,0.7) {$F$};
\node (c) at (-0.5,7) {$\bullet$};
\draw (-2,5)--(-0.5,7);
\draw (-3,3)--(-0.5,7);
\draw (2,5)--(-0.5,7);
\draw (1,3)--(-0.5,7);
\end{tikzpicture}
\caption{Generalized Thom space $ T_{\pi}$ associated with the fibration $\pi_{\mathcal{F}}:\mathcal{F}\longrightarrow F.$\label{Fig:ThomSpace}}\label{Fig:GTP}
\end{center}
\end{figure}
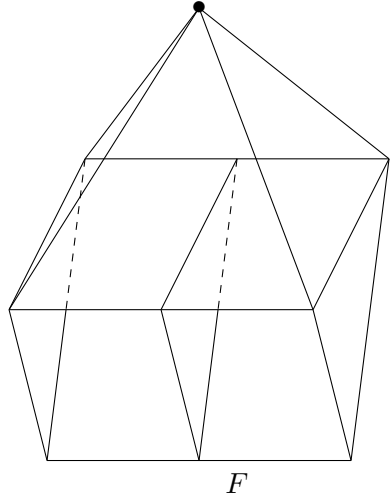

For each $t$, the space $T_\pi$ (more precisely, its compactification) becomes a compact Witt stratified space, see Figure \ref{Fig:ThomSpace}.  The family of metrics $(g^{TT_\pi}_t)_{t\in(0,1/2)}$ defined by \eqref{Eqn:MetricGTS} are all mutually quasi-isometric  and therefore they all compute the same $L^2$-signature. Consequently, from Theorem \ref{Prop:L2singnatureIndex} we get 
$$\sigma_{(2)}(T_\pi)=\ind \left(D^+_{T_\pi,g^{TT_\pi}_t}\right). $$
Using this relation, Theorem \ref{Thm:5.2} and Remark \ref{Rmk:VanishingBaseEven} Br\"uning proved the following result. 

\begin{theorem}[{\cite[Theorem 5.4]{B09}}]\label{Thm:5.4}
For $t>0$ sufficiently small, we have
\begin{align*}
\ind(Q_{<}({A}(t))(H),Q_{\geq }(A_0(t))(H))=\sigma_{(2)}(T_\pi)+\frac{1}{2}\dim\ker(A_0(t)).
\end{align*}
\end{theorem}

On the other hand, a remarkable theorem of Cheeger and Dai (\cite[Theorem 1.1]{CD09}) shows, still restricted to the Witt case, that this $L^2$-signature is equal to Dai's $\tau$ invariant of the fibration $\pi_\mathcal{F}:\mathcal{F}\longrightarrow F$ (see Section \ref{Section:Dai}), i.e.
\begin{align}\label{Eqn:CheegerDai}
\sigma_{(2)}(T_\pi)=\tau. 
\end{align}
From this formula and using Example \ref{Example:ProjBundle}, as in the Section \ref{Sect:PfoofSignatureFormula}, we obtain the following vanishing result. 
\begin{coro}\label{Coro:TauWitt}
For $t$ sufficiently small, we have
\begin{align*}
\ind(Q_{<}({A}(t))(H),Q_{\geq 0}(A_0(t))(H))-\frac{1}{2}\dim\ker(A_0(t))=\tau=0. 
\end{align*}
\end{coro}
We therefore see from Proposition \ref{Prop:IndexFormulaCasi} and Theorem \ref{Thm:S1SignatureThm} that the index of the signature operator on $M_0/S^1$, in the Witt case, computes Lott's equivariant $S^1$-signature. 
\begin{theorem}\label{Thm:IndD+}
Let $M$ be a closed, oriented Riemannian $4k+1$ dimensional manifold on which $S^1$ acts effectively and semi-freely by orientation preserving isometries. If the codimension of the fixed point set $M^{S^1}$ in $M$ is divisible by four, then $M/S^1$ is a Witt space  and the index of the signature operator is 
\begin{align*}
\ind(D^+)=\sigma_{S^1}(M)=\int_{M_0/S^1}L\left(T(M_0/S^1),g^{T(M_0/S^1)}\right).
\end{align*}
\end{theorem}

Observe that we have used the fact that the $\eta(M^{S^1})$ vanishes in the Witt case (see Section \ref{Section:VanishEta}).

\subsection{The index formula for the Dirac-Schr\"odinger signature operator}

In this last section we describe how to compute the index of the operator $\mathscr{D}^+$ using the techniques illustrated in Section \ref{Section:IndexSigOpWitt}. We obtain the complete index formula for this operator in the Witt case and for the non-Witt case we point out some difficulties in the calculations using these methods. \\

To begin with, we make no assumption on the parity of $N$, i.e. we do not distinguish between the Witt and the non-Witt case. As above, we use the geometric decomposition of $M_0/S^1$ in order to apply Theorem \ref{BBCThm4.17} to compute the index of $\mathscr{D}^+$ by adding up the index contributions of $Z_t$ and $U_t$. Motivated by the treatment of the signature operator discussed in the last section, consider , for $0<r<t<t_0/2$,  the following non-parabolic compatible super-symmetric Dirac systems  
\begin{align*}
\mathscr{D}_1\coloneqq &\gamma\left(\frac{\partial}{\partial r}+\bigstar\otimes \mathscr{A}(t+r)\right),\\
\mathscr{D}_2\coloneqq &-\gamma\left(\frac{\partial}{\partial r}-\bigstar\otimes \mathscr{A}(t-r)\right),
\end{align*}
where $\gamma$ and $\bigstar$ are defined as in last section and $\mathscr{A}(r)$ is the operator of Theorem \ref{Thm:LocalDesD1}. The corresponding graded operators, with respect to the super-symmetry $\bigstar$, are 
\begin{align*}
\mathscr{D}^+_1&=\frac{\partial}{\partial r}+\mathscr{A}(t+r),\\
\mathscr{D}^+_2&=-\left(\frac{\partial}{\partial r}-\mathscr{A}(t-r)\right).
\end{align*}
\begin{remark}\label{Rmk:EquivSobChainsPot}
As before, on $H\coloneqq L^2(\wedge T^*\mathcal{F}_t)$ and the associated Sobolev chain of $\mathscr{A}(t)$ is denoted by $H^{s}\coloneqq \dom(|\mathscr{A}(t)|^s)$. Since for $t>0$ fixed, $\mathscr{A}(t)-A(t)$ 
is a bounded operator, it follows by Remark \ref{Rmk:SobolevKatoPert} that the Sobolev chains of $\mathscr{A}(t)$ and $A(t)$ are isomorphic.
\end{remark}

As we did for the signature operator, we impose the complementary APS-type boundary conditions
\begin{align*}
\mathscr{B}^+_1\coloneqq & Q_{<}(\mathscr{A}(t))(\dom(|\mathscr{A}(t)|^{1/2})),\\
\mathscr{B}^+_2\coloneqq & Q_{\geq}(\mathscr{A}(t))(\dom(|\mathscr{A}(t)|^{1/2})),
\end{align*}
on $Z_t$ and $U_t$ respectively. If we denote both operators with their compatible elliptic boundary conditions by
\begin{align*}
\mathscr{D}^{+}_{Z_t,Q_{<}(\mathscr{A}(t))(H)}\coloneqq &\mathscr{D}_{1,{\mathscr{B}^+_1}},\\
\mathscr{D}^{+}_{U_t,Q_{\geq}(\mathscr{A}(t))(H)}\coloneqq &\mathscr{D}_{1,{\mathscr{B}^+_2}},
\end{align*}
then, as in last section, we can apply Theorem \ref{BBCThm4.17} to get decomposition formula 
\begin{align}\label{Eqn:Thm5.1OpPot}
\textnormal{ind}(\mathscr{D}^+)=\textnormal{ind}\left(\mathscr{D}^{+}_{Z_t,Q_{<}(\mathscr{A}(t))(H)}\right)+\textnormal{ind}\left(\mathscr{D}^{+}_{U_t,Q_{\geq }(\mathscr{A}(t))(H)}\right).
\end{align}

\subsubsection{The index formula for $\mathscr{D}^+$ in the Witt case}
Now we restrict ourselves to the Witt case, i.e. $N$ is odd. In order to obtain a vanishing result for the  index contribution of $U_t$ we need to modify the boundary conditions at $r=t$. As noted in Remark \ref{Rmk:EquivSobChainsPot}, for fixed $t>0$ the operator $A(t)$ is a Kato perturbation of $\mathscr{A}(t)$. Hence, by Theorem \ref{BThm5.9} and Theorem  \ref{BBCThm4.17} we obtain the decomposition formula of the index 
\begin{align}\label{Eqn:Thm5.1OpPotI}
\textnormal{ind}(\mathscr{D}^+)=\textnormal{ind}\left(\mathscr{D}^{+}_{Z_t,Q_{<}({A}(t))(H)}\right)+\textnormal{ind}\left(\mathscr{D}^{+}_{U_t,Q_{\geq }({A}(t))(H)}\right).
\end{align}

The following lemma relates these two boundary conditions at $r=t$. 
\begin{lemma}\label{Lemma:SplitKatoInd}
The following index identity holds true, 
\begin{align*}
 \ind(Q_{<}  (\mathscr{A}(t))(H) & ,  Q_{\geq }(A_0(t))(H))\\
=&\ind(Q_{<}({A}(t))(H),Q_{\geq }(A_0(t))(H))
+\ind(Q_{<}(\mathscr{A}(t))(H),Q_{\geq }(A(t))(H)).
\end{align*}
\end{lemma}

\begin{proof}
The strategy is to use the index formula of  Lemma \ref{Prop:A.13}. Using similar arguments as in the proof of Theorem \ref{Thm:5.3} one can see that 
\begin{align*}
&(Q_{<}({A}(t))(H),Q_{\geq }(A_0(t))(H)),\\
&(Q_{<}(\mathscr{A}(t))(H),Q_{\geq }(A(t))(H)),
\end{align*} 
are both Fredholm pairs. As in the proof of Corollary \ref{Coro:VanishingKatoIndexPert}, we can use  Remark \ref{Rmk:DiffKatoProj} and \cite[Lemma A.1]{BB01} with $A_1=\mathscr{A}(t)$, $A_2=A_0(t)$,  $\alpha_1=\alpha_2=0$ to prove that 
 \begin{align*}
 Q_>(\mathscr{A}(t))-Q_>(A_0(t))=\frac{1}{2\pi i}\int_{\text{Re}\:z=0}(\mathscr{A}(t)-z)^{-1}\left[\frac{1}{t}\left(\nu-\frac{\varepsilon}{2}\right)\right](A_0(t)-z)^{-1}dz,
 \end{align*}
is compact for fixed $t>0$. Indeed, both $\mathscr{A}(t)$ and $A_0(t)$ are discrete, so their resolvent is compact, and the perturbation term $(\nu-\varepsilon/2)$ is a bounded operator. Similarly, we see that all the differences
\begin{align*}
 &Q_<(\mathscr{A}(t))-Q_<(A_0(t)),\\
 &Q_>(\mathscr{A}(t))-Q_>(A(t)),\\
  &Q_>({A}(t))-Q_>(A_0(t)),
\end{align*}
are compact. Finally we can use Lemma \ref{Prop:A.13} with $P=Q_{<}(\mathscr{A}(r))$, $Q=Q_{<}({A}(r))$ and $B=Q_{\geq }(A_0(r))(H)$ to obtain the desired formula. 
\end{proof}

\begin{lemma}
For $t>0$ suficciently small we have
\begin{align*}
\ind(Q_\leq(A(t))(H),Q_>(\mathscr{A}(t)(H))=0
\end{align*}
\end{lemma}

\begin{proof}
The idea is to apply Corollary \ref{Coro:VanishingKatoIndexPert} with $A=A(t)$ and $B=\varepsilon/2t$ so that the sum $A+B=\mathscr{A}(t)$. From \eqref{SpecAvWittCase} and Lemma \ref{Lemma:VertOp1} we see that for $t$ small enough, 
\begin{align*}
|A(t)|\geq\frac{\sqrt{C}}{t}|A_V|\geq \frac{\sqrt{C}}{t}.
\end{align*}
Hence, the required condition of  Corollary \ref{Coro:VanishingKatoIndexPert} is, for $\mu\coloneqq \sqrt{C}/t$,
\begin{align*}
\frac{\sqrt{C}}{t}>\sqrt{2}\bnorm{\frac{\varepsilon}{2t}}=\frac{\sqrt{2}}{2t},
\end{align*}
that is, $C>1/2$, which we can be achieved in view of the proof of Lemma \ref{Lemma:VertOp1}.
\end{proof}

From this lemma we conclude, via Theorem \ref{BBCThm4.14}, that in the Witt case the decompositions \eqref{Eqn:Thm5.1OpPot} and \eqref{Eqn:Thm5.1OpPotI} are the same. \\

Now we describe the vanishing result for the index on $U_t$. Analogously as before, consider the operator  $\mathscr{D}^+_t\coloneqq  \mathscr{D}_{U_t,Q_{\geq}({A}(t))(H)}$ defined on the core
\begin{align*}
\mathcal{C}({D}^+_t)\coloneqq \{\sigma\in C^{1}_c((0,t],H^1))\:|\:Q_<({A}(t))\sigma(t)=0\}.
\end{align*}

\begin{theorem}\label{Thm:5.2Pot}
In the Witt case, for $t\in(0,t_0]$ sufficiently small we have 
\begin{align*}
\textnormal{ind}\left(\mathscr{D}^{+}_{U_t,Q_{\geq}({A}(t))(H)}\right)=0. 
\end{align*}
\end{theorem}

The strategy of the proof of this theorem is to argue by a deformation argument, as in the proof of \cite[Theorem 5.2]{B09}, splitting the index into a contribution of the space of vertical harmonic forms and a contribution of the complement. More precisely, for $\delta>0$ sufficiently small, let
\begin{align*}
P_\mathcal{H}(x)\coloneqq\frac{1}{2\pi i}\int_{|z|=\delta}(\Delta_{Y,x}-z)^{-1}dz, \quad\text{for $x\in F$,}
\end{align*}
 be the projection onto the space of vertical harmonic forms $\mathcal{H}(Y)$ and let $P_{\mathcal{H}^\perp}\coloneqq I-P_\mathcal{H}$ be the complementary projection. These two projections induce an orthogonal decomposition $L^2(\wedge T^*Y)=\mathcal{H}(Y)\oplus\mathcal{H}(Y)^\perp$ (see \cite[Equations (5.51), (5.52)]{B09} ). 

\begin{lemma}\label{Lemma:CommPH}
The projection $I\otimes P_\mathcal{H}$ commutes with $\mathscr{A}_V$ and with the principal symbol of $A_H(t)$.
\end{lemma}

\begin{proof}
From the relations
\begin{align*}
(\nu-1)d_V=&d_V\nu,\\
(\nu+1)d^\dagger_V=&d^\dagger_V\nu,
\end{align*}
it follows $\tilde{A}_{0V}\nu=\nu\tilde{A}_{0V}-(d_V-d^\dagger_V)$, for $\tilde{A}_{0V}=d_V+d^\dagger_V$. Next we compute 
\begin{align*}
\tilde{A}_{0V}^2\nu=&\tilde{A}_{0V}(\nu\tilde{A}_{0V}-(d_V-d^\dagger_V))\\
=&(\nu\tilde{A}_{0V}-(d_V-d^\dagger_V))\tilde{A}_{0V}-\tilde{A}_{0V}(d_V-d^\dagger_V)\\
=&\nu\tilde{A}^2_{0V}-(d_V-d^\dagger_V)\tilde{A}_{0V}-\tilde{A}_{0V}(d_V-d^\dagger_V)\\
=&\nu\tilde{A}^2_{0V}.
\end{align*}
On the other hand, since $\varepsilon\alpha+\alpha\varepsilon=0$, then $(\tilde{A}_{0V}\alpha)\varepsilon=-\tilde{A}_{0V}\varepsilon\alpha=\varepsilon(\tilde{A}_{0V}\alpha)$. Hence, $\Delta_Y\mathscr{A}_V=\mathscr{A}_V\Delta_Y$, as $\Delta_Y=(\tilde{A}_{0V}\alpha)^2$. In particular, this verifies that $I\otimes P_\mathcal{H}$ commutes with $\mathscr{A}_V$. For the second claim note that, as $\bar{\star}_V\Delta_Y=\Delta_Y\bar{\star}_V$, we have 
\begin{align*}
(I\otimes P_{\mathcal{H}})(c(f^\alpha)\bar{\star}_H\otimes \bar{\star}_V)(I\otimes{P_{\mathcal{H}}}^\perp)=0,
\end{align*}
for any horizontal $1$-form $f^\alpha$. 
\end{proof}

Let us define 
\begin{align*}
\mathscr{A}^{\delta}(r)\coloneqq &(I\otimes P_{\mathcal{H}})\mathscr{A}(r)(I\otimes P_{\mathcal{H}})+(I\otimes{P_{\mathcal{H}}}^\perp)\mathscr{A}(r)(I\otimes P_{\mathcal{H}}^\perp)\\
\eqqcolon &\mathscr{A}_{\mathcal{H}}(r)+\mathscr{A}_{{\mathcal{H}^\perp} }.
\end{align*}
By Lemma \ref{Lemma:CommPH} the difference $\mathscr{C}(r)\coloneqq \mathscr{A}^\delta(r)-\mathscr{A}(r)$
has uniformly bounded norm, i.e. there exists $\mathscr{C}>0$ such that $\norm{\mathscr{C}(r)}\leq \mathscr{C}$ for all $r\in(0,t]$. We now consider the ``deformed'' operator
\begin{align*}
\mathscr{D}^+_{t,\delta}\coloneqq \frac{\partial}{\partial r}+\mathscr{A}^{\delta}(r),
\end{align*}
defined on the core $\mathcal{C}(\mathscr{D}^+_{t,\delta})\coloneqq \mathcal{C}(\mathscr{D}^+_{t})=\{\sigma\in C^{1}_c((0,t],H^1))\:|\:Q_<(A^\delta(t))\sigma(t)=0\}$, where 
\begin{align*}
{A}^{\delta}(r)\coloneqq &(I\otimes P_{\mathcal{H}}){A}(r)(I\otimes P_{\mathcal{H}})+(I\otimes P_{\mathcal{H}}^\perp){A}(r)(I\otimes P_{\mathcal{H}}^\perp)\\
\eqqcolon &{A}_{\mathcal{H}}(r)+{A}_{{\mathcal{H}^\perp} }(r).
\end{align*}
Again, from Lemma \ref{Lemma:CommPH} it follows that $C(r)\coloneqq A^\delta(r)-A(r)$ has uniformly bounded operator. Since $\mathscr{D}^+_{t,\delta}-\mathscr{D}^+_{t}=\mathscr{C}(r)$ is uniformly bounded then, by Theorem \ref{BBCThm4.14}, 
\begin{align*}
\ind(\mathscr{D}^+_t )=\ind(\mathscr{D}^+_{t,\delta} )+\ind(Q_{<}(A(t))(H),Q_\geq(A^\delta(t))(H)).
\end{align*}
Now we introduce the operators 
\begin{align*}
\mathscr{D}^+_{t,\mathcal{H} /{\mathcal{H}^\perp} }\coloneqq\frac{\partial}{\partial r}+\mathscr{A}_{\mathcal{H} /{\mathcal{H}^\perp} }(r)
\end{align*}
defined on the core $\mathcal{C}(\mathscr{D}^+_{t,\mathcal{H} /{\mathcal{H}^\perp} })\coloneqq \{\sigma\in C^{1}_c((0,t],H^1))\:|\:Q_<(A_{\mathcal{H}/{\mathcal{H}^\perp} }(t))\sigma(t)=0\}$, which by orthogonality satisfy 
\begin{align*}
\ind(\mathscr{D}^+_{t,\delta})=\ind(\mathscr{D}^+_{t,\mathcal{H}})+\ind(\mathscr{D}^+_{t,{\mathcal{H}^\perp}} ). 
\end{align*}
Hence, we obtain the index relation
\begin{align*}
\ind(\mathscr{D}^+_t )=\ind(\mathscr{D}^+_{t,\mathcal{H} })+\ind(\mathscr{D}^+_{t,{\mathcal{H}^\perp} } )+\ind(Q_<(A(t))(H),Q_\geq(A^\delta(t))(H)). 
\end{align*}

\begin{remark}
In the Witt case, as $A(t)$ is invertible fot $t>0$ small enought, it is easy to see that this is also the case for the operator $A^\delta(t)$. 
\end{remark}

Now we show these three contributions to the index vanish for $t>0$ sufficiently small. 
\begin{proposition}\label{Prop:DefIndex1}
For $t>0$ small enough, $\ind(Q_<(A(t))(H),Q_\geq(A^\delta(t))(H))=0$.
\end{proposition}

\begin{proof}
As we are in the Witt case, we can assume \eqref{SpecAvWittCase} and therefore, by Lemma \ref{Lemma:VertOp1},
\begin{align*}
|A(t)|\geq\frac{\sqrt{C}}{t}|A_V|\geq\frac{\sqrt{C}}{t}.
\end{align*}
In order to apply the vanishing statement of Corollary \ref{Coro:VanishingKatoIndexPert} we must have 
\begin{align*}
\frac{\sqrt{C}}{t}>\sqrt{2}\norm{C(t)},
\end{align*}
which can always be achieved by making $t$ small,  since $C(t)$ is uniformly bounded. 
\end{proof}

\begin{proposition}\label{Prop:DefIndex2}
For $t>0$ small enough we have $\ind(\mathscr{D}^+_{t,{\mathcal{H}^\perp} } )=0$.
\end{proposition}

\begin{proof}
As we are in the Witt case we can assume \eqref{SpecAvWittCase}. Hence, it is easy to see that for the operator
\begin{align*}
{D}^+_{t,{\mathcal{H}^\perp} } \coloneqq \frac{\partial}{\partial r}+A_{{\mathcal{H}^\perp}}(r),
\end{align*}
defined on the core $\mathcal{C}(\mathscr{D}^+_{t,{\mathcal{H}^\perp} })$,  we can prove an analogue of Lemma \ref{Lemma:5.12}. That is, for $t>0$ small enough and $\sigma\in\mathcal{C}(\mathscr{D}^+_{t,{\mathcal{H}^\perp} })$, we can derive the estimate
\begin{align*}
\bnorm{D^+_{t,{\mathcal{H}^\perp}}\sigma}_{L^2((0,t],H)}\geq\frac{1}{2t}\bnorm{\left(A_{V,{\mathcal{H}^\perp}}+\frac{1}{2}\right)\sigma}_{L^2((0,t],H)},
\end{align*}
where $A_{V,{\mathcal{H}^\perp}}\coloneqq  (I\otimes P_{\mathcal{H}}^\perp)A_V(I\otimes P_{\mathcal{H}}^\perp)$. In particular, we can  show as before that $\ind(D_{t,{\mathcal{H}^\perp}})=0$. Now, the strategy of the proof is to show that $\mathscr{D}^+_{t,{\mathcal{H}^\perp} }$
is a Kato perturbation of ${D}^+_{t,{\mathcal{H}^\perp} }$ on $\mathcal{C}(\mathscr{D}^+_{t,{\mathcal{H}^\perp} })$.  If we are able to proof an estimate of the form 
\begin{align}\label{Eqn:KatoAim}
\bnorm{\frac{\varepsilon}{2r}\sigma}_{L^2((0,t],H)}
\leq b\norm{D^+_{t,{\mathcal{H}^\perp}}\sigma}_{L^2((0,t],H)},
\end{align}
with $b<1$, then by \cite[Theorem IV.5.22, pg. 236]{KATO} it would follow that 
$\textnormal{ind}\left(\mathscr{D}^{+}_{t,\mathcal{H}\perp}\right)=0$. The main idea is to use the computations used in the proof of Lemma \ref{Lemma:5.12}  to obtain the estimate \eqref{Eqn:KatoAim}. Concretely, in this proof the following equation is deduced for an element $\sigma\in\mathcal{C}(\mathscr{D}^+_{t,{\mathcal{H}^\perp}})$ (see \eqref{Eqn:AuxLemma5.12}),
\begin{align*}
\norm{{D}^+_{t,{\mathcal{H}^\perp}}\sigma(r) }_{H}^2=&\left(\norm{\sigma'(r)}^2_H-\frac{1}{4r^2}\norm{\sigma(r)}^2_H\right)+\frac{d}{dr}\inner{\sigma(r)}{{A}_{{\mathcal{H}^\perp}}(r)\sigma(r)}_H\\
&+r^{-2}\bnorm{\left({A}_{V,{\mathcal{H}^\perp}}+\frac{1}{2}\right)\sigma(r)}^2_H+r^{-1}\left\langle\left({A}_{V,{\mathcal{H}^\perp}}+\frac{1}{2}\right)\sigma(r),{L}^*\sigma(r)\right\rangle_H\notag\\
&-\inner{\sigma(r)}{A'_{H,{\mathcal{H}^\perp}}(0)\sigma(r)}_H+\norm{A_{H,{\mathcal{H}^\perp}}(r)\sigma(r)}_H^2.\notag
\end{align*}
Arguing as in the the mentioned proof,  we see that after integration between $0$ and $t$:
\begin{itemize}
\item The first term in brackets is non-negative by Hardy's inequality. 
\item The term containing the total derivative with respect to $r$ is also non-negative because $\sigma$ has compact support and because the boundary condition at $r=t$.
\end{itemize}

Thus, it follows that for $t$ small enough and $0<\beta<1$,
\begin{align}\label{Eqn:Lemma5.12Modif}
\bnorm{\frac{1}{r}\left(A_{V,{\mathcal{H}^\perp}}+\frac{1}{2}\right)\sigma}_{L^2((0,t],H)}\leq (1+\beta)\norm{D^+_{t,{\mathcal{H}^\perp}}\sigma}_{L^2((0,t],H)}.
\end{align}

On the other hand, 
\begin{align*}
\norm{\sigma}_H=&\bnorm{\left({A}_{V,{\mathcal{H}^\perp}}+\frac{1}{2}\right)^{-1}\left({A}_{V,{\mathcal{H}^\perp}}+\frac{1}{2}\right)\sigma(r)}_H\\
\leq &\bnorm{\left({A}_{V,{\mathcal{H}^\perp}}+\frac{1}{2}\right)^{-1}}\bnorm{\left({A}_{V,{\mathcal{H}^\perp}}+\frac{1}{2}\right)\sigma(r)}_H.
\end{align*}
The norm of the resolvent is given by (\cite[SectionV.5]{KATO}), 
\begin{align*}
\bnorm{\left({A}_{V,{\mathcal{H}^\perp}}+\frac{1}{2}\right)^{-1}}=\frac{1}{d^\perp}, 
\end{align*}
where $d^\perp\coloneqq {\text{dist}(-1/2,\spec(A_{V,{\mathcal{H}^\perp}}))}$. From this relation and \eqref{Eqn:Lemma5.12Modif} we conclude that
\begin{align}\label{Eqn:ResultKatoEstimate}
\bnorm{\frac{\varepsilon}{2r}\sigma}_{L^2((0,t],H)}\leq \frac{(1+\beta)}{2d^\perp}\norm{D^+_t\sigma}_{L^2((0,t],H)},
\end{align}
i.e. we have shown the desired estimate with $b=(1+\beta)/2d^\perp$. Observe however that we require the condition $b<1$, which translates to
\begin{align}
\frac{(1+\beta)}{2}<d^\perp.
\end{align}
This can always be achieved, in view of Theorem \ref{Thm:SpectralDecomp}, by rescaling the vertical metric.
\end{proof}

\begin{proposition}\label{Prop:DefIndex3}
For $t$ small enough we have $\ind(\mathscr{D}^+_{t,{\mathcal{H}} } )=0$.
\end{proposition}

\begin{proof}
We will proceed similarly as in the proof of \cite[Theorem 5.2]{B09}. Since we can identify the first order part of $\mathscr{A}_H$ with the odd signature operator $A_F$ of $F$  with coefficients in $\mathcal{H}(Y)$ we know, by the discussion of Section \ref{Section:VanishEta} and the proof of Lemma \ref{Lemma:VanishingEtaWitt}, that in the Witt case  there exists a self-adjoint involution $\mathscr{U}$ such that $\mathscr{U}A_F\mathscr{U}=-A_F$. In particular, $\mathscr{U}$ anti-commutes with the principal symbol of $\mathscr{A}_H$ and therefore
\begin{align*}
\mathscr{U}\mathscr{A}_{\mathcal{H}}(r)\mathscr{U}=-{\mathscr{A}}_\mathcal{H}(r)+\mathscr{C}_2(r),
\end{align*}
where  $\norm{\mathscr{C}_2(r)}_H\leq \mathscr{C}_2$ for $r\in(0,t]$.  Similarly,  we get an analogous formula for $A_\mathcal{H}$ since it has the same principal symbols as $\mathscr{A}_\mathcal{H}$, i.e 
\begin{align*}
\mathscr{U}{A}_{\mathcal{H}}(r)\mathscr{U}=-{A}_\mathcal{H}(r)+{C}_2(r),\quad\text{where  $\norm{{C}_2(r)}_H\leq C_2$ for $r\in(0,t]$}.
\end{align*}
Observe now that the operator $\mathscr{U}\mathscr{D}^+_t \mathscr{U}$ defined on the core
\begin{align*}
\mathcal{C}(\mathscr{U}\mathscr{D}^+_{t,\mathcal{H}}\mathscr{U})\coloneqq \{\sigma\in C^{1}_c((0,t],H^1))\:|\:Q_<(\mathscr{U} A_{\mathcal{H}} (t)\mathscr{U})\sigma(t)=0\}
\end{align*}
is given by 
\begin{align*}
\left(\frac{\partial }{\partial r}+\mathscr{U} \mathscr{A}_H(r)\mathscr{U} \right)\sigma(r)=-\left(-\frac{\partial }{\partial r}+\mathscr{A}_\mathcal{H}(r)-\mathscr{C}_2(r)\right)\sigma(r). 
\end{align*}
As we are in the Witt case we can use Lemma \ref{Lemma:VertOp1} and \cite[Theorem V.4.10]{KATO} to verify for $t>0$ sufficiently small the relation $Q_<(-A_{\mathcal{H}} (t)+C_2(t))=Q_<(-A_{\mathcal{H}})=Q_>(A_{\mathcal{H}})$ since $C_2(r)$ is uniformly bounded. We now compare $\mathscr{U} \mathscr{D}^+_t \mathscr{U} $ with the adjoint operator
\begin{align*}
(\mathscr{D}^+_{t,\mathcal{H}})^*\sigma(r)=\left(-\frac{\partial }{\partial r}+\mathscr{A}_H(r)\right)\sigma(r),
\end{align*}
defined on the core (Remark \ref{Rmk:AdjointBoundCond})
\begin{align*}
\mathcal{C}((\mathscr{D}^+_{t,\mathcal{H}})^*)\coloneqq \{\sigma\in C^{1}_c((0,t],H^1))\:|\:Q_\geq(A_{\mathcal{H}} (t))\sigma(t)=0\}.
\end{align*}
Since in the Witt case $A_\mathcal{H}$ is invertible then $Q_\geq(A_{\mathcal{H}} (t))=Q_>(A_{\mathcal{H}} (t))$ and therefore 
$\mathscr{U} \mathscr{D}^+_t \mathscr{U} =-(\mathscr{D}^+_{t,\mathcal{H}})^*$. Altogether, 
\begin{align*}
\ind(\mathscr{D}^+_{t,{\mathcal{H}} } )=\ind( \mathscr{U}\mathscr{D}^+_{t,{\mathcal{H}}} \mathscr{U})=(\mathscr{D}^+_{t,\mathcal{H}})^*=-\ind(\mathscr{D}^+_{t,{\mathcal{H}} } ), 
\end{align*}
which shows that $\ind(\mathscr{D}^+_{t,{\mathcal{H}} } )=0$. 
\end{proof}

\begin{proof}[Proof of Theorem \ref{Thm:5.2Pot}]
This follows now from the deformation argument described above, the vanishing results of Proposition \ref{Prop:DefIndex1}, Proposition \ref{Prop:DefIndex2} and Proposition \ref{Prop:DefIndex3}.
\end{proof}

Regarding the index contribution of $Z_t$ it is easy to see that, by deforming the metric close to $r=t$, we can adapt the proof Theorem \ref{Thm:5.3} to get the analogous formula,
\begin{align*}
\textnormal{ind}\left(\mathscr{D}^+_{Z_t,Q_{<}({A}(t))(H)}\right)&=\textnormal{ind}\left({D}^+_{
Z_t,Q_{<}(A_0(t))(H)}\right)+(Q_{<}({A}(r))(H),Q_{\geq }({A}_0(r))(H)), 
\end{align*}
where ${D}^+_{Z_t,Q_{<}(A_0(t))(H)}$ is the signature operator on the manifold with boundary $Z_t$. Using  \eqref{IndS1SignKern} we see for $t>0$ small enough 
\begin{align*}
\textnormal{ind}\left(\mathscr{D}^+_{Z_t,Q_{<}({A}(t))(H)}\right)=\sigma_{S^1}(M)-\frac{1}{2}\dim\ker(A_0(t))+(Q_{<}({A}(r))(H),Q_{\geq }({A}_0(r))(H)).
\end{align*}
Finally, from this formula and from the vanishing Theorem \ref{Thm:5.2Pot} we obtain a partial answer for the index of the operator $\mathscr{D}^+$ (compare with Theorem \ref{Thm:IndD+}). 

\begin{theorem}
Let $M$ be a closed, oriented Riemannian $4k+1$ dimensional manifold on which $S^1$ acts effectively and semi-freely by orientation preserving isometries. If the codimension of the fixed point set $M^{S^1}$ in $M$ is divisible by four, then $M/S^1$ is a Witt space  and the index of the graded Dirac-Schr\"odinger operator is 
\begin{align*}
\ind(\mathscr{D}^+)=\sigma_{S^1}(M)=\int_{M_0/S^1}L\left(T(M_0/S^1),g^{T(M_0/S^1)}\right).
\end{align*}
\end{theorem}
In this last formula we have again used that $\eta(M^{S^1})=0$ in the Witt case

\subsection{Perspectives}
This work presents a specific model for a formalism that can be applied to more general stratified spaces. We want to mention some paths where this point of view could lead future research:
\begin{enumerate}
\item Although the operator $\mathscr{D}$ is essentially self-adjoint in both the Witt and the non-Witt case, he have not been able to prove that its index (which is well defined!) computes $\sigma_{S^1}(M)$. Nevertheless, we are optimistic about this conjecture in view of Theorem \ref{THM2}. This index computation has been elusive because of two reasons:
\begin{enumerate}
\item The techniques used to prove the vanishing theorem of the index contribution of $U_t$ use the Witt condition. However, it seems that this is just a feature of the proof and it does not seem to be a fundamental obstruction. Choosing some appropriate gluing elliptic boundary condition might lead to such a vanishing result. 
\item When computing the adiabatic limit for the index contribution of $Z_t$ we needed to modify the boundary condition in order to obtain the right tangential operator, an essential ingredient of the signature Thorem \ref{Thm:SignThmMBound}. As seen before, by Theorem \ref{BBCThm4.14}, this change of boundary condition is compensated by a Kato index between the respective projections. The result of Cheeger and Dai in \cite{CD09}, used by  Br\"uning to compute this Kato index (Theorem \ref{Thm:5.4}), is only valid for the Witt case. For the non-Witt case we hope to proceed in a similar way in view of the generalization result \cite[Theorem 1.7]{MH05}. 
\end{enumerate}
\item Following \cite[Section 4.3]{L00}, it does not  seem hard to generalize the constructions of Chapter \ref{Sec:Induced} in the context of semi-free actions for general compact Lie groups.
\item In  \cite[Section 2.4]{L00} Lott defined, inspired in the index theorem \cite[Theorem 4.2]{APSI} for the spin-Dirac operator, the $\widehat{A}$-genus for $M/S^1$ which he proves it is always an integer. In the work \cite{AG16} of Albin and Gell-Redman they study the index formula of the spin-Dirac operator on a compact stratified space with one singular stratum. In their formalism, they impose a geometric Witt condition which requires the spectrum of the cone coefficient not to intersect the open interval  $(-1/2,1/2)$. In the Witt case their index formula computes Lott's $\widehat{A}$-genus. Nevertheless, as for the signature operator, the $\widehat{A}$-genus still makes sense in the non-Witt case and a natural question is  whether this integer comes again as an index of certain Dirac-type operator. Following the ideas developed in  Chapter \ref{Sec:Induced} one might try to push-down the spin-Dirac operator to $M/S^1$. The main problem might be that the obtained potential might not necessarily commute with the chirality involution of the spinor bundle. If this is the case, then one should explore instead pushing down an appropriate transversally elliptic operator to obtain the desired compatible potential. Of course, after achieving this, one should try to compute the index. 
\item Finally, we hope to be able to implement the addition of  these kind of potentials in order to study further topological invariants in  more general stratified spaces. In particular, it would be interesting to see which type of index formulas can be obtained when adding this concrete potential ``by hand" to the signature operator on more general stratified spaces which do not necessarily satisfy the Witt-condition. Perhaps adding this potential, or possibly variants of it, might bring a more geometric interpretation to the index formulas  obtained from imposing certain ideal boundary conditions.
 \end{enumerate}

\section{Regular singular operators}\label{App:RSO}

In this appendix we want to collect some important results on {\em first order regular singular operators}, introduced in the seminal work of Br\"uning and Seeley \cite{BS88}. \\

Let us consider the following setting:
\begin{enumerate}
\item  $H$ is a fixed Hilbert space. 
\item $\dom(S) \subseteq H $ is the common domain of the family of self-adjoint operators 
\begin{align*}
S(r)\coloneqq S_0+r^{\beta+1}S_1(r), \quad\textnormal{for $r\in(0,\infty)$ and $\beta>-\frac{1}{2}$}. 
\end{align*}
\item $S_0$ is a discrete operator on $\dom(S)$. We denote its associated orthonormal basis of eigenvectors by $\{e_\lambda\}_{\lambda\in\spec(S_0)}$.
\item There exists a constant $C_0>0$ such that 
\begin{align*}
\norm{S_1(r)(|S_0|+1)^{-1}}_H + \norm{(|S_0|+1)^{-1}S_1(r)}_H\leq C_0,
\end{align*}
uniformly in $r>0$. 
\end{enumerate}
In this context we define the operator
\begin{align*}
T\coloneqq \frac{\partial}{\partial r}+\frac{1}{r}S(r). 
\end{align*}
acting on $L^2([0,\infty),H)$ with domain $C^\infty_c((0,\infty), \dom(S))$. This operator has a formal adjoint, defined on the same domain, given by 
\begin{align*}
T^{\dagger}\coloneqq -\frac{\partial}{\partial r}+\frac{1}{r}S(r). 
\end{align*}

We can define two natural closed extensions of $T$:
\begin{itemize}
\item The {\em minimal extension} $T_\textnormal{min}\coloneqq \bar{T}$ is the closure of $T$, i.e. $x\in \dom(T_\textnormal{min})$ if and only if there exists a sequence $(x_n)_n\subset C^\infty_c((0,\infty), \dom(S))$ such that $x_n\longrightarrow x$ and $(Tx_n)_n\longrightarrow y$ for some $y\in L^2([0,\infty),H)$. In this case we define $T_\textnormal{min} x\coloneqq  y$. 
\item The {\em maximal extension} $T_\textnormal{max}\coloneqq (T^\dagger)^*$ is defined to be the adjoint operator of $T^\dagger$. 
\end{itemize}

\begin{theorem}[{\cite[Theorem 3.1]{BS88}}]\label{BS88Thm3.1}
$T_\textnormal{max}$ and $T_\textnormal{min}$ are Fredholm operators. The extensions of $T_\textnormal{min}$ are all Fredholm operators, and correspond to subspaces of the finite-dimensional space $\dom(T_\textnormal{max})/\dom (T_\textnormal{min})$.
\end{theorem}

\begin{lemma}[{\cite[Lemma 3.2]{BS88}}]\label{BS88Lemma3.2}
For $\lambda\in\spec(S_0)$ with $|\lambda|<1/2$, there are continuous linear functionals $c_\lambda$ defined on $\dom(T_\textnormal{max})$ such that for $r\in (0,1)$ and $0<\epsilon<1$, 
\begin{align*}
\bnorm{x(r)-\sum_{\substack{\lambda\in\spec(S_0),\\|\lambda|<1/2}}c_\lambda(x)r^{-\lambda} e_\lambda}_H\leq \epsilon r^{1/2}|\log r|^{1/2}+C_{\epsilon,x}r^{1/2}, \:\: \text{for $x\in\dom(T_\textnormal{max})$}.
\end{align*}
\end{lemma}

\begin{theorem}[{\cite[Theorem 3.2]{BS88}}]\label{BS88Thm3.2}
The closed extensions of $T_\textnormal{min}$ are classified by the subspaces $W$ of 
\begin{align*}
\bigoplus_{\substack{\lambda\in\spec(S_0),\\|\lambda|<1/2}}\ker(S_0-\lambda).
\end{align*}
For such subspace $W$, the corresponding operator $T_W$ has Fredholm index 
\begin{align*}
\ind(T_W)=\ind(T_\textnormal{min})+\dim W. 
\end{align*}
\end{theorem}

\section{A sequence of cut-off functions}\label{App:Seq}

In this appendix we are going to describe a family of cut-off functions introduced in \cite[Section 6]{BS87}. To begin choose two functions $\phi,\chi\in C^\infty(\mathbb{R})$ such that 
\begin{align*}
\begin{aligned}[c]
&0\leq\varphi \leq 1,\\
&\varphi(r)=1, \:\text{if $|r|\leq 1$},\\
&\varphi(r)=0, \:\text{if $|r|\geq 2$},
\end{aligned}
\qquad\text{and}\qquad
\begin{aligned}[c]
&0<\chi(r)\leq r,\:\text{if $r>0$},\\
& \chi(r)=0,\: r<0,\\
&\chi(r)=r, \:\text{if $0\leq r\leq 1$},\\
&\chi(r)=1, \:\text{if $r\geq 2$}.
\end{aligned}
\end{align*}

For $n\in\mathbb{N}$ and $n\geq 2$, put $\vartheta_n\coloneqq (\log n)^{-1/2}$. It is easy to verify $\vartheta_n\longrightarrow 0$ as $n\longrightarrow\infty$. Define a sequence of cut-off fun functions by
\begin{align*}
\psi_n(r)&\coloneqq \chi(r)^{\vartheta_n}(1-\varphi(nr)),\\
\psi_{nm}(r)&\coloneqq \psi_n(r)-\psi_m(r).
\end{align*} 

\begin{remark}
Observe that $\psi_n(r)=0$ whenever $|r|\leq 1/n$.  
\end{remark}

\begin{lemma}\label{Lemma:PropPsin}
The sequence $(\psi_n)_n$ satisfies,
\begin{enumerate}
\item It is uniformly bounded.
\item For each $r\in\mathbb{R}$,  $\psi_n(r)\longrightarrow 1$ and $\psi_{nm}(r)\longrightarrow 0$ as $n,m\longrightarrow \infty$. 
\item For each $\sigma\in L^2([0,2], H)$, where $H$ is a Hilbert space, $\psi_n \sigma\longrightarrow v$ in $L^2$.
\end{enumerate}
\end{lemma}

\begin{proof}
\begin{enumerate}
\item By definition of $\chi,\varphi$ and $\vartheta_n$ we estimate
\begin{align*}
|\psi_n(x)|\leq |\chi(r)|^{\vartheta_n}|(1-\varphi(nr))|\leq 2^{\vartheta_n}<2.
\end{align*}
\item Let us fix $r\in\mathbb{R}$ and choose $n\in\mathbb{N}$ such that  $nr>2$, then $(1-\varphi(nr))=1$ and therefore $\psi_n(r)=\chi(r)^{\vartheta_n}$. If $r\geq2$ or $r\leq 0$ then we are done by the definition of $\chi$. Assume now $r\in(0,2)$. In this case $0<\chi(r)<2$, and since $\vartheta_n\longrightarrow 0$ as $n\longrightarrow 0$ then the claim follows.
\item This follows from (1), (2) and from Lebesgue's dominated convergence theorem.
\end{enumerate}
\end{proof}
Now we study some properties of the sequence $(\psi'_n)_n$. First we compute
\begin{align*}
\psi'_n(r)=\vartheta_n\chi(r)^{\vartheta_n-1}\chi'(r)(1-\varphi(nr))-n\chi(r)^{\vartheta_n}\varphi'(nr). 
\end{align*}
Hence,
\begin{align}\label{Eqn:DPsi_n}
\psi'_n(r)^2=&\vartheta_n^2\chi(r)^{2\vartheta_n-2}\chi'(r)^2(1-\varphi(nr))^2+n^2\chi(r)^{2\vartheta_n}\varphi'(nr)^2\\
&-2\vartheta_n n\chi(r)^{2\vartheta_n-1}\chi'(r)(1-\varphi(nr))\varphi'(nr)\notag \\
\leq &C(\vartheta_n^2\chi(r)^{2\vartheta_n-2}\chi'(r)^2(1-\varphi(nr))^2+n^2\chi(r)^{2\vartheta_n}\varphi'(nr)^2),\notag 
\end{align}
by Young's inequality. 

\begin{lemma}\label{LemmaB2}
The sequence $\psi'_n(r)^2\longrightarrow 0$ uniformly in $r\geq 1$. 
\end{lemma}

\begin{proof}
For $r\geq 1$ we study the two terms separately of the estimate \eqref{Eqn:DPsi_n}. On the one hand, since $nr\geq 2$, then $\varphi(nr)=1$, thus 
\begin{align*}
|\vartheta_n^2\chi(r)^{2\vartheta_n-2}\chi'(r)^2(1-\varphi(nr))^2|\leq \vartheta_n^2 |\chi(r)^{2\vartheta_n-2}\chi'(r)^2|. 
\end{align*} 
Note that $|\chi'(r)|\leq C_1$ for some $C_1>0$ by construction, hence 
\begin{align*}
|\vartheta_n^2\chi(r)^{2\vartheta_n-2}\chi'(r)^2(1-\varphi(nr))^2|\leq  2\vartheta_n^2 C_1\longrightarrow 0. 
\end{align*} 
On the other hand, the term $|n^2\chi(r)^{2\vartheta_n}\varphi'(nr)^2)|$ is automatically zero since $\varphi'(nr)=0$ as $nr\geq2 $. 
\end{proof}

\begin{proposition}\label{Prop:AppSeqMain}
If $\sigma\in L^2([0,2],H)$ and $\norm{\sigma(r)}_H=O(r^{1/2})$ as $r\longrightarrow 0$, then 
\begin{align*}
\int_0^2 \psi'_{nm}(r)^2\norm{\sigma(r)}^2_H dr\longrightarrow 0, 
\end{align*}
as $n,m\longrightarrow \infty$. 
\end{proposition}

\begin{proof}
We split the integral into two parts. By Lemma \ref{LemmaB2} we see
\begin{align*}
\int_1^2 \psi'_{nm}(r)^2\norm{\sigma(r)}^2_H dr\longrightarrow 0, 
\end{align*}
Now we treat the contributions of the two terms of the estimate \eqref{Eqn:DPsi_n} between $0$ and $1$. The first can be estimated by the integral 
\begin{align*}
C\vartheta^2_n\int_0^1 r^{2\vartheta_n-1}dr=\frac{C}{2}\vartheta_n r^{2\vartheta_n}\bigg{|}^1_0=\frac{C}{2}\vartheta_n\longrightarrow 0, \quad n\longrightarrow \infty. 
\end{align*}
The second term, since $\varphi'(nr)=0$ if $nr\geq 2$,  can be estimated by 
\begin{align*}
Cn^2\int_0^{2/n}r^{2\vartheta_n+1}dr=Cn^2\frac{1}{2\vartheta_n+2} r^{2\vartheta_n+2}\bigg{|}^1_0 \leq  Cn ^{-2\vartheta_n}. 
\end{align*}
Finally note 
\begin{align*}
n ^{-2\vartheta_n} = \left(e^{\log n}\right)^{-2\vartheta_n}=e^{-2\vartheta_n\log n}=e^{-2(\log n)^{1/2}} \longrightarrow 0, \quad n\longrightarrow \infty. 
\end{align*}
\end{proof}

\begin{remark}
Similar results hold also for the sequence
\begin{align*}
\tilde{\psi}_n(r)\coloneqq &\varphi\left(\frac{r}{n}\right)\psi_n(r), \\
\tilde{\psi}_{nm}\coloneqq &\tilde{\psi}_n-\tilde{\psi}_m. 
\end{align*}
Observe that its derivative is given by
\begin{align*}
\tilde{\psi'}_n(r)=&\frac{1}{n}\varphi\left(\frac{r}{n}\right)\psi_n(r)+\varphi\left(\frac{r}{n}\right)\psi'_n(r).
\end{align*}
As before, we can estimate its square as the sum of tho contributions
\begin{align*}
\tilde{\psi'_n}(r)^2=&\frac{1}{n^2}\varphi\left(\frac{r}{n}\right)^2\psi_n(r)^2+\varphi\left(\frac{r}{n}\right)^2\psi'_n(r)^2\\
&+\frac{2}{n}\varphi\left(\frac{r}{n}\right)\varphi\left(\frac{r}{n}\right)'\psi_n(r)\psi'_n(r)\\
\leq & \tilde{C}\left(\frac{1}{n^2}\varphi\left(\frac{r}{n}\right)^2\psi_n(r)^2+\varphi\left(\frac{r}{n}\right)\psi'_n(r)\right).
\end{align*}
\end{remark}

Now we are going to describe an extension of the sequence $(\psi_n)_n$ to higher dimensions described in \cite[Lemma 5.1]{BS91}. For $(s,r)\in \mathbb{R}^h\times\mathbb{R}$ set
\begin{align}
\widetilde{\psi}_n(r,s)\coloneqq &\varphi\left(\frac{\norm{s}}{n}\right)\varphi\left(\frac{r}{n}\right)\psi_n(r), \label{Eqn:TildePsin}\\
\widetilde{\psi}_{nm}\coloneqq &\widetilde{\psi}_n-\widetilde{\psi}_m. \label{Eqn:TildePsinm}
\end{align}
From the properties of $\varphi$ and $\psi_n$ it follows that $(\widetilde{\psi}_n)_n$ is uniformly bounded and converges pointwise to zero. Moreover, 
\begin{align*}
\frac{\partial \widetilde{\psi}_n}{\partial r}(r,s)=&\frac{1}{n}\varphi\left(\frac{\norm{s}}{n}\right)\varphi '\left(\frac{r}{n}\right)\psi_n(r)+\varphi\left(\frac{\norm{s}}{n}\right)\varphi\left(\frac{r}{n}\right)\psi'_n(r),\\
\frac{\partial \widetilde{\psi}_n}{\partial \norm{s}}(r,s)=&\frac{1}{n}\varphi'\left(\frac{\norm{s}}{n}\right)\varphi\left(\frac{r}{n}\right)\psi_n(r),
\end{align*}
so we can estimate as before
\begin{align}\label{Eqn:EstimDTildePsi}
\left(\frac{\partial \widetilde{\psi}_n}{\partial r}(r,s) + \frac{\partial \widetilde{\psi}_n}{\partial \norm{s}}(r,s)\right)^2
&\leq \widetilde{C}\left(\varphi\left(\frac{\norm{s}}{n}\right)^2\varphi\left(\frac{r}{n}\right)^2\psi'_n(r)^2 \right.\\
&\left. +\frac{1}{n^2}\left(\varphi'\left(\frac{\norm{s}}{n}\right)^2\varphi \left(\frac{r}{n}\right)^2+\varphi\left(\frac{\norm{s}}{n}\right)^2\varphi '\left(\frac{r}{n}\right)^2\right)\psi_n(r)^2\right).\notag 
\end{align}

\pagenumbering{gobble}
\bibliographystyle{acm}
\bibliography{references} 
\end{document}